\documentclass[12 pt]{amsart}

\usepackage[page]{appendix}
\usepackage{etex}
\usepackage{enumitem}
\usepackage{amsmath}
\usepackage{amsxtra}
\usepackage{amscd}
\usepackage{amsthm}
\usepackage{amsfonts}

\usepackage{xcolor}
\def\darrow{\mathrel{\ThisStyle{\ooalign{$\SavedStyle\rightarrow$\cr%
  \hfil\textcolor{white}{\rule{2\LMpt}{1\LMex}}\kern2\LMpt\hfil}}}}

\usepackage{amssymb}
\usepackage{eucal}
\usepackage[all]{xy}
\usepackage{graphicx,setspace}
\usepackage{tikz}
\usetikzlibrary{matrix,arrows}
\usepackage{tikz-cd}
\usepackage{mathrsfs}
\usepackage{mathpazo}
\usepackage{euler}
\usepackage[colorinlistoftodos, textsize=tiny]{todonotes}
\usepackage{morefloats}
\usepackage{mathdots, mathtools}
\usepackage{pdfpages}
\usepackage{thm-restate}
\usepackage[utf8]{inputenc}
\usepackage{epigraph,hyperref}
\usepackage{csquotes}
\usepackage{url}
\usepackage{comment}
\usepackage[margin=1in]{geometry}
\usepackage[labelfont=bf]{caption}
\newcommand*{\DashedArrow}[1][]{\mathbin{\tikz [baseline=-0.25ex,-latex, dashed,#1] \draw [#1] (0pt,0.5ex) -- (1.3em,0.5ex);}}%

\RequirePackage{color}
\definecolor{myred}{rgb}{0.75,0,0}
\definecolor{mygreen}{rgb}{0,0.5,0}
\definecolor{myblue}{rgb}{0,0,0.65}

\usepackage{listings}

\definecolor{dkgreen}{rgb}{0,0.6,0}
\definecolor{gray}{rgb}{0.5,0.5,0.5}
\definecolor{mauve}{rgb}{0.58,0,0.82}

\usepackage{placeins}

\usepackage{xcolor}
\hypersetup{
    colorlinks,
    linkcolor={red!50!black},
    citecolor={blue!50!black},
    urlcolor={blue!80!black}
}

\usepackage{tikz}
\usetikzlibrary{matrix,arrows,decorations.pathmorphing}

\newtheorem{theorem}{Theorem}
\numberwithin{theorem}{subsection}
\newtheorem*{theorem*}{Theorem}
\newtheorem{proposition}[theorem]{Proposition}
\newtheorem{lemma}[theorem]{Lemma}
\newtheorem{corollary}[theorem]{Corollary}

\newtheorem{question}[theorem]{Question}

\theoremstyle{definition}
\newtheorem{defn}[theorem]{Definition}

\newtheorem{example}[theorem]{Example}

\theoremstyle{remark}
\newtheorem{remark}[theorem]{Remark}

\newcommand\nc{\newcommand}
\nc{\on}{\operatorname}
\nc\renc{\renewcommand}
\nc{\BR}{\mathbb R}
\nc{\BC}{\mathbb C}
\nc{\BQ}{\mathbb Q}
\nc{\BF}{\mathbb F}
\nc{\BZ}{\mathbb Z}
\nc{\BN}{\mathbb N}
\nc{\BS}{\mathbb S}
\nc{\BA}{\mathbb A}
\nc{\BP}{\mathbb P}
\nc{\BT}{\mathbb T}
\nc{\BG}{\mathbb G}
\nc{\Hom}{\on{Hom}}
\nc{\wt}{\widetilde}
\nc{\vspan}{\on{span}}
\nc{\ord}{\on{ord}}
\nc{\im}{\on{im}}
\nc{\Mat}{\on{Mat}}
\nc{\can}{\on{can}}
\nc{\coker}{\on{coker}}
\nc{\ev}{\on{ev}}
\nc{\Tr}{\on{Tr}}
\nc{\End}{\on{End}}
\nc{\swap}{\on{swap}}
\nc{\Set}{\on{Set}}
\nc{\bC}{{\mathbf C}}
\nc{\bc}{{\mathbf c}}
\nc{\bD}{{\mathbf D}}
\nc{\bd}{{\mathbf d}}
\nc{\bE}{{\mathbf E}}
\nc{\be}{{\mathbf e}}
\nc{\bF}{{\mathbf F}}
\nc{\bff}{{\mathbf f}}
\nc{\fa}{\mathfrak a}
\nc{\Def}{\on{Def}}
\renc{\mod}{\on{-mod}} 
\newcommand{\spec}{\text{spec}}
\nc{\adj}{\on{adj}}
\nc{\tensor}[3]{#1 \underset{#2}\otimes #3}
\nc{\Nat}{\on{Nat}}
\nc{\op}{\on{op}}
\nc{\Funct}{\on{Funct}}
\nc{\Ob}{\on{Ob}}
\nc{\fR}{\mathfrak{R}}
\nc{\Vect}{\on{Vect}}
\nc{\ns}{\on{non-spec}}
\nc{\GL}{\on{GL}}
\nc{\ol}{\overline}
\nc{\ul}{\underline}
\nc{\univ}{\on{univ}}
\nc{\Maps}{\on{Maps}}
\nc{\bdd}{\on{bdd}}
\nc{\cont}{\on{cont}}
\nc{\Sym}{\on{Sym}}
\nc{\Ind}{\on{Ind}}
\nc{\Res}{\on{Res}}
\nc{\Ann}{\on{Ann}}
\nc{\cI}{\mathcal{I}}
\nc{\pt}{\on{pt}}
\nc{\Bl}{\on{\Bl}}
\nc{\Spec}{\on{Spec}}
\nc{\Spf}{\on{Spf}}
\nc{\Cl}{\on{Cl}}
\nc{\litebox}{\boxed}
\nc{\cD}{\mathcal{D}}
\renc{\div}{\on{div}}
\nc{\codim}{\on{codim}}
\renc{\d}{\partial}
\nc{\Gr}{\on{Gr}}
\nc{\mc}{\mathcal}
\nc{\rk}{\on{rk}}
\nc{\OO}{\mathcal{O}}
\nc{\pp}{\mathfrak{p}}
\nc{\qq}{\mathfrak{q}}
\nc{\mm}{\mathfrak{m}}
\nc{\Gal}{\on{Gal}}
\nc{\Frob}{\on{Frob}}
\nc{\SL}{\on{SL}}
\nc{\SO}{\on{SO}}
\nc{\SU}{\on{SU}}
\nc{\supp}{\on{supp}}
\nc{\gug}{\mathfrak{g}}
\nc{\hug}{\mathfrak{h}}
\nc{\Lie}{\on{Lie}}
\nc{\sla}{\mathfrak{sl}}
\nc{\gla}{\mathfrak{gl}}
\nc{\mf}{\mathfrak}
\nc{\Ad}{\on{Ad}}
\nc{\ad}{\on{ad}}
\nc{\Aut}{\on{Aut}}
\renc{\scr}{\mathscr}
\newcommand*{\sheafext}{\mathscr{E}\kern -.5pt xt}

\makeatletter
\usepackage{hyperref}
\newcommand{\customlabel}[2]{%
	\protected@write \@auxout {}{\string \newlabel {#1}{{#2}{\thepage}{#2}{#1}{}} }%
	\hypertarget{#1}{#2}
}

\newcommand{\mybinom}[3][0.8]{\scalebox{#1}{$\dbinom{#2}{#3}$}}

\allowdisplaybreaks

\let\oldtocsection=\tocsection

\let\oldtocsubsection=\tocsubsection

\let\oldtocsubsubsection=\tocsubsubsection

\renewcommand{\tocsection}[2]{\hspace{0em}\oldtocsection{#1}{#2}}
\renewcommand{\tocsubsection}[2]{\hspace{3em}\oldtocsubsection{#1}{#2}}
\renewcommand{\tocsubsubsection}[2]{\hspace{6em}\oldtocsubsubsection{#1}{#2}}

\title[Inflectionary Invariants For Curve Singularities]{Inflectionary Invariants for Isolated Complete \\ Intersection Curve Singularities}

\date{\today}

\author[Anand P.~Patel]{Anand P.~Patel}

\author[Ashvin A. Swaminathan]{Ashvin A. Swaminathan}

\AtEndDocument{\bigskip{\footnotesize%
  \textsc{Department of Mathematics, Oklahoma State University, \mbox{Stillwater, OK 74075}} \par
  \textit{E-mail address}, Anand P.~Patel: \texttt{anand.patel@okstate.edu} \par
  \addvspace{\medskipamount}
  \textsc{Department of Mathematics, Princeton University, \mbox{Princeton, NJ 08544}} \par
  \textit{E-mail address}, Ashvin A. Swaminathan: \texttt{ashvins@math.princeton.edu}
}}


\begin{document}

\vspace*{-0.75cm}

\begin{abstract}
We investigate the role played by curve singularity germs in the enumeration of inflection points in families of curves acquiring singular members. Let $N \geq 2$, and consider an isolated complete intersection curve singularity germ $f \colon (\BC^N,0) \to (\BC^{N-1},0)$. We define a numerical function $m \mapsto \on{AD}_{(2)}^m(f)$ that naturally arises when counting $m^{\mathrm{th}}$-order weight-$2$ inflection points with ramification sequence $(0, \dots, 0, 2)$ in a $1$-parameter family of curves acquiring the singularity $f = 0$, and we compute $\on{AD}_{(2)}^m(f)$ for several interesting families of pairs $(f,m)$. In particular, for a node defined by $f \colon (x,y) \mapsto xy$, we prove that $\on{AD}_{(2)}^m(xy) = {{m+1} \choose 4},$ and we deduce as a corollary that
    $\on{AD}_{(2)}^m(f) \geq (\on{mult}_0 \Delta_f) \cdot {{m+1} \choose 4}$ for any $f$, where $\on{mult}_0 \Delta_f$ is the multiplicity of the discriminant $\Delta_f$ at the origin in the deformation space. Significantly, we prove that the function $m \mapsto \on{AD}_{(2)}^m(f) -(\on{mult}_0 \Delta_f) \cdot {{m+1} \choose 4}$ is an analytic invariant measuring how much the singularity ``counts as'' an inflection point. We prove similar results for weight-$2$ inflection points with ramification sequence $(0, \dots, 0, 1,1)$ and for weight-$1$ inflection points, and we apply our results to solve a number of related enumerative problems.
\end{abstract}

\maketitle


\tableofcontents

\section{Introduction}

The main objective of this paper is to make precise sense of and answer the following natural question that lies at the crossroads of singularity theory and \mbox{enumerative geometry:}
\begin{equation*}
\text{\emph{To what extent does a curve singularity germ ``count as'' an inflection point of a given type?}}
\end{equation*}
Our approach to answering the above question is to introduce a (countable) collection of new analytic singularity invariants, which we call \emph{automatic degeneracies}, that encode the extent to which curve singularity germs ``count as'' various interesting types of inflection points. We restrict our consideration to the case of isolated complete intersection curve singularity (ICIS) germs, as the theory is more tractable in this setting.

\subsection{Motivations}

The basic setup underlying the definition of automatic degeneracies is familiar to experts in two areas that have been subjects of significant research for at least the past three decades. The first area
concerns the problem of
finding a locally free replacement for the sheaf $\scr{P}_{X/B}^m(\scr{V})$ of $m^{\mathrm{th}}$\emph{-order (relative) principal
parts} associated to a family $X/B$ of
curves and a vector bundle $\scr{V}$ on the total space $X$. As it happens, the sheaf $\scr{P}_{X/B}^m(\scr{V})$ inconveniently fails to be locally free at the singular locus $\on{Sing}(X/B)$ of the family, so it is natural to seek to replace the sheaf $\scr{P}_{X/B}^m(\scr{V})$ with a
locally free sheaf $\scr{Q}_{X/B}^{m}(\scr{V})$ that agrees with $\scr{P}_{X/B}^m(\scr{V})$ on the complement of $\on{Sing}(X/B)$. Satisfactory (and sophisticated) replacements can be found in the important work of D.~Laksov and A.~Thorup~\cite{MR1318539,MR1786498,MR2007394}, E.~Esteves~\cite{MR1368706}, and L.~Gatto~\cite{MR1326732}; see \S~\ref{sec-strats} for a brief summary of these results. Our observation is that for many families $X/B$---especially those that are of interest to us in this paper---the double-dual sheaf $\scr{P}_{X/B}^m(\scr{V})^{\vee\vee}$, which we term the sheaf of $m^{\mathrm{th}}$\emph{-order invincible parts} associated to the pair $(X/B, \scr{V})$, serves
as a locally free replacement of the sheaf $\scr{P}_{X/B}^m(\scr{V})$. Owing to their versatility and relatively simple definition, one can use the sheaves of invincible parts to perform otherwise difficult calculations with relative ease, as we demonstrate in \S~\ref{sec-calc}--\ref{sec-autodegegs}.

The sheaves of principal parts constitute an effective tool for solving enumerative problems about inflection points (see~\cite[\S~7.5]{harris3264}). The standard approach to counting inflection points on smooth curves is to express the desired number in terms of the Chern classes of the sheaves of principal parts using the \emph{Porteous formula} (see~\cite[\S~12]{harris3264}), but this approach fails in the presence of singularities precisely because the sheaves of principal parts fail to be locally free. As we discuss in \S~\ref{sec-ptlininvar}--\ref{sec-ranwithit}, a few \emph{ad hoc} workarounds for this issue exist in the literature. For instance, one can count the number of \emph{hyperflexes} in a general pencil of plane curves of a given degree by setting up a Chern class calculation over the universal point-line incidence variety (see~\cite[\S~11.3.1]{harris3264}), but it is not known whether this method can be generalized to solve other enumerative problems. Another approach is developed in Z.~Ran's remarkable series of papers~\cite{MR2172162,MR2157135,MR3078931}, in which Ran answers enumerative questions about families of curves acquiring nodes by (essentially) replacing the families with iterated blowups. However, Ran's results have not been replicated for higher-order singularities, because his methods depend on specific properties of certain Hilbert schemes associated to nodal curves.
In this article, we demonstrate that by replacing the sheaves of principal parts rather than the families themselves, one obtains a more direct and broadly applicable strategy for studying inflectionary behavior in families of curves acquiring singular members.

The second area alluded to above concerns the problem of extracting a number from a module of finite rank over a Noetherian local ring that has finite colength in a free module. This number is called the \emph{Buchsbaum-Rim
multiplicity} of the module and is well-studied not only as a notion of interest in its own right in commutative algebra (e.g., see~\cite{kodiyalamsolo,brscant,BRFitt0,kodiyalammohan}), but also as a key concept in the work of T.~Gaffney, S.~Kleiman, and D.~Massey on equisingularity theory. In their series of papers~\cite{gkm2,gkm1,gkm4,gkm3}, Gaffney \emph{et al.} define a \emph{single} Buchsbaum-Rim multiplicity for each isolated complete intersection
singularity of any dimension, not just for curve singularities. As we explain in \S~\ref{sec-booksrim}, some of our automatic degeneracies
can be regarded as Buchsbaum-Rim multiplicities of certain modules associated to the sheaves of principal parts. Nonetheless, it is essential to observe that automatic degeneracies are quite different from the Buchsbaum-Rim multiplicity that arises in equisingularity theory. Indeed, apart from certain technical differences that we explain in Remark~\ref{rem-booksrim}, simply note that there is a \emph{countable collection} of automatic degeneracies---each of which corresponds to a different type of inflection point---associated to a given ICIS.

\subsection{Overview of Main Results} \label{sec-overthetop}

We briefly describe in rough terms what it means for an ICIS to ``count as'' a certain number of inflection points, to be made precise in \S~\ref{sec-backtobasics} and \S~\ref{sec-limdefs}. We work over an algebraically closed field $k$ of characteristic zero. Let $f$ be the germ of an ICIS $X_0$, and let $\pi \colon X \to B$ be a general $2$-parameter deformation of the ICIS cut out by $f$. Here, the map $\pi$ sends the singular point $0 \in X_0$ cut out by $f$ to the point $0 \in B$, and the fiber $\pi^{-1}(0)$ equal to $X_0$. We denote by $\Delta_f \subset B$ the \emph{discriminant} divisor parametrizing singular fibers of the map $\pi$.


Now, let $m \geq 2$ be an integer, let $\scr{P}_{X/B}^m$ denote the completion of the stalk of the sheaf of principal parts at $0 \in X$, and choose $m-1$ general elements of $\scr{P}_{X/B}^m$. Let $H_{(2)}^m \subset X$ denote the degeneracy locus of these elements on the smooth locus of the map $\pi$. The scheme $H_{(2)}^m$ has expected codimension $2$ in $X$ and can be thought of as the locus of $m^{\mathrm{th}}$\emph{-order weight-$2$ type-(a) inflection points} of a general linear system near the singular point $0 \in X$. We prove the following result in \S~\ref{sec-limdefs}:
\begin{theorem*}
The intersection multiplicity between the closure of $H_{(2)}^m$ and the total space of a general $1$-parameter subdeformation of $X/B$ is given by the following formula:
\begin{equation} \label{Eq:FundamentalFormula}
      \on{AD}_{(2)}^m(f) - (\on{mult}_0 \Delta_f) \cdot \on{AD}_{(2)}^m(xy),
\end{equation}
where $\on{AD}_{(2)}^m(f)$ denotes the $m^{\mathrm{th}}$\emph{-order weight-$2$ type-(a) automatic degeneracy} of the ICIS germ $f$ (see Definition~\ref{def-auto}), $\on{AD}_{(2)}^m(xy)$ is the value of this automatic degeneracy for the germ $(x,y) \mapsto xy$ of an ordinary double point (henceforth, \emph{node}), and $(\on{mult}_0 \Delta_f)$ denotes the multiplicity of $\Delta_f$ at $0 \in B$.
\end{theorem*}

The multiplicity in~\eqref{Eq:FundamentalFormula} is equal to \emph{the number of $m^{\mathrm{th}}$-order weight-$2$ type-(a) inflection points limiting to the singularity in a general $2$-parameter deformation}, and it
measures the extent to which the singularity ``counts as'' an $m^{\mathrm{th}}$-order weight-$2$ type-(a) inflection point.

For a given linear system on a family of curves, there are two types of weight-$2$ inflection points, which we call \emph{type (a)} and \emph{type (b)}, but only one type of weight-$1$ inflection point (see \S~\ref{sec-degens} for definitions). In addition to the aforementioned automatic degeneracy $\on{AD}_{(2)}^m(f)$, we also define automatic degeneracies denoted $\on{AD}_{(1,1)}^m(f)$ and $\on{AD}_{(1)}^m(f)$ for the weight-$2$ type-(b) and weight-$1$ cases, respectively (see Definition~\ref{def-auto}). Each $m^{\mathrm{th}}$-order automatic degeneracy of an ICIS can be thought of as a family of analytic invariants of the ICIS parametrized by the order $m$. Although the weight of an inflection point can be any positive integer, we only consider inflection points of weight at most $2$; it remains open as to whether our results can be extended to inflection points of higher weights.

Just like in the weight-$2$ type-(a) case, $\on{AD}_{(1,1)}^m(f)$ and $\on{AD}_{(1)}^m(f)$ encode the extent to which the ICIS cut out analytically-locally by $f = 0$ ``counts as'' an inflection point of the corresponding type. For the weight-$2$ type-(b) case, choose $m+1$ general elements of $\scr{P}_{X/B}^m$, and let $H_{(1,1)}^m \subset X$ denote the degeneracy locus of these elements on the smooth locus of the map $\pi$. The scheme $H_{(1,1)}^m$ also has expected codimension $2$ in $X$ and can be thought of as the locus of $m^{\mathrm{th}}$\emph{-order weight-$2$ type-(b) inflection points} of a general linear system near the singular point $0 \in X$. We prove the following result in \S~\ref{sec-limdefs}:

\begin{theorem*}
The intersection multiplicity between the closure of $H_{(1,1)}^m$ and the total space of a general $1$-parameter subdeformation of $X/B$ is given by the following formula:
\begin{equation} \label{eq-fund2}
         \on{AD}_{(1,1)}^m(f) - (\on{mult}_0 \Delta_f) \cdot \on{AD}_{(1,1)}^m(xy).
\end{equation}
\end{theorem*}

The multiplicity in~\eqref{eq-fund2} is equal to \emph{the number of $m^{\mathrm{th}}$-order weight-$2$ type-(b) inflection points limiting to the singularity in a general $2$-parameter deformation}, and it
measures the extent to which the singularity ``counts as'' an $m^{\mathrm{th}}$-order weight-$2$ type-(b) inflection point.

For the weight-$1$ case, let $\pi \colon X \to B$ be a general $1$-parameter deformation, choose $m$ general elements of $\scr{P}_{X/B}^m$, and let $H_{(1)}^m \subset X$ denote the degeneracy locus of these elements on the smooth locus of the map $\pi$. The scheme $H_{(1)}^m$ has expected codimension $1$ in $X$ and can be thought of as the locus of $m^{\mathrm{th}}$\emph{-order weight-$1$ inflection points} of a general linear system near the singular point $0 \in X$. We prove the following result in \S~\ref{sec-limdefs}:
\begin{theorem*}
The intersection multiplicity between the closure of $H_{(1)}^m$ and the singular fiber $X_0$ is given by the following formula:
\begin{equation} \label{eq-fund3}
    \on{AD}_{(1)}^m(f).
\end{equation}
\end{theorem*}
The multiplicity in~\eqref{eq-fund3} is equal to the \emph{the number of $m^{\mathrm{th}}$-order weight-$1$ inflection points limiting to the singularity in a general $1$-parameter deformation}, and it measures the extent to which the singularity ``counts as'' an $m^{\mathrm{th}}$-order weight-$1$  inflection point.

As we explain in \S~\ref{sec-wl}, our work on automatic degeneracy in the weight-$1$ case is related to the seminal work of C.~Widland and R.~Lax, who developed a general method for determining the number of weight-$1$ inflection points limiting to an ICIS in a family of Gorenstein curves over the course of their series of papers~\cite{lw2,lw1,lw3,lw4,lw5,lw6,MR1038736}.

The formulas~\eqref{Eq:FundamentalFormula},~\eqref{eq-fund2}, and~\eqref{eq-fund3} reduce the rather difficult problem of directly computing numbers of limiting inflection points to the problem of computing automatic degeneracies, which are more algebraically accessible and hence easier to compute.

We now provide a summary of our results on computing automatic degeneracies. In Theorem~\ref{thm-main2}, we explicitly compute all three automatic degeneracies associated to a node:
\begin{theorem*}
We have the following formulas:
\begin{equation} \label{eq-overview-thm1}
  \on{AD}_{(2)}^m(xy) = {{m+1} \choose 4}, \quad
  \on{AD}_{(1,1)}^m(xy) = {{m+2} \choose 4}, \quad
  \on{AD}_{(1)}^m(xy) = m(m-1).
\end{equation}
\end{theorem*}
The combinatorial formula ${{m+1} \choose 4}$ unsurprisingly appears in the work of Ran on the enumerative geometry of nodal curves (see~\cite{MR3078931} and \S~\ref{sec-ranwithit}). We also compute the weight-$2$ type-(a) automatic degeneracy of a cusp in Corollary~\ref{cor-cusp}:
\begin{theorem*}
We have the following formula:
\begin{equation} \label{eq-talkaboutacusp}
    \on{AD}_{(2)}^m(y^2 - x^3) = 2 \cdot {{m+1} \choose 4}
\end{equation}
\end{theorem*}

It remains open as to how one might determine automatic degeneracies as explicit functions of $m$ for higher-order singularities, such as a tacnode or triple point. Nonetheless, in Theorem~\ref{cor-worse}, we obtain the following lower bounds on automatic degeneracies:
\begin{theorem*}
We have the following inequalities:
\begin{align}
  & \on{AD}_{(2)}^m(f) \geq (\on{mult}_0 \Delta_f) \cdot {{m+1} \choose 4}, \quad
  \on{AD}_{(1,1)}^m(f) \geq (\on{mult}_0 \Delta_f) \cdot {{m+2} \choose 4}, \label{eq-overview-thm2}  \\
  & \hphantom{this is fat extra space}\on{AD}_{(1)}^m(f) \geq \delta_f \cdot m(m-1), \nonumber
  \end{align}
 where $f$ is any ICIS germ and $\delta_f$ denotes the $\delta$\emph{-invariant} of $f$.\footnote{Note that the first two inequalities in~\eqref{eq-overview-thm2} follow immediately from the formulas~\eqref{Eq:FundamentalFormula} and~\eqref{eq-fund2}.} \end{theorem*} We prove the following upper bounds on weight-$2$ automatic degeneracies of \emph{planar} ICIS germs $f$ in Theorem~\ref{thm-theupperbound}:
  \begin{theorem*}
  We have the following inequalities:
  \begin{equation} \label{eq-overview-thm-3}
  \on{AD}_{(2)}^m(f) \leq \mu_f \cdot \frac{3m-1}{m+1} \cdot {{m+1} \choose 4} \quad \text{and} \quad \on{AD}_{(1,1)}^m(f) \leq \mu_f \cdot \frac{3m-2}{m+2} \cdot {{m+2} \choose 4}
  \end{equation}
  where $\mu_f$ denotes the \emph{Milnor number} of $f$.
  \end{theorem*}
   We pay particular attention to the case of planar ICIS germs, because it has the virtue that explicit calculations are possible to perform by hand, albeit with significant effort. In \S~\ref{sec-algae}, we present an algorithm for computing a basis of the completion of the stalk of the sheaf of invincible parts at a planar ICIS. Theoretically, this algorithm can be used to compute the $m^{\mathrm{th}}$-order automatic degeneracies of any given planar ICIS for any \emph{fixed} value of $m$. We use this algorithm to compute several low-order automatic degeneracies of planar ICIS germs:
   \begin{theorem*}
   We have the following formulas:
   \begin{align}
       \on{AD}_{(1,1)}^2(f) & = \mu_f,\,\, \text{ for any planar ICIS germ $f$ (\emph{Theorem~\ref{thm-autodeg112ismiln})}}, \label{eq-lablist1} \\
       \on{AD}_{(2)}^3(y^t - x^s) & = (2t-3)(s-1),\,\, \text{ for $s \geq t \geq 2$ (\emph{Theorem~\ref{thm-23ab})}},\\
       \on{AD}_{(1)}^2(f) & = e_f,\,\, \text{ for any planar ICIS germ $f$ (\emph{Theorem~\ref{thm-tess}})},\\
       \on{AD}_{(2)}^4(y^2 - x^s) & = 6(s-1),\,\, \text{ for $s \geq 4$ (\emph{Theorem~\ref{thm-2t})}},\\
       \on{AD}_{(1)}^3(y^2 - x^s) & = 3s,\,\, \text{ for $s \geq 4$ (\emph{Theorem~\ref{thm-2t1})}},  \label{eq-lablist2}
   \end{align}
   where $e_f$ denotes the \emph{Hilbert-Samuel multiplicity} of the Jacobian ideal of $f$.
   \end{theorem*}
   Moreover, we demonstrate in \S~\ref{sec-probcorrect} that upon making certain simplifying assumptions, the algorithm in \S~\ref{sec-algae} can be used in conjunction with a computer algebra system like {\tt Macaulay2} or {\tt Singular} to determine ``expected values'' of the $m^{\mathrm{th}}$-order automatic degeneracies for any planar ICIS germ and any given order $m$.

Finally, as long as we assume that certain generality hypotheses are satisfied, it is possible to apply our results on automatic degeneracy to address classical enumerative problems about inflection points in $1$-parameter families of curves. We conclude this article with a detailed discussion of the following three problems:
\begin{enumerate}
  \item Counting the number of points at which the members of a general $1$-parameter family of curves of a given degree in $\BP_k^{m-2}$ have contact of order at least $m$ with a hyperplane, thus recovering by elementary means a result of Ran in~\cite[Example 3.21]{MR3078931} (note that taking $m = 4$ yields the problem of counting hyperflexes in a pencil of plane curves);
  \item Counting the number of \emph{septactic points} (i.e., points at which the osculating conic meets the curve with contact of order at least $7$) in a general pencil of plane curves of a given degree;
  \item Computing the $\lambda$ and $\delta_0$ terms of the divisors of weight-$2$ Weierstrass points of arbitrary degree on the Deligne-Mumford compactification $\ol{\mathcal{M}}_g$ of the moduli space of curves of genus $g$.
\end{enumerate}

The rest of this paper is organized as follows. In \S~\ref{sec-thebacks}, we provide relevant background material on the sheaves of principal parts and on inflection points, and we define and prove basic properties of the sheaves of invincible parts. In \S~\ref{sec-dealwithit}, we introduce the notion of automatic degeneracy, explain why automatic degeneracy can be used to determine the number of inflection points limiting to an ICIS (by proving the formulas~\eqref{Eq:FundamentalFormula}--\eqref{eq-fund3}), and discuss the relationship between automatic degeneracy and other known invariants and multiplicities. In \S~\ref{sec-calc}, we prove the main results about automatic degeneracies of nodes stated in~\eqref{eq-overview-thm1}. In \S~\ref{sec-getsworse}, we discuss automatic degeneracies of cusps, obtaining the formula stated in~\eqref{eq-talkaboutacusp}, we present an algorithm that can be used to compute automatic degeneracies of planar ICIS germs, and we obtain the bounds stated in~\eqref{eq-overview-thm2} and~\eqref{eq-overview-thm-3}.
In \S~\ref{sec-autodegegs}, we perform explicit calculations of several low-order automatic degeneracies, some of which are stated in~\eqref{eq-lablist1}--\eqref{eq-lablist2}. We conclude the paper in \S~\ref{sec-eg} with a discussion of the three classical enumerative applications described above. A number of questions about automatic degeneracy remain open, and we take care to point out these questions throughout the course of the article. A summary of these open questions is provided in \S~\ref{sec-openended}.

\section{Background Material} \label{sec-thebacks}

In this section, we begin by defining two key notions: the sheaves of principal parts and inflection points of linear systems on families of curves. We then define and prove basic properties of the sheaves of invincible parts. Throughout this section, we introduce notations and constructions that are used in the rest of the paper.

\subsection{The Sheaves of Principal Parts}

We work over an algebraically closed field $k$ of characteristic $0$. Let $X$ and $B$ be irreducible $k$-varieties (i.e., reduced separated schemes of finite type over $k$), with $X$ normal and $B$ smooth, and let $\pi \colon X \to B$ be a map of $k$-varieties. We view $X$ as the total space of a family over the base $B$ via the map $\pi$. For a point $b \in B$, we write $X_b$ for the fiber of the family lying over $b$.

Next, let $\scr{V}$ be a sheaf on $X$. For a point $p \in X$, we write $\scr{V}|_p \coloneqq \scr{V}_p \otimes_{\scr{O}_{X,p}} \kappa(p)$ for the fiber of $\scr{V}$ at $p$, where $\scr{V}_p$ denotes the stalk of $\scr{V}$ at $p$, $\scr{O}_{X,p}$ denotes the local ring of $X$ at $p$, and $\kappa(p)$ denotes the residue field of $X$ at $p$.

If $\scr{V}$ is locally free (i.e., a vector bundle), the associated sheaves of principal parts are defined as follows:
\begin{defn} \label{def-spps}
The $m^{\mathrm{th}}$\emph{-order sheaf of (relative) principal parts} associated to the pair $(X/B, \scr{V})$ is given by
$$\scr{P}^m_{X/B}(\scr{V}) \coloneqq {\pi_1}_*(\pi_2^* \scr{V} \otimes_{\scr{O}_{X \times_B X}} (\scr{O}_{X \times_B X}/\scr{I}_{\Delta}^m)),$$
where $\Delta \hookrightarrow X \times_B X$ is the diagonal closed subscheme\footnote{Note that $\Delta\hookrightarrow X \times_B X$ is indeed a closed subscheme because the fact that $X$ is Noetherian and separated implies that the map $\pi \colon X \to B$ is separated.}, $\scr{I}_\Delta$ denotes the ideal sheaf of $\Delta$, and $\pi_1$ and $\pi_2$ are the projections of $X \times_B X$ onto the left- and right-hand factors, respectively.
\end{defn}

The sheaves of principal parts are often referred to as \emph{sheaves of jets} in the literature on singularity theory and can be thought of as parametrizing $m^{\mathrm{th}}$-order Taylor expansions of sections of the vector bundle $\scr{V}$ along the fibers of the family $X/B$. For a modern treatment of the sheaves of principal parts (which were first defined by A.~Grothendieck in~\cite[\S~16]{MR0238860}) and their properties, refer to \cite[\S~7.2 and \S~11.1.1]{harris3264}.

Apart from Definition~\ref{def-spps}, we require only three facts about the sheaves of principal parts for the purposes of this paper. The first is that the sheaves of principal parts are coherent:
\begin{lemma} \label{lem-princoh}
The sheaf $\scr{P}_{X/B}^m(\scr{V})$ is coherent for any $m \geq 1$ and any vector bundle $\scr{V}$.
\end{lemma}
\begin{proof}
  Notice that by definition, $\scr{P}_{X/B}^m(\scr{V})$ is the pushforward of the evidently coherent sheaf $\pi_2^* \scr{V} \otimes_{\scr{O}_{X \times_B X}} (\scr{O}_{X \times_B X}/\scr{I}_{\Delta}^m)$ along the morphism $\pi_1$.
  Since $\Delta$ is locally Noetherian, and because the sheaf $\pi_2^* \scr{V} \otimes_{\scr{O}_{X \times_B X}} (\scr{O}_{X \times_B X}/\scr{I}_{\Delta}^m)$ is supported on $\Delta$, it suffices to show that $\pi_1|_\Delta \colon \Delta \to X$ is proper, but this is clear because the map $\pi_1|_\Delta$ is an isomorphism.
\end{proof}
The second fact concerns the fiber of the sheaf of principal parts at a point:
\begin{lemma} \label{lem-fiber}
For every point $p \in X(k)$, we have the identification
$$(\scr{P}_{X/B}^m(\scr{V}))|_p \simeq H^0((\scr{V} \otimes_{\scr{O}_X} (\scr{O}_{X}/\scr{I}_p^m))|_{X_{\pi(p)}}),$$
where $\scr{I}_p$ denotes the ideal sheaf of the point $p$.
\end{lemma}
\begin{proof}
  It suffices to treat the case where $X = \Spec S$ and $B = \Spec T$ are affine and $\scr{V} \simeq S$ is a free $S$-module of rank $1$. Letting $I_\Delta \subset S \otimes_T S$ be the ideal cutting out the diagonal $\Delta \hookrightarrow \Spec S \otimes_T S$, we have that
 \begin{equation} \label{eq-neededindefs}
  \scr{P}_{X/B}^m(\scr{V}) = (S \otimes_T S)/I_\Delta^m,
  \end{equation}
  where we view the right-hand side as an $S$-module by letting $s \in S$ act via multiplication by $s \otimes 1$. Notice that $S$ is of finite type over $k$ and hence of finite type over $T$, so we can express $S$ as $S = T[x_1, \dots, x_n]/(g_1, \dots, g_\ell)$ for some polynomials $g_1, \dots, g_\ell \in T[x_1, \dots, x_n]$. We thus have that
  \begin{align}
  & \scr{P}_{X/B}^m(\scr{V}) = \nonumber \\
  & T[x_1, \dots, x_n,u_1, \dots, u_n]/\big(g_1(x_1, \dots, x_\ell), \dots, g_n(x_1, \dots, x_\ell), \nonumber \\
  & \hphantom{T[x_1, \dots, x_n,u_1, \dots, u_n]/\big(((} g_1(u_1, \dots, u_\ell), \dots, g_n(u_1, \dots, u_\ell), (u_1 - x_1, \dots, u_\ell - x_\ell)^m\big), \nonumber
  \end{align}
  where $I_\Delta = (u_1 - x_1, \dots, u_\ell - x_\ell)$. Now, let $p \in X(k)$, let $I_{\pi(p)} \subset T$ denote the ideal cutting out $\pi(p) \in B$, and let $(x_1, \dots, x_\ell) = (a_1, \dots, a_\ell)$ be the coordinates of $p$ over $B$. Then the fiber of the sheaf of principal parts at $p$ is given by
  \begin{align}
  & (\scr{P}_{X/B}^m(\scr{V}))|_p =  T[u_1, \dots, u_n]/\big(I_{\pi(p)}, g_1, \dots, g_n, (u_1 - a_1, \dots, u_\ell - a_\ell)^m\big). \label{eq-expanditout}
  \end{align}
  On the other hand, we have that
  $$\scr{V} \otimes_{\scr{O}_X} (\scr{O}_{X}/\scr{I}_p^m) = S \otimes_S T[x_1, \dots, x_n]/\big(g_1, \dots, g_n, (I_{\pi(p)}, x_1 - a_1, \dots, x_\ell - a_\ell)^m\big),$$
  so restricting to the fiber $X_{\pi(p)}$ yields that
  \begin{equation} \label{eq-harrisfiber}
      (\scr{V} \otimes_{\scr{O}_X} (\scr{O}_{X}/\scr{I}_p^m))|_{X_{\pi(p)}} = T[x_1, \dots, x_n]/\big(I_{\pi(p)},g_1, \dots, g_n, (x_1 - a_1, \dots, x_\ell - a_\ell)^m\big).
  \end{equation}
  It is now evident that the expressions on the right-hand sides of~\eqref{eq-expanditout} and~\eqref{eq-harrisfiber} are the same, so we have the lemma.
\end{proof}

The third fact is that the sheaves of principal parts recursively fit into exact sequences:
\begin{lemma}[\protect{\cite[part (c) of Theorem 7.2]{harris3264}}]\label{prop-principalpartsseqexact}
  For every $m \geq 2$, we have the right-exact sequence
	\begin{center}
		\begin{tikzcd}[row sep = tiny]
			\scr{V} \otimes_{\scr{O}_X} (\Sym^{m-1} \Omega_{X/B}^1) \arrow{r} & \scr{P}^m_{X/B}(\scr{V}) \arrow{r} & \scr{P}^{m-1}_{X/B}(\scr{V}) \arrow{r} & 0
		\end{tikzcd}
	\end{center}
	where $\Omega_{X/B}^1 = {\pi_1}_* (\scr{I}_{\Delta}/\scr{I}_{\Delta}^2)$ denotes the sheaf of relative differentials associated to the family $X/B$ and where the map $\scr{P}_{X/B}^m(\scr{V}) \to \scr{P}_{X/B}^{m-1}(\scr{V})$ is the obvious map induced by the quotient map $\scr{O}_{X \times_B X}/\scr{I}_\Delta^m \to \scr{O}_{X \times_B X}/\scr{I}_\Delta^{m-1}$. Moreover, the above sequence is left-exact when $X/B$ is smooth.
\end{lemma}

In particular, it follows from Lemma~\ref{prop-principalpartsseqexact} that if $p \in X(k)$ is such that the sheaves of principal parts fail to be locally free at $p$, then $p$ is a singular point of the family. This lemma is particularly useful because, when the map $\pi$ is smooth, one can inductively apply the \emph{Whitney sum formula} (see~\cite[Theorem 5.3]{harris3264}) to the exact sequence in the lemma to compute the Chern classes of the sheaves of principal parts, and as we explain in the following section, loci of inflection points can be naturally described in terms of these Chern classes.

\subsection{Families of Curves and their Inflection Points} \label{sec-degens}

In the next definition, we specify the conditions that we require the family $X/B$ to satisfy in what follows.

\begin{defn}\label{def-admissiblefam}
We consider the family $X/B$ to be \emph{admissible} if the map $\pi$ is flat, each geometric fiber of $\pi$ is reduced, local complete intersection, and of pure dimension $1$, and the locus $\on{Sing}(X/B) \subset X$ of points at which $\pi$ fails to be smooth has codimension $2$ in $X$.
\end{defn}
Observe that if $X/B$ is admissible, then its fibers are Gorenstein, implying that the relative dualizing sheaf $\omega_{X/B}$ associated to the family is invertible.

We are now in position to define what it means for a point to be an inflection point of a linear system on a family of curves. The following is a generalization to the case of families of the definition of inflection points on individual curves given in~\cite[\S~7.5.1]{harris3264}.
\begin{defn} \label{def-inflect}
Let $X/B$ be an admissible family. We make the following definitions:
\begin{enumerate}
\item A \emph{linear system} on $X/B$ is a pair $(\scr{L}, \scr{E})$ consisting of a line bundle $\scr{L}$ on $X$ and a vector subbundle $\scr{E} \subset \pi_* \scr{L}$ on $B$. In the case where $B = \Spec k$, the data of a vector subbundle $\scr{E} \subset \pi_* \scr{L}$ is just a vector subspace $W \subset H^0(\scr{L})$, so we denote the linear system by $(\scr{L}, W)$.
\item For a point $p \in X(k) \setminus \on{Sing}(X/B)$, the \emph{vanishing sequence} of the linear system $(\scr{L},\scr{E})$ at $p$ is the sequence $\alpha_1 < \cdots < \alpha_r$ of distinct orders of vanishing at $p$ of those global sections of $\scr{L}|_{X_{\pi(p)}}$ contained in the subspace $\scr{E}|_{\pi(p)} \subset (\pi_*\scr{L})|_p = H^0(\scr{L}|_{X_{\pi(p)}})$. It is always the case that $r = \on{rk} \scr{E}$ (for a proof, see~\cite[\S~7.5.1]{harris3264}).
\item For a point $p \in X(k) \setminus \on{Sing}(X/B)$, the \emph{ramification sequence} of the linear system $(\scr{L},\scr{E})$ at $p$ is the sequence $\beta_1 \leq \cdots \leq \beta_r$ defined by $\beta_i = \alpha_i - i + 1$ for each $i \in \{1, \dots, r\}$.
\item A point $p \in X(k) \setminus \on{Sing}(X/B)$ is an \emph{inflection point} of the linear system $(\scr{L}, \scr{E})$ if the ramification sequence of the linear system at $p$ contains a nonzero term.
\item The \emph{weight} of an inflection point is the sum of the terms in its ramification sequence. All weight-$1$ inflection points have ramification sequence of the form $(0, \dots, 0, 1)$. There are two types of weight-$2$ inflection points, namely those with ramification sequence of the form $(0, \dots, 0, 0, 2)$, which we call \emph{type-(a)} points, and those with ramification sequence of the form $(0,\dots,0,1,1)$, which we call \emph{type-(b)} points.
\end{enumerate}
\end{defn}

We now explain how to use the sheaves of principal parts to describe the loci of weight-$1$, weight-$2$ type-(a), and weight-$2$ type-(b) inflection points of a linear system $(\scr{L},\scr{E})$ on the family $X/B$. We claim that we have a map
\begin{equation} \label{eq-xiclaim}
(\pi^* \pi_*\scr{L})|_p \to (\scr{P}_{X/B}^m(\scr{L}))|_p
\end{equation}
at the level of sheaf fibers for every $p \in X(k)$ and $m \geq 1$. To see why, notice that $(\pi^* \pi_*\scr{L})|_p \simeq H^0(\scr{L}|_{X_{\pi(p)}})$, and recall that $(\scr{P}_{X/B}^m(\scr{L}))|_p \simeq H^0((\scr{L} \otimes_{\scr{O}_X} (\scr{O}_{X}/\scr{I}_p^m)|_{X_{\pi(p)}}))$ from Lemma~\ref{lem-fiber}. Then the desired map~\eqref{eq-xiclaim} arises by applying the global sections functor to the natural map $\scr{L}|_{X_{\pi(p)}} \to (\scr{L} \otimes_{\scr{O}_X} (\scr{O}_{X}/\scr{I}_p^m))|_{X_{\pi(p)}}$. Moreover, it is a standard fact that as $p$ varies in $X$, the maps in~\eqref{eq-xiclaim} fit together to form a map of sheaves $\pi^* \pi_*\scr{L} \to \scr{P}_{X/B}^m(\scr{L})$ (see~\cite[part (c) of Theorem~11.2]{harris3264}). Precomposing this map with the pullback along $\pi$ of the inclusion $\scr{E} \hookrightarrow \pi_* \scr{L}$ yields a map
\begin{equation} \label{eq-deftau}
\tau \colon \pi^* \scr{E} \to \scr{P}_{X/B}^m(\scr{L}).
\end{equation}
We now specify the conditions for a point $p \in X(k)$ to be an inflection point in terms of the map $\tau$. We say that a point $p \in X(k)$ is an inflection point of
\begin{itemize}
    \item weight $1$ if and only if $p \in X(k) \setminus \on{Sing}(X/B)$ and the map on fibers $\tau|_p \colon (\pi^* \scr{E})|_p \to (\scr{P}_{X/B}^m(\scr{L}))|_p$ has rank at most $m-1$, where $m = \on{rk} \scr{E}$ (i.e., some section of $\scr{E}|_{\pi(p)}$ vanishes to order $m$ at $p$);
    \item weight $2$ and type (a) if and only if $p \in X(k) \setminus \on{Sing}(X/B)$ and the map on fibers $\tau|_p \colon (\pi^* \scr{E})|_p \to (\scr{P}_{X/B}^m(\scr{L}))|_p$ has rank at most $m-2$, where $m = \on{rk} \scr{E}+1$ (i.e., some section of $\scr{E}|_{\pi(p)}$ vanishes to order $m$ at $p$); and
    \item weight $2$ and type (b) if and only if $p \in X(k) \setminus \on{Sing}(X/B)$ and the map on fibers $\tau|_p \colon (\pi^* \scr{E})|_p \to (\scr{P}_{X/B}^m(\scr{L}))|_p$ has rank at most $m-1$, where $m = \on{rk} \scr{E} -1$ (i.e., some linearly independent pair of sections of $\scr{E}|_{\pi(p)}$ vanish to order $m$ at $p$).
\end{itemize}

When the map $\pi$ is smooth, the sheaf $\scr{P}_{X/B}^m(\scr{L})$ is locally free, so the map $\tau$ can be locally represented as a matrix with $m$ rows and $\on{rk} \scr{E}$ columns. In this case, the locus of inflection points of a given type is simply the \emph{degeneracy locus} of the map $\tau$ and is locally cut out by the ideal of maximal minors of a matrix representing $\tau$ (here, we take $m = \on{rk} \scr{E} + \{0, \pm 1\}$ according to the type of inflection point as specified in the itemized list above). If it has the expected codimension, the class of this degeneracy locus in the Chow ring of $X$ is simply given by the Chern classes $c_1(\tau)$ in the weight-$1$ case and $c_2(\tau)$ in the weight-$2$ case. The Chern classes of the map $\tau$ can be easily computed in terms of the Chern classes of the bundles $\pi^* \scr{E}$ and $\scr{P}_{X/B}^m(\scr{L})$ using the Porteous formula (see~\cite[\S~12, particularly Theorem 12.4]{harris3264}). If we take $B = \Spec k$, so that $X$ is an individual smooth curve, then the number of weight-$1$ inflection points associated to the linear system $(\scr{L}, \scr{E})$ on $C$ is finite and given by $\deg c_1(\tau)$. If we take $B$ to be $1$-dimensional, so that $X/B$ is a smooth $1$-parameter family of curves, then the number of weight-$2$ inflection points of a given type associated to the linear system $(\scr{L}, \scr{E})$ on $X/B$ is finite and given by $\deg c_2(\tau)$.

This is the standard strategy used to count, say, the number of \emph{flexes} on a general plane curve $C \subset \BP_k^2$ of a given degree (see~\cite[\S~7.5.2]{harris3264}). Indeed, flexes---which are defined to be points of plane curves at which the tangent line meets the curve with intersection multiplicity at least $3$---are simply weight-$1$ inflection points of the linear system $(\scr{L},W)$ where $\scr{L} = \scr{O}_C(1)$ is the pullback of $\scr{O}_{\BP_k^2}(1)$ to $C$ and $W \subset H^0(\scr{O}_C(1))$ is the subspace spanned by the pullbacks of global sections of $\scr{O}_{\BP_k^2}(1)$. Thus, the number of flexes on $C$ is given by $\deg c_1(\tau)$. This strategy breaks down, however, if we try to apply it to count the number of \emph{hyperflexes} in a general pencil $X/\BP_k^1$ of plane curves of a given degree. Indeed, hyperflexes---which are defined to be points of plane curves at which the tangent line meets the curve with intersection multiplicity at least $4$---are weight-$2$ type-(a) inflection points of the linear system $(\scr{L},W)$ where $\scr{L} = \scr{O}_X(1)$ is the pullback of $\scr{O}_{\BP_k^2}(1)$ to $X$ (via the map $X \to \BP_k^2$ that maps each fiber of the pencil to the corresponding plane curve) and $W \subset H^0(\scr{O}_X(1))$ is the subspace spanned by the pullbacks of global sections of $\scr{O}_{\BP_k^2}(1)$. As it happens, the pencil $X/\BP_k^1$ contains a finite but nonzero number of singular fibers, so the associated sheaves of principal parts fail to be locally free, and the number of hyperflexes cannot be obtained by simply computing $\deg c_2(\tau)$. In \S~\ref{sec-countatlast}, we discuss known workarounds for this issue and explain how to solve the hyperflex problem using our results on automatic degeneracy.

\subsection{The Sheaves of Invincible Parts} \label{sec-strats}

Let $X/B$ be an admissible family, let $U \coloneqq X \setminus \on{Sing}(X/B)$ denote the smooth locus of the family, and let $\scr{V}$ be a vector bundle on $X$. As we have explained in the preceding sections, the key difficulty that arises when studying inflection points on $X/B$ is that the sheaves of principal parts fail to be locally free over $\on{Sing}(X/B)$. In this section, we introduce our own approach to addressing this issue, which is to replace the sheaves of principal parts with a system of vector bundles---namely, their double-duals---satisfying the same properties that make the sheaves of principal parts so useful for describing inflection points.

We want our replacement sheaves to be locally free on the entire total space $X$ and to be isomorphic to the sheaves of principal parts on the complement of $\on{Sing}(X/B)$. The idea of seeking such locally free replacements for the sheaves of principal parts over singular curves dates back to work of Laksov and Thorup. In their paper~\cite{MR1318539}, they introduce the notion of a \emph{Wronski algebra system}, which is motivated and defined as follows.\footnote{For a detailed and comprehensive summary of the literature on the problem of finding locally free replacements for the sheaves of principal parts, refer to the expository paper~\cite{gattoricolfi} of Gatto and A.~Ricolfi.}

    In some sense, the failure of the sheaf of relative differentials $\Omega_{X/B}^1$ to be locally free is the reason why the sheaves of principal parts $\scr{P}_{X/B}^m(\scr{V})$ fail to be locally free; indeed, one can use the local-freeness of $\Omega_{X/B}^1$ in the case where the family $X/B$ is smooth in conjunction with the exact sequence in Lemma~\ref{prop-principalpartsseqexact} to inductively deduce the local-freeness of $\scr{P}_{X/B}^m(\scr{V})$. But because the relative dualizing sheaf $\omega_{X/B}$ is the unique locally free replacement for $\Omega_{X/B}^1$ (more generally, $\omega_{X/B}^{\otimes m}$ is the unique locally free replacement for $\Sym^m \Omega_{X/B}^1$), it is natural to ask whether one can come up with a sequence of sheaves $\scr{Q}_{X/B}^m(\scr{V})$ that satisfy the same basic properties as the sheaves of principal parts, but fit into exact sequences having tensor powers of the relative dualizing sheaf as the kernel. We are thus led to the following definition:
    \begin{defn}\label{def-wronski}
    A \emph{Wronski algebra system} associated to the pair $(X/B, \scr{V})$ is a sequence of sheaves $\scr{Q}_{X/B}^m(\scr{V})$ satisfying the following properties:
    \begin{enumerate}
    \item For each $m \geq 1$ we have maps $\psi^m \colon \scr{P}_{X/B}^m(\scr{V}) \to \scr{Q}_{X/B}^m(\scr{V})$ such that the following diagram commutes, with each row being exact:
    \end{enumerate}
	\begin{center}
	\begin{tikzcd}
	  & \scr{K}_{X/B}^m(\scr{V}) \arrow{r} \arrow{d} & \scr{P}^m_{X/B}(\scr{V}) \arrow{r} \arrow{d}{\psi^m} & \scr{P}^{m-1}_{X/B}(\scr{V}) \arrow{r} \arrow{d}{\psi^{m-1}} & 0 \\
0 \arrow{r} & \scr{V} \otimes_{\scr{O}_X} \omega_{X/B}^{\otimes(m-1)} \arrow{r} & \scr{Q}^m_{X/B}(\scr{V}) \arrow{r} & \scr{Q}^{m-1}_{X/B}(\scr{V}) \arrow{r} & 0
	\end{tikzcd}
	\end{center}
	\begin{enumerate}
	    \item[] where $\scr{K}_{X/B}^m(\scr{V}) \coloneqq \scr{V} \otimes_{\scr{O}_X} (\Sym^{m-1} \scr{\Omega}_{X/B}^1)$.
	\end{enumerate}
\begin{enumerate}
    \item[(b)] The map $\psi^1 \colon \scr{V} = \scr{P}_{X/B}^1(\scr{V}) \to \scr{Q}_{X/B}^1(\scr{V})$ is an isomorphism.
    \end{enumerate}
    \end{defn}
   The following fundamental result on Wronski algebra systems is due to Esteves:
   \begin{theorem}[{\cite[Theorem 2.6]{MR1368706}}]\label{thm-esteves}
   Let $X/B$ be an admissible family (with $X$ not necessarily normal), and let $\scr{V}$ be a vector bundle on $X$. There exists a unique Wronski algebra system associated to the pair $(X/B, \scr{V})$.
   \end{theorem}
The proof of Theorem~\ref{thm-esteves} is fairly involved and does not appear to provide an easy-to-use description of what the Wronski algebra system is on a given family $X/B$. Nevertheless, because we are taking the total space $X$ to be normal in this paper, it turns out that the Wronski algebra system in Theorem~\ref{thm-esteves} is simply given by the double-duals of the sheaves of principal parts. For simplicity of terminology, we make the following definition:
   \begin{defn}\label{def-invincparts}
  We say that $\scr{P}^m_{X/B}(\scr{V})^{\vee\vee}$ is the $m^{\mathrm{th}}$\emph{-order sheaf of invincible parts} associated to the pair $(X/B, \scr{V})$.
\end{defn}
We then have the following result:
\begin{proposition}\label{thm-replace}
	For every $m \geq 1$, the sheaf of invincible parts $\scr{P}^m_{X/B}(\scr{V})^{\vee\vee}$ is the unique locally free sheaf on $X$ whose restriction to $U$ is isomorphic to $\scr{P}^m_{X/B}(\scr{V})|_U$. The sheaves $\scr{P}^m_{X/B}(\scr{V})^{\vee\vee}$ form a Wronski algebra system for the pair $(X/B, \scr{V})$, where for each $m$ we take $\psi^m$ to be equal to the canonical evaluation map
	$$\on{can}_{\on{ev}} \colon \scr{P}_{X/B}^m(\scr{V}) \longrightarrow \scr{P}_{X/B}^m(\scr{V})^{\vee\vee}.$$
\end{proposition}

Before we prove Proposition~\ref{thm-replace}, we recall a few useful facts about sheaves.

\begin{lemma}\label{prop-sheafprop}
Let $X$ be a Noetherian integral scheme, and let $\scr{F},\scr{F}'$ be coherent sheaves on $X$. We have the following properties:
  \begin{enumerate}
  \item The dual sheaf $\scr{F}^\vee$ is reflexive.
  \item Suppose that $X$ is normal and that $\scr{F}, \scr{F}'$ are reflexive sheaves. Then,
    \begin{enumerate}
      \item[(i)] If $\scr{F}, \scr{F}'$ differ on a closed subset of codimension at least $2$, then $\scr{F} \simeq \scr{F}'$.
      \item[(ii)] A map $\scr{F}|_U \to \scr{F}'|_U$ on the complement $U$ of a closed subset of codimension at least $2$ extends to a map $\scr{F} \to \scr{F}'$ on all of $X$.
    \end{enumerate}
\end{enumerate}
\end{lemma}
\begin{proof}
  The above properties are well-known; a good reference for the basic facts about reflexive sheaves is~\cite{MR597077}, in which property (a) is Corollary 1.2 and property (b) is Proposition 1.6. 
\end{proof}

We are now in position to prove Proposition~\ref{thm-replace}.

\begin{proof}[Proof of Proposition~\ref{thm-replace}]
Suppose the Wronski algebra system in Theorem~\ref{thm-esteves} is composed of sheaves $\scr{Q}_{X/B}^m(\scr{V})$ and maps $\psi^m \colon \scr{P}_{X/B}^m(\scr{V}) \to \scr{Q}_{X/B}^m(\scr{V})$. Because $X$ was taken to be Noetherian and normal and because $\scr{P}_{X/B}^m(\scr{V})$ is coherent by Lemma~\ref{lem-princoh}, we have that $\scr{P}_{X/B}^m(\scr{V})^\vee$ is reflexive by property (a) in Lemma~\ref{prop-sheafprop}. Taking $\scr{F} = \scr{Q}_{X/B}^m(\scr{V})$ and $\scr{F}' = \scr{P}_{X/B}^m(\scr{V})^{\vee\vee}$, part (i) of property (b) in Lemma~\ref{prop-sheafprop} then tells us that $\scr{Q}_{X/B}^m(\scr{V}) \simeq \scr{P}_{X/B}^m(\scr{V})^{\vee\vee}$. In particular, we deduce that $\scr{P}_{X/B}^m(\scr{V})^{\vee\vee}$ is locally free on all of $X$. A similar argument shows that $\scr{K}_{X/B}^m(\scr{V})^{\vee\vee} \simeq \scr{V} \otimes_{\scr{O}_X} \omega_{X/B}^{\otimes(m-1)}$ is also locally free.

We now claim that the sequence
\begin{equation} \label{eq-invincpartsexactseq}
    	\begin{tikzcd}[row sep = tiny]
			0 \arrow{r} & \scr{V} \otimes_{\scr{O}_X} \omega_{X/B}^{\otimes(m-1)} \arrow{r} & \scr{P}^m_{X/B}(\scr{V})^{\vee\vee} \arrow{r} & \scr{P}^{m-1}_{X/B}(\scr{V})^{\vee\vee} \arrow{r} & 0
		\end{tikzcd}
\end{equation}
obtained by taking the double-dual of the sequence in Lemma~\ref{prop-principalpartsseqexact} is short exact. Indeed, notice that the sequence in~\eqref{eq-invincpartsexactseq} is exact upon restricting to the open subscheme $U \subset X$ (because the sheaves of principal parts are locally free on $U$), and so we have that
\begin{equation} \label{eq-urestrict}
\ker(\scr{V} \otimes_{\scr{O}_X} \omega_{X/B}^{\otimes(m-1)} \to \scr{P}^m_{X/B}(\scr{V})^{\vee\vee})|_U = \coker(\scr{P}^m_{X/B}(\scr{V})^{\vee\vee} \to \scr{P}^{m-1}_{X/B}(\scr{V})^{\vee\vee}))|_U = 0.
\end{equation}
It then follows from part (ii) of property (b) in Lemma~\ref{prop-sheafprop} that the equalities in~\eqref{eq-urestrict} hold upon lifting the restriction to $U$, thus proving the claim.

Upon making the identification $\scr{K}_{X/B}^m(\scr{V})^{\vee\vee} \simeq \scr{V}\otimes_{\scr{O}_X} \omega_{X/B}^{\otimes(m-1)}$, we obtain the following commutative diagram in which each row is exact:
\begin{equation}\label{eq-triedprince}
\begin{tikzcd}
     & \scr{K}_{X/B}^m(\scr{V}) \arrow{r} \arrow{d} & \scr{P}^m_{X/B}(\scr{V}) \arrow{r} \arrow{d}{\on{can}_{\on{ev}}} & \scr{P}^{m-1}_{X/B}(\scr{V}) \arrow{r} \arrow{d}{\on{can}_{\on{ev}}} & 0 \\
    0 \arrow{r} & \scr{V} \otimes_{\scr{O}_X} \omega_{X/B}^{\otimes(m-1)} \arrow[swap]{r} & \scr{P}^m_{X/B}(\scr{V})^{\vee\vee} \arrow[swap]{r}  & \scr{P}^{m-1}_{X/B}(\scr{V})^{\vee\vee} \arrow{r} & 0
  \end{tikzcd}
\end{equation}
But by Theorem~\ref{thm-esteves}, we know that the Wronski algebra system composed of the sheaves $\scr{Q}_{X/B}^m(\scr{V})$ and the maps $\psi^m$ is unique. It follows that this system must be equal to the Wronski algebra system in~\eqref{eq-triedprince}, as desired.
\end{proof}

\begin{remark} \label{rem-dualexact}
   Note that it follows from Proposition~\ref{thm-replace} that the dual sheaves $\scr{K}_{X/B}^m(\scr{V})^\vee$ and $\scr{P}_{X/B}^m(\scr{V})^{\vee}$ are also locally free and that the sequence
   \begin{equation} \label{eq-dualexact}
       \begin{tikzcd}[row sep = tiny]
			0 \arrow{r} & \scr{P}^{m-1}_{X/B}(\scr{V})^\vee \arrow{r} & \scr{P}^m_{X/B}(\scr{V})^\vee \arrow{r} & \mathscr{K}_{X/B}^m(\scr{V})^\vee \arrow{r} &  0
		\end{tikzcd}
   \end{equation}
   is short exact. Indeed, for each $m$, we have that the sheaves $\scr{K}_{X/B}^m(\scr{V})^\vee$ and $\scr{P}_{X/B}^m(\scr{V})^\vee$ are respectively isomorphic to the triple-dual sheaves $\scr{K}_{X/B}^m(\scr{V})^{\vee\vee\vee}$ and $\scr{P}_{X/B}^m(\scr{V})^{\vee\vee\vee}$, which are locally free because they are the duals of the locally free sheaves $\scr{K}_{X/B}^m(\scr{V})^{\vee\vee}$ and $\scr{P}_{X/B}^m(\scr{V})^{\vee\vee}$. Moreover, the exactness of the sequence in~\eqref{eq-dualexact} follows from the exactness of the sequence obtained by dualizing the sequence in~\eqref{eq-invincpartsexactseq}.
\end{remark}

We conclude this section by using the sequence in~\eqref{eq-invincpartsexactseq} to compute the first two Chern classes of the sheaves of invincible parts of a line bundle.

\begin{proposition}\label{prop-chernconquest}
Let $\scr{L}$ be a line bundle on $X$. Then we have
\begin{align*}
c(\scr{P}_{X/B}^m(\scr{L})^{\vee\vee}) & = 1 + m \cdot c_1(\scr{L}) + \mybinom[0.8]{m}{2} \cdot c_1(\omega_{X/B}) + \\
& \hphantom{====} \mybinom[0.8]{m}{2}\cdot c_1(\scr{L})^2 + \left(3\mybinom[0.8]{m+1}{4} - \mybinom[0.8]{m}{3}\right) \cdot c_1(\omega_{X/B})^2 + \\
& \hphantom{====} \left(3\mybinom[0.8]{m+1}{3}-2\mybinom[0.8]{m}{2}\right) \cdot c_1(\omega_{X/B}) \cdot c_1(\scr{L}) + \cdots
\end{align*}
\end{proposition}
\begin{proof}
By~\eqref{eq-invincpartsexactseq}, we have the short exact sequence
\begin{equation} \label{eq-repeatseq}
		\begin{tikzcd}[row sep = tiny]
			0 \arrow{r} & \scr{L} \otimes_{\scr{O}_X} \omega_{X/B}^{\otimes(m-1)} \arrow{r} & \scr{P}^m_{X/B}(\scr{L})^{\vee\vee} \arrow{r} & \scr{P}^{m-1}_{X/B}(\scr{L})^{\vee\vee} \arrow{r} & 0
		\end{tikzcd}
	\end{equation}
	Applying the Whitney sum formula in conjunction with the splitting principle to the sequence in~\eqref{eq-repeatseq} yields that
	\begin{align}
& 	c(\scr{P}^m_{X/B}(\scr{L})^{\vee\vee}) =     c(\scr{L} \otimes_{\scr{O}_X} \omega_{X/B}^{\otimes(m-1)}) ) \cdot c(\scr{P}^{m-1}_{X/B}(\scr{L})^{\vee\vee}) = \nonumber \\
	& (1 +  c_1(\scr{L}) + (m-1) \cdot c_1(\omega_{X/B})) \cdot (1 + c_1(\scr{P}^{m-1}_{X/B}(\scr{L})^{\vee\vee}) + c_2(\scr{P}^{m-1}_{X/B}(\scr{L})^{\vee\vee})) + \cdots =  \nonumber \\
	&  1 +  c_1(\scr{L}) + (m-1) \cdot c_1(\omega_{X/B}) + c_1(\scr{P}^{m-1}_{X/B}(\scr{L})^{\vee\vee}) + \nonumber \\
	&  \qquad \big( c_1(\scr{L})  + (m-1) \cdot c_1(\omega_{X/B}) \big) \cdot c_1(\scr{P}^{m-1}_{X/B}(\scr{L})^{\vee\vee}) + c_2(\scr{P}^{m-1}_{X/B}(\scr{L})^{\vee\vee})+\cdots \label{eq-codim2terms}
	\end{align}
	An easy induction yields that the first Chern class is given as follows:
\begin{equation}\label{eq-c1}
c_1(\scr{P}^m_{X/B}(\scr{L})^{\vee\vee}) = 1 + m \cdot c_1(\scr{L}) + \mybinom[0.8]{m}{2} \cdot c_1(\omega_{X/B}).
\end{equation}
	As for the second Chern class, substituting the result of~\eqref{eq-c1} into the terms of weight $2$ in~\eqref{eq-codim2terms} and applying induction once more yields that
	\begin{align*}
	c_2(\scr{P}^m_{X/B}(\scr{L})^{\vee\vee}) & =  \sum_{i = 2}^m \left[(i-1) \cdot c_1(\scr{L})^2 + \left(\mybinom[0.7]{i-1}{2} + (i-1)^2\right) \cdot c_1(\omega_{X/B}) \cdot c_1(\scr{L}) + \right. \\
& \hphantom{\sum_{i = 2}^m ===} \left. (i-1) \cdot \mybinom[0.7]{i-1}{2} \cdot c_1(\omega_{X/B})^2 \right],
	\end{align*}
	and evaluating the above sum using the standard identities for summing consecutive squares and cubes gives the desired formula.
\end{proof}

\section{Defining Automatic Degeneracy}\label{sec-dealwithit}

We showed in Proposition~\ref{thm-replace} that the sheaves of invincible parts are a natural locally free replacement for the sheaves of principal parts. The purpose of this section is to demonstrate that the sheaves of invincible parts can be used to study the enumerative geometry of inflection points on singular families of curves, just like the sheaves of principal parts on smooth families of curves.

In this section, we motivate and define the three different types of automatic degeneracy, derive the formulas~\eqref{Eq:FundamentalFormula}--\eqref{eq-fund3}, and discuss the relationship between automatic degeneracies and singularity invariants that already exist in the literature.

\subsection{The Definition} \label{sec-backtobasics}

Let $X/B$ be an admissible $1$-parameter family, and let $(\scr{L}, \scr{E})$ be a linear system on $X/B$. Since the canonical map $\on{can}_{\on{ev}} \colon \scr{P}_{X/B}^m(\scr{L}) \to \scr{P}_{X/B}^m(\scr{L})^{\vee\vee}$ is an isomorphism away from the locus $\on{Sing}(X/B)$ of singular points of the family, it follows from the discussion at the end of \S~\ref{sec-degens} that the locus of inflection points of a given type is given by the intersection with $X \setminus \on{Sing}(X/B)$ of the degeneracy locus of the composite map
\begin{equation} \label{eq-defxi}
\xi \coloneqq \on{can}_{\on{ev}} \circ \tau \colon \pi^* \scr{E} \longrightarrow \scr{P}^{m}_{X/B}(\scr{L}) \longrightarrow \scr{P}^{m}_{X/B}(\scr{L})^{\vee\vee},
\end{equation}
where $m \in \on{rk} \scr{E} + \{0, \pm 1\}$ depending on the type of inflection point we are interested in. The class of this degeneracy locus in the Chow ring of $X$ is given by the Chern classes $c_1(\xi)$ in the weight-$1$ case and $c_2(\xi)$ in the weight-$2$ case. Although it is easy to compute the classes $c_i(\xi)$ using the Porteous formula (in combination with Proposition~\ref{prop-chernconquest}), it is considerably more difficult to ascertain how the points of $\on{Sing}(X/B)$ contribute to these classes. The purpose of automatic degeneracy is to measure the extent to which the classes $c_i(\xi)$ are supported at a given point of $\on{Sing}(X/B)$ in the \emph{most general} setting possible and in a way that is \emph{intrinsic} to the singularity.


Let $f \colon (k^N,0) \to (k^{N-1},0)$ be an ICIS germ with component functions $f_1, \dots, f_{N-1} \in \widehat{\scr{O}}_{k^N,0} \simeq k[[x_1, \dots, x_N]]$. The coordinate ring of the germ is $k[[x_1, \dots, x_N]]/(f_1, \dots, f_{N-1})$, which we write as $\widehat{\scr{O}}_{k^N,0}/(f)$ for short. Let $\on{Def}^1(f) \simeq \on{Ext}^1_{\widehat{\scr{O}}_{k^N,0}/(f)}(\Omega_{(\widehat{\scr{O}}_{k^N,0}/(f))/k}^1, S)$ be the $k$-vector space of first-order deformations of $f$, where $\Omega_{(\widehat{\scr{O}}_{k^N,0}/(f))/k}^1$ denotes the module of relative differentials of $\widehat{\scr{O}}_{k^N,0}/(f)$ over $k$. Letting $d = \dim_k \Def^1(f)$ and $(\delta_1(f), \dots, \delta_d(f))$ be a basis of $\Def^1(f)$, the versal deformation space of $f$ is given by $\mathfrak{B} \coloneqq \Spf k[[\alpha_1, \dots, \alpha_d]]$, and the total space of the family of curves lying over $\mathfrak{B}$ is
$$\mathfrak{X} \coloneqq \Spf \widehat{\scr{O}}_{k^N,0}[[\alpha_1, \dots, \alpha_d]]\bigg/ \left(f + \sum_{i = 1}^d \alpha_i \cdot \delta_i(f)\right).$$
Note that because $f$ is a complete intersection germ, we have that $\mathfrak{X}$ is normal and Cohen-Macaulay and that $\mathfrak{X}/\mathfrak{B}$ is an admissible family. We denote by $0 \in \mathfrak{B}$ the origin, by $0 \in \mathfrak{X}$ the singular point cut out by the germ $f$, and by $X_0$ the fiber of $\mathfrak{X}/\mathfrak{B}$ over $0 \in \mathfrak{B}$. We now introduce terminology for the types of deformations that we will consider in the sequel:
\begin{defn} \label{def-deftypes}
We say that a family $X/B$ is a \emph{$1$-parameter deformation} of $f$ if $B = \Spf k[[t]]$ and there is a map $\scr{O}_{\mathfrak{B}} = k[[\alpha_1, \dots, \alpha_d]] \to k[[t]]$ defined by $\alpha_i \mapsto a_i \cdot t$ where $a_i \in k$ and $X = \mathfrak{X} \times_{\mathfrak{B}} B$. We say that a family $\wt{X}/\wt{B}$ is a \emph{$2$-parameter deformation} of $f$ if $\wt{B} = \Spf k[[s,t]]$ and there is a map $\scr{O}_{\mathfrak{B}} = k[[\alpha_1, \dots, \alpha_d]] \to k[[s,t]]$ defined by $\alpha_i \mapsto a_i \cdot s + b_i \cdot t$ where $a_i,b_i \in k$ and $\wt{X} = \mathfrak{X} \times_{\mathfrak{B}} \wt{B}$. For each $i \in \{1, 2\}$, we say that an $i$-parameter deformation is \emph{general} if it is general as an element of the $k$-vector space of all $i$-parameter deformations.
\end{defn}



 Let $X/B$ be a general $1$-parameter deformation. Let $n \in m + \{0, \pm 1\}$, and consider a list $\ol{\tau} = (\tau_1, \dots, \tau_n)$ of elements of
 \begin{equation} \label{eq-defmodparts}
     \scr{P}_{X/B}^m \coloneqq (\scr{O}_X \otimes_{\scr{O}_B} \scr{O}_X)/I_\Delta^m
 \end{equation}
 viewed as a module over the ring $\scr{O}_X$ of functions on $X$ via the action of $\scr{O}_X$ on the left-hand tensor factor.\footnote{Because both $X$ and $B$ are affine schemes, we can think of $\scr{P}_{X/B}^m$ over $\scr{O}_X$ as just a module over a ring, rather than as a sheaf of modules over a sheaf of rings.} It follows from the proof of Lemma~\ref{lem-fiber} that we can think of $\scr{P}_{X/B}^m$ as being the completion of the stalk of the sheaf of principal parts of any line bundle on a $1$-parameter family of curves acquiring a singularity cut out analytically-locally by $f = 0$. Let $\tau \colon \scr{O}_X^n \to \scr{P}_{X/B}^m$ be the map of $\scr{O}_X$-modules that sends the $\ell^{\mathrm{th}}$ standard basis vector of the free module $\scr{O}_X^n$ to $\tau_\ell \in \scr{P}_{X/B}^m$ for each $\ell \in \{1, \dots, n\}$, and let $\xi = \on{can}_{\on{ev}} \circ \tau$ as before. (The maps $\tau$ and $\xi$ defined here are intended to imitate the maps $\tau$ and $\xi$ defined in~\eqref{eq-deftau} and~\eqref{eq-defxi}.) Let $Z_{\ol{\tau}}^m \subset \Spf \scr{O}_X$ be the degeneracy locus of the map $\xi$, and let $I_{\ol{\tau}}^m \subset \scr{O}_X$ be the ideal cutting out $Z_{\ol{\tau}}^m$.
 Given a basis $(e_1, \dots, e_m)$ of the free $\scr{O}_X$-module ${\scr{P}_{X/B}^m}^\vee$, the ideal $I_{\ol{\tau}}^m$ is generated by the maximal minors of the matrix
\begin{equation}\label{eq-matrixeq}
M_{\ol{\tau}}^m \coloneqq \left[\begin{array}{ccc} e_1(\tau_1) & \cdots & e_1(\tau_n) \\ \vdots & \ddots & \vdots \\ e_m(\tau_1) & \cdots & e_m(\tau_n)  \end{array} \right]
\end{equation}
We now define the automatic degeneracies:
\begin{defn} \label{def-auto}
    The $m^{\mathrm{th}}$\emph{-order automatic degeneracies} of the ICIS germ $f$ are given by
    \begin{align*}
    & \text{\emph{weight-$1$:\hphantom{type--(a)}}}\quad \on{AD}_{(1)}^{m}(f) \coloneqq \min_{X/B} \min_{\ol{\tau}}\big( \dim_k \scr{O}_X/(I_{X_0}+I_{\ol{\tau}}^m)\big), \\
    & \text{\emph{weight-$2$ type-(a)}:}\quad \on{AD}_{(2)}^{m}(f) \coloneqq \min_{X/B} \min_{\ol{\tau}}\big(\dim_k \scr{O}_X/I_{\ol{\tau}}^m\big),\\
       & \text{\emph{weight-$2$ type-(b)}:}\quad \on{AD}_{(1,1)}^{m}(f) \coloneqq \min_{X/B} \min_{\ol{\tau}}\big(\dim_k \scr{O}_X/I_{\ol{\tau}}^m\big),
    \end{align*}
    where the outer minima are taken over all $1$-parameter deformations $X/B$ and the inner minima are taken over all choices of the list $\ol{\tau}$ of $n$ elements of $\scr{P}_{X/B}^m$ (here, $n = m$ in the weight-$1$ case, $n = m-1$ in the weight-$2$ type-(a) case, and $n = m+1$ in the weight-$2$ type-(b) case).
\end{defn}

\begin{remark} \label{rem-generalassump}
Observe that each minimum in Definition~\ref{def-auto} is, if finite, achieved for a general choice of the $1$-parameter deformation $X/B$ and the list $\ol{\tau}$ of elements of $\scr{P}_{X/B}^m$. For a given linear system $(\scr{L},\scr{E})$ on a given family of curves $X/B$, it may \emph{not} be true that the list $\ol{\tau}$ of elements obtained by taking the analytic-local germs of images in $\scr{P}_{X/B}^m(\scr{L})$ of local basis elements of $\pi^* \scr{E}$ under the map $\tau$ in~\eqref{eq-deftau} is general enough to attain any of the minima in Definition~\ref{def-auto}.
\end{remark}

\begin{remark} \label{rem-arcs}
   One could have alternatively defined the automatic degeneracies using arbitrary $1$-parameter \emph{arcs} in the the deformation space as opposed to \emph{lines}, but it turns out that the resulting definition is equivalent to Definition~\ref{def-auto}. The fact that these two definitions are equivalent can be proven in two steps as follows. Let $\wt{\on{AD}}_*^m(f)$ for $* \in \{(1),(2),(1,1)\}$ denote automatic degeneracies computed with respect to arcs. The first step is to note that the two definitions are evidently equivalent in the nodal case. Second, the arguments in \S~\ref{sec-limdefs} hold whether automatic degeneracy is defined using lines or arcs. In particular, Proposition~\ref{prop-flatness1} and Proposition~\ref{prop-flatness2} hold for both definitions, so it follows that the formulas~\eqref{Eq:FundamentalFormula}--\eqref{eq-fund3} hold for both definitions; i.e., we have that
   \begin{align*}
       &\on{AD}_*^m(f) - (\on{mult}_0 \Delta_f) \cdot \on{AD}_*^m(xy) = \#(\text{limiting weight-$2$ inflection points}) = \\
       & \qquad \wt{\on{AD}}_*^m(f) - (\on{mult}_0 \Delta_f) \cdot \wt{\on{AD}}_*^m(xy) \quad \text{for $* \in \{(2),(1,1)\}$,} \\
       & \on{AD}_{(1)}^m(f) = \#(\text{limiting weight-$1$ inflection points})  = \wt{\on{AD}}_{(1)}^m(f).
   \end{align*}
   In other words, one can translate the calculation of an automatic degeneracy into a computation of the multiplicity of a divisorial scheme at $0$ in the base $\mathfrak{B}$ of a versal deformation of the singularity, and the multiplicity of a divisor at $0$ can be computed using general test lines or arcs.
   We choose to work with lines in Definition~\ref{def-auto} because it simplifies the computations performed in \S~\ref{sec-getsworse} and \S~\ref{sec-autodegegs}.
\end{remark}


\begin{remark}
From Definition~\ref{def-auto} and Remark~\ref{rem-arcs}, it is evident that the automatic degeneracies are \emph{analytic invariants}, in the sense that they only depend on the analytic isomorphism class of the singularity. We discuss the relationship between automatic degeneracies and other singularity invariants and multiplicities in \S~\ref{sec-invarrels}.
\end{remark}

\subsection{Application to Counting Limiting Inflection Points} \label{sec-limdefs}

We now prove the formulas~\ref{Eq:FundamentalFormula}--\ref{eq-fund3}, which reduce the problem of counting the number of inflection points of a given type limiting to an ICIS to the problem of computing automatic degeneracies.

We start with the case of weight-$1$ inflection points. Let $f$ be the germ of an ICIS, and suppose that $\on{AD}_{(1)}^m(f) < \infty$. Let $X/B$ be a general $1$-parameter deformation, let $\ol{\tau}$ be a list of $m$ elements of $\scr{P}_{X/B}^m$ achieving the minimum value $\dim_k \scr{O}_X/(I_{X_0} + I_{\ol{\tau}}^m) = \on{AD}_{(1)}^m(f)$. 
We interpret the intersection multiplicity $\dim_k \scr{O}_X/(I_{X_\eta} + I_{\ol{\tau}}^m)$ of the geometric generic fiber $X_\eta$ of the family with $Z_{\ol{\tau}}^m$ as being \emph{the number of $m^{\mathrm{th}}$-order weight-$1$ inflection points limiting toward the ICIS in a general $1$-parameter deformation}. The following proposition implies that $$\dim_k \scr{O}_X/(I_{X_\eta} + I_{\ol{\tau}}^m) =  \on{AD}_{(1)}^m(f),$$
thus proving the formula~\eqref{eq-fund3}.
\begin{proposition} \label{prop-flatness1}
The restriction $\pi|_{Z_{\ol{\tau}}^m}$ of the map $\pi \colon X \to B$ to $Z_{\ol{\tau}}^m$ is finite and flat. In particular, the fiber multiplicity of $\pi|_{Z_{\ol{\tau}}^m}$ is constant.
\end{proposition}
\begin{proof}
We first handle flatness. Notice that $Z_{\ol{\tau}}^m$ is Cohen-Macaulay because it is a determinantal subscheme having the expected codimension in the scheme $\Spf \scr{O}_X$, which is itself Cohen-Macaulay because $X$ is Cohen-Macaulay. That the map $\pi|_{Z_{\ol{\tau}}^m}$ is flat over $B$ now follows from Miracle Flatness (see~\cite[Theorem 23.1]{MR1011461}), because it has finite fibers (since $\dim_k \scr{O}_X/(I_{X_{0}} + I_{\ol{\tau}}^m) < \infty$) and because its target space $B$ is regular.

We now handle finiteness. Because the map $\pi|_{Z_{\ol{\tau}}^m}$ has finite fibers and is evidently of finite type, it is quasifinite. By the geometric version of the Weierstrass Preparation Theorem (see~\cite[Th\'{e}or\`{e}me~1]{weierprep}), it follows that $\pi|_{Z_{\ol{\tau}}^m}$ is finite.
\end{proof}

We next consider the case of weight-$2$ type-(a) inflection points (the type-(b) case is similar). Let $f$ be the germ of an ICIS, and suppose that $\on{AD}_{(2)}^m(f) < \infty$. Consider a general $2$-parameter deformation $\wt{\pi} \colon \wt{X}\to \wt{B}$, let the inclusion $\wt{B} \subset \mathfrak{B}$ be given by a map $\phi \colon \scr{O}_{\mathfrak{B}} \to \scr{O}_{\wt{B}} = k[[s,t]]$, and let $X/B$ be the $1$-parameter deformation given by post-composing $\phi$ with the projection map $k[[s,t]] \to k[[t]]$ sending $s \mapsto 0$. Letting $B^\perp \subset \wt{B}$ be the line given by $t = 0$, we view $\wt{B}$ as a family over $B^\perp$ via the projection map $k[[s,t]] \to k[[s]]$ sending $t \mapsto 0$. Then $B$ is the fiber of $\wt{B}/B^\perp$ over $s = 0$, and $X/B$ is the subfamily of curves lying above this fiber. Also, let $\eta \in B^\perp$ denote the geometric generic point, let $B' \coloneqq \wt{B}\times_{B^\perp} \eta$ over $\eta$, and let $X' \coloneqq \wt{X} \times_{\wt{B}} B'$. Note that the family $X'/B'$ can be thought of as a general perturbation of the family $X/B$.

Let $\ol{\tau}$ be a list of $m-1$ general elements of $\scr{P}_{\wt{X}/\wt{B}}^m$, let $\wt{Z}_{\ol{\tau}}^m \subset \Spf \scr{O}_{\wt{X}}$ denote the corresponding degeneracy locus, cut out by the ideal $\wt{I}_{\ol{\tau}}^m \subset \scr{O}_{\wt{X}}$. 
Then $$ \dim_k \big((\scr{O}_{\wt{X}}/\wt{I}_{\ol{\tau}}^m) \otimes_{\scr{O}_{\wt{B}}} \scr{O}_B \big) = \dim_k \scr{O}_X/(\wt{I}_{\ol{\tau}}^m \otimes_{\scr{O}_{\wt{B}}} \scr{O}_B) = \on{AD}_{(2)}^m(f).$$
The following proposition implies that $$ \dim_k \big((\scr{O}_{\wt{X}}/\wt{I}_{\ol{\tau}}^m) \otimes_{\scr{O}_{\wt{B}}} \scr{O}_{B'} \big) =  \on{AD}_{(2)}^m(f).$$

\begin{proposition} \label{prop-flatness2}

Let $\ol{\pi}$ denote the composition of the map $\wt{\pi} \colon \wt{X} \to \wt{B}$ with the projection map $\wt{B} \to B^\perp$. Then the restriction $\ol{\pi}|_{\wt{Z}_{\ol{\tau}}^m}$ of the map $\ol{\pi} \colon \wt{X} \to B^\perp$ to $\wt{Z}_{\ol{\tau}}^m$ is finite and flat. In particular, the fiber multiplicity of $\ol{\pi}|_{\wt{Z}_{\ol{\tau}}^m}$ is constant.
 \end{proposition}
 \begin{proof}
 We omit the proof because it is essentially identical to that of Proposition~\ref{prop-flatness1}.
\end{proof}

We interpret the total multiplicity of $(\wt{Z}_{\ol{\tau}}^m \times_{\wt{B}} B') \setminus \on{Sing}(\wt{X}/\wt{B})$ as being \emph{the number of $m^{\mathrm{th}}$-order weight-$2$ type-(a) inflection points limiting toward the ICIS in a general $2$-parameter deformation}. Since the total multiplicity of $\wt{Z}_{\ol{\tau}}^m \times_{\wt{B}} B'$ is given by $\dim_k \big((\scr{O}_{\wt{X}}/\wt{I}_{\ol{\tau}}^m) \otimes_{\scr{O}_{\wt{B}}} \scr{O}_{B'} \big) = \on{AD}_{(2)}^m(f)$, all that remains to prove~\eqref{Eq:FundamentalFormula} is to show that the total multiplicity of $(\wt{Z}_{\ol{\tau}}^m \times_{\wt{B}} B') \cap \on{Sing}(\wt{X}/\wt{B})$ is equal to $(\on{mult}_0 \Delta_f) \cdot \on{AD}_{(2)}^m(xy)$. But this follows because $(\wt{Z}_{\ol{\tau}}^m \times_{\wt{B}} B') \cap \on{Sing}(\wt{X}/\wt{B})$ consists of $(\on{mult}_0 \Delta_f)$-many points, each of which is a node, and the multiplicity of $\wt{Z}_{\ol{\tau}}^m \times_{\wt{B}} B'$ at each of these nodes is $\on{AD}_{(2)}^m(xy)$ because the list $\ol{\tau}$ was chosen to be general. We have thus proven~\eqref{Eq:FundamentalFormula}; the proof of~\eqref{eq-fund2} is entirely analogous.

The multiplicity $\on{mult}_0 \Delta_f$ can be easily computed in terms of Milnor numbers, about which we require the following two facts. The first is known as the \emph{L\^{e}-Greuel formula} and provides an easy way of inductively computing the Milnor number:

\begin{theorem}[\protect{\cite[Theorem 3.7.1]{dongtran} and~\cite[Korollar 5.5]{greuel}}] \label{def-milnnumber}
    Let $f \colon (k^N,0) \to (k^\ell,0)$ be the germ of an isolated complete intersection singularity with component functions $f_1, \dots, f_\ell \in \widehat{\scr{O}}_{k^N,0} \simeq k[[x_1, \dots, x_N]]$. Letting $J(f)$ be the Jacobian ideal of $f$ (which is generated by the maximal minors of the matrix of first-order partial derivatives of the $f_i$), the following formulas hold:
    \begin{enumerate}
        \item When $\ell = 1$, we have that $$\mu_f = \dim_k k[[x_1, \dots, x_N]]/J(f).$$
        \item For $\ell \geq 2$, suppose that $f' \colon (k^N,0) \to (k^{\ell-1},0)$ with component functions $f_1, \dots, f_{\ell-1}$ is also the germ of an isolated complete intersection singularity. Then we have that
        $$\mu_{f} + \mu_{f'} = \dim_k k[[x_1, \dots, x_N]]/(f_1, \dots, f_{\ell-1},J(f')).$$
    \end{enumerate}
\end{theorem}

The second fact is that the number of nodes ``nearby'' an ICIS in a general $2$-parameter deformation can be expressed in terms of Milnor numbers.

\begin{proposition}[\protect{\cite[Proposition 3.6.4]{dongtran}}] \label{prop-milnorsum}
   Consider an ICIS germ $f$. The multiplicity of the discriminant locus at the ICIS in a general $2$-parameter deformation is given by $\mu_f + \mu_{f'}$, where $f'$ is the germ of the total space of a general $1$-parameter deformation of $f$.
\end{proposition}

Since a general $1$-parameter deformation of a \emph{planar} ICIS cut out analytically-locally by $f = 0$ has a smooth total space, it follows from Proposition~\ref{prop-milnorsum} that $\on{mult}_0 \Delta_f = \mu_f$ in the planar case.

\subsection{Relationship to Well-Known Invariants and Multiplicities} \label{sec-invarrels}

It is natural to wonder whether automatic degeneracies are related in some meaningful way to invariants and multiplicities of curve singularities that arise in the literature. In this section, we provide a brief discussion of a number of such relationships.

\subsubsection{The Milnor Number} We have already observed that there is a geometric interpretation of the relationship between the weight-$2$ automatic degeneracies and the Milnor number. Indeed,~\eqref{Eq:FundamentalFormula} and~\eqref{eq-fund2}, together with Proposition~\ref{prop-milnorsum}, tell us that the number of $m^{\mathrm{th}}$-order weight-$2$ inflection points of a given type limiting toward an ICIS in a general $2$-parameter deformation admits a simple closed-form expression in terms of the order $m$, the corresponding $m^{\mathrm{th}}$-order weight-$2$ automatic degeneracy, and the Milnor numbers of the ICIS and of the total space of a general $1$-parameter deformation. Moreover, we shall demonstrate in \S~\ref{sec-boundit} how to obtain lower and upper bounds on the $m^{\mathrm{th}}$-order weight-$2$ automatic degeneracies of a planar ICIS in terms of the order $m$ and its Milnor number. The Milnor number also turns out to be equal to certain weight-$2$ automatic degeneracies of planar ICISs, as we show in Theorem~\ref{thm-autodeg112ismiln} and Remark~\ref{rem-milnrefer}.

\subsubsection{The Widland-Lax Multiplicity} \label{sec-wl}

By~\eqref{eq-fund3}, the $m^{\mathrm{th}}$-order weight-$1$ automatic degeneracy of an ICIS is equal to the number of $m^{\mathrm{th}}$-order weight-$1$ inflection points limiting toward the ICIS in a general $1$-parameter deformation. But this automatic degeneracy can also be interpreted as measuring the number of $m^{\mathrm{th}}$-order weight-$1$ inflection points that the ICIS ``counts as.'' Incidentally, the problem of determining how many $m^{\mathrm{th}}$-order weight-$1$ inflection points that a given curve singularity ``counts as'' has been studied extensively by Widland and Lax, who authored several papers devoted to extending the theory of inflection points to Gorenstein singular curves (see~\cite{lw2,lw1,lw3,lw4,lw5,lw6,MR1038736}). The work of Widland and Lax on Gorenstein curves has been extended to the case of arbitrary integral singular curves by E.~Ballico and Gatto in~\cite{ballicogatto} and to fields $k$ of arbitrary characteristic by Laksov and Thorup in~\cite{MR1318539}.

We now briefly describe the work of Widland and Lax concerning inflection points on Gorenstein singular curves, following the detailed exposition provided in~\cite[\S~4]{gattoricolfi}. Let $C$ be an integral projective Gorenstein curve over $k$, and let $p \in C(k)$ be any point (smooth or singular). Because we have stipulated that $C$ is Gorenstein, the dualizing sheaf $\omega_C \coloneqq \omega_{C/\on{Spec} k}$ is invertible, so the stalk $\omega_{C,p}$ is a free $\scr{O}_{C,p}$-module generated by an element $\sigma_p \in \omega_{C,p}$. Now, let $(\scr{L}, W)$ be a linear system on $C$, and let $(\sigma_1, \dots, \sigma_m)$ be a basis of $W$. By identifying the stalk $\scr{L}_p$ (which is a free $\scr{O}_{C,p}$-module of rank $1$) with $\scr{O}_{C,p}$, we may regard each $\sigma_i$ as an element of $\scr{O}_{C,p}$. Letting
$$d \colon \scr{O}_C \to \Omega_C^1 \to \omega_C$$
be the composition of the universal derivation $\scr{O}_C \to \Omega_C^1$ with the natural map $\Omega_C^1 \to \omega_C$, define elements $\sigma_i^{(j)} \in \scr{O}_{C,p}$ for each $i \in \{1, \dots, m\}$ and $j \in \{0, \dots, m-1\}$ by
$$\sigma_i^{(0)} = \sigma_i \in \scr{O}_{C,p} \quad \text{and} \quad d(\sigma_i^{(j-1)}) = \sigma_i^{(j)} \cdot \sigma_p \in \scr{\omega}_{C,p}.$$
Then the \emph{Widland-Lax multiplicity} $\on{WL}_{p}(\scr{L},W)$ of the linear system $(\scr{L}, W)$ at $p$ is defined to be the order of vanishing at $p$ of the determinant of the following $m \times m$ matrix, which can be thought of as a generalization of the Wronskian:
$$\left[\begin{array}{ccc} \sigma_1^{(0)} & \cdots & \sigma_m^{(0)} \\ \vdots & \ddots & \vdots \\ \sigma_1^{(m-1)} & \cdots & \sigma_m^{(m-1)} \end{array}\right]$$
The point $p$ is defined to be a \emph{Widland-Lax inflection point} of the linear system $(\scr{L},W)$ if $\on{WL}_p(\scr{L},W) > 0$. If $p$ is a smooth point of $C$, then $\on{WL}_p(\scr{L},W)$ is equal to the weight of $p$ as an inflection point.

Now consider an admissible $1$-parameter family of curves $\pi \colon X \to B$ of arithmetic genus $g$ with a point $p \in X$ such that the fiber $X_{\pi(p)}$ is irreducible and has an isolated singularity at $p$, and let $(\scr{L}, \scr{E})$ be a linear system on the family. Then the locus of Widland-Lax inflection points $p' \in X(k)$, counted with multiplicity $\on{WL}_{p'}(\scr{L}|_{X_{\pi(p')}},\scr{E}|_{\pi(p')})$, forms a divisor on $X$ that is finite and flat over the base $B$. There are two consequences: firstly, $\on{WL}_{p'}(\scr{L}|_{X_{\pi(p')}},\scr{E}|_{\pi(p')})$ is equal to the number of weight-$1$ inflection points limiting toward the point $p'$, and secondly, the so-called \emph{total inflection}
$$\sum_{p' \in X_b(k)} \on{WL}_{p'}(\scr{L}|_{X_{b}},\scr{E}|_{b})$$ of the fiber $X_b$ is independent of the choice of the point $b \in B(k)$. In the case where the fiber $X_b$ is smooth, the total inflection is given by the \emph{Pl\"{u}cker formula} (see~\cite[Theorem 7.13]{harris3264}), which states that
\begin{equation} \label{eq-genpluck}
    \sum_{p' \in X_b(k)} \on{WL}_{p'}(\scr{L}|_{X_{b}},\scr{E}|_{b}) = m \cdot \on{deg} c_1(\scr{L}|_{X_b}) + m(m-1)\cdot (g-1). \footnote{The formula~\eqref{eq-genpluck} is referred to as the \emph{Brill-Segre formula} in~\cite{gattoricolfi}.}
\end{equation}
It follows that the formula~\eqref{eq-genpluck} holds for all $b \in B(k)$, regardless of whether the fiber $X_b$ is smooth. In particular, the formula~\eqref{eq-genpluck} holds for any linear system on any Gorenstein curve (this result is~\cite[Proposition 1]{MR1038736}).

We now return to the context of a single integral projective Gorenstein curve $C$ of genus $g$ over $k$. In~\cite[\S~4.2]{gattoricolfi}, Gatto and Ricolfi show how to use the Pl\"{u}cker formula~\eqref{eq-genpluck} to compute $\on{WL}_p(\scr{L},W)$ for an isolated singular point $p \in C(k)$. Let $\nu \colon \wt{C} \to C$ be a partial normalization of the curve $C$ on a neighborhood of the point $p$. Then the arithmetic genus of $\wt{C}$ is equal to $g - \delta_p$, where $\delta_p$ is the $\delta$-invariant of the singularity at $p$ (see \S~\ref{sec-bounditbelow} for the definition of the $\delta$-invariant, which can be thought of as measuring the number of double points that a singularity ``counts as''). Let $\wt{W}$ be the $k$-vector subspace of $H^0(\nu^*\scr{L})$ spanned by the pullbacks of sections in $W$. Using the Pl\"{u}cker formula~\eqref{eq-genpluck} together with the fact that the map $\nu$ is an isomorphism away from $\nu^{-1}(p)$, we obtain the following equalities, which are stated in~\cite[Proof of Proposition 4.8]{gattoricolfi}:
\begin{align}
\on{WL}_p(\scr{L},W) & = \sum_{p' \in C(k)} \on{WL}_{p'}(\scr{L},W) - \sum_{p' \in \wt{C}(k)} \on{WL}_{p'}(\nu^* \scr{L}, \wt{W}) + \sum_{p' \in \nu^{-1}(p)} \on{WL}_{p'}(\nu^* \scr{L}, \wt{W}) \nonumber \\
& =  \delta_p \cdot m (m-1)  + \sum_{p' \in \nu^{-1}(p)} \on{WL}_{p'}(\nu^* \scr{L}, \wt{W}). \label{eq-gattoricolfthm}
\end{align}
Because the points $p' \in \nu^{-1}(p)$ are smooth points of $\wt{C}$, it is easy to compute $\on{WL}_{p'}(\nu^* \scr{L},W)$ for each of these points in any given example. Moreover, the $\delta$-invariant can be readily computed by means of the Milnor-Jung formula (see~\cite[Theorem 10.5]{milne}). Thus,~\eqref{eq-gattoricolfthm} provides an easy-to-use formula for computing the Widland-Lax multiplicity.

One might wonder how the Widland-Lax multiplicity compares to the weight-$1$ automatic degeneracy. The fundamental distinction between these two notions is that the Widland-Lax multiplicity is defined globally and depends on the specific curve and linear system under consideration, whereas the $m^{\mathrm{th}}$-order weight-$1$ automatic degeneracy of an ICIS is defined locally and depends only on the order $m$ and the analytic isomorphism class of the ICIS. It is natural to ask under what conditions we have the equality
\begin{equation} \label{eq-adwall}
\on{AD}_{(1)}^m(f) \overset{\text{?}}= \on{WL}_p(\scr{L},W)
\end{equation}
where $(\scr{L},W)$ is a linear system on $C$ with $\dim_k W = m$. As it happens, equality does hold in many examples, as we demonstrate in Remarks~\ref{rem-wladnode},~\ref{rem-wlad12}, and~\ref{rem-cuspAD1dontknow}, although in the last of these remarks, we also discuss an example of when equality fails to hold. In addition, note that it follows from~\eqref{eq-gattoricolfthm} that we have the lower bound
\begin{equation} \label{eq-wlbound}
\on{WL}_p(\scr{L},W) \geq \delta_p \cdot m(m-1).
\end{equation}
In \S~\ref{sec-bounditbelow}, we show that the bound~\eqref{eq-wlbound} holds upon replacing the Widland-Lax multiplicity with the $m^{\mathrm{th}}$-order weight-$1$ automatic degeneracy. Finally, it remains open to study whether the constructions of Widland-Lax can be used to find global analogues of the two types of weight-$2$ automatic degeneracies.

\subsubsection{The Buchsbaum-Rim Multiplicity} \label{sec-booksrim}

Let $M$ be a proper submodule of finite colength in a free module $F$ of rank $r$ over a Noetherian local ring $A$ of dimension $d$ over $k$. For each integer $n \geq 0$, let $\mc{S}_n$ denote the $n^{\mathrm{th}}$ graded component of the symmetric algebra of $F$, and let $\mc{R}_n \subset \mc{S}_n$ denote the $n^{\mathrm{th}}$ graded component of the \emph{Rees algebra} of $M$, which is equal to the $A$-subalgebra generated by $M$ in the symmetric algebra of $F$. In~\cite[\S~3.1]{MR0159860}, D.~Buchsbaum and D.~Rim proved that the quantity $\dim_k (\mc{S}_n/\mc{R}_n)$ is equal to a polynomial of degree $d+r-1$ in $n$ for all sufficiently large $n$. The coefficient of $\frac{1}{(d+r-1)!} \cdot n^{d+r-1}$ in this polynomial is a quantity known as the \emph{Buchsbaum-Rim multiplicity} of $F/M$ and is denoted $\on{BR}(F/M)$. The Buchsbaum-Rim multiplicity can be thought of as a generalization for modules of the \emph{Hilbert-Samuel multiplicity} for ideals.

Consider an ICIS cut out analytically-locally by $f = 0$, let $B$ be a $1$-parameter deformation, and take $A = \scr{O}_X$, $M = \scr{P}_{X/B}^m$, and $F = {\scr{P}_{X/B}^m}^{\vee\vee}$. Then we have $r = \on{rk} F = m$ and $d = \dim A = 2$. Given general elements $\tau_1, \dots, \tau_{m+1} \in M$,~\cite[part (iii) of Proposition 2.5]{gkm1} tells us that because $A$ is Cohen-Macaulay, the submodule $U \coloneqq \sum_{i = 1}^{m+1} A \cdot \tau_i \subset M$ is a \emph{minimal reduction} of $M$, meaning that $U$ is minimal among all submodules $U' \subset M$ with the property that the Rees algebra of $M$ is integral over the Rees algebra of $U'$. Then, it follows from~\cite[Theorem 1.2]{BRFitt0} that we have the equality of Buchsbaum-Rim multiplicities
\begin{equation} \label{eq-br}
\on{BR}(F/M) = \on{BR}(F/U),
\end{equation}
and that the quantities in~\eqref{eq-br} are both equal to the colength $\dim_k A/I_{\ol{\tau}}^m$. Thus, for a general choice of the $1$-parameter deformation $B$, we have the equality
\begin{equation} \label{eq-br2}
\on{AD}_{(1,1)}^m(f) = \on{BR}({\scr{P}_{X/B}^m}^{\vee\vee}/\scr{P}_{X/B}^m).
\end{equation}

As it happens, the weight-$2$ type-(a) automatic degeneracy can also be expressed as a Buchsbaum-Rim multiplicity, albeit in a far more contrived way. Indeed, let $F$ be a free $A$-module of rank $r = m-1$, and let $(e_1, \dots, e_{m-1})$ be a basis of $F$. Let $\tau_1, \dots, \tau_{m-1} \in \on{P}^m(f)$ be general elements, and consider the submodule $$U'' \coloneqq \sum_{j = 1}^m A \cdot \left(\sum_{i = 1}^{m-1} \theta_j(\tau_i) \cdot e_i\right) \subset F.$$ Then we have the equality
\begin{equation} \label{eq-brinfty}
\on{AD}_{(2)}^m(f) = \on{BR}(F/U'').
\end{equation}
It remains open as to whether the weight-$2$ type-(a) automatic degeneracy can be expressed in an ``intrinsic'' way in terms of the local principal parts module as we managed to do for the type-(b) case in~\eqref{eq-br2}.

\begin{remark} \label{rem-booksrim}
   The Buchsbaum-Rim multiplicities in~\eqref{eq-br2} and~\eqref{eq-brinfty} are \emph{not} the same as the Buchsbaum-Rim multiplicity that arises in the work of Gaffney, Kleiman, and Massey on equisingularity theory (see the series of papers~\cite{gkm2,gkm1,gkm4,gkm3}). Indeed, while the multiplicities in~\eqref{eq-br2} and~\eqref{eq-brinfty} depend on partial derivatives of the component functions of $f$ of order up to $m-1$ and are \emph{not} in general equisingularity invariants (as we explain in Remark~\ref{rem-notequi}), the multiplicity considered by Gaffney \emph{et al.~}depends only on the first-order partial derivatives of the component functions of $f$ and is an equisingularity invariant. There is also the obvious distinction that the multiplicities in~\eqref{eq-br2} and~\eqref{eq-brinfty} form a countable collection for each ICIS as $m$ varies through the positive integers, whereas Gaffney \emph{et al.~}associate a single multiplicity to each ICIS.
\end{remark}

\subsubsection{Other Invariants and Multiplicities}

In Theorem~\ref{thm-secordcodim1}, Example~\ref{eg-toss}, and Theorem~\ref{thm-tess}, we show that for any planar ICIS cut out analytically-locally by $f = 0$, the $2^{\mathrm{nd}}$-order weight-$1$ automatic degeneracy is equal to two well-known multiplicities: (1) the intersection multiplicity of the singularity with a \emph{generic polar}, which is related to B.~Teissier's notion of \emph{polar invariant} (see~\cite{tesspolarvars}), and (2) the Hilbert-Samuel multiplicity of the Jacobian ideal of $f$.

We conclude by remarking that it remains open to find other interesting relationships between automatic degeneracies and well-known singularity invariants and multiplicities, and that there appears to be much room for further investigation in this direction.

\section{Automatic Degeneracies of a Node} \label{sec-calc}

In this section, we compute the automatic degeneracies of a node. To do this, we employ the following three-step procedure:
\begin{enumerate}
  \item Find a basis of the dual of the complete module of relative principal parts;
  \item Compute the minors of the degeneracy matrix using the basis obtained in part (a);
  \item Compute the colength of the degeneracy ideal associated to the matrix obtained in part (b).
\end{enumerate}
As we explain in \S~\ref{sec-getsworse}, it is in general quite difficult, if not impossible, to use the above procedure for a given ICIS to compute its automatic degeneracies as explicit functions of $m$. Nevertheless, we demonstrate in this section that this procedure can be executed in the case of a node.

\begin{theorem} \label{thm-main2}
We have the following two results:
\begin{enumerate}
\item The weight-$1$ $m^{\mathrm{th}}$-order automatic degeneracy of a node, cut out analytically-locally by $xy = 0$, is given by the formula
$$\on{AD}_{(1)}^m(xy) = m(m-1),$$
so in a general $1$-parameter deformation of a node, the number of $m^{\mathrm{th}}$-order weight-$1$ inflection points limiting to the node is given by $m(m-1)$.
\item The weight-$2$ $m^{\mathrm{th}}$-order automatic degeneracies of a node, cut out analytically-locally by $xy = 0$, are given by the formulas
$$\on{AD}_{(2)}^m(xy) = {{m+1} \choose {4}} \quad \text{and} \quad \on{AD}_{(1,1)}^m(xy) = {{m+2} \choose {4}}$$
\end{enumerate}
\end{theorem}
\begin{remark} \label{rem-wladnode}
   Let $C$ be a projective integral Gorenstein curve with a node at a point $p \in C(k)$, and let $(\scr{L},W)$ be a linear system on $C$ such that $\dim_k W = m$ and such that the sequence of orders of vanishing of sections in $W$ along each branch of the node at $p$ is given by $(0, 1, 2, \dots, m-1)$. Then from~\eqref{eq-gattoricolfthm}, we deduce that $\on{WL}_p(\scr{L},W) = \delta_p \cdot m(m-1) = m(m-1)$, because the $\delta$-invariant of a node is equal to $1$ and because the Widland-Lax multiplicities of the two preimages of the node in the normalization are both $0$. Thus, in this case, it follows from part (a) of Theorem~\ref{thm-main2} that equality holds in~\eqref{eq-adwall}; i.e., we have that $\on{AD}_{(1)}^m(xy) = \on{WL}_p(\scr{L},W)$.
\end{remark}
\begin{remark}
   Note that in a general $2$-parameter deformation of a node, it follows from the formulas~\eqref{Eq:FundamentalFormula} and~\eqref{eq-fund2} that the number of $m^{\mathrm{th}}$-order weight-$2$ inflection points of either type limiting to the node is equal to $0$.
\end{remark}
\begin{proof}[Proof of Theorem~\ref{thm-main2}]

Let $X/B$ be a general $1$-parameter deformation of a planar ICIS germ $f \in k[[x,y]]$. Since we consider various different types of singularities in the sequel, we introduce the clearer notation $\on{P}^m(f) \coloneqq \scr{P}_{X/B}^m$ for the modules of principal parts. The first step is to find a basis of $\on{P}^m(f)^{\vee}$ that is ``nice enough'' to make it feasible to compute the automatic degeneracies of a node for every $m$. Before we can do this, however, we need a more explicit description of $\on{P}^m(f)$. Note that $\scr{O}_{B} = k[[t]]$ and that $\scr{O}_X$ is the $\scr{O}_{B}$-algebra given by
\begin{equation*}
\mathscr{O}_{X} = k[[x,y]][[t]]/\left(\textstyle f - t \cdot \sum_i a_i \cdot \delta_i(f)\right)
\end{equation*}
where the list $(\delta_i(f))_i \subset k[[x,y]]$ is a basis of the $k$-vector space $\on{Def}^1(f)$ and $a_i \in k$ are general. Note that $\sum_i a_i \cdot \delta_i(f) \in (k[[x,y]])^\times$, so the relation $f - t \cdot \sum_i a_i \cdot \delta_i(f) = 0$ defining $\mathscr{O}_{X}$ as a quotient of $k[[x,y]][[t]]$ can be rearranged to obtain the relation $t = \left( \sum_i a_i \cdot \delta_i(f)\right)^{-1} \cdot f$ expressing $t$ in terms of $x,y$. Using this relation, we obtain the following explicit description of $\on{P}^m(f)$:
\begin{align}
\scr{P}_{X/B}^m & = \big(\scr{O}_X \otimes_{\scr{O}_{B}} \scr{O}_X\big)/I_\Delta^m \nonumber \\
& = \left(\frac{k[[x,y]][[t]]}{\left(f(x,y) - t \cdot \sum_i a_i \cdot \delta_i(f)(x,y)\right)} \otimes_{k[[t]]} \frac{k[[u,v]][[t]]}{\left(f(u,v) - t \cdot \sum_i a_i \cdot \delta_i(f)(u,v)\right)}\right)\bigg/(u-x,v-y)^m \nonumber \\
& = R[[u,v]]/\left(\textstyle \left(\sum_i a_i \cdot \delta_i(f)(u,v)\right)^{-1} \cdot f(u,v) - \left(\sum_i a_i \cdot \delta_i(f)(x,y)\right)^{-1} \cdot f(x,y), (u-x,v-y)^m\right) \label{eq-defPm(f)} \raisetag{-1.4\normalbaselineskip}
\end{align}
where we have put $R \coloneqq k[[x,y]]$ for ease of notation. The expression of $\on{P}^m(f)$ given in~\eqref{eq-defPm(f)} will come in handy in \S~\ref{sec-getsworse}, where we discuss automatic degeneracies of arbitrary planar singularities. But $\on{Def}^1(xy)$ is a $1$-dimensional $k$-vector space, generated by $1 \in k[[x,y]]$. Thus, in the case $f = xy$,~\eqref{eq-defPm(f)} takes on the much simpler form
$$\on{P}^m(xy) =  R[[u,v]]/(uv - xy, (u-x,v-y)^m).$$

\subsection{Finding a Basis of \texorpdfstring{$\on{P}^m(xy)^{\vee}$}{Pm(xy)v}} \label{sec-basisnode}

We now execute step (a) of the procedure outlined above. The following lemma tells us that functionals in $\on{P}^m(xy)^\vee$ satisfy \mbox{a handy property.}

\begin{lemma}\label{lem-dualsend}
  Let $m$ be a positive integer, and let $\phi \in \on{P}^m(xy)^\vee$. For every $i \in \{0, \dots, m\}$, we have $x^i \mid \phi(u^i)$ and $y^i \mid \phi(v^i)$.
\end{lemma}
\begin{proof}
The lemma is obvious when $i = 0$. For convenience, let the relation $(u-x)^{m-i}(v-y)^i$ be denoted by $R_i$ for each $i \in \{0, \dots, m\}$. Next, observe that every term other than $(-x)^{m-1} \cdot v$ in relation $R_1$ contains a factor of $y$, so $$y \mid \phi(R_1 - (-x)^{m-1} \cdot v) = \phi(0 - (-x)^{m-1} \cdot v) = (-x)^{m-1}\cdot \phi(v).$$ It follows that $y \mid \phi(v)$, so the lemma holds when $i = 1$. Further observe that every term other than $(-x)^{m-2} \cdot v^2$ in relation $R_2$ either contains a factor of $y^2$ or contains a factor of $y \cdot v$, so $$y^2 \mid \phi(R_2 - (-x)^{m-2} \cdot v^2) = \phi(0 - (-x)^{m-2} \cdot v^2) = (-x)^{m-2}\cdot \phi(v^2).$$
It follows that $y^2 \mid \phi(v^2)$, so the lemma holds when $i = 2$. Continuing in this manner by inductively assuming that, for some $j \in \{0, \dots, m-1\}$, the lemma holds for every $i \in \{0, \dots, j\}$, one can use relation $R_{j+1}$ to deduce that $y^{j+1} \mid \phi(v^{j+1})$. It follows that $y^i \mid \phi(v^i)$ for every $i \in \{0, \dots, m\}$. Since the setup is symmetric under $(x,u) \leftrightarrow (y,v)$, the same argument demonstrates that $x^i \mid \phi(u^i)$ for every $i \in \{0, \dots, m\}$.
\end{proof}

In the next lemma, we use Lemma~\ref{lem-dualsend} to construct a basis of $\on{P}^m(xy)^\vee$.

\begin{lemma}\label{lem-dualbasis}
For each $i \in \{0, \dots, m-1\}$, there exists a unique functional $e_i \in \on{P}^m(xy)^\vee$ with the following two properties:
\begin{enumerate}
\item $e_i(u^j) = \delta_{ij} \cdot x^j$ for each $j \in \{0, \dots, m-1\}$; and
\item $d_{ij} \coloneqq e_i(v^j)/y^j \in R^\times$ for each $j \in \{1, \dots, m\}$.
\end{enumerate}
 Moreover, the list $(e_0, \dots, e_{m-1})$ forms a basis of $\on{P}^m(xy)^\vee$ as an $R$-module.
\end{lemma}
\begin{proof}
Observe that specifying a map $R[[u,v]]/(uv-xy) \to R$ is equivalent to specifying the images of the powers of $u$ and $v$. For each $i \in \{0, \dots, m-1\}$, let $\wt{e}_i \colon R[[u,v]]/(uv-xy) \to R$ be any map satisfying the condition that $\wt{e}_i(u^j) = \delta_{ij} \cdot x^j$ for each $j \in \{0, \dots, m-1\}$. In order for $\wt{e}_i$ to descend to a map $e_i \colon \on{P}^m(xy) \to R$, the condition $\wt{e}_i(R_\ell) = 0$ must be satisfied for each $\ell \in \{0, \dots, m\}$. We claim that the condition $\wt{e}_i(R_0) = 0$ merely serves to specify the value of $\wt{e}_i(u^m)$. To see why this claim holds, observe that
\begin{align*}
 \wt{e}_i(R_0) & = 0 \Longleftrightarrow  \sum_{j = 0}^{m} \mybinom[0.7]{m}{j} \cdot (-1)^{m-j}\cdot x^{m-j} \cdot \wt{e}_i(u^j) = 0 \Longleftrightarrow \\
 \wt{e}_i(u^m) & = \sum_{j = 0}^{m-1} \mybinom[0.7]{m}{j} \cdot (-1)^{m-j+1}\cdot x^{m-j} \cdot \wt{e}_i(u^j) = \mybinom[0.7]{m}{i} \cdot (-1)^{m-i+1}\cdot x^m.
\end{align*}
Thus, $\wt{e}_i$ satisfies the condition $\wt{e}_i(R_0) = 0$ if and only if $\wt{e}_i(u^m)$ is given as above. In much the same manner, the condition $\wt{e}_i(R_1) = 0$ determines the value of $\wt{e}_i(v)$; indeed, notice that we have
\begin{align*}
 & \wt{e}_i(R_1) = 0 \Longleftrightarrow \\
 & (-1)^{m-1} \cdot x^{m-1} \cdot \wt{e}_i(v) + \left(\sum_{j = 1}^{m-1} \mybinom[0.7]{m-1}{j} \cdot (-1)^{m-1-j} \cdot x^{m-j}y \cdot \wt{e}_i(u^{j-1})\right) + \\
 & \quad \left(\sum_{j = 0}^{m-1} \mybinom[0.7]{m-1}{j} \cdot (-1)^{m-j} \cdot x^{m-1-j}y \cdot \wt{e}_i(u^j)\right) = 0 \Longleftrightarrow \\
 & \wt{e}_i(v)  =  (-1)^i \cdot \mybinom{m}{i+1} \cdot y.
\end{align*}
We can continue in this manner by using the condition $\wt{e}_i(R_{\ell+1})$ and the already-specified values of $\wt{e}_i(v^n)$ for $n \in \{1, \dots, \ell\}$ to solve for $\wt{e}_i(v^{\ell+1})$. After much laborious computation, it follows by strong induction that
\begin{align*}
& \wt{e}_i(R_\ell) = 0 \text{ for all $\ell \in \{1, \dots,m \}$} \Longleftrightarrow \\
& \wt{e}_i(v^\ell) = (-1)^i \cdot \frac{ \ell(m-i)}{m(\ell+i)} \cdot \mybinom[0.7]{m}{i} \cdot \mybinom[0.7]{m+\ell-1}{\ell} \cdot y^\ell \text{ for all $\ell \in \{1, \dots,m \}$}.
\end{align*}
Notice in particular that $d_{i\ell} = \wt{e}_i(v^\ell)/y^\ell \in \BZ \setminus\{0\} \subset R^\times$ for all choices of $i$ and $\ell$. With the values specified as above, the maps $\wt{e}_i \colon R[[u,v]]/(uv-xy) \to R$ satisfy the conditions $\wt{e}_i(R_\ell) = 0$ and therefore descend to maps $e_i \colon \on{P}^m(xy) \to R$. Moreover, since the maps $\wt{e}_i$ satisfy points (a) and (b) in the statement of the lemma, so do the maps $e_i$. Finally, because the elements $u^\ell, v^\ell$ for $\ell \in \{0, \dots, m\}$ generate $\on{P}^m(xy)$, and because we have specified the values of $e_i(u^\ell)$ and $e_i(v^\ell)$ for each $\ell$, it follows that we have completely determined the maps $e_i$.

It is evident that the list $(e_0, \dots, e_{m-1})$ is linearly independent, so it remains to check that this list spans all of $\on{P}^m(xy)^\vee$. Let $\phi \in \on{P}^m(xy)^\vee$ be any element; observe by Lemma~\ref{lem-dualsend} that there exist $a_i \in R$ such that $\phi(u^i) = a_i \cdot x^i$ for each $i \in \{0, \dots, m-1\}$. Then the functional $\psi = \phi - \sum_{i = 0}^{m-1} a_i \cdot e_i \in \on{P}^m(xy)^\vee$ has the property that $\psi(u^i) = 0$ for every $i \in \{0, \dots, m-1\}$. Inductively tracing through the relations $\psi(R_i) = 0$ as we did in the previous paragraph, we deduce that $\psi(u^m) = 0$ and that $\psi(v^\ell) = 0$ for every $\ell \in \{1, \dots, m\}$, so in fact $\psi$ is the zero functional, and we have $\phi = \sum_{i = 0}^{m-1} a_i \cdot e_i$, implying that the list $(e_0, \dots, e_{m-1})$ spans all of $\on{P}^m(xy)^\vee$.
\end{proof}

\subsection{The Weight-$2$ Case}

\subsubsection{Step (b): Computing the Minors} \label{sec-firststepb}

For now, we restrict our consideration to the case of weight-$2$ type-(a) inflection points; we use our findings for the type-(a) case to study the case of weight-$2$ type-(b) inflection points at the end of \S~\ref{sec-compit} and the case of weight-$1$ inflection points in \S~\ref{sec-codim1}.

Let $\ol{\tau} = (\tau_1, \dots, \tau_{m-1})$ be a list of general elements of $\on{P}^m(xy)$ (i.e., achieving the minimum degeneracy in the sense of Definition~\ref{def-auto}). Note that for each $\ell \in \{1, \dots, m-1\}$,
\begin{align}
\tau_{\ell} = \sum_{i,j \geq 0} c_{ij}^{(\ell)} \cdot u^iv^j & = \sum_{i, j \geq 0} (xy)^i \cdot \big( c_{(i+j)i}^{(\ell)} \cdot u^{j} + c_{i(i+j)}^{(\ell)} \cdot v^{j}\big) \label{eq-helpertime}
\intertext{for units $c_{ij}^{(\ell)} \in R^\times$. Applying the functionals $e_i$ to $\tau_{\ell}$ is somewhat painful, because we need to use the relations $R_0$ and $R_m$ to express the powers of $u$ and $v$ of degree greater than or equal to $m$ in terms of smaller-degree powers of $u$ and $v$ in order to apply Lemma~\ref{lem-dualbasis}. It ends up being far too cumbersome to write out $e_i(\tau_{\ell})$ explicitly as a power series in $x$ and $y$, but we can nonetheless express $e_i(\tau_\ell)$ in a sufficiently convenient form as follows. Notice that every term of the right-hand side of~\eqref{eq-helpertime} with $i > 0$ or $j \geq m$ can be expressed as a multiple of a term with $i = 0$ and $j < m$. Thus, there exist units $\alpha_j^{(\ell)}, \beta_j^{(\ell)} \in R^\times$ that satisfy the following two properties:
\begin{enumerate}
\item The constant terms of $\alpha_j^{(\ell)},\beta_j^{(\ell)}$ are respectively given by the constant terms of $c_{j0}^{(\ell)},c_{0j}^{(\ell)}$;
\item We have that $e_i(\tau_{\ell})$ is given by
\end{enumerate}}
e_i(\tau_{\ell}) & =  \alpha_i^{(\ell)} \cdot x^i + \sum_{j = 1}^{m-1}  d_{ij} \beta_j^{(\ell)} \cdot y^{j}. \nonumber
\end{align}
Substituting the above result into~\eqref{eq-matrixeq}, we find that the degeneracy matrix is given by
 \begin{align*}
 & M_{\ol{\tau}}^m = \left[\begin{array}{ccc}
 \alpha_0^{(1)} + \overset{m-1}{\underset{j = 1}\sum} d_{1j} \beta_j^{(1)}  \cdot y^j & \cdots & \alpha_0^{(m-1)} + \overset{m-1}{\underset{j = 1}\sum} d_{1j} \beta_j^{(m-1)} \cdot y^j \\
 \alpha_1^{(1)} \cdot x + \overset{m-1}{\underset{j = 1}\sum} d_{2j} \beta_j^{(1)}  \cdot y^j &  \cdots & \alpha_1^{(m-1)} \cdot x + \overset{m-1}{\underset{j = 1}\sum} d_{2j} \beta_j^{(m-1)} \cdot y^j \\
 \vdots &  \ddots & \vdots \\
 \alpha_{m-1}^{(1)} \cdot x^{m-1} + \overset{m-1}{\underset{j = 1}\sum} d_{mj} \beta_j^{(1)}  \cdot y^j &  \cdots & \alpha_{m-1}^{(m-1)} \cdot x^{m-1} + \overset{m-1}{\underset{j = 1}\sum} d_{mj} \beta_j^{(m-1)} \cdot y^j
 \end{array}\right]
 \end{align*}
For each $i \in \{1, \dots, m\}$, let $\Xi_i$ denote the maximal $(m-1) \times (m-1)$ minor of $M_{\ol{\tau}}^m$ obtained by computing the determinant of the matrix that results from deleting the $(m-i+1)^{\mathrm{th}}$ row of $M_{\ol{\tau}}^m$. The degeneracy ideal $I_{\ol{\tau}}^m$ is generated by the $\Xi_i$'s, so we need to obtain a useful description of these minors. However, since we were unable to give an explicit description of the coefficients $\alpha_j^{(n)}$ and $\beta_j^{(n)}$ that appear in the entries of $M_{\ol{\tau}}^m$, we are consequently unable to determine the $\Xi_i$'s explicitly. Fortunately, however, it is possible to provide a description of the $\Xi_i$'s that is adequate for the purpose of computing automatic degeneracies. In the following section, we introduce a convenient system of representing elements of $R$ that allows us to obtain such an adequate description of the $\Xi_i$'s.

\subsubsection{``Root Expansions'' of Power Series}

The space $\BZ_{\geq 0} \times \BZ_{\geq 0}$ of pairs of nonnegative integers forms a partially ordered set under the relation $(i,j) \leq (i', j')$ if and only if $i \leq i'$ and $j \leq j'$, with equality if and only if $i = i'$ and $j = j'$. We make use of this structure in the next lemma:
\begin{lemma}\label{lem-rootexpansion}
  Let $h = \sum_{i,j \geq 0} c_{ij} \cdot x^i y^j \in R$ be nonzero. There exists a unique finite subset $\Sigma \subset \BZ_{\geq 0} \times \BZ_{\geq 0}$, along with units $r_{ij} \in R^\times$ for each $(i,j) \in \Sigma$ that are not necessarily unique, such that the following conditions are satisfied:
  \begin{enumerate}
  \item $(i,j) \not\leq (i',j')$ for all $(i,j), (i', j') \in \Sigma$; and
  \item $h = \sum_{(i,j) \in \Sigma} r_{ij} \cdot x^i y^j$.
  \end{enumerate}
\end{lemma}
\begin{proof}
Uniqueness, as is often the case, holds trivially. If uniqueness fails, so that we have two distinct such sets $\Sigma$ and $\Sigma'$, then for any $(i,j) \in \Sigma \setminus \Sigma'$, it would be possible to express $0$ as a sum with the coefficient of the $x^i y^j$ term being nonzero, an absurdity.

As for existence, it suffices to show that we can reduce to the case where $h$ is expressible as a finite sum of distinct monomials in $x$ and $y$ with coefficients in $R^\times$. Indeed, the lemma is obvious given such an expression of $h$, for one can simply induct on the number of terms in the sum. We now demonstrate that we can reduce to this case. Let $c_{ij} \cdot x^iy^j$ be a (nonzero) term of $h$ having minimal degree, and let $\gamma_{ij} = \sum_{(i',j') \geq (i,j)} c_{i'j'} \cdot x^{i'-i} y^{j'-j}$. Then, for each $n \in \{0, \dots, j-1\}$, let $i_n$ be the smallest among all $i'$ with the property that $c_{i'n} \neq 0$, and let $\gamma_n' = \sum_{i' \geq i_n} c_{i'n} \cdot x^{i'-i_n}$. Similarly, for each $n \in \{j+1, \dots, i + j\}$, let $j_n$ be the smallest among all $j'$ with the property that $c_{nj'} \neq 0$, and let $\gamma_n'' = \sum_{j' \geq j_n} c_{nj'} \cdot y^{j'-j_n}$. We then have that
$$h = \left(\sum_{n = 0}^{j-1} \gamma_n' \cdot x^{i_n}y^n\right) + (\gamma_{ij} \cdot x^iy^j) + \left(\sum_{n = j+1}^{i+j} \gamma_n'' \cdot x^ny^{j_n}\right),$$
where $\gamma_{ij} \in R^\times$ and $\gamma_n', \gamma_n'' \in R^\times$ for each $n \in \{0, \dots, i + j\} \setminus \{j\}$. We have thus expressed $h$ as a finite sum of distinct monomials in $x$ and $y$ with coefficients in $R^\times$, which is the desired form.
\end{proof}

\begin{defn} \label{def-roots}
With notation as in Lemma~\ref{lem-rootexpansion}, we say that $\Sigma$ is the set of \emph{roots} of $h$ and that the expression of $h$ in point (b) is \mbox{a \emph{root expansion} of $h$.}
\end{defn}

The choice of terminology in Definition~\ref{def-roots} is motivated by the fact that one can visualize the partially ordered set $\BZ_{\geq 0} \times \BZ_{\geq 0}$ as a directed graph and the nonzero terms of $h$ as a directed subgraph; then the roots of $h$ are simply those nodes that have no parents on the subgraph corresponding to $h$.

\subsubsection{Back to Step (b): Computing the Minors} \label{sec-backtothefuture}

We now express the minors $\Xi_i$ in terms of their root expansions.

\begin{proposition}\label{prop-minorcalc}
For a general choice of $c_{j'0}^{(\ell)}$ and $c_{0j'}^{(\ell)}$ with $j' \in \{0, \dots, m-1\}$, the root expansion of $\Xi_i$ is given for each $i \in \{1, \dots, m\}$ by
$$\Xi_i = \sum_{j = 0}^{m-2} \gamma_{ij} \cdot x^{\kappa_{m-2-j} + \max\{0,i-j-1\}} y^{\kappa_j},$$
where $\gamma_{ij} \in R^\times$ for each $i,j$ and $\kappa_j = \sum_{i = 0}^j j$ is the $j^{\mathrm{th}}$ \emph{triangular number} for each $j \geq 0$.
\end{proposition}
\begin{proof}
Fix $i \in \{1, \dots, m\}$. To determine the roots of $\Xi_i$, we ask the following question: for every nonnegative integer $\ell'$, what is the smallest $\ell''$ so that $\Xi_i$ has a nonzero term proportional to $x^{\ell''}y^{\ell'}$? Before we answer this question in full generality, let us work out the argument in the easiest case, namely when $\ell' = 0$.
For this case, we want to compute the smallest power of $x$ that appears as a term in $\Xi_i$; this smallest power is evidently the same as that which arises from computing the following determinant, obtained by deleting all nonzero powers of $y$ from the entries of the minor defining $\Xi_i$:
\begin{equation}\label{eq-testdet}
 \left|\begin{array}{ccc}
 \alpha_0^{(1)}  & \cdots & \alpha_0^{(m-1)}  \\
 \alpha_1^{(1)} \cdot x  &  \cdots & \alpha_1^{(m-1)} \cdot x \\
 \vdots &  \ddots & \vdots \\
 \alpha_{m-1}^{(1)} \cdot x^{m-1}  &  \cdots & \alpha_{m-1}^{(m-1)} \cdot x^{m-1}
 \end{array}\right|_{m-i+1}
\end{equation}
where the subscript is meant to indicate that we have deleted the $(m-i+1)^{\mathrm{th}}$ row. It follows by inspection of~\eqref{eq-testdet} that the smallest $\ell''$ so that $\Xi_i$ has a nonzero term of the form $\gamma_{i0} \cdot x^{\ell''}$ is $\ell'' = \kappa_{m-2} + i - 1$, where $\gamma_{i0} \in R$.

Now let us deal with the case when $\ell' > 0$. An entry of the minor defining $\Xi_i$ can either contribute a factor of $x$ through the term $\alpha_j^{(\ell)} \cdot x^j$ or contribute a factor of $y$ through one of the terms $\beta_j^{(\ell)} \cdot y^j$. Since we are looking for the smallest $\ell''$ so that $\Xi_i$ has a nonzero term proportional to $x^{\ell''}y^{\ell'}$, we want the $y$-factors to come from the bottom-most rows of the minor defining $\Xi_i$, so that the $y$-factors are essentially replacing the largest $x$-factors.
But notice that in computing $\Xi_i$, we cannot choose the same power of $y$ from any two of the bottom-most rows. To see why this claim is true, consider the matrix obtained from $M_{\ol{\sigma}}^m(xy)$ by deleting all powers of $x$, and compute any $2 \times 2$ minor of it:
\begin{align*}
& \left|\begin{array}{cc}
\overset{m-1}{\underset{j = 1}\sum} b_{aj} \beta_j^{(c)}  \cdot y^j & \overset{m-1}{\underset{j = 1}\sum} b_{aj} \beta_j^{(c')} \cdot y^j \\  \overset{m-1}{\underset{j = 1}\sum} b_{a'j} \beta_j^{(c)}  \cdot y^j & \overset{m-1}{\underset{j = 1}\sum} b_{a'j} \beta_j^{(c')} \cdot y^j \end{array}\right| = \\
& \left(\overset{m-1}{\underset{j = 1}\sum} b_{aj} \beta_j^{(c)}  \cdot y^j\right)\left(\overset{m-1}{\underset{j = 1}\sum} b_{a'j} \beta_j^{(c')} \cdot y^j \right) - \left(\overset{m-1}{\underset{j = 1}\sum} b_{aj} \beta_j^{(c')} \cdot y^j\right)\left(\overset{m-1}{\underset{j = 1}\sum} b_{a'j} \beta_j^{(c)}  \cdot y^j\right) = \\
& \sum_{\substack{1 \leq j, j' \leq m-1\ \\ j \neq j'}} \big(b_{aj} b_{a'j'} \beta_j^{(c)}\beta_{j'}^{(c')} -  b_{aj} \beta_j^{(c')}b_{a'j'} \beta_{j'}^{(c)}\big) \cdot y^{j + j'},
 \end{align*}
 where in the last step above, we could restrict the sum by stipulating that the indices $j$ and $j'$ be different because the summand evidently vanishes when we set $j = j'$. In other words, choosing the same power of $y$ from any two rows yields a contribution of $0$. It follows that the only possible values of $\ell'$ are the triangular numbers $\kappa_j$ for each $j \in \{0, \dots, m-2\}$, and it further follows that the smallest $\ell''$ so that $\Xi_i$ has a nonzero term of the form $\gamma_{ij} \cdot x^{\ell''}y^{\kappa_j}$ is $\ell'' = \kappa_{m-2-j} + \max\{0,i-j-1\}$, where $\gamma_{ij} \in R$.

 It is not easy to determine explicit expressions for the coefficients $\gamma_{ij}$, but doing so is unnecessary for our purposes. All we need is the following fact: the constant term of each $\gamma_{ij}$ is a polynomial in the constant terms of the coefficients $\alpha_{j'}^{(\ell)}, \beta_{j'}^{(\ell)}$ for $j' \in \{0, \dots, m-1\}$. Recall that we established in \S~\ref{sec-firststepb} that the constant term of $\alpha_{j'}^{(\ell)}$ is $c_{j'0}^{(\ell)}$ and the constant term of $\beta_{j'}^{(\ell)}$ is $c_{0j'}^{(\ell)}$. Thus, each $\gamma_{ij}$ is a polynomial in the coefficients $c_{j'0}^{(\ell)}$ and $c_{0j'}^{(\ell)}$ for $j' \in \{0, \dots, m-1\}$ of the original germs $\sigma_1, \dots, \sigma_{m-1}$. In particular, the constant terms of all of the $\gamma_{ij}$'s depend on only finitely many of the coefficients of the germs $\sigma_1, \dots, \sigma_{m-1}$. We conclude that under the generality assumption in the statement of the proposition, the coefficients $\gamma_{ij}$ are units in $R$, and so the monomials $x^{\kappa_{m-2-j} + \max\{0,i-j-1\}}y^{\kappa_j}$ for $j \in \{0, \dots, m-2\}$ are the roots of $\Xi_i$ for every pair $i \in \{1, \dots, m\}$, as desired.
\end{proof}

\subsubsection{Step (c): Computing the Colength} \label{sec-compit}

We are now ready to combine the results from previous steps to compute the automatic degeneracy $\on{AD}_{(2)}^m(xy)$.

\begin{lemma}\label{lem-adcalc}
We have that
$$I_{\ol{\tau}}^m = (\{x^{\kappa_{m-2-j}} y^{\kappa_j} : j \in \{0, \dots, m-2\}),$$
for a general choice of $c_{j'0}^{(\ell)}$ and $c_{0j'}^{(\ell)}$ with $j' \in \{0, \dots, m-1\}$.
\end{lemma}
\begin{proof}
By definition, $I_{\ol{\tau}}^m$ is generated by the minors $\Xi_1, \dots, \Xi_m$, the root expansions of which we determined in Proposition~\ref{prop-minorcalc}. To prove the desired equality of ideals, it suffices to show that the residue of each $x^{\kappa_{m-2-j}} y^{\kappa_j}$ in $R/I_{\ol{\tau}}^m$ is equal to $0$. Observe that for each $i \in \{1, \dots, m-1\}$, the relation $\Xi_i = 0$ in $R/I_{\ol{\tau}}^m$ can be expressed as a relation on the monomials $x^{\kappa_{m-2-j}}y^{\kappa_j}$ for each $j \in \{0, \dots, m-2\}$ with coefficient given by $\gamma_{ij} \cdot x^{\max\{0, i-j-1\}}$. Thus, we may view the relations $\Xi_i = 0$ for $i \in \{1, \dots, m-1\}$ as a system of equations in the variables $x^{\kappa_{m-2-j}}y^{\kappa_j}$; putting this system into matrix form yields
\scriptsize
$$ \left[\begin{array}{ccccc} \gamma_{10} & \gamma_{11} & \cdots & \gamma_{1(m-3)} & \gamma_{1(m-2)} \\ x \cdot \gamma_{20} & \gamma_{21} & \cdots & \gamma_{2(m-3)} & \gamma_{2(m-2)} \\ \vdots & \vdots & \ddots & \vdots & \vdots \\ x^{m-3} \cdot \gamma_{(m-2)0} & x^{m-4} \cdot \gamma_{(m-2)1} & \cdots & \gamma_{(m-2)(m-3)} & \gamma_{(m-2)(m-2)} \\ x^{m-2} \cdot \gamma_{(m-1)0} & x^{m-3} \cdot \gamma_{(m-1)1} & \cdots & x \cdot \gamma_{(m-1)(m-3)} & \gamma_{(m-1)(m-2)} \end{array}\right] \cdot \left[\begin{array}{c} x^{\kappa_{m-2}} \\ x^{\kappa_{m-3}}y^{\kappa_1} \\ \vdots \\ x^{\kappa_{1}}y^{\kappa_{m-3}} \\ y^{\kappa_{m-2}} \end{array}\right] = \left[\begin{array}{c} 0 \\ 0 \\ \vdots \\ 0 \\ 0 \end{array}\right].$$
\normalsize
To solve the above system of equations, we simply put the associated augmented matrix into row echelon form. After doing this, the first $m-2$ entries of the last row are $0$, so as long as the constant term of $\gamma_{(m-1)(m-2)}$ is nonzero, so that $\gamma_{(m-1)(m-2)}$ is a unit, we deduce that $y^{\kappa_{m-2}} = 0$. Going up one row, the first $m-3$ entries of the second-to-last row are $0$, so as long as $\gamma_{(m-2)(m-3)}$ is a unit, we deduce that $x^{\kappa_1} y^{\kappa_{m-3}} = 0$. Continuing inductively in this manner, we find that $x^{\kappa_{m-2-j}} y^{\kappa_j} = 0$ as long as $\gamma_{(i+1)i}$ is a unit for each $i \in \{0, \dots, m-2\}$. By Proposition~\ref{prop-minorcalc}, this condition on the $\gamma_{(i+1)i}$'s will be satisfied for a general choice of the coefficients $c_{j'0}^{(\ell)}$ and $c_{0j'}^{(\ell)}$ for $j' \in \{0, \dots, m-1\}$. Thus, we have shown that $I_{\ol{\tau}}^m$ is generated by the monomials in the set $\{x^{\kappa_{m-2-j}} y^{\kappa_j} : j \in \{0, \dots, m-2\}$
as long as the generality condition on the coefficients $c_{j'0}^{(\ell)}$ and $c_{0j'}^{(\ell)}$ is satisfied.
\end{proof}

Given the result of Lemma~\ref{lem-adcalc}, we are finally ready to compute the colength:

\begin{lemma} \label{lem-counter}
We have that
$$\dim_k R/(\{x^{\kappa_{m-2-j}} y^{\kappa_j} : j \in \{0, \dots, m-2\}) = {{m+1} \choose {4}}.$$
\end{lemma}
\begin{proof}
Clearly, there is a unique basis of $$R/(\{x^{\kappa_{m-2-j}} y^{\kappa_j} : j \in \{0, \dots, m-2\})$$
with the property that each basis vector is a monomial in $x$ and $y$ with coefficient $1$. This basis may be equivalently described as follows: consider the directed graph $\BZ_{\geq 0} \times \BZ_{\geq 0}$, and remove all nodes that either are equal to or are children of the nodes $(\kappa_{m-2-j}, \kappa_j)$ for $j \in \{0, \dots, m-2\}$. Then the number of nodes that remain in the graph is the desired dimension.

To compute the number of remaining nodes, we sum the number that remain in the ``ray'' of nodes of the form $(-, \ell)$ over $\ell \in \{0, \dots, \kappa_{m-2}-1\}$. This sum is most easily computed by splitting it into the chunks $\kappa_j \leq \ell \leq \kappa_{j+1} - 1$ for $j \in \{0, \dots, m-3\}$: indeed, for each $\ell$ in this interval, the number of nodes that remain in the corresponding ``ray'' is simply $\kappa_{m-2-j}$. Thus, the total number of nodes that remain is given by
\begin{align*}
\sum_{\ell = 0}^{\kappa_{m-2} - 1} \#(\text{nodes of type $(-,\ell)$ that remain}) & = \sum_{j = 0}^{m-3}\sum_{\ell = \kappa_j}^{\kappa_{j+1} - 1} \kappa_{m-2-j} \\
& = \sum_{j = 0}^{m-3} \kappa_{m-2-j} \cdot (\kappa_{j+1} - \kappa_j) \\
& = \sum_{j = 0}^{m-3} \frac{(m-2-j)(m-2-j+1)}{2} \cdot (j+1) \\
& = {{m+1} \choose {4}},
\end{align*}
where in the last step above, we have appealed to the standard identities for summing consecutive squares and cubes to obtain the desired formula.
\end{proof}

Lemmas~\ref{lem-adcalc} and~\ref{lem-counter} together imply that $\on{AD}_{(2)}^m(xy) = {{m+1} \choose {4}}$. This concludes the proof of Theorem~\ref{thm-main2} in the case of weight-$2$ type-(a) inflection points.

  We now briefly sketch the proof of the case of weight-$2$ type-(b) inflection points, because the argument is very similar to that of the type-(a) case. We start with a list of general elements $\ol{\tau} = (\tau_1, \dots, \tau_{m})$ in $\on{P}^{m-1}(xy)$, to which we associate a degeneracy matrix $M_{\ol{\tau}}^{m-1}$. The maximal $(m-1) \times (m-1)$ minors of $M_{\ol{\tau}}^{m-1}$ all have the same form as the minor $\Xi_1$ computed in Proposition~\ref{prop-minorcalc}. Indeed, under a generality condition similar to the type specified in the statement of Proposition~\ref{prop-minorcalc}, the root expansions of these minors are all given by a linear combination with coefficients in $R^\times$ of the monomials $x^{\kappa_{m-2-j}}y^{\kappa_j}$ for $j \in \{0, \dots, m-2\}$, and these minors give $m$ maximally linearly independent relations on the monomials $x^{\kappa_{m-2-j}}y^{\kappa_j}$. Thus, the degeneracy ideal $I_{\ol{\tau}}^{m-1}$ is generated by the monomials $x^{\kappa_{m-2-j}}y^{\kappa_j}$, and the claimed equality $\on{AD}_{(1,1)}^{m-1}(xy) = \dim_k R/I_{\ol{\tau}}^{m-1} =  {{m+1} \choose 4}$ then follows from Lemma~\ref{lem-counter}.

  \subsection{The Weight-$1$ Case} \label{sec-codim1}

  We now compute $\on{AD}_{(1)}^{m-1}(xy)$. To do this, we start with a list of general elements $\ol{\tau} = (\tau_1, \dots, \tau_{m-1})$ in $\on{P}^{m-1}(xy)$, to which we associate a degeneracy matrix $M_{\ol{\tau}}^{m-1}$. The determinant of $M_{\ol{\tau}}^{m-1}$ has the same form as the minor $\Xi_1$ computed in Proposition~\ref{prop-minorcalc}. Indeed, under a generality condition similar to the type specified in the statement of Proposition~\ref{prop-minorcalc}, the root expansion of this determinant is given by
\begin{equation} \label{eq-anlocdegen}
\sum_{i = 0}^{m-2} \gamma_i \cdot x^{\kappa_{m-2-i}} y^{\kappa_i}
\end{equation}
where $\gamma_i \in R^\times$ for each $i \in \{0, \dots, m-2\}$. Interestingly, the expression in~\eqref{eq-anlocdegen} can be factored as follows:
$$\sum_{i = 0}^{m-2} \gamma_i \cdot x^{\kappa_{m-2-i}} y^{\kappa_i} = \prod_{i = 1}^{m-2} (\alpha_i \cdot x^i - \beta_i \cdot y^{m-1-i})$$
where $\alpha_i, \beta_i \in R^\times$ for $i \in \{1, \dots, m-2\}$. We therefore arrive at the following result:
\begin{theorem}\label{thm-anlocsingtype}
The divisor of $(m-1)^{\mathrm{th}}$-order weight-$1$ inflection points in a general $1$-parameter deformation of a node is given by the reduced union of $m-2$ hypercuspidal branches defined by equations of the form $\alpha_i \cdot x^i - \beta_i \cdot y^{m-1-i} = 0$ where $\alpha_i, \beta_i \in R^\times$ for $i \in \{1, \dots, m-2\}$.
\end{theorem}
\begin{remark}
  The result stated in Theorem~\ref{thm-anlocsingtype} was first proven by S.~Cautis in~\cite[Theorem~3.25]{MR2708742}, although Cautis arrived at the result via a monodromy argument as opposed to a direct local calculation.
\end{remark}
We now return to the calculation of $\on{AD}_{(1)}^{m-1}(xy)$. We have that
\begin{equation} \label{eq-ad1calc}
I_{\ol{\tau}}^{m-1} = \left(xy,\,\sum_{i = 0}^{m-2} \gamma_i \cdot x^{\kappa_{m-2-i}} y^{\kappa_i}\right) = (xy, \gamma_0 \cdot x^{\kappa_{m-2}} + \gamma_{m-2} \cdot y^{\kappa_{m-2}}).
\end{equation}
under the aforementioned generality condition. We compute the colength of $I_{\ol{\tau}}^{m-1}$ in the following lemma:
\begin{lemma} \label{lem-ad1calcs}
We have that
$$\dim_k R/(xy, \gamma_0 \cdot x^{\kappa_{m-2}} + \gamma_{m-2} \cdot y^{\kappa_{m-2}}) = (m-1)(m-2).$$
\end{lemma}
\begin{proof}
The relation $\gamma_0 \cdot x^{\kappa_{m-2}} + \gamma_{m-2} \cdot y^{\kappa_{m-2}} = 0$ expresses $x^{\kappa_{m-2}}$ as a unit multiple of $y^{\kappa_{m-2}}$ and, taken together with the relation $xy = 0$, implies that $y^{\kappa_{m-2} + 1} = \gamma_0 \cdot \gamma_{m-2}^{-1} \cdot xy \cdot x^{\kappa_{m-2}-1} = 0$.
It follows that the set of monomials $$\{x^i : 0 \leq i \leq \kappa_{m-2}-1\} \cup \{y^i : 0 \leq i \leq \kappa_{m-2}\}$$ forms a basis of $R/(xy, x^{\kappa_{m-2}}, y^{\kappa_{m-2}})$ as a $k$-vector space, so the desired colength is equal to the size of this set, which is given by $2 \cdot \kappa_{m-2} = (m-1)(m-2)$.
\end{proof}
Combining~\eqref{eq-ad1calc} with Lemma~\ref{lem-ad1calcs}, we deduce that $\on{AD}_{(1)}^{m-1}(xy) = \dim_k R/I_{\ol{\tau}}^{m-1} = (m-1)(m-2)$. This concludes the proof of Theorem~\ref{thm-main2}.
\end{proof}
We conclude this section by providing two examples in which we use Theorem~\ref{thm-main2} (as well as Theorem~\ref{thm-anlocsingtype}) to reproduce known results about weight-$1$ inflection points limiting to a node.
\begin{example} \label{eg-backtorefer}
  Consider weight-$1$ inflection points on a plane curve $C \subset \BP_k^2$ associated to the linear system $(\scr{O}_C(1), W)$, where $W \subset H^0(\scr{O}_C(1))$ is the $k$-vector subspace spanned by pullbacks of linear forms on $\BP_k^2$ vanishing at a given point. It is well-known that in a general $1$-parameter family of smooth plane curves specializing to a plane curve with a node, the number of such inflection points limiting to the node is $2$. This fact agrees with Theorem~\ref{thm-main2}, which tells us that the number of $2^{\mathrm{nd}}$-order weight-$1$ inflection points limiting to a node in a general $1$-parameter deformation is equal to $2(2-1) = 2$.
\end{example}

\begin{example}
  As stated in~\cite[\S~4, part (s)]{diazharrisseveriI}, in a general $1$-parameter family of smooth plane curves specializing to a plane curve with a node, ``a total of six flexes of the nearby smooth curves will approach the node, comprising two smooth arcs. Each arc will be simply tangent to one of the branches [of the node, and] hence will have intersection number $3$ [with the singular fiber].'' Now, recall that flexes on a smooth plane curve $C \subset \BP_k^2$ are $3^{\mathrm{rd}}$-order weight-$1$ inflection points with respect to the linear system $(\scr{O}_{C}(1), W)$, where $W \subset H^0(\scr{O}_C(1))$ is the $k$-vector subspace spanned by pullbacks to $C$ of sections in $H^0(\scr{O}_{\BP_k^2}(1))$. Thus, Theorem~\ref{thm-main2} tells us that the number of $3^{\mathrm{rd}}$-order weight-$1$ inflection points limiting to a node in a general $1$-parameter deformation is equal to $3(3-1) = 6$, which agrees with the first part of the quoted result. Moreover, Theorem~\ref{thm-anlocsingtype} tells us that the divisor of $3^{\mathrm{rd}}$-order weight-$1$ inflection points in a general $1$-parameter deformation of a node is the reduced union of two branches with equations of the form $\alpha_1 \cdot x - \beta_1 \cdot y^2 = 0$, which is simply tangent to the branch $y = 0$ of the node, and $\alpha_2 \cdot x^2 - \beta_2 \cdot y = 0$, which is simply tangent to the branch $x = 0$ of the node. We have thus also reproduced the second part of the quoted result.
\end{example}

\subsection{Flecnodes as Limits of Inflection Points}

  We have established that in a general $2$-parameter deformation of a node, the number of $m^{\mathrm{th}}$-order weight-$2$ inflection points of either type limiting to the node is equal to $0$. However, it is interesting to consider what happens when we take the elements $\tau_\ell$ used to define the degeneracy scheme to be special in some way. For example, a \emph{flecnode} on a plane curve $C \subset \BP_k^2$ is a node at which one of the two branches is ``flexed,'' meaning that the tangent line to that branch meets it at the point of tangency with intersection multiplicity $3$. 
  Note that this tangent line meets the nodal curve at the node with intersection multiplicity $4$, so it is reasonable to ask the following question: in a general $2$-parameter deformation of a flecnode, is there a nonzero number of hyperflexes limiting to the flecnode? The next theorem tells us that the answer is yes.

  \begin{theorem}
    Take a list of elements $\ol{\tau} = (\tau_1, \dots, \tau_{m-1})$ in $\on{P}^m(xy)$ that is general among all lists of $m-1$ elements with $\tau_1$ vanishing to order $m$ along the branch $y = 0$ of the node. Then $$\dim_k R/I_{\ol{\tau}}^m = {{m+1} \choose 4} + 1,$$ so in this case, the number of $m^{\mathrm{th}}$-order weight-$2$ type-(a) inflection points limiting to a node in a general $2$-parameter deformation is equal to $1$.
  \end{theorem}
  \begin{proof}
  For each $\ell \in \{1, \dots, m-1\}$, we write
$\tau_\ell = \sum_{i,j \geq 0} c_{ij}^{(\ell)} \cdot u^iv^j$. The condition that $\tau_1$ vanishes to order $m$ along the branch $y = 0$ simply means that $c_{i0}^{(1)} = 0$ for $i \in \{0, \dots, m-1\}$. Just as we argued in \S~\ref{sec-firststepb}, $\alpha_{m-1}^{(1)} \in R^\times$, as well as $\alpha_j^{(\ell)} \in R^\times$ for $j \in \{0, \dots, m-1\}, \, \ell \in \{2, \dots, m-1\}$ and $\beta_j^{(\ell)} \in R^\times$ for $j \in \{0, \dots, m-1\}, \, \ell \in \{1, \dots, m-1\}$ that satisfy the following two properties:
\begin{enumerate}
\item The constant terms of $\alpha_j^{(\ell)},\beta_j^{(\ell)}$ are respectively given by the constant terms of $c_{j0}^{(\ell)},c_{0j}^{(\ell)}$;
\item We have that $e_i(\tau_{\ell})$ is given by
\end{enumerate}
$$e_i(\tau_{\ell}) =  \begin{cases} \alpha_i^{(\ell)} \cdot x^i + \sum_{j = 1}^{m-1}  d_{ij} \beta_j^{(\ell)} \cdot y^{j} & \quad \text{if $i = m-1$ and $\ell = 1$ or if $\ell \in \{2, \dots, m-1\}$} \\ \sum_{j = 1}^{m-1}  d_{ij} \beta_j^{(1)} \cdot y^{j} &\quad \text{if $i \in \{0, \dots, m-2\}$ and $\ell = 1$} \end{cases}$$
Then the degeneracy matrix $M_{\ol{\tau}}^m$ has the form
  \begin{align*}
 & M_{\ol{\tau}}^m = \\
 & \left[\begin{array}{cccc}
 \scriptstyle\overset{m-1}{\underset{j = 1}\sum} d_{1j} \beta_j^{(1)}  \cdot y^j & \scriptstyle\alpha_0^{(2)} + \overset{m-1}{\underset{j = 1}\sum} d_{1j} \beta_j^{(2)} \cdot y^j & \scriptstyle\cdots &\scriptstyle \alpha_0^{(m-1)} + \overset{m-1}{\underset{j = 1}\sum} d_{1j} \beta_j^{(m-1)} \cdot y^j \\
  \scriptstyle\vdots & \scriptstyle\vdots &  \scriptstyle\ddots & \scriptstyle\vdots \\
 \scriptstyle\overset{m-1}{\underset{j = 1}\sum} d_{(m-1)j} \beta_j^{(1)}  \cdot y^j &\scriptstyle \alpha_{m-2}^{(2)} \cdot x^{m-2} + \overset{m-1}{\underset{j = 1}\sum} d_{(m-1)j} \beta_j^{(2)} \cdot y^j &\scriptstyle \cdots &\scriptstyle \alpha_{m-2}^{(m-1)} \cdot x^{m-2} + \overset{m-1}{\underset{j = 1}\sum} d_{(m-1)j} \beta_j^{(m-1)} \cdot y^j \\
 \scriptstyle\alpha_{m-1}^{(1)} \cdot x^{m-1} + \overset{m-1}{\underset{j = 1}\sum} d_{mj} \beta_j^{(1)}  \cdot y^j &\scriptstyle \alpha_{m-1}^{(2)} \cdot x^{m-1} + \overset{m-1}{\underset{j = 1}\sum} d_{mj} \beta_j^{(2)} \cdot y^j &\scriptstyle \cdots &\scriptstyle \alpha_{m-1}^{(m-1)} \cdot x^{m-1} + \overset{m-1}{\underset{j = 1}\sum} d_{mj} \beta_j^{(m-1)} \cdot y^j
 \end{array}\right]
 \end{align*}
 By mimicking the proof of Proposition~\ref{prop-minorcalc}, one checks that the root expansion of the maximal $(m-1) \times (m-1)$ minor $\Xi_i$ of $M_{\ol{\tau}}^m$ obtained by computing the determinant of the matrix that results from deleting the $(m-i+1)^{\mathrm{th}}$ row is given for each $i \in \{1, \dots, m\}$ by
$$\Xi_i = \begin{cases}\sum_{j = 1}^{m-2} \gamma_{ij} \cdot x^{\kappa_{m-2-j}} y^{\kappa_j} &\quad\text{if $i = 1$} \\ \sum_{j = 0}^{m-2} \gamma_{ij} \cdot x^{\kappa_{m-2-j} + \max\{0,i-j-1\}} y^{\kappa_j}  &\quad\text{if $i \in \{2, \dots m\}$}  \end{cases}$$
where $\gamma_{ij} \in R^\times$ for each $i,j$. Moreover, by mimicking the proof of Lemma~\ref{lem-adcalc}, we find that the degeneracy ideal is given by
$$I_{\ol{\tau}}^m = (\{\Xi_i : i \in \{1, \dots, m\}\}) =  (\{x^{\kappa_{m-2-j}} y^{\kappa_j} : j \in \{1, \dots, m-2\}) \cup \{x^{\kappa_{m-2}+1}\}.$$
By mimicking the proof of Lemma~\ref{lem-counter}, we find that $\dim_k R/I_{\ol{\tau}}^m = {{m+1} \choose 4} + 1$.

Lastly, we claim that the number of $m^{\mathrm{th}}$-order weight-$2$ type-(a) inflection points limiting to the flecnode in a general $2$-parameter deformation is
$$\dim_k R/I_{\ol{\tau}}^m - (\on{mult}_0 \Delta_{xy}) \cdot \on{AD}_{(2)}^m(xy) = {{m+1} \choose 4} + 1 - 1 \cdot {{m+1} \choose 4} = 1,$$
Consider a general $2$-parameter deformation $\wt{X}/\wt{B}$ of the node. Then $\scr{O}_{\wt{B}} \simeq k[[s,t]]$, and $\scr{O}_{\wt{X}}$ is the $\scr{O}_{\wt{B}}$-algebra given by
$$\scr{O}_{\wt{X}} = k[[s,t,x,y]]/(xy - (t-s)).$$
The discriminant locus of the family lies over the line $s = t$.
As in \S~\ref{sec-limdefs}, let $B \subset \wt{B}$ be cut out by $s = 0$ and $X = \wt{X} \times_{\wt{B}} B$, and let $B' \subset \wt{B}$ be the line parallel to $s = 0$ passing through the geometric generic point $\eta \in k[[s]]$ and $X' = \wt{X} \times_{\wt{B}} B'$.
For each $\ell \in \{1, \dots, m-1\}$, take a general element $\wt{\tau}_\ell \in \scr{P}_{\wt{X}/\wt{B}}^m$ whose restriction to $X$ is equal to the element $\tau_\ell \in \on{P}^m(xy)$ specified in the theorem statement. Then $\wt{\tau}_\ell$ is given by
$$\wt{\tau}_\ell = \sum_{i,j \geq 0} \wt{c}_{ij}^{(\ell)} \cdot u^iv^j \in \scr{P}_{\wt{X}/\wt{B}}^m,$$
where each $\wt{c}_{ij}^{(\ell)} \in \scr{O}_{\wt{X}}$ is such that the restriction to $s = 0$ is given by $c_{ij}^{(\ell)} \in R$. 
For each $\ell \in \{1, \dots, m-1\}$, let $\tau_\ell'$ be the image of $\wt{\tau}_\ell$ under the natural map $\scr{P}_{\wt{X}/\wt{B}}^m \to \scr{P}_{X'/B'}^m$, and let $\ol{\tau}' = (\tau_1', \dots, \tau_{m-1}')$. We now prove that the list of elements $\ol{\tau}'$ is sufficiently general as to achieve the minimal automatic degeneracy in the sense of Definition~\ref{def-auto}; more precisely, we claim that
$$\dim_k \scr{O}_{X'}/I_{\ol{\tau}'}^m =  \on{AD}_{(2)}^m(xy).$$
To prove this, it suffices to assume the following: each of the $\wt{c}_{ij}^{(\ell)}$, when expressed as a power series in the variables $x,y,t$, has the property that the coefficients are polynomials (as opposed to power series) in $s$. Having made this assumption, we can evaluate the elements $\tau_\ell'$ at values $s \in k$. Thus, letting $\ol{\tau}'|_{s=1} = (\tau_1'|_{s=1}, \dots, \tau_{m-1}'|_{s=1})$, it suffices by specialization to show that
$$\dim_k R/I_{\ol{\tau}'|_{s=1}}^m = \on{AD}_{(2)}^m(xy),$$
but this holds because the elements $\tau_\ell'|_{s=1} \in \on{P}^m(xy)$ are general by construction (since the elements $\wt{\tau}_\ell \in \scr{P}_{\wt{X}/\wt{B}}^m$ were taken to be general). 
With this result, the claim follows by mimicking the proof in \S~\ref{sec-limdefs} of the formula~\eqref{Eq:FundamentalFormula}.
  \end{proof}

\section{Automatic Degeneracies of Higher-Order Singularities} \label{sec-getsworse}

The computations of $m^{\mathrm{th}}$-order automatic degeneracies for the node performed in \S~\ref{sec-calc} are difficult to reproduce for higher-order singularities. Given an ICIS cut out analytically-locally by $f = 0$, the primary obstacle to obtaining a formula for its $m^{\mathrm{th}}$-order automatic degeneracies is finding bases of the free modules $\on{P}^m(f)^\vee$ for all $m \geq 1$. Recall that such bases are needed to write down the matrix in~\eqref{eq-matrixeq}. The construction of bases of the modules $\on{P}^m(xy)^{\vee}$ in \S~\ref{sec-basisnode} relies on the symmetry and simplicity of the solitary equation $xy = 0$ defining the node, and it remains open as to whether a similar such construction can be made for any other ICIS.

In this section, we present a range of different results relating to the calculation of automatic degeneracies for higher-order singularities. Briefly, these results are as follows:
\begin{enumerate}
\item We obtain a formula for the $m^{\mathrm{th}}$-order weight-2 type-(a) automatic degeneracy of a cusp;
\item We provide an algorithm for obtaining a basis of $\on{P}^m(f)^{\vee}$ for a planar ICIS cut out analytically-locally by $f = 0$; and
\item We obtain lower and upper bounds for the $m^{\mathrm{th}}$-order weight-$2$ automatic degeneracies of each type for an arbitrary planar ICIS.
\end{enumerate}
We use the algorithm described in point (b) to compute low-order weight-$1$ and weight-$2$ automatic degeneracies of certain families of planar singularities in the next section, \S~\ref{sec-autodegegs}.

\subsection{The Case of Cusps} \label{sec-extendtocusp}

It is not clear whether one can directly compute the automatic degeneracies of a cusp in the manner of \S~\ref{sec-calc}, but there is a way to directly determine the number of weight-$2$ type-(a) inflection points limiting to the cusp in a general $2$-parameter deformation. Indeed, we have the following theorem:

\begin{theorem} \label{thm-cuspinf}
The number of $m^{\mathrm{th}}$-order weight-$2$ type-(a) inflection points limiting to a cusp in a general $2$-parameter deformation is equal to $0$ for every $m \geq 2$.
\end{theorem}
\begin{remark}
The basic intuition underlying the proof of Theorem~\ref{thm-cuspinf} is that cusps ``ought not to count as'' hyperflexes, for example, because there is no line in $\mathbb{A}_k^2 = \Spec k[x,y]$ that meets the cuspidal curve $V(y^2 - x^3) \subset \BA_k^2$ with intersection multiplicity at least $4$ at the cusp.
\end{remark}
\begin{proof}[Proof of Theorem~\ref{thm-cuspinf}]
Let $R = k[[x,y]]$ as in \S~\ref{sec-calc}. Then we can take $R/(y^2 - x^3)$ to be the analytic-local coordinate ring of the cusp, and the map of $k$-algebras $R/(y^2 - x^3) \to k[[t]]$ taking $x \mapsto t^2$ and $y \mapsto t^3$ is an isomorphism onto its image, which is equal to the sub-$k$-algebra $A \coloneqq k[[t^2, t^3]] \subset k[[t]]$. Because the numerical semigroup generated by $2$ and $3$ is equal to $\{i \in \mathbb{Z} : i \geq 2\}$, every $p \in A$ can be expressed as $p = c_0 +  \sum_{i = 2}^\infty c_i t^i$ for some coefficients $c_i \in k$.

Now choose $m-1$ general elements $p_1, \dots, p_{m-1} \in A$, and let $V$ be the $(m-1)$-dimensional $k$-vector subspace of $A$ that they span. We claim that the highest
order of vanishing of a nonzero element of $V$ is equal to $m-1$, and not greater than or equal to $m$. Notice that to prove the theorem, it suffices to prove the claim. Indeed, suppose it were true that a positive number of weight-$2$ type-(a) inflection points were limiting to the cusp in a general $2$-parameter deformation. Letting $X/B$ be a general $1$-parameter deformation, it follows from~\eqref{Eq:FundamentalFormula} that for a list of $m-1$ general elements $\ol{\tau}$ in $\scr{P}_{X/B}^m$, we have
\begin{equation} \label{eq-cuspineq}
\dim_k \scr{O}_X/I_{\ol{\tau}}^m - (\on{mult}_0 \Delta_{y^2 - x^3}) \cdot \on{AD}_{(2)}^m(xy) > 0.
\end{equation}
Then by upper-semicontinuity, this would remain true if we used a \emph{special} choice of elements $\ol{\tau}$ rather than \emph{general} elements to define the locus of inflection points. Indeed, if we take $\ol{\tau}$ to be the list of elements of $\scr{P}_{X/B}^m$ induced by the elements $p_1, \dots, p_{m-1} \in A$, then the inequality in~\eqref{eq-cuspineq} would continue to hold. Thus, we would find that the cusp is a limit of a positive number of inflection points with respect to the elements of $\scr{P}_{X/B}^m$ induced by $p_1, \dots, p_{m-1} \in A$. It would follow that $V$ contains an element---namely, the limit of elements vanishing to order $m$ at the nearby inflection points---that vanishes to order at least $m$ at the cusp, contradicting the claim.

We now prove the claim. Writing $p_j = c_{0j} +  \sum_{i = 2}^\infty c_{ij} t^i$ for each $j \in \{1, \dots, m-1\}$, the condition that the elements $p_1, \dots, p_{m-1}$ are general implies that none of the maximal minors of the following $m \times (m-1)$ matrix of coefficients is equal to zero:
$$\left[\begin{array}{cccc}c_{01} & c_{02} & \cdots & c_{0(m-1)} \\ c_{21} & c_{22} & \cdots & c_{2(m-1)} \\ \vdots & \vdots & \ddots & \vdots \\ c_{(m-1)1} & c_{(m-1)2} & \cdots & c_{(m-1)(m-1)} \\ c_{m1} & c_{m2} & \cdots & c_{m(m-1)} \end{array}\right]$$
It follows that by taking a suitable $k$-linear combination of $p_1, \dots, p_{m-1}$, we can find a nonzero $p \in V$ vanishing to order at least $m-1$ at $t = 0$ (i.e., $c_0 = c_2 = \cdots = c_{m-1} = 0$ in the power series expansion $p = c_0 + \sum_{i = 2}^\infty c_i t^i$). It further follows that no $k$-linear combination of $p_1, \dots, p_{m-1}$ vanishes to order at least $m$ at $t = 0$.
\end{proof}
While we typically seek to compute the number of limiting inflection points by first computing the relevant automatic degeneracy and then applying~\eqref{Eq:FundamentalFormula}, it turns out that for the cusp in the weight-$2$ type-(a) case, it is easier to proceed in reverse.
\begin{corollary} \label{cor-cusp}
The $m^{\mathrm{th}}$-order weight-$2$ type-(a) automatic degeneracy of a cusp, cut out analytically-locally by $y^2 - x^3 = 0$, is given by the formula
$$\on{AD}_{(2)}^m(y^2 - x^3) = 2 \cdot {{m+1} \choose {4}}$$
\end{corollary}
\begin{proof}
  This follows immediately by combining~\eqref{Eq:FundamentalFormula} with Theorem~\ref{thm-cuspinf} and by noting that $\on{mult}_0 \Delta_{y^2 - x^3} = 2$.
\end{proof}
\begin{remark}
    A similar argument cannot be made for the case of weight-$2$ type-(b) inflection points. Indeed, if we were to try to replicate the argument used to prove Theorem~\ref{thm-cuspinf} in the type-(b) case, we would require the patently absurd condition that $m+1$ (rather than $m-1$) general elements of $A$ have the property that no two linearly independent elements of their span vanish to order $m$ at $t = 0$. Na\"{i}vely, one might expect that the cusp must therefore be the limit of a nonzero number of weight-$2$ type-(b) inflection points in a general $2$-parameter deformation,
    but this expectation turns out to be flawed, as we demonstrate in Example~\ref{eg-cuspy}.
\end{remark}
In \S~\ref{sec-autodegegs}, we compute other automatic degeneracies of the cusp by means of the algorithm introduced in the next section.

\subsection{An Algorithm for Finding a Basis of \texorpdfstring{$\on{P}^m(f)^{\vee}$}{Pm(f)v}} \label{sec-algae}

As in \S~\ref{sec-calc}, we take $R = k[[x,y]]$. Let $f \in R$ be the germ of a planar ICIS. The condition that the singularity is isolated implies that $\gcd\left(\frac{\d f}{\d x},\frac{\d f}{\d y}\right) = 1$.

The first step in computing automatic degeneracies of $f$ is to find a basis of the $R$-module $\on{P}^m(f)^\vee$; in this section, we present an algorithm for doing so. Recall from~\eqref{eq-defPm(f)} that $\on{P}^m(f)$ admits the following explicit description:
\begin{align}
    & \on{P}^m(f) = \nonumber \\
    & \qquad R[[u,v]]/\left(\textstyle \left(\sum_i a_i \cdot \delta_i(f)(u,v)\right)^{-1} \cdot f(u,v) - \left(\sum_i a_i \cdot \delta_i(f)(x,y)\right)^{-1} \cdot f(x,y), (u-x,v-y)^m\right) \label{eq-Pmf2ndgoround}
    \raisetag{-1.4\normalbaselineskip}
\end{align}
where $(\delta_i(f))_i \subset \on{Def}^1(f)$ is a basis and the $a_i \in k$ are general. We saw in \S~\ref{sec-calc} that in the case where $f = xy$, the description in~\eqref{eq-Pmf2ndgoround} takes on a particularly simple form, because $\on{Def}^1(xy)$ is a $1$-dimensional $k$-vector space. However, for more complicated singularities, the description in~\eqref{eq-Pmf2ndgoround} becomes cumbersome. To alleviate this problem, we introduce an auxiliary $R$-module $\on{SP}^m(f)$, defined for any planar ICIS germ $f$ as follows:
$$\on{SP}^m(f) \coloneqq R[[u,v]]/(f(u,v) - f(x,y),(u-x,v-y)^m).$$
Notice that for $f = xy$ we have $\on{P}^m(xy) = \on{SP}^m(xy)$ and that for any $f$ we can write $$\on{P}^m(f) = \on{SP}^m\left(\textstyle \left(\sum_i a_i \cdot \delta_i(f)\right)^{-1} \cdot f\right)$$
Thus, it suffices to give an algorithm for finding a basis of $\on{SP}^m(f)^\vee$; indeed, to get a basis of $\on{P}^m(f)^\vee$, we simply apply the algorithm to $\on{SP}^m\left(\textstyle \left(\sum_i a_i \cdot \delta_i(f)\right)^{-1} \cdot f\right)$. In fact, for the purpose of computing automatic degeneracies of a planar ICIS germ $f$, it suffices to work exclusively with the modules $\on{SP}^m(f)$. To see how, first consider the following definition:
\begin{defn} \label{def-sd}
We define the following quantities for each planar ICIS germ $f$ and positive integer $m$:
\begin{align*}
    & \text{\emph{weight-$1$:\hphantom{type--(a)}}}\quad \on{SD}_{(1)}^{m}(f) \coloneqq  \min_{\ol{\tau}}\big( \dim_k R/(f,I_{\ol{\tau}}^m)\big), \\
    & \text{\emph{weight-$2$ type-(a)}:}\quad \on{SD}_{(2)}^{m}(f) \coloneqq  \min_{\ol{\tau}}\big(\dim_k R/I_{\ol{\tau}}^m\big),\\
       & \text{\emph{weight-$2$ type-(b)}:}\quad \on{SD}_{(1,1)}^{m}(f) \coloneqq  \min_{\ol{\tau}}\big(\dim_k R/I_{\ol{\tau}}^m\big),
    \end{align*}
    where the minima are taken over all choices of the list $\ol{\tau}$ of $n$ elements of $\on{SP}^m(f)$ and $I_{\ol{\tau}}^m$ is the ideal cutting out the degeneracy locus of the elements $\ol{\tau}$ in $\on{SP}^m(f)^{\vee\vee}$ (here, $n = m$ in the weight-$1$ case, $n = m-1$ in the weight-$2$ type-(a) case, and $n = m+1$ in the weight-$2$ type-(b) case).
    \end{defn}
    From Definition~\ref{def-sd}, observe that to compute $\on{AD}_*^m(f)$ for $* \in \{(1),(2),(1,1)\}$, it suffices to compute $\on{SD}_*^m(r \cdot f)$ for \emph{some} unit $r \in R^\times$ and show that this value is independent of the choice of $r$.\footnote{We rely on this strategy to compute automatic degeneracies in \S~\ref{sec-upperboundsgo} and \S~\ref{sec-autodegegs}.} Note that $\on{SD}_*^m(r \cdot f)$ for $r \in R^\times$ is defined entirely in terms of $\on{SP}^m(r \cdot f)$ and does not depend on taking a general $1$-parameter deformation (as does $\on{P}^m(f)$). Thus, $\on{SD}_*^m(r \cdot f)$ is somewhat easier to compute than $\on{AD}_*^m(f)$.

In what follows, we find that it is more convenient to work with the variables $a$ and $b$ defined by $a = u - x$ and $b = v - y$. In terms of $a$ and $b$, we have
$$\on{SP}^m(f) = R[[a,b]]\bigg/\left(\sum_{d = 1}^{m-1} \sum_{s = 1}^d \frac{1}{s!}\frac{1}{(d-s)!} \frac{\d^d f}{\d x^s \d y^{(d-s)}} a^s b^{d-s},\,(a,b)^m\right),$$
where we have expressed $f(u,v) - f(x,y)$ as a Taylor bi-series expansion in the variables $u$ and $v$ centered about $(u,v) = (x,y)$ and then performed the substitutions $a = u - x$ and $b = v - y$.

\subsubsection{The Basic Idea} \label{sec-joecatchphrase}
Take $m \geq 2$, and let $K^m(f)$ denote the kernel of the natural surjection $\on{SP}^m(f) \twoheadrightarrow \on{SP}^{m-1}(f)$. Dualizing the short exact sequence
\begin{equation}\label{eq-minusmap}
\begin{tikzcd}
0 \arrow{r} & K^m(f) \arrow{r} & \on{SP}^m(f) \arrow{r} & \on{SP}^{m-1}(f) \arrow{r} & 0
\end{tikzcd}
\end{equation}
gives the exact sequence
\begin{equation}\label{eq-longexact}
\begin{tikzcd}
0 \arrow{r} & \on{SP}^{m-1}(f)^\vee \arrow{r} & \on{SP}^m(f)^\vee \arrow{r} & K^m(f)^\vee \arrow{r} & \on{Ext}^1(\on{SP}^{m-1}(f),R) \arrow{r} & \cdots
\end{tikzcd}
\end{equation}
It follows from the exactness of~\eqref{eq-dualexact} that the map $\on{SP}^m(f)^\vee \to K^m(f)^\vee$ is in fact surjective, so we obtain a short exact sequence
\begin{equation} \label{eq-luckyduck}
\begin{tikzcd}
0 \arrow{r} & \on{SP}^{m-1}(f)^\vee \arrow{r} & \on{SP}^m(f)^\vee \arrow{r} & K^m(f)^\vee \arrow{r} & 0
\end{tikzcd}
\end{equation}
Moreover, it follows from Remark~\ref{rem-dualexact} that for every $m$, the $R$-modules $\on{SP}^m(f)^\vee$ and $K^m(f)^\vee$ are free of ranks $m$ and $1$, respectively, so in particular, the sequence in~\eqref{eq-luckyduck} splits. Consequently, we can construct bases of the modules $\on{SP}^m(f)^\vee$ inductively: if we can exhibit an element of $\on{SP}^m(f)^\vee$ whose image in $K^m(f)^\vee$ generates all of $K^m(f)^\vee$, then we can simply append that element to a previously constructed basis of $\on{SP}^{m-1}(f)^\vee$ to obtain a basis of $\on{SP}^m(f)^\vee$.

\begin{remark}
In deriving the algorithm, we do not ever use the fact that the map $\on{SP}^m(f)^\vee \to K^m(f)^\vee$ is surjective. Indeed, the surjectivity of this map follows immediately by applying the algorithm to construct a basis of $\on{SP}^m(f)^\vee$. All we need is that this map is nonzero, as explained in \S~\ref{sec-algs}.
\end{remark}

\begin{remark}
Note that the $R$-module $\on{Ext}^1(\on{SP}^{m-1}(f),R)$ is of rank $0$ and has finite length. When $m = 2$, we have $\on{SP}^{m-1}(f) = \on{SP}^1(f) \simeq R$, so $\on{Ext}^1(\on{SP}^1(f),R) = 0$. Incidentally, it is \emph{not} true that $\on{Ext}^1(\on{SP}^{m-1}(f),R) = 0$ for every $m$; indeed, one easily verifies by hand or using {\tt Macaulay2} that for $m = 3$ and $f = y^2 - x^2$, the $R$-module $\on{Ext}^1(\on{SP}^2(f),R)$ has rank $0$ and length $1$. Nonetheless, continuing with the example where $m = 3$ and $f = y^2 - x^2$, one also readily checks that the map $K^2(f)^\vee \to \on{Ext}^1(\on{SP}^2(f),R)$ is the zero map, which is consistent with the fact that the map $\on{SP}^2(f)^\vee \to K^2(f)^\vee$ is surjective.
\end{remark}

\subsubsection{Explicit Presentations} \label{sec-exppres}
Before we proceed with deriving the algorithm, we provide explicit presentations of the $R$-modules $\on{SP}^m(f)$ and $K^m(f)^\vee$, and we use these presentations to obtain a preliminary description of $\on{SP}^m(f)^\vee$.

We start by describing $\on{SP}^m(f)$. As in \S~\ref{sec-backtothefuture}, let $\kappa_j = \sum_{i = 0}^j i$ denote the $j^{\mathrm{th}}$ triangular number. Because $\on{SP}^m(f)$ is obtained as a quotient of $R[[a,b]]/(a,b)^m$, it follows that the $\kappa_m$-many monomials $a^ib^j$ for $0 \leq i + j \leq m-1$ form a set of generators of $\on{SP}^m(f)$ over $R$. The only relations on these generators arise from setting $f(u,v) - f(x,y) = 0$; in terms of $a$ and $b$, these relations are given explicitly as follows for $m \geq 2$ (the case $m = 1$ is trivial: note that $\on{SP}^m(f) = R$, so there are no relations):
\begin{equation}\label{eq-relish}
r_{ij} \coloneqq \sum_{d = 1}^{m-1-(i+j)} \sum_{s = 0}^d \frac{1}{s!}\frac{1}{(d-s)!} \frac{\d^d f}{\d x^s \d y^{(d-s)}} a^{s+i} b^{d-s+j} = 0 \quad \text{for} \quad 0 \leq i + j \leq m-2.
\end{equation}
Notice that there are $\kappa_{m-1}$-many relations $r_{ij} = 0$ in~\eqref{eq-relish}, one for each pair $(i,j)$ satisfying $0 \leq i + j \leq m-2$, so we obtain an exact sequence

\begin{equation}\label{eq-trialexact}
\begin{tikzcd}
R^{\kappa_{m-1}} \arrow{r} & R^{\kappa_m} \arrow{r} & \on{SP}^m(f) \arrow{r} & 0
\end{tikzcd}
\end{equation}

\noindent Observe that for any integer $m' \in \{1 ,\dots, m\}$, we have a natural surjection $\on{SP}^m(f) \twoheadrightarrow \on{SP}^{m'}(f)$; this surjection gives rise to the following diagram:

\begin{equation}\label{eq-doubleexact}
\begin{tikzcd}
  R^{\kappa_{m-1}} \arrow{r} \arrow[two heads]{d} & R^{\kappa_m} \arrow{r} \arrow[two heads]{d} & \on{SP}^m(f) \arrow{r} \arrow[two heads]{d} & 0 \\
  R^{\kappa_{m'-1}} \arrow{r} & R^{\kappa_{m'}} \arrow{r} & \on{SP}^{m'}(f) \arrow{r} & 0
\end{tikzcd}
\end{equation}
where the first two downward surjections are the natural ones and can be thought of as follows. Regard $R^{\kappa_m}$ as a free module with basis elements corresponding to the elements $(i,j)$ of the poset $\{(i,j) : 0 \leq i + j \leq m-1\}$; viewing this poset as a ``triangle'' with the $d^{\mathrm{th}}$ row of the triangle containing the pairs $(i,j)$ with $i + j = d$, the surjection $R^{\kappa_m} \to R^{\kappa_{m'}}$ can be thought of as killing rows $m'$ through $m-1$ of the triangle.

\begin{lemma} \label{lem-exact}
   The sequence in ~\eqref{eq-trialexact} is left-exact, and thus the resulting short exact sequence
   \begin{equation} \label{eq-exactpres}
   \begin{tikzcd}
0 \arrow{r} & R^{\kappa_{m-1}} \arrow{r} & R^{\kappa_m} \arrow{r} & \on{SP}^m(f) \arrow{r} & 0
\end{tikzcd}
\end{equation}
   is a length-$1$ resolution of $\on{SP}^m(f)$ by free $R$-modules.
\end{lemma}
\begin{proof}
It suffices to prove that the relations $r_{ij} = 0$ are linearly independent over $R$. We proceed by induction on the order $m \geq 1$. For the base cases, when $m = 1$, there are no relations, so the claim is vacuous; when $m = 2$, notice that there is exactly one relation in~\eqref{eq-relish}, namely $r_{00} = 0$, so the claim is clear. Now take $m \geq 3$, and assume that the claim holds for the order $m-1$. Let $r \coloneqq \sum_{0 \leq i + j \leq m-2} \alpha_{ij}\cdot r_{ij} = 0$ be some relation on the $r_{ij}$ with $\alpha_{ij} \in R$.
Observe that the image of $r$ under the natural surjection $R^{\kappa_{m}} \twoheadrightarrow R^{\kappa_{m-1}}$ (defined as in~\eqref{eq-doubleexact}) is given by
$$\ol{r} \coloneqq \sum_{0 \leq i + j \leq m-3} \alpha_{ij}\cdot\left(\sum_{d = 1}^{m-2-(i+j)} \sum_{s = 0}^d \frac{1}{s!}\frac{1}{(d-s)!} \frac{\d^d f}{\d x^s \d y^{(d-s)}} a^{s+i} b^{d-s+j} \right) \in R^{\kappa_{m-1}}.$$
Setting $r = 0$ in $R^{\kappa_m}$ implies that its image $\ol{r}$ in $R^{\kappa_{m-1}}$ is also zero, but by the inductive hypothesis, the fact that $\ol{r} = 0$ implies that $\alpha_{ij} = 0$ for $0 \leq i + j \leq m-3$. It remains to show that $\alpha_{ij} = 0$ for $i + j = m-2$; we handle these coefficients by contradiction as follows. Let $i'\in \{0 , \dots, m-2\}$ be the smallest integer so that for $j' = m-2 -i'$ we have $\alpha_{i'j'} \neq 0$. Note that $(i,j) = (i',j')$ is the only pair such that $i + j = m-2$ \emph{and} $r_{ij}$ contains a term proportional to $a^{i'+1}b^{j'}$ \emph{and} $\alpha_{ij} \neq 0$. But then no other term in $r$ can cancel out the term proportional to $a^{i'+1}b^{j'}$ in $\alpha_{i'j'}\cdot r_{i'j'}$, which contradicts the fact that $r = 0$. It follows that $\alpha_{ij} = 0$ for all $0 \leq i + j \leq m-2$, implying that the relation $r$ is trivial and thus proving the claim.
\end{proof}

Next, we obtain the following description of $K^m(f)^\vee$:

\begin{lemma} \label{lem-thebottomline}
We have an isomorphism $R \overset{\sim}\longrightarrow K^m(f)^\vee$ taking $1 \in R$ to the unique functional $\theta_m \in K^m(f)^\vee$ that sends $a^{m-1}$ to $\left(\frac{\d f}{\d y}\right)^{m-1}$.
\end{lemma}
\begin{proof}
Observe that $K^m(f)$ is the quotient of the free $R$-module on the generators $a^ib^j$ for $i + j = m-1$ by the relations $r_{i'j'}$ for $i' + j' = m-2$. Thus, a functional $\theta \in K^m(f)^\vee$ is determined by specifying $\theta(a^ib^j)$ in such a way that $\theta(r_{i'j'}) = 0$. Note that the relations $r_{i'j'} = 0$ take on the following simple form:
\begin{equation} \label{eq-simplerelish}
\frac{\d f}{\d x} a^{i'+1}b^{j'} = -\frac{\d f}{\d y} a^{i'}b^{j'+1} \quad \text{for} \quad i' + j' = m-2
\end{equation}
Combining these relations and applying the functional $\theta$ yields the following divisibility properties:
\begin{equation} \label{eq-divisibility}
\left(\frac{\d f}{\d y}\right)^{i}\left(\frac{\d f}{\d x}\right)^{j} \bigg\vert\, \theta(a^ib^j) \quad \text{for} \quad i + j = m-1
\end{equation}
For a given pair $(i, j)$ such that $i + j = m-1$, once we fix a value for $\theta(a^ib^j)$ that satisfies the corresponding divisibility property in~\eqref{eq-divisibility}, the entire functional $\theta$ is determined by the relations in~\eqref{eq-simplerelish}. This gives the desired characterization of $K^m(f)^\vee$.
\end{proof}

Finally, we use these explicit presentations of $\on{SP}^m(f)$ and $K^m(f)^\vee$ to obtain the following description of $\on{SP}^m(f)^\vee$:

\begin{lemma} \label{lem-duapartia}
We have the following basic properties of the $R$-module $\on{SP}^m(f)^\vee$:
\begin{enumerate}
    \item Any functional $\phi \in \on{SP}^m(f)^\vee$ is determined by specifying its values $\phi(a^ib^j)$ on the generators of $\on{SP}^m(f)$, subject to the relations $\phi(r_{ij}) = 0$.
    \item Given $\phi \in \on{SP}^m(f)^\vee$, there exists $c \in R$ such that for every pair $(i,j)$ with $i + j =
m-1$, we have
$$\phi(a^ib^j) = c \cdot (-1)^j \left(\frac{\d f}{\d y}\right)^i \left(\frac{\d f}{\d x}\right)^j$$
    In fact, the functional $\phi$ may be chosen so that $c \neq 0$.
\end{enumerate}
\end{lemma}
\begin{proof}
  Part (a) follows from the fact that to be an element of $\on{SP}^m(f)^\vee = \Hom_R(\on{SP}^m(f),R)$ is to be an element of $\Hom_R(R^{\kappa_m}, R)$ that vanishes on $R^{\kappa_{m-1}} \subset R^{\kappa_m}$, where, as in Lemma~\ref{lem-exact}, we view $R^{\kappa_m}$ as the module of generators $a^ib^j$ and $R^{\kappa_{m-1}}$ as the submodule of relations. As for part (b), the element $c$ is given by taking the image of $\phi$ in $K^m(f)^\vee$ and identifying that image with an element of $R$ via the isomorphism given in Lemma~\ref{lem-thebottomline}. To see why we can find $\phi$ with $c \neq 0$, notice that the map $\on{SP}^m(f)^\vee \to K^m(f)^\vee$ is not the zero map; indeed, if it were zero, the exactness of the sequence in~\eqref{eq-longexact} would imply that the map $\on{SP}^{m-1}(f)^\vee \to \on{SP}^m(f)^\vee$ is an isomorphism, but this is impossible because
  \begin{equation*}
  \on{rk} \on{SP}^{m-1}(f)^\vee = m-1 \neq m = \on{rk} \on{SP}^m(f)^\vee. \hspace*{\fill} \qedhere
  \end{equation*}
\end{proof}

\subsubsection{The Algorithm} \label{sec-algs} As we stated in \S~\ref{sec-joecatchphrase}, the basic idea behind our algorithm is to inductively extend a basis of $\on{SP}^{m-1}(f)^\vee$ to a basis of $\on{SP}^m(f)^\vee$ by explicitly constructing a preimage in $\on{SP}^m(f)^\vee$ of the generator $\theta_m$ of $K^m(f)^\vee$. For the base case of the induction, we simply let $\wt{\theta}_1 \in \on{SP}^1(f)^\vee$ be the functional defined by $\wt{\theta}_1(a^0b^0) = 1$; then $\wt{\theta}_1$ is clearly a preimage of $\theta_1$. In what follows, we take $m \geq 2$, and we assume by induction that we have constructed preimages $\wt{\theta}_i \in \on{SP}^i(f)^\vee$ of $\theta_i$ for each $i \in \{1 , \dots,  m-1\}$.

Choose any $\phi \in \on{SP}^m(f)^\vee$ such that for every pair $(i,j)$ with $i + j =
m-1$, we have
$$\phi(a^ib^j) = c \cdot (-1)^j \left(\frac{\d f}{\d y}\right)^i \left(\frac{\d f}{\d x}\right)^j$$
where $c \in R \setminus \{0\}$. Note that the existence of such a functional is guaranteed by Lemma~\ref{lem-duapartia} and that the image of $\phi$ in $K^m(f)^\vee$ is equal to $c \cdot \theta_m$. We regard the functional $\phi$ as the ``input'' to the algorithm, and at the end of this section, we shall provide an explicit construction of an input $\phi$.

Let $c' \in R$ be an irreducible element such that $c' \mid c$. We now claim that it suffices to exhibit a functional $\phi' \in \on{SP}^{m-1}(f)^\vee$ with the property that $c' \mid (\phi+\phi')(a^ib^j)$ for every pair $(i,j)$. Indeed, given such a $\phi'$, note that $\tfrac{1}{c'} \cdot(\phi + \phi')$ is a functional in $\on{SP}^m(f)^\vee$ such that for every pair $(i,j)$ with $i + j = m-1$, we have
$$\frac{1}{c'} \cdot (\phi + \phi')(a^ib^j) = \frac{c}{c'} \cdot (-1)^j \left(\frac{\d f}{\d y}\right)^i \left(\frac{\d f}{\d x}\right)^j$$
which is to say that the image of $\tfrac{1}{c'} \cdot(\phi + \phi')$ under the map $\on{SP}^m(f)^\vee \to K^m(f)^\vee$ is $\frac{c}{c'}\cdot \theta_m$, whereupon we can replace $\phi$ with $\frac{1}{c'}\cdot(\phi + \phi')$ and repeat the process described in this paragraph until we have exhausted all irreducible factors of $c$; we can then take $\wt{\theta}_m$ to be the resulting functional.

We next claim that it further suffices to exhibit for each $m' \in \{1, \dots, m\}$ a functional $\phi_{m'} \in \on{SP}^{m'}(f)^\vee$ such that the functional $\phi - \sum_{\ell=m'}^{m-1} \phi_\ell \in \on{SP}^m(f)^\vee$ has the property that for each pair $(i,j)$ with $m'-1 \leq i + j \leq m-1$,
$$c' \bigg\vert\, \left(\phi-\sum_{\ell=m'}^m \phi_\ell \right)(a^ib^j).$$
Indeed, given such functionals $\phi_1, \dots, \phi_m$, we can simply take $\phi' = -\sum_{\ell = 1}^m \phi_\ell$. We construct these functionals using another induction in the following lemma:

\begin{lemma}
We can construct a functional $\phi_{m'} \in \on{SP}^{m'}(f)^\vee$ for each $m' \in \{1, \dots, m\}$ such that for each pair $(i,j)$ with $m'-1 \leq i + j \leq m-1$,
$$\left(\phi-\sum_{\ell=m'}^{m-1} \phi_\ell \right)(a^ib^j) = c' \cdot d_{ij}$$
for an element $d_{ij} \in \left(\frac{\d f}{\d x},\frac{\d f}{\d y}\right)^{i+j} \subset R$ satisfying the following explicit description if $c' \nmid \frac{\d f}{\d y}$:
\begin{itemize}
    \item If $(m-1)-(i+j)+1 \leq i+j$ and $i = 0$ or $j = 0$, then $d_{ij}$ is a linear
combination with coefficients in $R$ of $(m-1)-(i+j)+1$ terms.
\begin{itemize}
\item[$\circ$] If $j = 0$, these terms are given by $\left(\frac{\d f}{\d y}\right)^{i-\ell}\left(\frac{\d f}{\d x}\right)^\ell$ for $\ell \in \{0 , \dots, (m-1) - i\}$.
\item[$\circ$] If $i = 0$, these terms are given by $\left(\frac{\d f}{\d y}\right)^\ell \left(\frac{\d f}{\d x}\right)^{j-\ell}$ for $\ell \in \{0 , \dots, (m-1) - j\}$.
\end{itemize}
\item If $(m-1)-(i+j)+1 \leq i+j$ and $i,j > 0$, then $d_{ij}$ is a linear combination with
coefficients in $R$ of the $(m-1)-(i+j)+2$ terms $\left(\frac{\d f}{\d y}\right)^{i+1-\ell} \left(\frac{\d f}{\d x}\right)^{j-1+\ell}$ for $\ell \in \{0 , \dots,  (m-1) - (i+j) + 1\}$.
\item If $(m-1)-(i+j)+1 > i+j$, then $d_{ij}$ is a linear combination with coefficients in $R$ of the
$i+j+1$ terms $\left(\frac{\d f}{\d y}\right)^{i+j-\ell} \left(\frac{\d f}{\d x}\right)^\ell$ for $\ell \in \{0 , \dots,  i+j\}$.
\end{itemize}
Otherwise, if $c' \mid \frac{\d f}{\d y}$, then we have $c' \nmid \frac{\d f}{\d x}$, 
so a similar description of $d_{ij}$ may be obtained as above by switching the variables $x$ and $y$.
\end{lemma}
\begin{proof}
We prove the lemma by reverse induction on $m'$, starting from $m' = m$ and going down to $n = 1$. For the base case, where $m' = m$, note that the lemma clearly holds by taking $\phi_m = 0$. Now assume that the lemma holds for some $m' \in \{3 , \dots, m\}$. To see that the lemma holds for $m'-1$, we need to construct the functional $\phi_{m'-1} \in \on{SP}^{m'-1}(f)^\vee$; note that $\phi_{m'-1}$, if it exists, may be replaced with any translate of itself by an element of $\on{SP}^{m'-2}(f)^\vee$. We now invoke the inductive hypothesis of the ``outer'' induction that we are doing on the rank $m$. By this hypothesis, the functional $\phi_{m'-1}$, if it exists, may be taken to be given by $h \cdot \wt{\theta}_{m'-1}$ for some element $h \in R$. Thus, it suffices to construct this $h$.

Suppose we have constructed elements $g,h \in R$ such that
\begin{equation} \label{eq-constructgh}
\left(\phi - \sum_{\ell = m'}^m \phi_\ell\right)(a^{m'-2}b^0) = c' \cdot g + h \cdot \left(\frac{\d f}{\d y}\right)^{m'-2}
\end{equation}
Then if we were to take $\phi_{m'-1} = h \cdot \wt{\theta}_{m'-1}$, we would have that
$$\left(\phi - \sum_{\ell = m'-1}^m \phi_\ell\right)(a^{m'-2}b^0) = c' \cdot g.$$
Moreover, if $c' \nmid \frac{\d f}{\d y}$, then the relations $\left(\phi - \sum_{\ell = m'}^m \phi_\ell\right)(r_{ij}) = 0$ for pairs $(i,j)$ with $i+j = m'-3$\footnote{Recall that these relations hold by part (a) of Lemma~\ref{lem-duapartia}.} would together imply that for pairs $(i,j)$ with $i + j = m'-2$, we have
\begin{equation} \label{eq-divrepeat}
c' \bigg\vert\, \left(\phi - \sum_{\ell = m'-1}^m \phi_\ell\right)(a^ib^j).
\end{equation}
Further still, by the ``inner'' induction we are doing on the index $m'$, we would have that~\eqref{eq-divrepeat} also holds for every pair $(i,j)$ with $m'-1 \leq i+j \leq m-1$. Thus, it suffices to construct $g,h$ satisfying~\eqref{eq-constructgh}.

We now show how to find the elements $g,h$. Take the relations $\left(\phi - \sum_{\ell = m'}^m \phi_\ell\right)(r_{ij}) = 0$ for pairs $(i,j)$ with $i+j = m'-3$ and combine them to obtain a single relation between the values of $\left(\phi - \sum_{\ell = m'}^m \phi_\ell\right)(a^{m'-2}b^0)$ and $\left(\phi - \sum_{\ell = m'}^m \phi_\ell\right)(a^0b^{m'-2})$. The resulting relation has the following form:
\begin{align}
& \left(\frac{\d f}{\d x}\right)^{m'-2} \cdot \left(\phi - \sum_{\ell = m'}^m \phi_\ell\right)(a^{m'-2}b^0) + (-1)^{m'-3}
\cdot \left(\frac{\d f}{\d y}\right)^{m'-2} \cdot \left(\phi - \sum_{\ell = m'}^m \phi_\ell\right)(a^0b^{m'-2}) + \label{eq-theothers} \\
& \qquad\qquad [\text{other terms}] = 0. \notag
\end{align}
Take the terms labelled ``[other terms]'' in the relation~\eqref{eq-theothers}, and split them into two pieces: those terms divisible by $\left(\frac{\d f}{\d x}\right)^{m'-2}$, which we call
``[$x$-divisible terms]'' and those that are not, which we call ``[$x$-indivisible terms].'' Thus, we have
\begin{equation} \label{eq-words}
\text{[other terms]} = \text{[$x$-divisible terms]} + \text{[$x$-indivisible terms]},
\end{equation}
and combining~\eqref{eq-theothers} with~\eqref{eq-words} yields that
$$\left(\frac{\d f}{\d x}\right)^{m'-2} \bigg\vert \left([x\text{-indivisible terms}] + (-1)^{m'-3} \cdot \left(\frac{\d f}{\d y}\right)^{m'-2} \cdot \left(\phi - \sum_{\ell = m'}^m \phi_\ell\right)(a^0b^{m'-2})\right)$$
The inner inductive hypothesis provides us with an explicit description of all the terms in [other terms]. From this description, we deduce that every term in [other terms] is divisible by $c'$ and further that every term in [$x$-indivisible terms] is divisible by $\left(\frac{\d f}{\d y}\right)^{m'-2}$. We can then take
\begin{align*}
g & = -\frac{1}{c'} \cdot \left(\frac{\d f}{\d x}\right)^{2-m'} \cdot [x\text{-divisible terms}], \text{ and} \\
h & = -\left(\frac{\d f}{\d x}\right)^{2-m'} \cdot  \left( \left(\frac{\d f}{\d y}\right)^{2-m'} \cdot [x\text{-indivisible terms}] + (-1)^{m'-3} \cdot \left(\phi - \sum_{\ell = m'}^m \phi_\ell\right)(a^0b^{m'-2})\right)
\end{align*}
It is then a tedious calculation to check using the relations $r_{ij} = 0$ for pairs $(i,j)$ with $i+j = n-3$ and the explicit description given by the inner inductive hypothesis that the resulting functional $\phi - \sum_{\ell = m'-1}^m \phi_\ell$ satisfies the explicit description given in the statement of the lemma. We omit the calculation for the sake of brevity.

We have thus constructed the functionals $\phi_2, \dots, \phi_m$. The functional $\phi_1 \in \on{SP}^1(f)^\vee$ may simply be taken to be the unique functional such that $\phi_1(a^0b^0) = -\phi(a^0b^0)$.
\end{proof}

All that remains is to demonstrate how to exhibit an input functional $\phi$ with the property that its image in $K^m(f)^\vee$ is nonzero. Consider the relations $\phi(r_{ij}) = 0$ from part (a) of Lemma~\ref{lem-duapartia}, and perform the following operations: replace each instance of $\phi(a^ib^0)$ with $1$ for every $i$, and replace each instance of $\phi(a^ib^j)$ with the variable $z_{ij}$ for pairs $(i,j)$ with $j > 0$. The result can be thought of as a system of linear equations in the $z_{ij}$, and we claim that this system of equations has a unique solution $\{z_{ij} = a_{ij}\}$ over the localization $R_{\frac{\d f}{\d y}}$. Indeed, we can solve for the $z_{ij}$ by reverse induction on the quantity $i+j$ as follows. When $i' + j' = m-2$, the relation $\phi(r_{i'j'}) = 0$ implies that
\begin{equation} \label{eq-algadd1}
\left(\frac{\d f}{\d y}\right) \cdot z_{i'(j'+1)} = z_{(i'+1)j'}  \cdot -\left(\frac{\d f}{\d x}\right)
\end{equation}
Now suppose that we have solved for $z_{ij}$ when $m' < i + j \leq m-1$. The relation $\phi(r_{i'j'}) = 0$ for $i' + j' = m'-1$ can be expressed as
\begin{equation} \label{eq-algadd2}
\left(\frac{\d f}{\d y}\right) \cdot z_{i'(j'+1)} = z_{(i'+1)j'}  \cdot -\left(\frac{\d f}{\d x}\right) + [\text{additional terms}],
\end{equation}
where the only $z_{ij}$'s that appear in ``$[\text{additional terms}]$'' are those that we assumed had been solved for in the inductive hypothesis. From~\eqref{eq-algadd1} and~\eqref{eq-algadd2}, one readily observes that it is possible to solve for the $z_{ij}$ when $\frac{\d f}{\d y}$ is invertible, as claimed. Taking $N$ to be the largest power of $\frac{\d f}{\d y}$ occurring in the denominators of the $a_{ij}$, define a functional $\phi_1$ by stipulating that $\phi_1(a^ib^0) = \left(\frac{\d f}{\d y}\right)^N$ and $\phi_1(a^ib^j) = \left(\frac{\d f}{\d y}\right)^N \cdot a_{ij}$ for pairs $(i,j)$ with $j > 0$.

Similarly, consider the relations $\phi(r_{ij}) = 0$, and perform the following operations: replace each instance of $\phi(a^0b^j)$ with $1$ for every $j$, and replace each instance of $\phi(a^ib^j)$ with the variable $z_{ij}$ for pairs $(i,j)$ with $i > 0$. The resulting system of linear equations in the $z_{ij}$ has a unique solution $\{z_{ij} = b_{ij}\}$ over the localization $R_{\frac{\partial f}{\partial x}}$, and $N$ is the largest power of $\frac{\partial f}{\partial x}$ occurring in the denominators of the $b_{ij}$. Define a functional $\phi_2$ by stipulating that $\phi_2(a^0b^j) = \left(\frac{\partial f}{\partial x}\right)^N$ and $\phi_2(a^ib^j) = \left(\frac{\partial f}{\partial x}\right)^N \cdot b_{ij}$ for pairs $(i,j)$ with $i > 0$.

We then take $\phi = \phi_1 + \phi_2$. The corresponding value of $c$ is given by $$\left(\frac{\partial f}{\partial y}\right)^{N-m+1} + (-1)^{m-1} \cdot \left(\frac{\partial f}{\partial x}\right)^{N-m+1} $$ which is coprime to each of $\frac{\d f}{\d x}$ and $\frac{\d f}{\d y}$ because we have stipulated that $\on{gcd}\left(\frac{\d f}{\d x}, \frac{\d f}{\d y}\right) = 1$. This completes the algorithm.

Given the algorithm described above, it should be possible --- at least in theory --- for one to compute the weight-$1$ and weight-$2$ automatic degeneracies of any fixed order $m$ and for any fixed $f$ cutting out a planar ICIS. We apply this algorithm to calculate automatic degeneracies in this manner in \S~\ref{sec-autodegegs}.

\subsection{Bounds on Automatic Degeneracies} \label{sec-boundit}

The basis of $\on{SP}^m(f)^\vee$ produced by the algorithm in \S~\ref{sec-algs} is difficult to write out explicitly for all $m$, which in turn makes it difficult to find a formula for the $m^{\mathrm{th}}$-order automatic degeneracies of a given ICIS (other than a node) as a function of $m$. However, it is natural to wonder whether we can at least obtain lower and upper bounds as functions of $m$. This section is devoted to finding such bounds.

\subsubsection{Lower Bounds} \label{sec-bounditbelow}
In the following theorem, we obtain meaningful lower bounds on the $m^{\mathrm{th}}$-order automatic degeneracies of an arbitrary ICIS.
\begin{theorem} \label{cor-worse}
For an ICIS cut out analytically-locally by $f = 0$, the $m^{\mathrm{th}}$-order automatic degeneracies satisfy the lower bounds
\begin{align*}
\on{AD}_{(1)}^m(f) & \geq \delta_{f} \cdot m(m-1)\\
\on{AD}_{(2)}^m(f) & \geq (\on{mult}_0 \Delta_f) \cdot {{m+1} \choose {4}} \\
\on{AD}_{(1,1)}^m(f) & \geq (\on{mult}_0 \Delta_f) \cdot {{m+2} \choose {4}}
\end{align*}
\end{theorem}
\begin{proof}
The weight-$2$ bounds follow immediately from~\eqref{Eq:FundamentalFormula} and~\eqref{eq-fund2}, so it remains to consider the weight-$1$ case. It was known classically (see~\cite{cayley2} and~\cite{MR1505600}) that there exists a deformation of the singularity germ $f$ such that the number of nodes lying on the deformed curve is equal to $\delta_f$. The support of the degeneracy scheme at each of these nodes is at least $\on{AD}_{(1)}^m(xy) = m(m-1)$, so by upper-semicontinuity, we must have that $\on{AD}_{(1)}^m(f) \geq \delta_f \cdot m(m-1)$.
\end{proof}

\subsubsection{Upper Bounds in the Planar Weight-$2$ Case} \label{sec-upperboundsgo}

To obtain an upper bound on each of the $m^{\mathrm{th}}$-order weight-$2$ automatic degeneracies of a planar ICIS cut out analytically-locally by $f = 0$, we perform two simplifying specializations: first, we specialize the module $\on{SP}^m(f)$ itself, and second, we make a special choice of the elements $\tau_\ell \in \on{SP}^m(f)$ with respect to which we are computing the automatic degeneracies.

We begin by introducing the setup required to specialize the module of principal parts:

\begin{defn} \label{def-grandparts}
We define the \emph{grand modules of $m^{\mathrm{th}}$-order principal parts} as follows. For $m \in \{1 ,2\}$, we take $R_m \coloneqq R$ and $\on{GP}^m(f) \coloneqq \on{SP}^m(f)$. Next, let $m \geq 3$ be an integer, let $f$ be the germ of a planar ICIS, and let $R_m \coloneqq R[t_2, \dots, t_{m-1}]$. As before, let $\kappa_m = \sum_{i = 0}^m i$ denote the $m^{\mathrm{th}}$ triangular number. View ${R_m}^{\kappa_m}$ as the free module on the generators $a^ib^j$ for pairs $(i,j)$ with $0 \leq i+j \leq m-1$ as in \S~\ref{sec-exppres}, and consider the $R_m$-submodule $M \subset {R_m}^{\kappa_m}$ generated by the set of elements $\{r_{ij}' : 0 \leq i + j \leq m-2\}$ defined by
$$r_{ij}' \coloneqq \frac{\d f}{\d x}a^{i+1}b^j + \frac{\d f}{\d y}a^ib^{j+1} + \sum_{d = 2}^{m-1-(i+j)} \left(\prod_{\ell =2}^d t_{\ell+i+j}\right) \cdot  \sum_{s = 0}^d \frac{1}{s!}\frac{1}{(d-s)!} \frac{\d^d f}{\d x^s \d y^{(d-s)}} a^{s+i} b^{d-s+j}.$$
We then put $\on{GP}^m(f) \coloneqq {R_m}^{\kappa_m}/M$.
\end{defn}

For $m \geq 3$ and a point $\vec{t} = (T_2, \dots, T_{m-1}) \in k^{m-2}$, let $\mathfrak{p}_{\vec{t}}$ denote the prime ideal of $R_m$ defined by $\mathfrak{p}_{\vec{t}} = (t_2 - T_2, \dots, t_{m-1} - T_{m-1})$, and let
$$\on{GP}^m(f)_{\vec{t}} \coloneqq R_m/\mathfrak{p}_{\vec{t}} \otimes_{R_m} \on{GP}^m(f).$$
Note that if we take $T_2 = \cdots = T_{m-1} = 1$ and $\vec{t} = (T_2, \dots, T_{m-1})$, then we have that $\on{GP}^m(f)_{\vec{t}} \simeq \on{SP}^m(f)$. The following lemma says that this isomorphism holds for all choices of $\vec{t}$, as long as none of the coordinates of $\vec{t}$ are equal to zero.

\begin{lemma} \label{lem-varysame}
Let $m \geq 3$ and $\vec{t} = (T_2, \dots, T_{m-1}) \in (k^\times)^{m-2}$. The map $\Phi_{\vec{t}} \colon \on{GP}^m(f)_{\vec{t}} \to \on{SP}^m(f)$ whose values on the generators $a^ib^j$ for pairs $(i,j)$ with $0 \leq i+j \leq m-1$ are given by
$$\Phi_{\vec{t}}(a^ib^j) = \begin{cases} a^0b^0 & \text{if} \quad i = j = 0 \\ \left(\prod_{\ell = 1 + i + j}^{m-1} T_\ell\right) \cdot a^ib^j & \text{otherwise} \end{cases}$$
is a well-defined isomorphism of $R$-modules.
\end{lemma}
\begin{proof}
  Clearly, if we substitute $\vec{t}$ into the relations $r_{ij}'$, we have $\Phi_{\vec{t}}(r_{ij}') = \left(\prod_{\ell = 2 + i + j}^{m-1} T_\ell\right) \cdot r_{ij}$, so $\Phi_{\vec{t}}$ takes relations to relations and is therefore well-defined. That $\Phi_{\vec{t}}$ is an isomorphism follows by simply observing that the inverse map is given on the generators by
  \begin{equation*}
 \Phi_{\vec{t}}^{-1}(a^ib^j) = \begin{cases} a^0b^0 & \text{if} \quad i = j = 0 \\ \left(\prod_{\ell = 1 + i + j}^{m-1} T_\ell^{-1}\right) \cdot a^ib^j & \text{otherwise} \end{cases} \hspace*{\fill} \qedhere\end{equation*}
\end{proof}

The maps $\Phi_{\vec{t}}$ induce dual isomorphisms $\Phi_{\vec{t}}^\vee \colon \on{SP}^m(f)^\vee \to \on{GP}^m(f)_{\vec{t}}^\vee$. Let $(\wt{\theta}_1, \dots, \wt{\theta}_m)$ be the basis of $\on{SP}^m(f)^\vee$ produced by the algorithm in \S~\ref{sec-algs}, and for each $i \in \{1, \dots, m\}$, let $\wt{\theta}_{i_{\vec{t}}} \coloneqq \Phi_{\vec{t}}^\vee(\wt{\theta}_i) = \wt{\theta}_i \circ \Phi_{\vec{t}}$. Then $(\wt{\theta}_{1_{\vec{t}}}, \dots, \wt{\theta}_{m_{\vec{t}}})$ is a basis of $\on{GP}^m(f)_{\vec{t}}^\vee$.

Now, to be able to use the grand modules of principal parts to study automatic degeneracies, we must first show that their duals are free. Observe that for every $m \geq 2$ and $\ell \in \{1 , \dots, m\}$ we have the following two natural maps, the second being the dual of the first:
\begin{equation} \label{eq-grandsurj}
\on{GP}^m(f) \twoheadrightarrow R_m \otimes_{R_\ell} \on{GP}^\ell(f) \quad \text{and} \quad R_m \otimes_{R_\ell} \on{GP}^\ell(f)^\vee \hookrightarrow \on{GP}^m(f)^\vee.
\end{equation}
Letting $\on{GK}^m(f)$ denote the kernel of the first map in~\eqref{eq-grandsurj} when $\ell = m-1$, we obtain the following exact sequence:
\begin{equation} \label{eq-grandexact}
    \begin{tikzcd}
0 \arrow{r} & R_m \otimes_{R_{m-1}} \on{GP}^{m-1}(f)^\vee \arrow{r} & \on{GP}^m(f)^\vee \arrow{r} & \on{GK}^m(f)^\vee \arrow{r} \arrow{r} & \cdots
\end{tikzcd}
\end{equation}

\begin{lemma} \label{lem-grandfree}
Let $m \geq 3$. The $R_m$-module $\on{GP}^m(f)^\vee$ is free of rank $m$. In particular, we can construct elements $\vartheta_\ell \in \on{GP}^\ell(f)^\vee$ for $\ell \in \{1 , \dots, m\}$ of $\on{GP}^m(f)^\vee$ with the following properties:
\begin{itemize}
    \item The list $(\vartheta_1, \dots, \vartheta_m)$ forms a basis of $\on{GP}^m(f)^\vee$.\footnote{Here we are abusing notation, writing $\vartheta_\ell$ for what is actually $1 \otimes \vartheta_\ell \in R_m \otimes_{R_\ell} \on{GP}^\ell(f)^\vee \hookrightarrow \on{GP}^m(f)\vee$.}
    \item For $\ell \in \{1 , \dots, m\}$, we have that $\wt{\theta}_{\ell_{\vec{t}}}$ is mapped to $1 \otimes \vartheta_\ell$ under the identification $\on{GP}^m(f)_{\vec{t}}^\vee = R_m/\mathfrak{p}_{\vec{t}} \otimes_{R_m} \on{GP}^m(f)^\vee$.
\end{itemize}
\end{lemma}
\begin{proof}
To begin with, we take $\vartheta_\ell = \wt{\theta}_\ell$ for each $\ell \in \{1, 2\}$. Next, for each $\ell \in \{3, \dots, m\}$, we define the functional $\vartheta_\ell$ on the generators $a^ib^j$ as follows:
$$\vartheta_\ell(a^ib^j) = \begin{cases} 0 & \text{if} \quad i = j = 0 \\ \left(\prod_{n = 1+i+j}^{m-1} t_n\right) \cdot \wt{\theta}_\ell(a^ib^j) & \text{otherwise} \end{cases}$$
For every pair $(i,j)$ with $0 \leq i + j \leq m-2$, upon substituting the formula for $r_{ij}'$ given in Definition~\ref{def-grandparts} into the above definition of $\vartheta_\ell$, we find that $$\vartheta_\ell(r_{ij}') = \left(\prod_{n = 2+i+j}^{m-1} t_n\right) \cdot \wt{\theta}_\ell(r_{ij}) = 0,$$ so $\vartheta_\ell$ is a well-defined element of $\on{GP}^\ell(f)^\vee$. Moreover, upon substituting $\vec{t} \in (k^\times)^{m-2}$ for $(t_1, \dots, t_m)$ into the definition of $\vartheta_\ell$, we obtain the functional $\wt{\theta}_{\ell_{\vec{t}}}$.

Now, observe that a modification of the proof of Lemma~\ref{lem-thebottomline} implies that $\on{GK}^m(f)^\vee$ is a free $R_m$-module of rank $1$, and further observe that the image of $\vartheta_m$ in $\on{GK}^m(f)^\vee$ is a generator. It follows that the sequence in~\eqref{eq-grandexact} is a short exact sequence ending in a free module, and so it splits, implying that
$$\on{GP}^m(f)^\vee \simeq R_m \otimes_{R_{m-1}} \on{GP}^{m-1}(f)^\vee \oplus \on{GK}^m(f)^\vee.$$
That the list $(\vartheta_1, \dots, \vartheta_m)$ forms a basis of $\on{GP}^m(f)^\vee$ now follows by inducting on $m$.
\end{proof}

We are now in position to define the degeneracy scheme associated to the grand principal parts module. In the manner of \S~\ref{sec-backtobasics}, let $\ol{\tau} = (\tau_1, \dots, \tau_n)$ be a list of general elements of $\on{SP}^m(f)$, where $n = m \pm 1$, and for each $\ell \in \{1, \dots, n \}$, we write $\tau_\ell = \sum_{i,j \geq 0} c_{ij}^{(\ell)} \cdot u^iv^j$,
where $c_{ij}^{(\ell)} \in R$. We can view the elements $\tau_\ell$ as elements of $\on{GP}^m(f)$ via the inclusion $R_m \otimes_R \on{SP}^m(f) \hookrightarrow \on{GP}^m(f)$. We defined the module $\on{GP}^m(f)$ using the variables $a$ and $b$, with respect to which we have that
\begin{equation} \label{eq-taudefs}
\tau_{\ell} = \sum_{i,j \geq 0} \alpha_{ij}^{(\ell)} \cdot a^ib^j, \quad \text{where} \quad \alpha_{ij}^{(\ell)} = \sum_{i'+j' \geq i+j} {{i'} \choose {i}} {{j'} \choose {j}} \cdot c_{i'j'}^{(\ell)} \cdot x^{i'-i}y^{j'-j}.
\end{equation}
Observe that $\alpha_{ij}^{(\ell)} \in R^\times$ as long as $\tau_\ell \in \on{SP}^m(f)$ is chosen so that $c_{ij}^{(\ell)} \in R^\times$.

Let $\xi_{\ell}$ denote the image of $\tau_{\ell}$ under the canonical map $\on{can}_{\on{ev}} \colon \on{GP}^m(f) \to \on{GP}^m(f)^{\vee\vee}$, and let $(\vartheta_1^\vee, \dots, \vartheta_m^\vee)$ denote the dual basis to the basis $(\vartheta_1, \dots, \vartheta_m)$ that we constructed for $\on{GP}^m(f)^\vee$ in Lemma~\ref{lem-grandfree}. With respect to this dual basis, we have that
$$\xi_{\ell} = \sum_{i = 1}^m \vartheta_i(\tau_{\ell}) \cdot \vartheta_i^\vee.$$
Now, let $Z_{\ol{\tau}}^m \subset \BA_R^{m-2} \coloneqq \Spec R_m$ be the degeneracy locus of the elements $\xi_{1}, \dots, \xi_{n}$. Then $Z_{\ol{\tau}}^m$ is cut out by the ideal $I_{\ol{\tau}}^m \subset R_m$ generated by the maximal minors of the matrix
$$M_{\ol{\tau}}^m \coloneqq \left[\begin{array}{ccc} \vartheta_1(\tau_{1}) & \cdots & \vartheta_1(\tau_{n}) \\ \vdots & \ddots & \vdots \\ \vartheta_m(\tau_{1}) & \cdots & \vartheta_m(\tau_{n})  \end{array} \right]$$

Now, the map of rings $k[t_2, \dots, t_{m-1}] \to R[t_2, \dots, t_{m-1}]/I_{\ol{\tau}}^m$ induces a map of schemes $Z_{\ol{\tau}}^m \to \BA_k^{m-2}$, so we may view $Z_{\ol{\tau}}^m$ as a family of automatic degeneracy schemes over $\BA_k^{m-2}$. For a general choice of the elements $\tau_\ell \in \on{SP}^m(f)$, the fiber $(Z_{\ol{\tau}}^m)_{\vec{1}}$ of the family over the closed point $\vec{1} = (1, 1, \dots, 1) \in \BA_k^{m-2}$ is a $0$-dimensional scheme with length equal to $\on{SD}_{(2)}^m(f)$ or $\on{SD}_{(1,1)}^m(f)$, according as $n = m-1$ or $n = m+1$. The following lemma says that we can make a general choice of the elements $\tau_\ell$ so that this property holds not only at $\vec{1}$, but also at the geometric generic fiber of the family:
\begin{lemma} \label{lem-genextend}
Let $\eta$ denote the geometric generic point of $\BA_k^{m-2}$. For a general choice of the elements $\tau_\ell \in \on{SP}^m(f)$, the fiber $(Z_{\ol{\tau}}^m)_\eta$ of the family $Z_{\ol{\tau}}^m/\BA_k^{m-2}$ is a $0$-dimensional scheme with length equal to $\on{SD}_{(2)}^m(f)$ or $\on{SD}_{(1,1)}^m(f)$, according as $n = m-1$ or $n = m+1$.
\end{lemma}
\begin{proof}
By upper-semicontinuity, the length of $(Z_{\ol{\tau}}^m)_{\eta}$ is at most  the length of $(Z_{\ol{\tau}}^m)_{\vec{1}}$. If the length of $(Z_{\ol{\tau}}^m)_{\eta}$ is strictly smaller than the length of $(Z_{\ol{\tau}}^m)_{\vec{1}}$, then there exists a closed point $\vec{t} = (T_2, \dots, T_{m-1}) \in \BA_k^{m-2}$ such that the length of $(Z_{\ol{\tau}}^m)_{\vec{t}}$ is strictly smaller than the length of $(Z_{\ol{\tau}}^m)_{\vec{1}}$. But $(Z_{\ol{\tau}}^m)_{\vec{t}}$ is the degeneracy locus in $\on{SP}^m(f)^{\vee\vee}$ of the elements
$\Phi_{\vec{t}}(\tau_{\ell}) \in \on{SP}^m(f)$ for $\ell \in \{1, \dots, n\}$. Thus, the elements $\Phi_{\vec{t}}(\tau_{\ell})$ achieve a smaller degeneracy than the elements $\tau_\ell$, contradicting the assumption that the elements $\tau_\ell$ are general.
\end{proof}

In the following theorem, we use the constructions introduced above to prove the desired upper bounds.

\begin{theorem} \label{thm-theupperbound}
  We have the bounds
  \begin{equation*}
  \on{AD}_{(2)}^m(f) \leq \mu_f \cdot \frac{3m-1}{m+1} \cdot {{m+1} \choose 4} \quad \text{and} \quad \on{AD}_{(1,1)}^m(f) \leq \mu_f \cdot \frac{3m-2}{m+2} \cdot {{m+2} \choose 4}
  \end{equation*}
\end{theorem}
\begin{proof}
  It suffices to prove the bounds in the theorem statement where we replace ``$\on{AD}$'' with ``$\on{SD}$'' because $\mu_{r \cdot f} = \mu_f$ for any planar ICIS germ $f$ and unit $r \in R^\times$. By upper-semicontinuity, if $(Z_{\ol{\tau}}^m)_{\vec{0}}$ is indeed $0$-dimensional, the length of $(Z_{\ol{\tau}}^m)_{\vec{0}}$ is greater than or equal to the length of $(Z_{\ol{\tau}}^m)_\eta$. Now, choose the elements $\tau_\ell \in \on{SP}^m(f)$ to be general so that the length of $(Z_{\ol{\tau}}^m)_\eta$ is equal to $\on{SD}_{(2)}^m(f)$ or $\on{SD}_{(1,1)}^m(f)$, according as $n = m-1$ or $n = m+1$; note that such a choice is possible by Lemma~\ref{lem-genextend}. Applying upper-semicontinuity again, it follows that the length of $(Z_{\ol{\tau}}^m)_{\vec{0}}$ with respect to \emph{any} choice of the elements $\tau_\ell \in \on{SP}^m(f)$ is greater than or equal to $\on{SD}^m_{(2)}(f)$ or $\on{SD}_{(1,1)}^m(f)$, according as $n = m-1$ or $n = m+1$. Thus, it remains to show that for \emph{some} choice of the elements $\tau_\ell \in \on{SP}^m(f)$, the fiber $(Z_{\ol{\tau}}^m)_{\vec{0}}$ has length at most $\mu_f \cdot \frac{3m-1}{m+1} \cdot {{m+1} \choose 4}$ or $\mu_f \cdot \frac{3m-2}{m+2} \cdot {{m+2} \choose 4}$, according as $n = m-1$ or $n = m+1$.

  Recall that $(Z_{\ol{\tau}}^m)_{\vec{0}}$ is the scheme cut out in the ring $R$ by the ideal generated by the maximal minors of the matrix
  \begin{equation} \label{eq-tet0mats}
  \left[\begin{array}{ccc} \wt{\theta}_{1_{\vec{0}}}(\tau_{1}) & \cdots & \wt{\theta}_{1_{\vec{0}}}(\tau_{n}) \\ \vdots & \ddots & \vdots \\ \wt{\theta}_{m_{\vec{0}}}(\tau_{1}) & \cdots & \wt{\theta}_{m_{\vec{0}}}(\tau_{n})  \end{array} \right]
  \end{equation}
  where by $\wt{\theta}_{\ell_{\vec{0}}}$ we mean the image of $1 \otimes \vartheta_\ell$ under the identification $R_m / \mathfrak{p}_{\vec{0}} \otimes_{R_m} \on{GP}^m(f)^\vee = \on{GP}^m(f)_{\vec{0}}^\vee$. From the explicit description of the functionals $\vartheta_\ell$ provided in the proof of Lemma~\ref{lem-grandfree}, we deduce that $\wt{\theta}_{\ell_{\vec{0}}}$ has the following explicit description:
  \begin{equation} \label{eq-tet0forms}
  \wt{\theta}_{\ell_{\vec{0}}}(a^ib^j) = \begin{cases} (-1)^j \left(\frac{\d f}{\d y}\right)^i \left(\frac{\d f}{\d x}\right)^j & \text{if} \quad i+j = \ell-1\\ 0 & \text{otherwise} \end{cases}
  \end{equation}
    Substituting~\eqref{eq-tet0forms} into each of the matrix entries in~\eqref{eq-tet0mats} yields a matrix, the row-$i$, column-$j$ entry of which is given by
    \begin{equation} \label{eq-entrydenied}
    \sum_{\ell = 0}^{i-1} \alpha_{(i-1-\ell)\ell}^{(j)} \cdot (-1)^\ell \left(\frac{\d f}{\d y}\right)^{i-1-\ell} \left(\frac{\d f}{\d x}\right)^\ell
    \end{equation}
    We now make a specific choice of the elements $\tau_\ell \in \on{SP}^m(f)$: we take them to be general among all choices of $\ol{\tau}$ satisfying the property that the coefficients $\alpha_{ij}^{(\ell)}$ are elements of $k \subset R$.

    In the following lemma, we compute the length of $(Z_{\ol{\tau}}^m)_{\vec{0}}$ for the node $f = xy$.
\begin{lemma} \label{lem-firsttime}
     For the choice of the elements $\tau_\ell \in \on{SP}^m(xy)$ specified above, the length of $(Z_{\ol{\tau}}^m)_{\vec{0}}$ is equal to $\frac{3m-1}{m+1} \cdot {{m+1} \choose 4}$ or $\frac{3m-2}{m+2} \cdot {{m+2} \choose 4}$, according as $n = m-1$ or $n = m+1$.
  \end{lemma}
  \begin{proof}
    For the above choice of the elements $\tau_\ell \in \on{SP}^m(f)$, the expression in~\eqref{eq-entrydenied} is a homogeneous polynomial of degree $i-1$ the symbols $\frac{\d f}{\d y}$ and $\frac{\d f}{\d x}$ for each pair $(i,j)$ with coefficients in $k^\times$. Thus, when $f = xy$, the row-$i$, column-$j$ entry of the matrix in~\eqref{eq-tet0mats} is what is referred to in the literature as ``semi-weighted homogeneous,'' meaning that the root expansion of the entry is a homogeneous polynomial in $x$ and $y$ with coefficients in $k$. In~\cite{damon1}, Damon obtains a formula for the colength (which he terms the ``Macaulay-Bezout number'') of the ideal of maximal minors of a matrix with semi-weighted homogeneous entries. This formula was reproven in~\cite{damonredone}, where it is expressed in the following easy-to-use form:
    \begin{lemma}[\protect{\cite[Lemma 5.6]{damonredone}}] \label{thm-damon}
      Let $m' \geq 1$, and let $M$ be an $(m'+1) \times m'$ matrix with entries in $R$ such that for each pair $(i,j)$, the row-$i$, column-$j$ entry of $M$ is semi-weighted homogeneous with root expansion having degree $d_{ij} \geq 0$. Then the colength of the ideal of maximal $m' \times m'$ minors of the matrix is, if finite, given by
      $$\sum_{1 \leq i_1 < i_2 \leq m'+1} d_{i_1i_1}d_{i_2(i_2-1)}.$$
    \end{lemma}
    Substituting in $m' = m-1$ and $d_{ij} = i-1$ or $m' = m$ and $d_{ij} = j-1$ according as $n = m-1$ or $n = m+1$ into the formula in Lemma~\ref{thm-damon} yields the claimed formulas for the length of $(Z_{\ol{\tau}}^m)_{\vec{0}}$. (Note that for the case $n = m+1$, we replace the matrix in~\eqref{eq-tet0mats} with its transpose to apply Lemma~\ref{thm-damon}.) \end{proof}

     The next lemma helps us relate the lengths of $(Z_{\ol{\tau}}^m)_{\vec{0}}$ for the node and $(Z_{\ol{\tau}}^m)_{\vec{0}}$ for arbitrary planar ICIS germs $f$.

  \begin{lemma} \label{lem-worsethannodes}
  Let $n \geq 1$, and let $I \subset k[[x',y']]$ be an ideal with $\dim_k k[[x',y']]/I < \infty$. Let $f \in R$ be the germ of a planar ICIS, and let $\phi \colon k[[x',y']] \to R$ be the map sending $x' \mapsto \frac{\d f}{\d y}$ and $y' \mapsto \frac{\d f}{\d x}$. Then we have the inequality
  $$\dim_k (R/\phi(I)R) \leq \mu_f \cdot \dim_k (k[[x',y']]/I).$$
  \end{lemma}
  \begin{proof}
    Notice that we have the isomorphism of $k[[x',y']]$-modules $$R/\phi(I)R \simeq R \otimes_{k[[x',y']]} (k[[x',y'']]/I),$$ where we view $R$ as a $k[[x',y']]$-algebra via the map $\phi$. Suppose there is a surjective map of $k[[x',y']]$-modules $\psi \colon (k[[x',y']])^{\mu_f} \twoheadrightarrow R$. Then by right-exactness of the tensor product, we can tensor $\psi$ with $k[[x',y']]/I$ to obtain a surjective map of $(k[[x',y']]/I)$-modules $(k[[x',y']]/I)^{\mu_f}   \twoheadrightarrow R/\phi(I)R$, from which the desired inequality of $k$-vector space dimensions is obvious. By Nakayama's Lemma, a minimal generating set for $R/\phi(I)R$ as a $(k[[x',y']]/I)$-module is given by taking lifts of a $k$-vector space basis of $$(k[[x',y']]/(x',y')) \otimes_{k[[x',y']]} (R/\phi(I)R) \simeq R/J(f),$$ but $\dim_k R/J(f) = \mu_f$, so it follows that the desired surjection $\psi$ must exist.
  \end{proof}
  The theorem now follows by combining Lemmas~\ref{lem-firsttime} and~\ref{lem-worsethannodes}. Indeed, it is clear from~\eqref{eq-entrydenied} and our choice of the elements $\tau_\ell \in \on{SP}^m(f)$ that the ideal defining $(Z_{\ol{\tau}}^m)_{\vec{0}}$ for $f$ is generated by elements of $k\big[\big[\frac{\d f}{\d y}, \frac{\d f}{\d x}\big]\big]$, so using notation from the statement of Lemma~\ref{lem-worsethannodes}, we deduce that the ideal of $R$ defining $(Z_{\ol{\tau}}^m)_{\vec{0}}$ for $f$ is the image under the map $\phi$ of the ideal of $k[[x',y']]$ defining $(Z_{\ol{\tau}}^m)_{\vec{0}}$ for the node. Thus, by Lemma~\ref{lem-worsethannodes}, we have that the length of $(Z_{\ol{\tau}}^m)_{\vec{0}}$ for $f$ is less than or equal to $\mu_f$ times the length of $(Z_{\ol{\tau}}^m)_{\vec{0}}$ for the node, which we computed in Lemma~\ref{lem-firsttime}.

  This concludes the proof of Theorem~\ref{thm-theupperbound}.
\end{proof}

Recall that in the case of the node, we found that the weight-$1$ automatic degeneracy is a polynomial in $m$ of degree $2$ and that the weight-$2$ automatic degeneracies are polynomials in $m$ of degree $4$ (see Theorem~\ref{thm-main2}). It is natural to wonder whether a similar such result holds for an arbitrary ICIS. From the lower bound in Theorem~\ref{cor-worse} and the upper bound in Theorem~\ref{thm-theupperbound}, we deduce the following corollary on the asymptotic growth of weight-$2$ automatic degeneracies:

\begin{corollary} \label{cor-growth}
For any ICIS cut out analytically-locally by $f = 0$, we have that
$$\on{AD}_{(2)}^m(f) = \Omega(m^4) \quad \text{and} \quad \on{AD}_{(1,1)}^m(f) = \Omega(m^4).$$
Moreover, for a planar ICIS cut out analytically-locally by $f = 0$, we have that
$$\on{AD}_{(2)}^m(f) = \Theta(m^4) \quad \text{and} \quad \on{AD}_{(1,1)}^m(f) = \Theta(m^4).$$
\end{corollary}

The results in Corollary~\ref{cor-growth} lead us to pose the following natural question:

\begin{question} \label{question-polynomiality}
  Given an ICIS, are any of the automatic degeneracies identically equal to a polynomial in $m$ or equal to a polynomial in $m$ for all sufficiently large $m$? If so, what is the degree of the polynomial?
\end{question}
Although answering Question~\ref{question-polynomiality} in its entirety remains open, note that we can answer it in the following cases:
\begin{enumerate}
    \item For each $i \in \{1, 2\}$, the weight-$i$ automatic degeneracies of a node are identically given by polynomials in $m$ of degree $2i$ (by Theorem~\ref{thm-main2}).
    \item The weight-$2$ type-(a) automatic degeneracy of a cusp is identically given by a polynomial in $m$ of degree $4$ (by Corollary~\ref{cor-cusp}).
    \item For planar singularities, if either of the weight-$2$ automatic degeneracies is equal to a polynomial in $m$ for all sufficiently large $m$, then that polynomial is of degree $4$ (by Corollary~\ref{cor-growth}).
    \end{enumerate}

\section{Examples of Computing Automatic Degeneracies} \label{sec-autodegegs}

In this section, we compute low-order automatic degeneracies for various singularities.

\subsection{Conditions for Automatic Degeneracies to be Zero}

The following theorem tells us exactly when the $m^{\mathrm{th}}$-order automatic degeneracies of an ICIS are equal to $0$.

\begin{theorem} \label{thm-autodeg0}
  Consider an ICIS cut out analytically-locally by $f = 0$. Then $\on{AD}_{(1)}^m(f) = 0$ if and only if $m = 1$, $\on{AD}_{(2)}^m(f) = 0$ if and only if $m = 2$, and $\on{AD}_{(1,1)}^m(f) = 0$ if and only if $m = 1$.
\end{theorem}
\begin{proof}
Combining the lower bounds obtained in Theorem~\ref{cor-worse} with the fact that $\delta_f \geq 1$ and $\on{mult}_0 \Delta_f \geq 1$ implies that the ``only if'' parts of the claims are true.

As for the ``if'' parts of the claims, recall that $\on{P}^1(f) \simeq R$, so $\on{P}^1(f)^\vee$ is generated by the functional $\wt{\theta}_1$ that sends a fixed generator $\tau \in \on{P}^1(f)$ to $1$. Take elements $\alpha_1, \alpha_2 \in R$, and consider the associated degeneracy matrices with respect to the basis $(\wt{\theta}_1)$ of $\on{P}^1(f)^\vee$:
\begin{equation} \label{eq-smallmats}
\left[\begin{array}{c} \alpha_1 \end{array}\right] \quad \text{and} \quad \left[\begin{array}{cc} \alpha_1 & \alpha_2 \end{array}\right]
\end{equation}
The ideal of maximal minors of each of the matrices in~\eqref{eq-smallmats} is the unit ideal of $R$ and thus has colength $0$ as long as $\alpha_1 \in R^\times$, and this happens for a general choice of the $\alpha_i$. It follows that $\on{AD}_{(1)}^1(f) = \on{AD}_{(1,1)}^1(f) = 0$. Now, the exact sequence in Lemma~\ref{prop-principalpartsseqexact} splits on the right-hand side whem $m = 2$, so we have that
$$\on{P}^2(f) \simeq \on{P}^1(f) \oplus \ker(\on{P}^2(f) \to \on{P}^1(f)),$$
implying that there is a functional $\wt{\theta}_2 \in \on{P}^2(f)^\vee$ such that $(\wt{\theta}_1, \wt{\theta}_2)$ forms a basis of $\on{P}^2(f)^\vee$. With respect to this basis, the degeneracy matrix associated to an element $$(\alpha_1 \cdot \tau, \beta) \in \on{P}^1(f) \oplus \ker(\on{P}^2(f) \to \on{P}^1(f)) \simeq \on{P}^2(f)$$
is given by
\begin{equation} \label{eq-smallmats2}
\left[\begin{array}{c} \alpha_1 \\ \wt{\theta}_2(\beta) \end{array}\right]
\end{equation}
The ideal of maximal minors of the matrix in~\eqref{eq-smallmats2} is the unit ideal of $R$, and thus has colength $0$, as long as $\alpha_1\in R^\times$, and this happens for a general choice of $(\alpha_1 \cdot \tau, \beta)$. It follows that $\on{AD}_{(2)}^2(f) = 0$.
\end{proof}
\begin{remark}
   Recall that there is no such thing as a $1^{\mathrm{st}}$-order weight-$1$, $2^{\mathrm{nd}}$-order weight-$2$ type-(a), or $1^{\mathrm{st}}$-order weight-$2$ type-(b) inflection point, so it is not useful to interpret the result of Theorem~\ref{thm-autodeg0} in terms inflection points limiting toward an ICIS.
\end{remark}

\subsection{An Explicit Basis of $\on{SP}^m(f)^\vee$ for $1 \leq m \leq 4$} \label{sec-expibasis}

Take a planar ICIS cut out analytically-locally by $f = 0$. By applying the algorithm in \S~\ref{sec-algae}, we find that the maps $\wt{\theta}_i$ for $i \in \{1, 2, 3, 4\}$ defined as follows have the property that for each $m \in \{1, 2, 3, 4\}$, the list $(\wt{\theta}_1, \dots, \wt{\theta}_m)$ forms a basis of $\on{SP}^m(f)^\vee$: given $\tau = \sum_{i,j \geq 0} \alpha_{ij} \cdot a^ib^j \in \on{SP}^m(f)$, we obtain
  \begin{align*}
    \wt{\theta}_1(\tau) & = \alpha_{00}\\
    \wt{\theta}_2(\tau) & = \alpha_{10} \cdot \frac{\d f}{\d y} + \alpha_{01} \cdot \left(-\frac{\d f}{\d x}\right)\\
    \wt{\theta}_3(\tau) & = \alpha_{10} \cdot \left(\frac{\d^2 f}{\d x \d y} \frac{\d f}{\d y}  - \frac{1}{2}\frac{\d^2 f}{\d y^2} \frac{\d f}{\d x} \right) + \alpha_{01} \cdot \left( - \frac{1}{2} \frac{\d^2 f}{\d x^2}\frac{\d f}{\d y}\right) + \\
    & \hphantom{===} \alpha_{20} \cdot \left(\frac{\d f}{\d y}\right)^2 - \alpha_{11} \cdot \frac{\d f}{\d y}\frac{\d f}{\d x} +
 \alpha_{02} \cdot \left(\frac{\d f}{\d x}\right)^2\\
    \wt{\theta}_4(\tau) & = \alpha_{10} \cdot \left(\left(\frac{\d f}{\d x \d y}\right)^2 \frac{\d f}{\d y} + \frac{1}{2}\left(\frac{\d^3 f}{\d x^2 \d y}\left(\frac{\d f}{\d y}\right)^2-\frac{\d^3 f}{\d x \d y^2} \frac{\d f}{\d x}\frac{\d f}{\d y}-\frac{\d^2 f}{\d x \d y}\frac{\d^2 f}{\d y^2} \frac{\d f}{\d x}+\frac{1}{3}\frac{\d^3 f}{\d y^3}\right)\right) + \\
    & \hphantom{=====} \alpha_{01} \cdot \left(-\frac{1}{2}\left(\frac{\d^2 f}{\d x^2}\frac{\d^2 f}{\d x \d y}\frac{\d f}{\d y}+\frac{1}{3}\frac{\d^3 f}{\d x^3} \left(\frac{\d f}{\d y}\right)^2\right)\right) + \\
    & \hphantom{===} \alpha_{20} \cdot \left(2\frac{\d^2 f}{\d x \d y} \left(\frac{\d f}{\d y}\right)^2 - \frac{\d^2 f}{\d y^2}\frac{\d f}{\d x}\frac{\d f}{\d y}\right) + \\
    & \hphantom{=====} \alpha_{11} \cdot \left(-\frac{\d^2 f}{\d x \d y}\frac{\d f}{\d x}\frac{\d f}{\d y} + \frac{1}{2}\left(\frac{\d^2 f}{\d y^2} \left(\frac{\d f}{\d x}\right)^2 - \frac{\d^2 f}{\d x^2}\left(\frac{\d f}{\d y}\right)^2\right)\right) + \alpha_{02} \cdot \left(\frac{\d^2 f}{\d x^2} \frac{\d f}{\d x}\frac{\d f}{\d y}\right) + \\
  & \hphantom{===} \alpha_{30} \cdot \left(\frac{\d f}{\d y}\right)^3  -  \alpha_{21} \cdot \left(\frac{\d f}{\d y}\right)^2 \left(\frac{\d f}{\d x}\right)  + \alpha_{12} \cdot \left(\frac{\d f}{\d y}\right)\left(\frac{\d f}{\d x}\right)^2 - \alpha_{03} \cdot \left(\frac{\d f}{\d x}\right)^3
  \end{align*}
  \begin{remark}
  Note that the basis elements listed above are complicated in the sense that $\wt{\theta}_m$ is defined in terms of polynomial expressions---which grow longer as $m$ increases---in all of the partial derivatives of $f$ up to and including order $m-1$. Consequently, as $m$ increases, it becomes recursively more difficult to apply the algorithm in \S~\ref{sec-algae} to find a basis for $\on{SP}^m(f)^\vee$. Thus, in the remainder of this section, we restrict our attention to computing automatic degeneracies corresponding to $m$-values with $m \in \{1, 2, 3, 4\}$.
  \end{remark}

\subsection{Computation of $\on{AD}_{(1,1)}^2(f)$}

The following theorem tells us that $2^{\mathrm{nd}}$-order type-(b) automatic degeneracy equals the Milnor number in the planar case.

\begin{theorem} \label{thm-autodeg112ismiln}
The $2^{\mathrm{nd}}$-order weight-$2$ type-(b) automatic degeneracy of a planar ICIS cut out analytically-locally by $f = 0$ is given by $$\on{AD}_{(1,1)}^2(f) = \mu_f.$$ In particular, the number of $2^{\mathrm{nd}}$-order weight-$2$ type-(b) inflection points limiting to any planar ICIS in a general $2$-parameter deformation is equal to $0$.
\end{theorem}
\begin{proof}
There are two strategies to prove the theorem: (1) we can directly compute the automatic degeneracy, or (2) we can attempt to use the lower and upper bounds obtained in Theorems~\ref{cor-worse} and~\ref{thm-theupperbound}. We provide both proofs for the sake of completeness.

  \textbf{Strategy I}: It suffices to show that $\on{SD}_{(1,1)}^2(r \cdot f) = \mu_f$ for any $r \in R^\times$. Take elements $\tau_1, \tau_2, \tau_3 \in \on{SP}^2(r \cdot f)$ with $\tau_\ell = \sum_{0 \leq i+j \leq 1} \alpha_{ij}^{(\ell)} \cdot a^ib^j$, and write $\ol{\tau} = (\tau_1, \tau_2, \tau_3)$ as usual. The associated degeneracy matrix $M_{\ol{\tau}}^m$ has the following form with respect to the basis $(\wt{\theta}_1, \wt{\theta}_2)$ computed in \S~\ref{sec-expibasis}:
  \begin{align*} 
  & M_{\ol{\tau}}^m = \\
  & \left[ \begin{array}{ccc} \alpha_{00}^{(1)} & \alpha_{00}^{(2)} & \alpha_{00}^{(3)} \\ \alpha_{10}^{(1)} \cdot \frac{\d(r \cdot f)}{\d y} + \alpha_{01}^{(1)} \cdot \left(-\frac{\d(r \cdot f)}{\d x}\right) & \alpha_{10}^{(2)} \cdot \frac{\d(r \cdot f)}{\d y} + \alpha_{01}^{(2)} \cdot \left(-\frac{\d(r \cdot f)}{\d x}\right) & \alpha_{10}^{(3)} \cdot \frac{\d(r \cdot f)}{\d y} + \alpha_{01}^{(3)} \cdot \left(-\frac{\d(r \cdot f)}{\d x}\right) \end{array} \right] \nonumber
  \end{align*}
  The ideal $I_{\ol{\tau}}^m$ of maximal minors of the matrix $M_{\ol{\tau}}^m$ is given by
  \begin{align*}
  I_{\ol{\tau}}^m & = \left((\alpha_{00}^{(1)}\alpha_{10}^{(2)} - \alpha_{00}^{(2)}\alpha_{10}^{(1)}) \cdot \frac{\d(r \cdot f)}{\d y} + (\alpha_{00}^{(1)}\alpha_{01}^{(2)} -\alpha_{00}^{(2)}\alpha_{01}^{(1)}) \cdot \frac{\d(r \cdot f)}{\d x}, \right. \label{eq-idears112} \\
  & \hphantom{===} \left. (\alpha_{00}^{(1)}\alpha_{10}^{(3)} - \alpha_{00}^{(3)}\alpha_{10}^{(1)}) \cdot \frac{\d(r \cdot f)}{\d y} + (\alpha_{00}^{(1)}\alpha_{01}^{(3)} -\alpha_{00}^{(3)}\alpha_{01}^{(1)}) \cdot \frac{\d(r \cdot f)}{\d x}, \right. \nonumber \\
  & \hphantom{===} \left.  (\alpha_{00}^{(2)}\alpha_{10}^{(3)} - \alpha_{00}^{(3)}\alpha_{10}^{(2)}) \cdot \frac{\d(r \cdot f)}{\d y} + (\alpha_{00}^{(2)}\alpha_{01}^{(3)} -\alpha_{00}^{(3)}\alpha_{01}^{(2)}) \cdot \frac{\d(r \cdot f)}{\d x} \right) \nonumber
  \end{align*}
  For a general choice of the elements $\tau_1, \tau_2, \tau_3 \in \on{SP}^2(f)$, the matrix
  \begin{equation} \label{eq-furtherexplain1}
  \left[\begin{array}{cc} \alpha_{00}^{(1)}\alpha_{10}^{(2)} - \alpha_{00}^{(2)}\alpha_{10}^{(1)} & \alpha_{00}^{(1)}\alpha_{01}^{(2)} -\alpha_{00}^{(2)}\alpha_{01}^{(1)} \\ \alpha_{00}^{(1)}\alpha_{10}^{(3)} - \alpha_{00}^{(3)}\alpha_{10}^{(1)} & \alpha_{00}^{(1)}\alpha_{01}^{(3)} -\alpha_{00}^{(3)}\alpha_{01}^{(1)} \\ \alpha_{00}^{(2)}\alpha_{10}^{(3)} - \alpha_{00}^{(3)}\alpha_{10}^{(2)} & \alpha_{00}^{(2)}\alpha_{01}^{(3)} -\alpha_{00}^{(3)}\alpha_{01}^{(2)} \end{array}\right]
  \end{equation}
  has the maximum possible rank of $2$. Indeed, a general choice of $\tau_1, \tau_2, \tau_3 \in \on{SP}^2(f)$ amounts to a general choice of the coefficients $\alpha_{ij}^{(\ell)} \in R$, and one can readily find specific choices of $\alpha_{ij}^{(\ell)} \in \BZ$ such that the matrix in~\eqref{eq-furtherexplain1} has rank $2$. It follows that $I_{\ol{\tau}}^m = \left(\frac{\d(r \cdot f)}{\d x}, \frac{\d(r \cdot f)}{\d y}\right)$. It follows that we have
  \begin{equation*}
  \on{SD}_{(1,1)}^2(f) = \dim_k R\bigg/\left(\frac{\d(r \cdot f)}{\d x}, \frac{\d(r \cdot f)}{\d y}\right) = \mu_{r \cdot f} = \mu_f.
  \end{equation*}

  \textbf{Strategy II}: Simply notice that the lower and upper bounds for $\on{AD}_{(1,1)}^2(f)$ obtained in Theorems~\ref{cor-worse} and~\ref{thm-theupperbound} are both equal to $\mu_f$ when $m = 2$. The statement about the number of limiting inflection points follows from the fact that the lower bound in Theorem~\ref{cor-worse} is equal to $\on{AD}_{(1,1)}^2(f)$.
\end{proof}
\begin{example} \label{eg-cuspy}
  We can interpret $2^{\mathrm{nd}}$-order weight-$2$ type-(b) inflection points geometrically as follows: given a curve $C$ over $k$ with a map $\phi \colon C \to \BP_k^2$, a smooth point $p \in C(k)$ is an inflection point of this sort if the image of $C$ in $\BP_k^2$ has a cusp at $\phi(p)$. Theorem~\ref{thm-autodeg112ismiln} tells us that the number of such inflection points limiting to a planar ICIS in a general $2$-parameter deformation is $0$, a result that may be counterintuitive. 
\end{example}

\subsection{Computation of $\on{AD}_{(2)}^3(y^t - x^s)$} \label{sec-23ab}

Let $f$ be a planar ICIS germ, take elements $\tau_1, \tau_2 \in \on{SP}^3(f)$ with expansions $\tau_\ell = \sum_{0 \leq i+j \leq 2} \alpha_{ij}^{(\ell)} \cdot a^ib^j$, and write $\ol{\tau} = (\tau_1, \tau_2)$ as usual. The associated degeneracy matrix $M_{\ol{\tau}}^3$ has the following form with respect to the basis $(\wt{\theta}_1, \wt{\theta}_2, \wt{\theta}_3)$ computed in \S~\ref{sec-expibasis}:
\begin{align*} 
    & M_{\ol{\tau}}^3 =  \\
    & \left[\begin{array}{cc} \scriptstyle\alpha_{00}^{(1)} & \scriptstyle\alpha_{00}^{(2)} \\  \scriptstyle\alpha_{10}^{(1)} \cdot \frac{\d f}{\d y} + \alpha_{01}^{(1)} \cdot \left(-\frac{\d f}{\d x}\right) & \scriptstyle\alpha_{10}^{(2)} \cdot \frac{\d f}{\d y} + \alpha_{01}^{(2)} \cdot \left(-\frac{\d f}{\d x}\right) \\ \scriptstyle\left(\alpha_{10}^{(1)} \cdot \left(\frac{\d^2 f}{\d x \d y} \frac{\d f}{\d y}  - \frac{1}{2}\frac{\d^2 f}{\d y^2} \frac{\d f}{\d x} \right) + \alpha_{01}^{(1)} \cdot \left( - \frac{1}{2} \frac{\d^2 f}{\d x^2}\frac{\d f}{\d y}\right) + \right.   &
    \scriptstyle\left(\alpha_{10}^{(2)} \cdot \left(\frac{\d^2 f}{\d x \d y} \frac{\d f}{\d y}  - \frac{1}{2}\frac{\d^2 f}{\d y^2} \frac{\d f}{\d x} \right) + \alpha_{01}^{(2)} \cdot \left( - \frac{1}{2} \frac{\d^2 f}{\d x^2}\frac{\d f}{\d y}\right) + \right. \\ \scriptstyle\quad \left. \alpha_{20}^{(1)} \cdot \left(\frac{\d f}{\d y}\right)^2 - \alpha_{11}^{(1)} \cdot \frac{\d f}{\d y}\frac{\d f}{\d x} +
 \alpha_{02}^{(1)} \cdot \left(\frac{\d f}{\d x}\right)^2\right) &\scriptstyle  \quad \left. \alpha_{20}^{(2)} \cdot \left(\frac{\d f}{\d y}\right)^2 - \alpha_{11}^{(2)} \cdot \frac{\d f}{\d y}\frac{\d f}{\d x} +
 \alpha_{02}^{(2)} \cdot \left(\frac{\d f}{\d x}\right)^2 \right) \end{array} \right]
\end{align*}
Let $\Xi_i$ denote the minor of the matrix $M_{\ol{\tau}}^3$ obtained by removing the $(4-i)^{\mathrm{th}}$ row and taking the determinant. Then $\Xi_1$ is given by
\begin{equation*}
\Xi_1 = (\alpha_{00}^{(1)}\alpha_{10}^{(2)} - \alpha_{00}^{(2)}\alpha_{10}^{(1)}) \cdot \frac{\d f}{\d y} + (\alpha_{00}^{(1)}\alpha_{01}^{(2)} -\alpha_{00}^{(2)}\alpha_{01}^{(1)}) \cdot \frac{\d f}{\d x}
\end{equation*}
The degeneracy ideal is given by $I_{\ol{\tau}}^3 = (\Xi_1, \Xi_2, \Xi_3) \subset R$. As long as the condition
$$(\alpha_{00}^{(1)}\alpha_{10}^{(2)} - \alpha_{00}^{(2)}\alpha_{10}^{(1)}) \cdot (\alpha_{00}^{(1)}\alpha_{01}^{(2)} -\alpha_{00}^{(2)}\alpha_{01}^{(1)}) \in R^\times$$
is satisfied, which happens for a general choice of the elements $\tau_1, \tau_2 \in \on{SP}^3(f)$, we have that $\frac{\d f}{\d x} = \gamma_1 \cdot \frac{\d f}{\d y}$ as elements of $R/I_{\ol{\tau}}^3$ for some $\gamma_1 \in R^\times$; moreover, we can replace $\Xi_1$ with $\Xi_1' = \frac{\d f}{\d x} - \gamma_1 \cdot \frac{\d f}{\d y}$. It follows that $\Xi_i = \Xi_i'$ as elements of $R/I_{\ol{\tau}}^3$ for $i \in \{2, 3\}$, where the $\Xi_i'$ are defined by
\begin{align*}
\Xi_2' & = \left((\alpha_{00}^{(1)}\alpha_{10}^{(2)} - \alpha_{00}^{(2)}\alpha_{10}^{(1)}) \cdot \left(\frac{\d^2 f}{\d x \d y} - \frac{\gamma_1}{2}\frac{\d^2 f}{\d y^2} \right) + (\alpha_{00}^{(1)}\alpha_{01}^{(2)} -  \alpha_{00}^{(2)}\alpha_{01}^{(1)}) \cdot  - \frac{1}{2} \frac{\d^2 f}{\d x^2}  \right) \cdot \frac{\d f}{\d y} + \\
& \hphantom{===} \big((\alpha_{00}^{(1)}\alpha_{20}^{(2)} -  \alpha_{00}^{(2)}\alpha_{20}^{(1)}) - \gamma_1 \cdot (\alpha_{00}^{(1)}\alpha_{11}^{(2)} - \alpha_{00}^{(2)}\alpha_{11}^{(1)}) + \gamma_1^2 \cdot (\alpha_{00}^{(1)}\alpha_{02}^{(2)} - \alpha_{00}^{(2)}\alpha_{02}^{(1)}) \big) \cdot \left(\frac{\d f}{\d y}\right)^2 \\
\Xi_3' & = \left((\beta_{00}^{(1)}\alpha_{10}^{(2)} - \beta_{00}^{(2)}\alpha_{10}^{(1)}) \cdot \left(\frac{\d^2 f}{\d x \d y} - \frac{\gamma_1}{2}\frac{\d^2 f}{\d y^2} \right) + (\beta_{00}^{(1)}\alpha_{01}^{(2)} -  \beta_{00}^{(2)}\alpha_{01}^{(1)}) \cdot - \frac{1}{2} \frac{\d^2 f}{\d x^2}  \right) \cdot \left(\frac{\d f}{\d y}\right)^2 + \\
& \hphantom{===} \big((\beta_{00}^{(1)}\alpha_{20}^{(2)} -  \beta_{00}^{(2)}\alpha_{20}^{(1)}) - \gamma_1 \cdot (\beta_{00}^{(1)}\alpha_{11}^{(2)} - \beta_{00}^{(2)}\alpha_{11}^{(1)}) + \gamma_1^2 \cdot (\beta_{00}^{(1)}\alpha_{02}^{(2)} - \beta_{00}^{(2)}\alpha_{02}^{(1)}) \big) \cdot \left(\frac{\d f}{\d y}\right)^3
\end{align*}
and where for the sake of convenience we have put $\beta_{00}^{(\ell)} = \alpha_{10}^{(\ell)} - \gamma_1 \cdot \alpha_{01}^{(\ell)}$ for each $\ell \in \{1, 2\}$. We then have that
\begin{equation} \label{eq-reducedautodeg}
\on{SD}_{(2)}^3(f) = \dim_k R/\left(\Xi_1', \Xi_2', \Xi_3'\right)
\end{equation}
Although~\eqref{eq-reducedautodeg} may be easily computed for any given $f$, it is difficult to find a formula for~\eqref{eq-reducedautodeg} that holds for arbitrary $f$. Nevertheless, the following theorem demonstrates that a formula can be found for a large family of planar singularities, namely the hypercuspidal singularities:

\begin{theorem} \label{thm-23ab}
  Let $s \geq t \geq 2$ be integers. The $3^{\mathrm{rd}}$-order weight-$2$ type-(a) automatic degeneracy of a planar ICIS cut out analytically-locally by $y^t - x^s = 0$ is given by
  $$\on{AD}_{(2)}^3(y^t - x^s) = (2t-3)(s-1).$$
  In particular, the number of $3^{\mathrm{rd}}$-order weight-$2$ type-(a) inflection points limiting to this singularity in a general $2$-parameter deformation is equal to \mbox{$(t-2)(s-1)$.}
\end{theorem}
\begin{proof}
    The result is true when $s = t = 2$ by Theorem~\ref{thm-main2}, so take $s,t$ such that $(s,t) \neq (2,2)$. It suffices to show that
    $$\on{SD}_{(2)}^3(r \cdot (y^t - x^s)) = (2t-3)(s-1)$$
    for any unit $r \in R^\times$. Put $f = r \cdot (y^t - x^s)$; from the first generator $\Xi_1'$ of the ideal $I_{\ol{\tau}}^3$ in~\eqref{eq-reducedautodeg}, we deduce the following equality of elements of $R/I_{\ol{\tau}}^3$:
    \begin{equation} \label{eq-togetnextgamma}
    x^{s-1} \cdot \left(- r \cdot s + \left(-\frac{\d r}{\d x}+ \gamma_1 \cdot \frac{\d r}{\d y}\right) x  \right) = y^{t-1} \cdot \left(\gamma_1 \cdot r \cdot t  + \left(-\frac{\d r}{\d x} + \gamma_1 \cdot \frac{\d r}{\d y}\right)y \right)
    \end{equation}
    It follows from~\eqref{eq-togetnextgamma} that there exists $\gamma_1' \in R^\times$ such that $x^{s-1} = \gamma_1' \cdot y^{t-1}$; moreover, we can replace $\Xi_1'$ with $\Xi_1'' = x^{s-1} - \gamma_1' \cdot y^{t-1}$. Next, observe that as elements of $R/I_{\ol{\tau}}^3$, we have the following equalities:
    $$\frac{\d f}{\d y} = \nu_1 \cdot y^{t-1} , \quad \frac{\d^2 f}{\d x \d y} - \frac{\gamma_1}{2} \frac{\d^2 f}{\d y^2} = \nu_2 \cdot y^{t-2} , \quad -\frac{1}{2} \frac{\d^2 f}{\d x^2} = \nu_3 \cdot x^{s-2} ,$$
    where the $\nu_i$ are elements of $R^\times$ and are defined as follows:
    \begin{align*}
    \nu_1 & = r \cdot t  + \frac{\d r}{\d y} \cdot (y - \gamma_1' \cdot x),\\
    \nu_2 & = -\frac{\gamma_1}{2} \cdot r \cdot t(t-1) + \left(\frac{\d r}{\d x}  -\gamma_1 \cdot \frac{\d r}{\d y}\right) \cdot ty + \\
    & \hphantom{===}  \gamma_1' \cdot \left(-\frac{\d^2 r}{\d x \d y} + \frac{\gamma_1}{2} \cdot \frac{\d^2 r}{\d y^2}\right) \cdot xy + \left( \frac{\d^2 r}{\d x \d y}  -\frac{\gamma_1}{2} \cdot \frac{\d^2 r}{\d y^2} \right)\cdot y^2,\\
    \nu_3 & = \frac{1}{2} \cdot r \cdot s(s-1) + \frac{\d r}{\d x} \cdot bx - \frac{1}{2 \gamma_1'} \cdot \frac{\d^2 r}{\d x^2} \cdot xy + \frac{1}{2} \cdot \frac{\d^2 r}{\d x^2} \cdot x^2
    \end{align*}
With this notation, $\Xi_2'$ and $\Xi_3'$ are given as follows:
    \begin{align*}
\Xi_2' & = (\alpha_{00}^{(1)}\alpha_{01}^{(2)} -  \alpha_{00}^{(2)}\alpha_{01}^{(1)}) \cdot   \nu_{1}\nu_{3} \cdot x^{s-2}y^{t-1} + \left((\alpha_{00}^{(1)}\alpha_{10}^{(2)} - \alpha_{00}^{(2)}\alpha_{10}^{(1)}) \cdot  \nu_1\nu_2 + \right. \\
& \hphantom{==} \left. \big((\alpha_{00}^{(1)}\alpha_{20}^{(2)} -  \alpha_{00}^{(2)}\alpha_{20}^{(1)}) - \gamma_1 \cdot (\alpha_{00}^{(1)}\alpha_{11}^{(2)} - \alpha_{00}^{(2)}\alpha_{11}^{(1)}) + \gamma_1^2 \cdot (\alpha_{00}^{(1)}\alpha_{02}^{(2)} - \alpha_{00}^{(2)}\alpha_{02}^{(1)}) \big) \cdot \nu_1^2 \cdot y\right)y^{2t-3}, \\
\Xi_3' & = (\beta_{00}^{(1)}\alpha_{01}^{(2)} -  \beta_{00}^{(2)}\alpha_{01}^{(1)}) \cdot   \nu_1^2 \nu_3 \cdot  x^{s-2}y^{2t-2} + \left((\beta_{00}^{(1)}\alpha_{10}^{(2)} - \beta_{00}^{(2)}\alpha_{10}^{(1)}) \cdot  \nu_1^2\nu_2 + \right. \\
& \hphantom{==} \left. \big((\beta_{00}^{(1)}\alpha_{20}^{(2)} -  \beta_{00}^{(2)}\alpha_{20}^{(1)}) - \gamma_1 \cdot (\beta_{00}^{(1)}\alpha_{11}^{(2)} - \beta_{00}^{(2)}\alpha_{11}^{(1)}) + \gamma_1^2 \cdot (\beta_{00}^{(1)}\alpha_{02}^{(2)} - \beta_{00}^{(2)}\alpha_{02}^{(1)}) \big) \cdot \nu_1^3 \cdot y\right)y^{3t-4}.
\end{align*}
 For a general choice of the elements $\tau_1, \tau_2 \in \on{SP}^3(f)$, the coefficients of $x^{s-2}y^{t-1}$ and $y^{2t-3}$ in the above expansion of $\Xi_2'$ and the coefficients of $x^{s-2}y^{2t-2}$ and $y^{3t-4}$ in the above expansion of $\Xi_3'$ are all invertible elements of $R$ with distinct constant coefficients (this claim can be checked by showing that it holds for specific choices of $\alpha_{ij}^{(\ell)} \in \BZ$). Thus, there exist elements $\gamma_2, \gamma_3 \in R^\times$ such that $\gamma_2 - \gamma_3 \in R^\times$ and such that for some $\kappa_1, \kappa_2 \in R$ we have $$\kappa_1 \cdot \Xi_2' = x^{s-2}y^{t-1} - \gamma_2 \cdot y^{2t-3} \quad \text{and} \quad \kappa_2 \cdot \Xi_3' = x^{s-2}y^{2t-2} - \gamma_3 \cdot y^{3t-4},$$ where the equalities hold as elements of $R$. Thus, stipulating that $\Xi_2' = \Xi_3' = 0$ is equivalent to stipulating that the following equalities of elements of $R/I_{\ol{\tau}}^3$ hold:
  \begin{align*}
  & x^{s-2}y^{t-1} - \gamma_2 \cdot y^{2t-3} = 0, \\
  & x^{s-2}y^{2t-2} = \gamma_3 \cdot y^{3t-4} = \gamma_3 \cdot \frac{\kappa_1 \cdot y^{t-1} \cdot \Xi_2' - \kappa_3 \cdot \Xi_3'}{\gamma_3 - \gamma_2} = 0.
  \end{align*}
  We therefore obtain the following explicit form for the ideal $I_{\ol{\tau}}^3$:
  \begin{equation} \label{eq-simplycola}
  I_{\ol{\tau}}^3 = (x^{s-1} - \gamma_1' \cdot y^{t-1}, x^{s-2}y^{t-1} - \gamma_2 \cdot y^{2t-3}, x^{s-2}y^{2t-2}).
  \end{equation}
  To compute the colength $\dim_k R/I_{\ol{\tau}}^3$, it suffices to find a finite collection of monomials in $R$ the mod-$I_{\ol{\tau}}^3$ residues of which form a basis of $R/I_{\ol{\tau}}^3$.

  We handle the cases $s = t$ and $s > t$ separately. First suppose that $s = t$. In this case, the ideal $I_{\ol{\tau}}^3$ is homogeneous, so it is easy to use the formula
  \begin{equation} \label{eq-anandlem}
    \dim_k R/I = \sum_{\ell = 0}^\infty \dim_k ((R/I) \otimes_R \mathfrak{m}^\ell)/((R/I) \otimes_R \mathfrak{m}^{\ell+1}), \quad \text{where $\mathfrak{m} = (x,y) \subset R$}
    \end{equation}
  which works for any ideal $I \subset R$, to compute $\dim_k R/I_{\ol{\tau}}^3$. Since none of the three generators of $I_{\ol{\tau}}^3$ in~\eqref{eq-simplycola} have degree less than $t-1$, we deduce that
  \begin{equation} \label{eq-help1}
  \dim_k ((R/I_{\ol{\tau}}^3) \otimes_R \mathfrak{m}^\ell)/((R/I_{\ol{\tau}}^3) \otimes_R \mathfrak{m}^{\ell+1}) = \ell+1
  \end{equation}
  for $\ell \in \{0, \dots, t-2\}$, because the monomials $x^e y^{\ell - e}$ for $e \in \{0, \dots,\ell\}$ form a basis. Now, the first generator of $I_{\ol{\tau}}^3$ in~\eqref{eq-simplycola} has degree $t-1$, and the other two generators of $I_{\ol{\tau}}^3$ have degrees greater than or equal to $2t-3$. For a general choice of the elements $\tau_1, \tau_2 \in \on{SP}^3(f)$, for each $i \in \{0, \dots, t-3\}$, the elements $x^ey^{i-e}(x^{s-1} - \gamma_1' \cdot y^{t-1})$ for $e \in \{0, \dots, i\}$ form linearly independent relations on the monomials $x^e y^{t-1+i - e}$ for $e \in \{0, \dots, t-1+i\}$. Indeed, for each $e \in \{0, \dots, t-1+i\}$, the relation $x^ey^{i-e}(x^{s-1} - \gamma_1' \cdot y^{t-1}) = 0$ is equivalent to saying that $x^{e+s-1}y^{i-e}$ is a unit multiple of $x^ey^{i-e+t-1}$. We deduce that
  \begin{equation} \label{eq-help2}
  \dim_k ((R/I_{\ol{\tau}}^3) \otimes_R \mathfrak{m}^{t-1+i})/((R/I_{\ol{\tau}}^3) \otimes_R \mathfrak{m}^{t+i}) = t-1.
  \end{equation}
  The second generator of $I_{\ol{\tau}}^3$ in~\eqref{eq-simplycola} has degree $2t-3$, and the third generator has degree $3t-4$. For a general choice of the elements $\tau_1, \tau_2 \in \on{SP}^3(f)$, for each $i \in \{0, \dots, t-2\}$, the elements $x^ey^{t-2+i-e}(x^{s-1} - \gamma_1' \cdot y^{t-1})$ and $x^{e'}y^{i-e'}(x^{s-2}y^{t-1} - \gamma_2 \cdot y^{2t-3})$ for $e \in \{0, \dots, t-2+i\}$ and $e' \in \{0, \dots, i\}$ form a set of linearly independent relations on the monomials $x^e y^{2t-3+i - e}$ for $e \in \{0, \dots, 2t-3+i\}$. Indeed, we have the following observations:
  \begin{itemize}
  \item For each $e \in \{0, \dots, t-2+i\}$, the relation $x^ey^{t-2+i-e}(x^{s-1} - \gamma_1' \cdot y^{t-1}) = 0$ is equivalent to saying that $x^{e+s-1}y^{t-2+i-e}$ is a unit multiple of $x^ey^{2t-3+i-e}$;
  \item The relation $x^{s-2}y^{i+t-1} - \gamma_2 \cdot y^{i+2t-3} = 0$ is equivalent to saying that $x^{s-2}y^{i+t-1}$ is a unit multiple of $y^{i+2t-3}$;
  \item For each $e' \in \{1, \dots, i\}$, the two relations $x^{e'-1}y^{t-2+i-(e'-1)}(x^{s-1} - \gamma_1' \cdot y^{t-1}) = 0$ and $x^{e'}y^{i-e'}(x^{s-2}y^{t-1} - \gamma_2 \cdot y^{2t-3}) = 0$ are together equivalent to saying that $x^{e'-1+s-1}y^{t-2+i-(e'-1)}$ is a unit multiple of $x^{e'-1}y^{2t-3+i-(e'-1)}$ (this is just a restatement of the first itemized observation above) and that $x^{e'}y^{2t-3+i-e'}$ is a unit multiple of $x^{e'-1}y^{2t-3+i-(e'-1)}$.
  \end{itemize}
  Combining the above observations yields that
  \begin{equation} \label{eq-help3}
  \dim_k ((R/I_{\ol{\tau}}^3) \otimes_R \mathfrak{m}^{2t-3+i})/((R/I_{\ol{\tau}}^3) \otimes_R \mathfrak{m}^{2t-2+i}) = t - 2 - i.
  \end{equation}
  It follows that $\mathfrak{m}^{3t-5} \subset I_{\ol{\tau}}^3$, so we have that
  \begin{equation} \label{eq-help4}
  \dim_k ((R/I_{\ol{\tau}}^3) \otimes_R \mathfrak{m}^\ell)/((R/I_{\ol{\tau}}^3) \otimes_R \mathfrak{m}^{\ell+1}) = 0
  \end{equation}
  for $\ell \geq 3t-4$. Substituting the results of~\eqref{eq-help1},~\eqref{eq-help2},~\eqref{eq-help3}, and~\eqref{eq-help4} into the formula in~\eqref{eq-anandlem} yields $\dim_k R/I_{\ol{\tau}}^3 = (2t-3)(t-1)$, as desired.

  Now suppose that $s > t$; since $t \geq 2$, we have that $s \geq 3$. For convenience, let $V$ denote the $k$-vector space $R/I_{\ol{\tau}}^3$. To begin with, notice that the relation $x^{s-1} = \gamma_1' \cdot y^{t-1}$ implies that the mod-$I_{\ol{\tau}}^3$ residues of the monomials $x^iy^j$ for $i \in \{0, \dots, s-2\}$ and $j \geq 0$ span $V$. Further observe that the elements of $I_{\ol{\tau}}^3$ do not impose any relations on the monomials $x^iy^j$ for $i \in \{0, \dots, s-2\}$ and $j \in \{0, \dots, t-2\}$, so these $(s-1)(t-1)$ monomials are linearly independent; let $V_1 \subset V$ be the $k$-vector subspace spanned by these monomials.

  Next, notice that the relation $x^{s-2}y^{t-1} = \gamma_2 \cdot y^{2t-3}$ implies that the mod-$I_{\ol{\tau}}^3$ residues of the monomials $x^iy^j$ for $i \in \{0, \dots, s-3\}$ and $j \geq t-1$ span $V/V_1$. Further observe that the elements of $I_{\ol{\tau}}^3$ do not impose any relations on the monomials $x^iy^j$ for $i \in \{0, \dots, s-3\}$ and $j \in \{t-1, \dots, 2t-4\}$, so these $(s-2)(t-2)$ monomials are linearly independent;\footnote{Note that when $t = 2$, there are no such monomials, which is reflected by the fact that the formula for the number of monomials $(s-2)(t-2)$ gives $0$ when $t= 2$. One can assume $t \geq 3$ for the rest of the proof.} let $V_2 \subset V$ be the $k$-vector subspace spanned by $V_1$ along with these monomials.

  Continuing in this manner, notice that the relations $x^{s-1} = \gamma_1' \cdot y^{t-1}$ and $x^{s-2}y^{t-1} = \gamma_2 \cdot y^{2t-3}$ together imply that $xy^{2t-3} = \frac{\gamma_1'}{\gamma_2} \cdot y^{2t-2}$. Thus, the mod-$I_{\ol{\tau}}^3$ residues of the monomials $y^j$ for $j \geq 2t-3$ span $V/V_2$. We claim that the mod-$I_{\ol{\tau}}^3$ residues of the monomials $y^j$ for $j \in \{ 2t-3, \dots, 3t-6\}$ in fact form a basis of $V/V_2$. To prove this claim, we first show that we have the equality $y^{3t-5} = 0$ as elements of $R/I_{\ol{\tau}}^3$. This follows by repeatedly applying the relations (1) $x^{s-2}y^{t-1} = \gamma_2 \cdot y^{2t-3}$ and (2) $x^{s-1} = \gamma_1' \cdot y^{t-1}$ in the following order: first apply (1), and then alternately apply (1) followed by (2). \mbox{Doing so, we find that}
  \begin{align*}
  & y^{3t-5} \,\overset{\text{by (1)}}{\propto}\, x^{s-2}y^{2t-3} \,\overset{\text{by (1)}}{\propto}\, x^{2s-4}y^{t-1} \,\overset{\text{by (2)}}{\propto}\, x^{s-3}y^{2t-2} \,\overset{\text{by (1)}}{\propto}\, x^{2s-5}y^t  \,\overset{\text{by (2)}}{\propto}\, \\
  & \hphantom{===} x^{s-4}y^{2t-1} \,\overset{\text{by (1)}}{\propto}\, \cdots \,\overset{\text{by (2)}}{\propto}\, x^{s-t}y^{3t-5} \,\overset{\text{by (1)}}{\propto}\, x^{2s-t-2}y^{2t-3} \,\overset{\text{by (2)}}{\propto}\, x^{s-t-1}y^{3t-4} = 0,
  \end{align*}
  where by $p \propto q$ we mean that there exists $\rho \in R^\times$ such that $p = \rho \cdot q$ and where the final equality follows from the fact that $y^{3t-4} = 0$ as elements of $R/I_{\ol{\tau}}^3$. It remains to show that the mod-$I_{\ol{\tau}}^3$ residues of the monomials $y^j$ for $j \in \{ 2t-3, \dots, 3t-6\}$ are linearly independent. It evidently suffices to show that $y^{3t-6} \neq 0$ as elements of $R/I_{\ol{\tau}}^3$.

  Let $S$ be the set of monomials given by
  \begin{align*}
  & S= \{x^{s-1+i}y^{2t-5-i} : 0 \leq i \leq t-4 \} \cup \{x^{i}y^{3t-6-i} : 0 \leq i \leq t-3 \}\, \cup \\
  & \hphantom{= \vspan_k((}  \{x^{2s-2}y^{t-4},x^{s-2}y^{2t-4},x^{2s-3}y^{t-3}\},
  \end{align*}
  let $W_1$ be the $k$-vector space given by
\begin{align} \label{eq-defw1}
W_1 & = \vspan_k\big(\{x^{s-1+i}y^{2t-5-i} - \gamma_1' \cdot x^i y^{3t-6-i} : 0 \leq i \leq t-4\}\,\cup \\
& \hphantom{= \vspan_k((} \{x^{s-1+i}y^{2t-5-i} - \gamma_2 \cdot x^{1+i}y^{3t-7-i} : 0 \leq i \leq t-4 \} \,\cup \nonumber \\
& \hphantom{= \vspan_k((} \{x^{2s-2}y^{t-4} - \gamma_1' \cdot x^{s-1}y^{2t-5}, x^{s-2}y^{2t-4} - \gamma_2 \cdot y^{3t-6}, \nonumber \\
& \hphantom{= \vspan_k((x^{2s-2}y^{t-4}} x^{2s-3}y^{t-3}- \gamma_1' \cdot x^{s-2}y^{2t-4}\}\big) \subset I_{\ol{\tau}}^3, \nonumber
\end{align}
and notice that $W_1 \subset \vspan_k(S)$.\footnote{In each set used to define $S$ and $W_1$, if the upper limit of the index $i$ is negative, we take the set to be empty.} We now explain the motivation behind defining the set $S$ and the $k$-vector space $W_1$. Let $G$ be the set of generators of $I_{\ol{\tau}}^3$ given in~\eqref{eq-simplycola}, and consider the set $\Gamma$ defined by $$\Gamma \coloneqq \{x^iy^j \cdot g : i, j \geq 0, \, g \in G\}.$$
Given $\gamma \in \Gamma$, we say that $\gamma$ is supported on a monomial $x^iy^j$ if the coefficient (which is an element of $k$) of $x^iy^j$ in $\gamma$ is nonzero. We endow $\Gamma$ with the structure of a graph by declaring that two elements of $\Gamma$ are connected by an edge if there is some monomial $x^iy^j$ on which they are both supported. Let $\Gamma' \subset \Gamma$ be the connected component containing all elements that are supported on $y^{3t-6}$. In this setup, $S$ is the set of monomials on which some element of $\Gamma'$ is supported, and $W_1$ is the $k$-vector space span of the elements of $\Gamma'$. From this alternative characterization of $W_1$, it is clear that there exists a $k$-vector subspace $W_2 \subset I_{\ol{\tau}}^3$ such that $W_2 \subset \vspan_k(\{x^iy^j \colon i,j \geq 0\} \setminus S)$ and such that we have the decomposition $I_{\ol{\tau}}^3 = W_1 \oplus W_2$. Thus, to show that $y^{3t-6} \neq 0$, it suffices to show that $y^{3t-6} \not\in W_1$. It follows by inspecting the right-hand side of~\eqref{eq-defw1} that no $k$-linear combination of the generators has the property that it is supported only on $y^{3t-6}$ and on no other monomials, as desired.
\end{proof}

\begin{remark} \label{rem-milnrefer}
   It follows from Theorem~\ref{thm-23ab} that for the ICIS cut out analytically-locally by $y^2 - x^s = 0$ (in the ADE classification of singularities, this singularity is said to be of type $A_{s-1}$), we have $\on{AD}_{(2)}^3(y^2 - x^s) = \mu_{y^2 - x^s}$.
\end{remark}

\subsection{Computation of $\on{AD}_{(1)}^2(y^t - x^s)$}
By relying on some of the work done to prove Theorem~\ref{thm-23ab}, we obtain the analogous result for the $2^{\mathrm{nd}}$-order weight-$1$ case:
\begin{theorem} \label{thm-secordcodim1}
Let $s \geq t \geq 2$ be integers. The $2^{\mathrm{nd}}$-order weight-$1$ automatic degeneracy of a planar ICIS cut out analytically-locally by $y^t - x^s = 0$ is given by
  $$\on{AD}_{(1)}^2(y^t - x^s) = s(t-1).$$
  In particular, the number of $2^{\mathrm{nd}}$-order weight-$1$ inflection points limiting to this singularity in a general $1$-parameter deformation is equal to \mbox{$s(t-1)$.}
\end{theorem}
\begin{remark} \label{rem-wlad12}
   Let $s \geq t \geq 2$ be integers with $g \coloneqq \on{gcd}(s,t)$, let $C$ be a projective integral Gorenstein curve with a singular point $p \in C(k)$ cut out analytically-locally by $y^t - x^s = 0$, and let $(\scr{L},W)$ be a linear system on $C$ such that $\dim_k W = 2$ and such that the sequence of orders of vanishing of sections in $W$ at $p$ is given by $(0,t)$. Then the sequence of orders of vanishing of sections in $W$ at each point $p'$ lying above $p$ in the partial normalization of $C$ at $p$ is given by $(0, \frac{t}{g})$. Thus, it follows from~\eqref{eq-gattoricolfthm} that
   \begin{equation} \label{eq-helpereq1}
       \on{WL}_p(\scr{L},W) = \delta_p \cdot 2(2-1) + g \cdot \left(\frac{t}{g}-1\right) = 2\delta_p + t-g.
   \end{equation}
   Moreover, the Milnor-Jung formula (see~\cite[Theorem 10.5]{milne}) tells us that
   \begin{equation} \label{eq-helpereq3}
   2\delta_{y^t - x^s} = \mu_{y^t - x^s} + r_{y^t - x^s} - 1,
   \end{equation}
   where $r_{y^t - x^s}$ equals the number of branches of $y^t - x^s = 0$. But we have $\delta_{y^t - x^s} = \delta_p$, $\mu_{y^t - x^s} = (s-1)(t-1)$, and $r_{y^t - x^s} = g$, so~\eqref{eq-helpereq3} becomes
  \begin{equation} \label{eq-helpereq2}
   2\delta_p = (s-1)(t-1) + g - 1.
   \end{equation}
   Combining~\eqref{eq-helpereq1} and~\eqref{eq-helpereq2} yields that
   $$\on{WL}_p(\scr{L},W) = ((s-1)(t-1)+g-1)+t-g = s(t-1).$$
   Thus, in this case, it follows from Theorem~\ref{thm-secordcodim1} that equality holds in~\eqref{eq-adwall}; i.e., we have that $\on{AD}_{(1)}^2(y^t - x^s) = \on{WL}_p(\scr{L},W)$.
\end{remark}
\begin{proof}[Proof of Theorem~\ref{thm-secordcodim1}]
It suffices to show that $$\on{SD}_{(1)}^2(r \cdot (y^t - x^s)) = s(t-1)$$
for any unit $r \in R^\times$. Put $f = r \cdot (y^t - x^s)$. Using notation from the proof of Theorem~\ref{thm-23ab} (and redefining $\tau_1, \tau_2 \in \on{SP}^3(f)$ to be their residues in $\on{SP}^2(f)$), the degeneracy ideal is given by
\begin{equation} \label{eq-listgens2fix}
I_{\ol{\tau}}^2 = (f, \Xi_1'') = (y^t - x^s, x^{s-1} - \gamma_1' \cdot y^{t-1}),
\end{equation}
and the desired automatic degeneracy is simply the colength of this ideal. As in the proof of Theorem~\ref{thm-23ab}, to compute the colength, it suffices to find a finite collection of monomials in $R$ the mod-$I_{\ol{\tau}}^2$ residues of which form a basis of $R/I_{\ol{\tau}}^2$.

We handle the cases $s = t$ and $s > t$ separately. First suppose that $s = t$; we need to show that the desired colength is equal to $t(t-1)$. In this case, the ideal $I_{\ol{\tau}}^2$ is homogeneous, so it is easy to use the formula~\eqref{eq-anandlem} to compute $\dim_k R/I_{\ol{\tau}}^2$. Since neither of the two generators of $I_{\ol{\tau}}^2$ in~\eqref{eq-listgens2fix} have degree less than $t-1$, we deduce that
  \begin{equation} \label{eq-helpaloo}
  \dim_k ((R/I_{\ol{\tau}}^2) \otimes_R \mathfrak{m}^\ell)/((R/I_{\ol{\tau}}^2) \otimes_R \mathfrak{m}^{\ell+1}) = \ell+1
  \end{equation}
  for $\ell \in \{0, \dots, t-2\}$, because the monomials $x^e y^{\ell - e}$ for $e \in \{0, \dots,\ell\}$ form a basis. Now, the second generator of $I_{\ol{\tau}}^2$ in~\eqref{eq-listgens2fix} has degree $t-1$, and the first generator has degree greater than $t-1$. Thus, for a general choice of the elements $\tau_1, \tau_2 \in \on{SP}^2(f)$, the element $x^{s-1} - \gamma_1' \cdot y^{t-1}$ is the only relation on the monomials $x^e y^{t-1 - e}$ for $e \in \{0, \dots, t-1\}$, so we deduce that
  \begin{equation} \label{eq-helpaloo2}
  \dim_k ((R/I_{\ol{\tau}}^2) \otimes_R \mathfrak{m}^{t-1})/((R/I_{\ol{\tau}}^2) \otimes_R \mathfrak{m}^{t}) = t-1.
  \end{equation}
  The first generator of $I_{\ol{\tau}}^2$ in~\eqref{eq-listgens2fix} has degree $t$. For a general choice of the elements $\tau_1, \tau_2 \in \on{SP}^2(f)$, for each $i \in \{0, \dots, t-2\}$, the elements $x^ey^{i+1-e}(x^{s-1} - \gamma_1' \cdot y^{t-1})$ and $x^{e'}y^{i-e'}(x^s - y^t)$ for $e \in \{0, \dots, i+1\}$ and $e' \in \{0, \dots, i\}$ form a set of linearly independent relations on the monomials $x^e y^{t+i - e}$ for $e \in \{0, \dots, t+i\}$. Indeed, we have the following observations:
  \begin{itemize}
      \item For each $e \in \{0, \dots, i+1\}$, the relation $x^ey^{i+1-e}(x^{s-1} - \gamma_1' \cdot y^{t-1}) = 0$ is equivalent to saying that $x^{e+s-1}y^{i+1-e}$ is a unit multiple of $x^ey^{i+t-e}$.
      \item For each $e' \in \{0, \dots, i\}$, the two relations $x^{e'+1}y^{i+1-(e'+1)}(x^{s-1} - \gamma_1' \cdot y^{t-1})$ and $x^{e'}y^{i-e'}(x^s - y^t)$ are together equivalent to saying that $x^{e'+1+s-1}y^{i+1-(e'+1)}$ is a unit multiple of $x^{e'}y^{i+t-e'}$ (this is just a restatement of the first itemized observation above) and that $x^{e'+1}y^{i+t-(e'+1)}$ is a unit multiple of $x^{e'}y^{i+t-e'}$.
  \end{itemize}
  Combining the above observations yields that
  \begin{equation} \label{eq-helpaloo3}
  \dim_k ((R/I_{\ol{\tau}}^2) \otimes_R \mathfrak{m}^{t+i})/((R/I_{\ol{\tau}}^2) \otimes_R \mathfrak{m}^{t+i+1}) = t - 2 - i.
  \end{equation}
  It follows that $\mathfrak{m}^{2t-2} \subset I_{\ol{\tau}}^2$, so we have that
  \begin{equation} \label{eq-helpaloo4}
  \dim_k ((R/I_{\ol{\tau}}^2) \otimes_R \mathfrak{m}^\ell)/((R/I_{\ol{\tau}}^2) \otimes_R \mathfrak{m}^{\ell+1}) = 0
  \end{equation}
  for $\ell \geq 2t-1$. Substituting the results of~\eqref{eq-helpaloo},~\eqref{eq-helpaloo2},~\eqref{eq-helpaloo3}, and~\eqref{eq-helpaloo4} into the formula in~\eqref{eq-anandlem} yields $\dim_k R/I_{\ol{\tau}}^2 = t(t-1)$, as desired.

Now suppose that $s > t$. For convenience, let $V$ denote the $k$-vector space $R/I_{\ol{\tau}}^2$. To begin with, notice that the relation $x^{s-1} = \gamma_1' \cdot y^{t-1}$ implies that the mod-$I_{\ol{\tau}}^2$ residues of the monomials $x^iy^j$ for $i \in \{0, \dots, s-2\}$ and $j \geq 0$ span $V$. Further observe that the elements of $I_{\ol{\tau}}^2$ do not impose any relations on the monomials $x^iy^j$ for $i \in \{0, \dots, s-2\}$ and $j \in \{0, \dots, t-2\}$, so these $(s-1)(t-1)$ monomials are linearly independent; let $V_1 \subset V$ be the $k$-vector subspace spanned by these monomials.

  Next, notice that the relations $x^{s-1} = \gamma_1' \cdot y^{t-1}$ and $x^s = y^t$ together imply that $xy^{t-1} \propto y^t$, so the mod-$I_{\ol{\tau}}^2$ residues of the monomials $y^j$ for $j \geq t-1$ span $V/V_1$. Furthermore, the relation $xy^{t-1} \propto y^t$ implies the relation $x^{s+1} \propto xy^t \propto y^{t+1}$. Repeating this process inductively, we obtain the relation $x^{s+i+2} \propto y^{t+i+2}$ from the relations $x^{s+i} \propto y^{t+i}$ and $x^{s+i+1} \propto y^{t+i+1}$. In particular, we find that $x^{s+t-2} \propto y^{2(t-1)}$, but notice that $x^{2(s-1)} = (\gamma_1')^2 \cdot y^{2(t-1)}$, so we find that $x^{s+t-2} \propto x^{2(s-1)}$, which implies that $y^{2(t-1)} \propto x^{s+t-2} = 0$ as elements of $R/I_{\ol{\tau}}^2$, because we took $s > t$. We claim that the mod-$I_{\ol{\tau}}^2$ residues of the monomials $y^j$ for $j \in \{t-1, \dots, 2t-3\}$ are linearly independent. It suffices to show that $y^{2t-3} \neq 0$ as elements of $R/I_{\ol{\tau}}^2$.

  Let $S$ be the set of monomials given by
  $$S= \{x^iy^{2t-3-i} : 0 \leq i \leq t-2\} \cup \{x^{s+i}y^{t-3-i} : 0 \leq i \leq t-3\} \cup \{x^{s-1}y^{t-2}\},$$ let $W_1$ be the $k$-vector space given by
\begin{align}
W_1 & = \vspan_k\big(\{x^{s+i}y^{t-3-i} - x^iy^{2t-3-i} : 0 \leq i \leq t-3\}\,\cup \label{eq-defw1the2nd}\\
& \hphantom{= \vspan_k((} \{x^{s+i}y^{t-3-i} - \gamma_1' \cdot x^{i+1}y^{2t-4-i} : 0 \leq i \leq t-3\}\cup\{x^{s-1}y^{t-2}-\gamma_1' \cdot y^{2t-3}\}\big) \subset I_{\ol{\tau}}^2, \nonumber
\end{align}
and notice that $W_1 \subset \vspan_k(S)$.\footnote{In each set used to define $S$ and $W_1$, if the upper limit of the index $i$ is negative, we take the set to be empty.} By the same sort of argument that we used in the proof of Theorem~\ref{thm-23ab}, there is a $k$-vector subspace $W_2 \subset I_{\ol{\tau}}^2$ such that $W_2 \subset \vspan_k(\{x^iy^j \colon i,j \geq 0\} \setminus S)$ and such that we have the decomposition $I_{\ol{\tau}}^2 = W_1 \oplus W_2$. Thus, to show that $y^{2t-3} \neq 0$, it suffices to show that $y^{2t-3} \not\in W_1$. It follows by inspecting the right-hand side of~\eqref{eq-defw1the2nd} that no $k$-linear combination of the generators has the property that it is supported only on $y^{2t-3}$ and on no other monomials, as desired.
\end{proof}

\begin{example} \label{eg-toss}
For each order $m$, we have introduced three different automatic degeneracy invariants associated to ICIS germs. Among all of these invariants, it appears that only one---the $2^{\mathrm{nd}}$-order weight-$1$ automatic degeneracy---has been previously examined in the literature. Consider the following well-studied enumerative problem: given a plane curve $C \subset \BP_k^2$ of degree $d$ and a fixed point $p \in \BP_k^2$, how many lines through $p$ are tangent to $C$? There are two known ways of answering this question. The first is to notice that the desired number is equal to the degree of the dual curve of $C$; thus if $C$ is smooth, the answer is $d(d-1)$. The second is to notice that the points of tangency of such lines with the curve $C$ are $2^{\mathrm{nd}}$-order weight-$1$ inflection points with respect to the linear system $(\scr{O}_C(1), W)$, where $W \subset H^0(\scr{O}_C(1))$ is the $k$-vector space of linear forms vanishing at the point $p$. If $C$ is smooth, the number of such inflection points on $C$ is then given by $\deg c_1(\scr{P}_C^2(\scr{O}_C(1)))$. Evaluating this Chern class degree using the Pl\"{u}cker formula (see~\cite[Theorem 7.13]{harris3264}) yields the desired number $d(d-1)$.

It is natural to ask whether the calculations in the previous paragraph can be extended to the case where the curve $C$ is singular. As it happens, the answer is yes: In 1834, J.~Pl\"{u}cker himself showed that if the only singularities of $C$ are nodes and cusps, the degree of the dual of $C$ is given by
\begin{equation} \label{eq-basicplucker}
d(d-1) - 2\delta - 3\kappa,
\end{equation}
where $\delta$ is the number of nodes and $\kappa$ is the number of cusps on the curve $C$. A further generalization of the formula in~\eqref{eq-basicplucker} can be found in~\cite[App.~II]{dualpack}, where Teissier shows that the degree of the dual of any plane curve $C$ with $n$ isolated singularities cut out analytically-locally by $f_1 = 0, \dots, f_n = 0$ is given by
\begin{equation} \label{eq-tesspluck}
    d(d-1) - \sum_{i = 1}^n e_{f_i},
\end{equation}
where $e_{f_i}$ is the Hilbert-Samuel multiplicity of the Jacobian ideal $J(f_i) \coloneqq \left(\frac{\d f_i}{\d x}, \frac{\d f_i}{\d y}\right) \subset R/(f_i)$. Recall that for a planar ICIS cut out analytically-locally by $f = 0$, the Hilbert-Samuel multiplicity $e_f$ of the ideal $J(f) = \left(\frac{\d f}{\d x}, \frac{\d f}{\d y}\right) \subset R/(f)$ is given by
\begin{equation} \label{eq-hilbsam}
e_f = \lim_{s \to \infty} \frac{1}{s} \cdot \dim_k R/\left(f, J(f)^s\right).
\end{equation}
(Note that formulas~\eqref{eq-basicplucker} and~\eqref{eq-tesspluck} agree because $e_{y^2 - x^2} = 2$ and $e_{y^2 - x^3} = 3$.)

One way to interpret the formula~\eqref{eq-tesspluck} is to think of the singularity with germ $f_i$ as diminishing the degree of the dual curve by an amount equal to $e_{f_i}$. But the degree of the dual curve counts $2^{\mathrm{nd}}$-order weight-$1$ inflection points, so the formula~\eqref{eq-tesspluck} suggests that the singularity with germ $f_i$ ``counts as'' $e_{f_i}$-many $2^{\mathrm{nd}}$-order weight-$1$ inflection points. We now arrive at an obvious question: is the number of $2^{\mathrm{nd}}$-order weight-$1$ inflection points limiting to a planar ICIS $f$ in a general $1$-parameter deformation equal to the number of inflection points that the singularity with germ $f$ ``counts as'' in the sense of the formula~\eqref{eq-tesspluck}? We answer this question in the affirmative in the following theorem:
\end{example}

\begin{theorem}  \label{thm-tess}
   We have for any ICIS germ $f \in R$ that \begin{equation} \label{eq-eqofmults}
    \on{AD}_{(1)}^2(f) = e_f.
\end{equation}
In particular, the number of $2^{\mathrm{nd}}$-order weight-$1$ inflection points limiting to this singularity in a general $1$-parameter deformation is equal to \mbox{$e_f$.}
\end{theorem}
\begin{proof}
We recycle notation from the proof of Theorem~\ref{thm-secordcodim1}. From~\eqref{eq-listgens2}, we have that
$$\on{AD}_{(1)}^2(f) = \dim_k R/I_{\ol{\tau}}^2 = \dim_k R/(f, \Xi_1'') = \dim_k R\bigg/\left(f, \frac{\d f}{\d x} - \gamma_1' \cdot \frac{\d f}{\d y}\right)$$
Thus, to prove the claimed equality~\eqref{eq-eqofmults}, it suffices to show for a general choice of $\alpha, \beta \in R^\times$ that
\begin{equation} \label{eq-needsmustrefer}
\dim_k R\bigg/\left(f, \alpha \cdot \frac{\d f}{\d x} - \beta \cdot \frac{\d f}{\d y}\right) = e_f.
\end{equation}
Note that the vanishing locus of $\alpha \cdot \frac{\d f}{\d x} - \beta \cdot \frac{\d f}{\d y}$, viewed as a subscheme of $\Spec R/(f)$, is a complete intersection. It follows that the Hilbert-Samuel multiplicity of the ideal $I \coloneqq \left(\alpha  \cdot \frac{\d f}{\d x} - \beta \cdot \frac{\d f}{\d y}\right) \subset R/(f)$ is equal to the colength of $I$ in $R/(f)$, which is the left-hand side of~\eqref{eq-needsmustrefer}. Thus, it suffices to show that the Hilbert-Samuel multiplicity of $I$ is equal to that of $J(f)$. Because $I \subset J(f)$, and because the Hilbert-Samuel multiplicity of an ideal is equal to that of any reduction, it further suffices to show that $I$ is a reduction of $J(f)$; i.e., we need only show that for some integer $s \geq 0$ we have the following equality of ideals of $R/(f)$:
\begin{equation} \label{eq-desireeideals}
I \cdot J(f)^s = J(f)^{s+1}.
\end{equation}
It is obvious that $I \cdot J(f)^s \subset J(f)^{s+1}$ for every $s$, so it suffices to show that $I \cdot J(f)^s \supset J(f)^{s+1}$ for some $s$. Consider the $R/(f)$-algebra
$$A \coloneqq k \oplus \bigoplus_{i \geq 1}^\infty J(f)^i/(\mathfrak{m} \cdot J(f)^{i}),$$
where $\mathfrak{m} = (x,y) \subset R/(f)$ denotes the maximal ideal.\footnote{The algebra $A$ is known as the \emph{fiber cone} of the ideal $J(f)$.} Observe that $\dim A = 1$, so for a general choice of $\alpha, \beta \in R^\times$, we have $\dim ((R/(f))/I) \otimes_{R/(f)} A = 0$, implying that $\dim_k ((R/(f))/I) \otimes A < \infty$. So, for some integer $s \geq 0$ we have
$$I \otimes_{R/(f)} (J(f)^s/(\mathfrak{m}\cdot J(f)^{s})) = J(f)^{s+1}/(\mathfrak{m} \cdot J(f)^{s+1}).$$
Thus, there is a finite list of elements $\sigma_1, \dots, \sigma_\ell \in J(f)^s/(\mathfrak{m} \cdot J(f)^s)$ such that the elements $\left(\alpha \cdot \frac{\d f}{\d x} - \beta \cdot \frac{\d f}{\d y}\right) \cdot \sigma_i$ generate $J(f)^{s+1}/(\mathfrak{m} \cdot J(f)^{s+1})$. Choose a lift $\wt{\sigma}_i \in J(f)^s$ of $\sigma_i$ for each $i \in \{1, \dots, \ell\}$. Then the elements $\left(\alpha \cdot \frac{\d f}{\d x} - \beta \cdot \frac{\d f}{\d y}\right) \cdot \wt{\sigma}_i \in I \cdot J(f)^s$ generate $J(f)^{s+1}$, so for this choice of $s$ we have $I \cdot J(f)^s \supset J(f)^{s+1}$, as desired.
\end{proof}

\begin{remark}
   For general $[\alpha : \beta ] \in \BP_k^1$, the element $\alpha \cdot \frac{\d f}{\d x} - \beta \cdot \frac{\d f}{\d y} \in J(f)$ is called a \emph{polar} of the singularity with germ $f$, and the quantity $\dim_k R\bigg/\left(f, \alpha \cdot \frac{\d f}{\d x} - \beta \cdot \frac{\d f}{\d y}\right)$---i.e., the intersection multiplicity of the singularity with the polar---is related to Teissier's notion of \emph{polar invariant} (see~\cite{tesspolarvars}). Another way to interpret the result of Theorem~\ref{thm-tess} is to say that the Hilbert-Samuel multiplicity of the Jacobian ideal of a planar ICIS is equal to its intersection multiplicity with a general polar, and another way to interpret the result of Theorem~\ref{thm-secordcodim1} is to say that the Hilbert-Samuel multiplicity of the hypercuspidal singularity with germ $y^t - x^s$ for $s \geq t \geq 2$ is given by $s(t-1)$.
\end{remark}

\subsection{Computation of $\on{AD}_{(2)}^4(y^2 - x^s)$}

While it remains open to compute $\on{AD}_{(2)}^4(y^t - x^s)$ for arbitrary pairs $(s,t)$ with $s \geq t \geq 2$, we show in the following theorem that the computation is feasible when we take $t = 2$, in which case the singularity under consideration is of type $A_{s-1}$.

\begin{theorem} \label{thm-2t}
  Let $s \geq 2$ be an integer. The $4^{\mathrm{th}}$-order weight-$2$ type-(a) automatic degeneracy of a planar ICIS cut out analytically-locally by $y^2 - x^s = 0$ is given by
  $$\on{AD}_{(2)}^4(y^2 - x^s) = \begin{cases} 5 & \text{if} \quad \text{$s = 2$} \\ 10 & \text{if} \quad \text{$s = 3$} \\ 6(s-1) & \text{if} \quad \text{$s \geq 4$} \end{cases}$$
  In particular, the number of $4^{\mathrm{th}}$-order weight-$2$ type-(a) inflection points limiting to this singularity in a general $2$-parameter deformation is equal to $0$ if $s \in \{2,3\}$ and $s-1$ if $s \geq 4$.
\end{theorem}
\begin{proof}
The case where $s = 2$ follows from Theorem~\ref{thm-main2}, and the case where $s = 3$ follows from Corollary~\ref{cor-cusp}.
 For the remainder of the proof, we take $s \geq 4$. It suffices to show that we have
 $$\on{SD}_{(2)}^4(r \cdot (y^2 - x^s)) = 6(s-1)$$
 for any $r \in R^\times$. Let $f = r \cdot (y^2 - x^s)$. Given elements $\tau_1, \tau_2, \tau_3 \in \on{SP}^4(f)$ with expansions $\tau_\ell = \sum_{0 \leq i+j \leq 3} \alpha_{ij}^{(\ell)} \cdot a^ib^j$, it is far too complicated to explicitly write down the associated degeneracy matrix with respect to the basis $(\wt{\theta}_1, \wt{\theta}_2, \wt{\theta}_3, \wt{\theta}_4)$ computed in \S~\ref{sec-expibasis}, let alone its maximal minors, as we did in \S~\ref{sec-23ab}. Nonetheless, we now show that it is possible to obtain a relatively concise expression of these minors in the specific case where $f = r \cdot (y^2 - x^s)$.

   Let $\Xi_i$ denote the minor of the degeneracy matrix obtained by removing the $(5-i)^{\mathrm{th}}$ row and taking the determinant. Let $\{i,j,k\} = \{1,2,3\}$, and let $\varepsilon_{ijk}$ denote the three-dimensional Levi-Civita symbol (i.e., $\varepsilon_{ijk}$ is equal to the sign of the permutation $(i,j,k)$ of the list $(1,2,3)$). Then, using index notation to suppress sums over the indices $i,j,k$, we have the following expressions for the minors $\Xi_i$:
\begin{align*}
    \Xi_1 & = \varepsilon_{ijk}\cdot \left(\nu_1 \cdot 8 \alpha_{00}^{(i)}\alpha_{01}^{(j)}\alpha_{02}^{(k)} \cdot y^3 + \nu_2 \cdot 2s(s-1)\alpha_{00}^{(i)}\alpha_{01}^{(j)}\alpha_{10}^{(k)} \cdot y^2x^{s-2}  - \nu_3 \cdot s^2\alpha_{00}^{(i)}\alpha_{01}^{(j)}\alpha_{10}^{(k)} \cdot x^{2s-2}\right) \\
    \Xi_2 & = \varepsilon_{ijk} \cdot \left(\nu_4 \cdot 16 \alpha_{00}^{(i)}\alpha_{01}^{(j)}\alpha_{03}^{(k)} \cdot y^4 + \nu_5 \cdot \tfrac{4}{3}s(s-1)(s-2) \alpha_{00}^{(i)}\alpha_{01}^{(j)}\alpha_{10}^{(k)} \cdot y^3x^{s-3} + \right. \\
    & \hphantom{=======} \left. \nu_6 \cdot 8t\alpha_{00}^{(i)}\alpha_{01}^{(j)}\alpha_{02}^{(k)} \cdot y^2x^{s-1} + \nu_7 \cdot (2s^2 \alpha_{00}^{(i)} \alpha_{01}^{(j)} \alpha_{11}^{(k)} - 4s^2\alpha_{00}^{(i)} \alpha_{02}^{(j)} \alpha_{10}^{(k)}) \cdot yx^{2s-2} + \right. \\
    & \hphantom{=======} \left. \nu_8 \cdot s^3 \alpha_{00}^{(i)} \alpha_{10}^{(j)} \alpha_{11}^{(k)} \cdot x^{3s-3}\right) \\
    \Xi_3 & = \varepsilon_{ijk} \cdot \left(\nu_9 \cdot 32 \alpha_{00}^{(i)} \alpha_{02}^{(j)} \alpha_{03}^{(k)} \cdot y^5 + \nu_{10} \cdot \tfrac{8}{3}s(s-1)(s-2)\alpha_{00}^{(i)} \alpha_{02}^{(j)} \alpha_{10}^{(k)} \cdot y^4 x^{s-3} + \right. \\
    & \hphantom{=======} \left. \nu_{11} \cdot 8t\alpha_{00}^{(i)} \alpha_{01}^{(j)} \alpha_{03}^{(k)} \cdot y^3 x^{s-1} +  \nu_{12} \cdot \tfrac{2}{3}s^2 (s-1)(s-2) \alpha_{00}^{(i)} \alpha_{01}^{(j)} \alpha_{10}^{(k)} \cdot y^2 x^{2s-4} + \right. \\
    & \hphantom{=======} \left. \nu_{13} \cdot 4t^2\alpha_{00}^{(i)} \alpha_{01}^{(j)} \alpha_{02}^{(k)} \cdot yx^{2s-2} + \nu_{14} \cdot s^4(\alpha_{00}^{(i)} \alpha_{01}^{(j)} \alpha_{30}^{(k)} + \alpha_{00}^{(i)} \alpha_{10}^{(j)} \alpha_{11}^{(k)} + \alpha_{00}^{(i)} \alpha_{20}^{(j)} \alpha_{11}^{(k)}) \cdot x^{4s-4} \right) \\
    \Xi_4 & = \varepsilon_{ijk} \cdot \left(\nu_{15} \cdot 64\alpha_{01}^{(i)} \alpha_{02}^{(j)} \alpha_{03}^{(k)} \cdot y^6 + \nu_{16} \cdot \tfrac{16}{3}s(s-1)(s-2)\alpha_{01}^{(i)} \alpha_{02}^{(j)} \alpha_{10}^{(k)} \cdot y^5x^{s-3} + \right. \\
    & \hphantom{=======} \left. \nu_{17} \cdot \tfrac{4}{3} s^2(s+1)(s-1) \alpha_{01}^{(i)} \alpha_{10}^{(j)} \alpha_{11}^{(k)} \cdot y^4x^{2s-4} - \nu_{18} \cdot 8s^2(s-1)\alpha_{01}^{(i)} \alpha_{02}^{(j)} \alpha_{10}^{(k)} \cdot y^3 x^{2s-3} + \right. \\
    & \hphantom{=======} \left. \nu_{19} \cdot 4s^3 \alpha_{01}^{(i)} \alpha_{02}^{(j)} \alpha_{10}^{(k)} \cdot yx^{3s-3} - \nu_{20} \cdot s^4\alpha_{01}^{(i)} \alpha_{10}^{(j)} \alpha_{11}^{(k)} \cdot x^{4s-4} \right)
\end{align*}
where $\nu_\ell \in R^\times$ for each $\ell \in \{1, \dots, 20\}$ (note that the $\nu_\ell$ depend on the choice of $r$ but are nevertheless always units).

Suppose the elements $\tau_\ell \in \on{SP}^4(f)$ are general, and let $[g]$ denote a term that is some general unit multiple of $g \in R$. Then by taking particular linear combinations of the minors $\Xi_i$ using the explicit formulas above, one can (albeit with painstaking effort) obtain the following identities:
\begin{align}
& \Xi_2' \coloneqq \Xi_2 + ([y] + [x^{s-3}]) \cdot \Xi_1  = [y^2 x^{s-1}] + [y x^{2s-2} + \gamma \cdot x^{3s-5}] \label{eq-blockrel2}\\
& \Xi_3' \coloneqq \Xi_3 + ([y] + [x^{s-3}]) \cdot \Xi_2 + ([y^2] + [yx^{s-3}] + [x^{2s-6}])\cdot \Xi_1 = \label{eq-blockrel6}\\
& \quad\qquad [yx^{2s-2}] +\begin{cases}  [x^8] & \quad \text{if $s = 4$} \\ [x^{3s-3}] & \quad \text{if $s \geq 5$} \end{cases}\nonumber\\
& \Xi_4' \coloneqq \Xi_4  + [x^{s-1}] \cdot \Xi_3  + ([y^2]+[yx^{s-1}])\cdot \Xi_2 + ([y^3] + [y^2x^{s-3}] + [yx^{2s-4}] + [x^{2s-3}]) \cdot \Xi_1   = \label{eq-blockrel11}\\
& \quad\qquad [yx^{3s-3}]+[x^{4s-4}]. \nonumber
\end{align}
where $\gamma = -\frac{\nu_3 \cdot s^2}{\nu_2 \cdot 2s(s-1)}$. Let $I_{\ol{\tau}}^4 = (\Xi_1, \Xi_2, \Xi_3, \Xi_4)$ denote the degeneracy ideal, and observe that $I_{\ol{\tau}}^4 = (\Xi_1, \Xi_2', \Xi_3', \Xi_4')$.

For convenience, let $V$ denote the $k$-vector space $R/I_{\ol{\tau}}^4$. Notice that the relation $\Xi_1 = 0$ implies that the mod-$I_{\ol{\tau}}^4$ residues of the monomials $y^ix^j$ for $i \in \{0, 1, 2\}$ and $j \geq 0$ span $V$. Further observe that the elements of $I_{\ol{\tau}}^4$ do not impose any relations on the monomials $y^ix^j$ satisfying either $i \in \{0, 1\}$ and $j \in \{0, \dots, s-2\}$ or $i = 2$ and $j \in \{0, \dots, s-3\}$, so these $3s-4$ monomials are linearly independent; let $V_1 \subset V$ be the $k$-vector subspace spanned by these monomials. Now notice that the relation $\Xi_1 = 0$ is the only relation in $I_{\ol{\tau}}^4$ involving the monomial $y^2x^{s-2}$, so $y^2x^{s-2} \neq 0$ as an element of $V/V_1$. Let $V_2 \subset V$ be the $k$-vector subspace (of dimension $3s-3$) spanned by $V_1$ along \mbox{with the monomial $y^2x^{s-2}$.}

  Next, notice from~\eqref{eq-blockrel2} that the relation $\Xi_2' = 0$ implies that the mod-$I_{\ol{\tau}}^4$ residues of the monomials $y^ix^j$ for $i \in \{0, 1\}$ and $j \geq s-1$ span $V/V_2$. Further observe that the elements of $I_{\ol{\tau}}^4$ do not impose any relations on the monomials $y^ix^j$ for $i \in \{0,1\}$ and $j \in \{s-1, \dots, 2s-3\}$, so these $2s-2$ monomials are linearly independent; let $V_3 \subset V$ be the $k$-vector subspace spanned by $V_2$ along with these monomials.

  Let $s \geq 5$. We now show that $x^{3s-3} = 0$ as elements of $R/I_{\ol{\tau}}^4$, which would imply that the mod-$I_{\ol{\tau}}^4$ residues of the monomials $x^i$ for $i \in \{2s-2, \dots, 3s-4\}$ span $V/V_3$. Combining the relations $x \cdot \Xi_1 = 0$ and $\Xi_3' = 0$ (see~\eqref{eq-blockrel6}) yields that $y^3x + [x^{2s-1}] = 0$. Substituting this new relation along with the relation $\Xi_3' = 0$ into the relation $y \cdot \Xi_2' = 0$ yields that $x^{3s-3} + [yx^{3s-5}] = 0$. But the relation $\Xi_3' = 0$ tells us that $x^{3s-3} + [yx^{2s-2}] = 0$, so we deduce that $yx^{2s-2} + [yx^{3s-5}] = 0$, implying that $0 = yx^{2s-2} = x^{3s-3}$, as desired.

  For the case $s \geq 5$, it remains to show that $x^{3s-4} \neq 0$ as elements of $R/I_{\ol{\tau}}^4$. The strategy is similar to that used in the proof of Theorem~\ref{thm-23ab}, although the relations $\Xi_1, \Xi_2', \Xi_3',\Xi_4'$ that generate $I_{\ol{\tau}}^4$ are so complicated that it is more effective to work modulo $\mathfrak{m}^{3s-3}$ (i.e., we show that $x^{3s-4} \not\in I_{\ol{\tau}}^4/\mathfrak{m}^{3s-3}$). Notice from~\eqref{eq-blockrel2}--\eqref{eq-blockrel11} that modulo $\mathfrak{m}^{3s-3}$, the ideal $I_{\ol{\tau}}^4$ is generated by $\Xi_1$, $[y^2 x^{s-1}] + [x^{3s-5}]$, and $yx^{2s-2}$. Let $S$ be the set of monomials given by
  $$S= \{x^{3s-4},y^2x^s, y^3x^2,y^2x^{2s-4}, x^{2s},y^3x^{s-2}\},$$ let $W_1$ be the $k$-vector space given by
\begin{align}
 W_1 & = \vspan_k\big(x^{s-2} \cdot \Xi_1, x^2 \cdot \Xi_1, x \cdot \Xi_2', x^{s-3} \cdot \Xi_2' \big) \nonumber \\
& = \vspan_k\big(\{[y^3x^{s-2}] + [y^2x^{2s-4} + \gamma \cdot x^{3s-4}], [y^3x^2] + [y^2x^s + \gamma \cdot x^{2s}], \label{eq-defw1the3rd}  \\
& \hphantom{= \vspan_k((x^{2s-2}y^{t-4}} [y^2x^s]+[x^{3s-4}], y^2x^{2s-4}\}\big) \subset I_{\ol{\tau}}^4/\mathfrak{m}^{3s-3}, \nonumber
\end{align}
and notice that $W_1 \subset \vspan_k(S)$.\footnote{In each set used to define $S$ and $W_1$, if the upper limit of the index $i$ is negative, we take the set to be empty.} By the same sort of argument that we used in the proof of Theorem~\ref{thm-23ab}, there is a $k$-vector subspace $W_2 \subset I_{\ol{\tau}}^4/\mathfrak{m}^{3s-3}$ such that $W_2 \subset \vspan_k(\{y^ix^j \colon 0 \leq i+j \leq 3s-4\} \setminus S)$ and such that we have the decomposition $I_{\ol{\tau}}^4/\mathfrak{m}^{3s-3} = W_1 \oplus W_2$. Thus, to show that $x^{3s-4} \neq 0$, it suffices to show that $x^{3s-4} \not\in W_1$. It follows by inspecting the right-hand side of~\eqref{eq-defw1the3rd} that no $k$-linear combination of the generators has the property that it is supported only on $x^{3s-4}$ and on no other monomials, as desired.

The proof of the case $s = 4$ involves a similar analysis, so we omit it for the sake of brevity.
\end{proof}

\subsection{Computation of $\on{AD}_{(1)}^3(y^2 - x^s)$}

\begin{theorem} \label{thm-2t1}
Let $s \geq 2$ be an integer. The $3^{\mathrm{rd}}$-order weight-$1$ automatic degeneracy of a planar ICIS cut out analytically-locally by $y^2 - x^s = 0$ is given by
  $$\on{AD}_{(1)}^3(y^2 - x^s) =  \begin{cases} 6 & \quad \text{if $s = 2$} \\
  8 & \quad \text{if $s = 3$} \\ 3s & \quad \text{if $s \geq 4$}\end{cases}$$
  In particular, the number of $3^{\mathrm{rd}}$-order weight-$1$ inflection points limiting to this singularity in a general $1$-parameter deformation is equal to $6$ if $s = 2$, $8$ if $s = 3$, and $3s$ if $s \geq 4$.
\end{theorem}
\begin{remark} \label{rem-cuspAD1dontknow}
    Let $s \geq 2$ be an integer, let $C$ be a projective integral Gorenstein curve with a singular point $p \in C(k)$ cut out analytically-locally by $y^2 - x^s = 0$, and let $(\scr{L},W)$ be a linear system on $C$ such that $\dim_k W = 3$ and such that the following properties hold:
    \begin{enumerate}
    \item The sequence of orders of vanishing of sections in $W$ at $p$ is given by $(0,2,3)$ if $s$ is even or $s = 3$ and by $(0,2,4)$ if $s \geq 5$ is odd.
    \item The sequence of orders of vanishing of sections in $W$ at each point $p'$ lying above $p$ in the partial normalization of $C$ at $p$ is given by $(0,1,2)$ when $s$ is even, by $(0,2,3)$ when $s  = 3$, and by $(0, 2, 4)$ when $s \geq 5$ is odd.
    \end{enumerate}
    Then, it follows from~\eqref{eq-gattoricolfthm} that
   \begin{equation} \label{eq-helpereq5}
       \on{WL}_p(\scr{L},W) = \delta_p \cdot 3(3-1) + \begin{cases} 2 \cdot 0 & \quad \text{if $s$ is even} \\ 2 & \quad \text{if $s = 3$} \\ 3 & \quad \text{if $s \geq 5$ is odd}\end{cases}
   \end{equation}
    By~\eqref{eq-helpereq3}, we have that
  \begin{equation} \label{eq-helpereq6}
   2\delta_p = s + g - 2,
   \end{equation}
   where $g = 2$ if $s$ is even and $g = 1$ if $s$ is odd.
   Combining~\eqref{eq-helpereq5} and~\eqref{eq-helpereq6} yields that
   $$\on{WL}_p(\scr{L},W) = \begin{cases} 6 & \quad \text{if $s = 2$} \\
  8 & \quad \text{if $s = 3$} \\ 3s & \quad \text{if $s \geq 4$}\end{cases}$$
   Thus, under the conditions enumerated above, it follows from Theorem~\ref{thm-2t1} that equality holds in~\eqref{eq-adwall}; i.e., we have that $\on{AD}_{(1)}^3(y^2 - x^s) = \on{WL}_p(\scr{L},W)$.

   Note that Theorem~\ref{thm-2t1} should \emph{not} be interpreted as saying that the number of \emph{flexes} limiting toward an ICIS cut out analytically-locally by $y^2 - x^s = 0$ in a $1$-parameter family of plane curves is given by $3s$ for $s \geq 4$. Indeed, by computing the relevant Widland-Lax multiplicity, one can show that the number of limiting flexes is $4(s-1)$ for $s \geq 4$. Thus, equality in~\eqref{eq-adwall} fails to hold for the linear system of lines in the plane.
\end{remark}
\begin{proof}[Proof of Theorem~\ref{thm-2t1}]
  The case where $s = 2$ follows from Theorem~\ref{thm-main2}.  For the remainder of the proof, we take $s \geq 3$. It suffices to show that $$\on{SD}_{(1)}^3(r \cdot (y^2 - x^s)) = \begin{cases} 8 & \quad \text{if $s = 3$} \\ 3s & \quad \text{if $s \geq 4$}\end{cases}$$
  for any $r \in R^\times$. Put $f = r \cdot (y^2 - x^s)$. Using notation from the proof of Theorem~\ref{thm-2t} (and redefining $\tau_1, \tau_2, \tau_3 \in \on{SP}^4(f)$ to be their residues in $\on{SP}^3(f)$), the degeneracy ideal is
\begin{equation} \label{eq-listgens2}
I_{\ol{\tau}}^3 = (f, \Xi_1) =  (y^2 - x^s, [y^3] + [x^{2s-2}]),
\end{equation}
and the desired automatic degeneracy is simply the colength of this ideal. As in the proof of Theorem~\ref{thm-2t}, to compute the colength, it suffices to find a finite collection of monomials in $R$ the mod-$I_{\ol{\tau}}^3$ residues of which form a basis of $R/I_{\ol{\tau}}^3$.

  For convenience, let $V$ denote the $k$-vector space $R/I_{\ol{\tau}}^3$. To begin with, notice that the relation $y^2 - x^s =0$ implies that the mod-$I_{\ol{\tau}}^3$ residues of the monomials $x^iy^j$ for $i \geq 0$ and $j \in \{0,1\}$ span $V$. Further observe that the elements of $I_{\ol{\tau}}^3$ do not impose any relations on the monomials $x^iy^j$ for $i \in \{0, \dots, s-1\}$ and $j \in \{0, 1\}$, so these $2s$ monomials are linearly independent; let $V_1 \subset V$ be the $k$-vector subspace spanned by these monomials.

  Next, notice that the relations $y^2 - x^s = 0$ and $[y^3] + [x^{2s-2}] =0$ together imply that we have the linearly independent relations $[yx^s] + [x^{2s-2}] =0$ and $[yx^{2s-2}] + [x^{2s}] =0$.
Now, upon combining these relations,
we deduce that $yx^{s+2} = x^{2s} = 0$ when $s \geq 4$ and $yx^{2s-2} = yx^4 = 0$ when $s = 3$. Thus, when $s \geq 4$, the mod-$I_{\ol{\tau}}^3$ residues of the monomials $x^i$ for $i \in \{s, \dots, 2s-1\}$ span $V/V_1$. When $s = 3$, however, the relation $[yx^s] + [x^{2s-2}] = [yx^3] + [x^4] = 0$ further implies that $[yx^4] + [x^5] = 0$. It follows that $x^5 = 0$, meaning that the mod-$I_{\ol{\tau}}^3$ residues of the monomials $x^i$ for $i \in \{3, \dots, 4\}$ span $V/V_1$ when $s = 3$. To prove the theorem, it therefore suffices to show that $x^{2s-1} \neq 0$ as elements of $R/I_{\ol{\tau}}^3$ for $s \geq 4$ and that $x^4 \neq 0$ as elements of $R/I_{\ol{\tau}}^3$ for $s = 3$.

Let $S$ be the set of monomials given by
$$S= \{x^{2s-1},xy^3,x^{s+1}y,y^2x^{s-1}\},$$
let $W_1$ be the $k$-vector space given by
\begin{equation} \label{eq-defw1the4th}
W_1 = \vspan_k(y^2x^{s-1} - x^{2s-1},xy^3-x^{s+1}y, [x^{2s-1}]+[xy^3]) \subset I_{\ol{\tau}}^3,
\end{equation}
and notice that $W_1 \subset \vspan_k(S)$. By the same sort of argument that we used in the proof of Theorem~\ref{thm-23ab}, there is a $k$-vector subspace $W_2 \subset I_{\ol{\tau}}^3$ such that $W_2 \subset \vspan_k(\{x^iy^j : i,j \geq 0\} \setminus S)$ and such that we have the decomposition $I_{\ol{\tau}}^3 = W_1 \oplus W_2$. Thus, to show that $x^{2s-1} \neq 0$, it suffices to show that $x^{2s-1} \not\in W_1$. It follows by inspecting the right-hand side of~\eqref{eq-defw1the4th} that no $k$-linear combination of the generators has the property that it is supported only on $x^{2s-1}$ and on no other monomials, as desired.

The case $s = 3$ admits an analogous proof.
\end{proof}

\subsection{Expected Values of Various Planar Automatic Degeneracies} \label{sec-probcorrect}

Thus far, we have obtained two types of automatic degeneracy formulas for planar singularities: Theorems~\ref{thm-main2} and~\ref{thm-cuspinf} give formulas for $m^{\mathrm{th}}$-order automatic degeneracies as functions of $m$ for a fixed singularity germ $f$, and Theorems~\ref{thm-autodeg0},~\ref{thm-autodeg112ismiln},~\ref{thm-23ab}, and~\ref{thm-2t} give formulas for automatic degeneracies of fixed order $m \leq 4$ as $f$ varies in a natural family of singularity germs. It is computationally challenging to obtain similar formulas of the first type for singularities other than the node and the cusp, as well as formulas of the second type for larger values of $m$ and for other families of singularities. Nonetheless, if we fix $m$ and $f$ and make certain simplifying modifications, it is possible to use a computer algebra system such as {\tt Macaulay2} or {\tt Singular} to produce ``expected values'' of $\on{AD}_{(1)}^m(f)$, $\on{AD}_{(2)}^m(f)$, and $\on{AD}_{(1,1)}^m(f)$. In this section, we present two different strategies for computing such expected values, and we apply the strategies to obtain a table of these values for various choices of $m$ and $f$.

{\bf Strategy I} (using {\tt Macaulay2}):
Let $m \geq 2$. Since computer algebra systems like {\tt Macaulay2} and {\tt Singular} tend to produce unreliable results for computations done over ``inexact fields,'' like the field $\BC$ of complex numbers, we take $k = \BQ$ or $k = \BF_p$ for a large prime $p$. Next, because {\tt Macaulay2} is better suited to handle computations involving homogeneous polynomials, rather than power series, we need to assume that the planar singularity germ $f \in R$ is actually an element of the polynomial ring $k[x,y] \subset R$; let $\wt{f} \in S \coloneqq k[[x,y,z]]$ be the homogenization of $f$. Now, instead of working with the $R$-module $\on{SP}^m(f)$, we work with the $S$-module $\wt{\on{SP}}^m(f)$ defined by
$$\wt{\on{SP}}^m(f) \coloneqq S[a,b]\bigg/ \left(\sum_{d = 1}^{m-1} \sum_{s = 1}^d \frac{1}{s!}\frac{1}{(d-s)!} \frac{\d^d \wt{f}}{\d x^s \d y^{(d-s)}} a^s b^{d-s},\,(a,b)^m\right)$$
Just as we found for the $R$-module $\on{SP}^m(f)$ in \S~\ref{sec-exppres}, the $S$-module $\wt{\on{SP}}^m(f)$ admits a presentation by free $S$-modules: indeed, in direct analogy with~\eqref{eq-trialexact}, we have a short exact sequence
\begin{equation}\label{eq-trialexact2}
\begin{tikzcd}
S^{\kappa_{m-1}} \arrow{r}{\phi} & S^{\kappa_m} \arrow{r} & \wt{\on{SP}}^m(f) \arrow{r} & 0
\end{tikzcd}
\end{equation}
where $S^{\kappa_m}$ has a basis given by the monomials $a^ib^j$ for $0 \leq i +j \leq m-1$, where $S^{\kappa_{m-1}}$ has a basis given by the relations
\begin{equation} \label{secondsforrelations}
\sum_{d = 1}^{m-1-(i+j)} \sum_{s = 0}^d \frac{1}{s!}\frac{1}{(d-s)!} \frac{\d^d \wt{f}}{\d x^s \d y^{(d-s)}} a^{s+i} b^{d-s+j} = 0
\end{equation}
for $0 \leq i+j \leq m-2$, and where the map $\phi$ simply expresses the relations~\eqref{secondsforrelations} in terms of the monomials $a^ib^j$. The following {\tt Macaulay2} function, called \texttt{presentationSk}, takes as input the homogeneous polynomial $\wt{f}$ and the integer $m$. As output, it returns a pair {\tt (mons, M)} where {\tt mons} is the list
$$(a^{m-1}, a^{m-2}b, \dots, b^{m-1}, a^{m-2}, a^{m-3}b, \dots, b^{m-2}, \dots, a, b, 1)$$
and {\tt M} is the matrix for the map $\phi$ with respect to the above-defined bases of $S^{\kappa_{m-1}}$ and $S^{\kappa_m}$, which we arrange in accordance with the ordering provided by {\tt mons} (i.e., the row-$i$, column-$j$ entry of the matrix {\tt M} is given by the coefficient of the $i^{\mathrm{th}}$ entry of {\tt mons} in the relation obtained by multiplying the $(i+m)^{\mathrm{th}}$ element of {\tt mons} by the relation in~\eqref{secondsforrelations} where $i = j = 0$).

\begin{lstlisting}
presentationSk:=(f,m)->(
    T = S[a,b]/(ideal(a,b))^m;
    B := flatten entries basis(T);
    DM := matrix {apply(B,k->sub(k,{a=>x,b=>y}))};
    Df1 := diff(DM,f);
    E := flatten apply(B,k->exponents k);
    Df := matrix {apply(length E,i->1/(((E_i)_0)!*((E_i)_1)!)*(flatten entries Df1)_i)};
    gn := matrix{{((sub(Df,T))*(transpose matrix {B}))_(0,0)-sub(f,T)}};
    (mons,M) := coefficients(super basis image gn);
    tg := S^(append(apply(flatten entries mons,k->(-(first degree k))),0));
    sc := S^(numcols M);
    A := matrix {apply(numcols M,i->0)};
    mp := sub(M||A,S);
    (append(flatten entries mons,1),map(tg,sc,mp))
    )
    \end{lstlisting}

Next, we explain how to encode the map $\on{can}_{\on{ev}} \colon \wt{\on{SP}}^m(f) \to \wt{\on{SP}}^m(f)^{\vee\vee}$ in {\tt Macaulay2}; i.e., we demonstrate how to write down a matrix representing the map of free $S$-modules $S^{\kappa_m} \to \wt{\on{SP}}^m(f)^{\vee\vee}$ induced by $\on{can}_{\on{ev}}$. To begin with, we need to better understand the module $\wt{\on{SP}}^m(f)^\vee$. We claim that $\wt{\on{SP}}^m(f)^\vee$ is in fact a free $S$-module of rank $m$; this claim can be proven by simply constructing a basis of $\wt{\on{SP}}^m(f)^\vee$ using a modification of the algorithm in \S~\ref{sec-algae}, where all instances of $f$ and its partial derivatives are replaced with the corresponding homogenizations. Now, notice that $\wt{\on{SP}}^m(f)^\vee = \ker \phi^\vee$, so we can think of $\wt{\on{SP}}^m(f)^\vee$ as being the $S$-module of relations (i.e., syzygies) on the columns of a matrix defining $\phi^\vee$. Thus, we can define $\wt{\on{SP}}^m(f)^\vee$ in {\tt Macaulay2} using the command {\tt A := syz transpose(M)}, which outputs a matrix {\tt A} embedding $\wt{\on{SP}}^m(f)^\vee$ as a submodule of $(S^{\kappa_{m}})^\vee$, using the dual basis of {\tt mons} as the basis of $(S^{\kappa_m})^\vee$. Then the matrix {\tt At} defined by the command {\tt At := transpose A} represents the map $(S^{\kappa_m})^{\vee\vee} \to \wt{\on{SP}}^m(f)^{\vee\vee}$, using the double dual basis of {\tt mons} as the basis of $(S^{\kappa_m})^{\vee\vee}$. Since the map $\on{can}_{\on{ev}} \colon S^{\kappa_m} \to (S^{\kappa_m})^{\vee\vee}$ is represented by the identity matrix when the basis of $(S^{\kappa_m})^{\vee\vee}$ is chosen to be the double dual of the basis of $S^{\kappa_m}$, and because of the commutativity of the diagram

\begin{center}
\begin{tikzcd}
S^{\kappa_m} \arrow{r} \arrow[swap]{d}{\on{can}_{\on{ev}}} &  \wt{\on{SP}}^m(f) \arrow{d}{\on{can}_{\on{ev}}} \\
(S^{\kappa_m})^{\vee\vee} \arrow{r} & \wt{\on{SP}}^m(f)^{\vee\vee}
\end{tikzcd}
\end{center}

\noindent it follows that the desired map $S^{\kappa_m} \to \wt{\on{SP}}^m(f)^{\vee\vee}$ is represented by the matrix {\tt At}.

To compute automatic degeneracies, we need to produce $m + \{0,\pm 1\}$ general elements of $\wt{\on{SP}}^m(f)$. We simulate this by generating $m+ \{0,\pm 1\}$ elements of $S^{\kappa_m}$ whose components are homogeneous polynomials of some fixed degree $d \geq 1$ in the variables $x$, $y$, $z$ with pseudo-random coefficients in the field $k$. In the following {\tt Macaulay2} function, called {\tt parameterMatrix}, we encode these $m+ \{0,\pm 1\}$ elements as column vectors of a matrix {\tt psi} (the code is written for $m -1$ elements, but it can be easily modified to handle the case of $m+\{0,1\}$ elements).\footnote{As it happens, we expect that it suffices to take $d = 1$, but ideally, we want to take $d$ to be as large as we can without causing the {\tt Macaulay2} code for the automatic degeneracy computation to take forever to halt.} Then an $m\times(m + \{0,\pm 1\})$ matrix expressing the images of these $(m + \{0,\pm 1\})$ elements in $\wt{\on{SP}}^m(f)$ is just the product matrix {\tt At*psi}. The function \texttt{parameterMatrix}, defined below, accepts $m$, $\wt{f}$, and $d$ as input and produces the matrix {\tt At*psi} as output.

\begin{lstlisting}
randPolyList:=D->(
    flatten apply(D,d->(
	    flatten entries ((matrix{{x,y}})*random(S^{2:d-1},S^1))
	    ))
    )
parameterMatrix:=(f,m,d)->(
    (mons,M) := presentationSk(f,m);
    T := ring(first mons);
    mons1 := drop(mons,-1);
    hv := diagonalMatrix(append(apply(mons1,k->z^(m-1-first degree k)),z^(m-1)));
  DM=matrix {append(apply(mons1,k->sub(k,{a=>x,b=>y})),1)};
    --L=randPolyList(D);
    --G=matrix apply(L,g->flatten entries diff(DM,g));
    G := hv*random(S^{length mons:d},S^(m-1));
    tg := S^(append(apply(mons1,k->(-(first degree k))),0));
    sc := S^(numcols G);
    psi := map(tg,sc,G);
    A := syz transpose(M);
    At := transpose A;
    At*psi
    )
\end{lstlisting}

All that remains is to obtain the ideal of maximal minors of the matrix {\tt At*psi}, add $f$ to the ideal for the weight-$1$ case, and compute the support at the prime ideal $(x,y) \subset S$ of the scheme that it cuts out. This is accomplished via the following code, where for example we take $k = \BF_{7919}$, $m = 4$, $f = y^2 - x^3$ in the weight-$2$ type-(a) case:

\begin{lstlisting}
S := (ZZ/7919)[x,y,z];
m := 4;
f := z*y^2-x^3;
I := saturate(minors(m-1,parameterMatrix(f,m,3)),z);
L := associatedPrimes(I);
for P in L do(
    if P!=ideal(x,y) then(
	I=saturate(I,P)
	))
hilbertPolynomial(I)
\end{lstlisting}

\FloatBarrier

\begin{table}
\centering
\renewcommand{\arraystretch}{1.5}
\begin{tabular}{c|c|c|c|c|c|c|c|c}
$f$ & Type & $\mu_f$ & $\on{AD}_{(2)}^3(f)$ & $\on{AD}_{(2)}^4(f)$ & $\on{AD}_{(1,1)}^2(f)$ & $\on{AD}_{(1,1)}^3(f)$ & $\on{AD}_{(1)}^2(f)$ & $\on{AD}_{(1)}^3(f)$ \\
\hline
\hline
        $xy$ &$A_1$ & 1 & 1 & 5 & 1 & 5 & 2 & 6 \\
 $y^2-x^3$ &$A_2$ & 2 & 2 & 10 & 2 & 10 & 3 & 8 \\
$y^2-x^4$ &$A_3$ & 3 &  3 & 18 & 3 & 15 & 4 & 12 \\
$y^2-x^5$ &$A_4$ & 4 &  4 & 24 & 4 & 20 & 5 & 15 \\
 $y^2-x^6$ &$A_5$ & 5 & 5 & 30 & 5 & 25 & 6 & 18 \\
\hline
$y^3-x^2y$ &$D_4$ & 4 &  6 & 29 & 4 & 24 & 6 & 18 \\
 $y^4-x^2y$ &$D_5$ & 5 & 7 & 36 & 5 & 29 & 7 & 20 \\
 $y^5-x^2y$ &$D_6$ & 6 & 8 & 45 & 6 & 34 & 8 & 24\\
\hline
$y^3-x^4$ &$E_6$ & 6 &  9 & 44 & 6 & 36 & 8 & 22 \\
\hline
$y^3x-x^3$ &$E_7$ & 7 &  10 & 53 & 7 & 41 & 9 & 26\\
\hline
$y^3-x^5$ &$E_8$ & 8 &  12 & 62 & 8 & 48 & 10 & 29\\
\end{tabular}
\caption{Here we present expected values of the three types of automatic degeneracy for various orders $m$ and planar ICIS germs $f$. The weight-$2$ type-(a) values were computed using Strategy I with $k = \mathbb{F}_{7919}$ and $d=1$, whereas the weight-$2$ type-(b) and weight-$1$ values were computed using Strategy II with $k = \BQ$. For convenience, we indicate the ADE type and Milnor number $\mu_f$ for each $f$. Note that the results of the preceding sections agree with the corresponding values in the table above.}\label{t:1}
\end{table}

{\bf Strategy II} (using {\tt Singular}): The main advantage of Strategy I is that we never needed to explicitly compute a basis of $\wt{\on{SP}}^m(f)^\vee$, as this task was left to the computer. On the other hand, the main disadvantage of Strategy I is that the code has a very long runtime, even for small values of $m$, because so much of the calculation is left to the computer and because {\tt Macaulay2} is not optimized for local calculations. The purpose of Strategy II is to provide a quick way of computing expected automatic degeneracy values when a basis of $\on{SP}^m(f)^\vee$ is already known.

    As in Strategy I, let $m \geq 2$, let $f \in R$ be a planar singularity germ, and let $k = \BQ$ or $k = \BF_p$ for a large prime $p$. We can manually apply the algorithm in \S~\ref{sec-algae} to obtain a basis $(\wt{\theta}_1, \dots, \wt{\theta}_m)$ of $\on{SP}^m(f)$. This is feasible for small values of $m$, as we showed for $m \leq 4$ in \S~\ref{sec-expibasis}, but becomes computationally challenging for larger values of $m$.\footnote{Note that this step would be rendered considerably easier if the algorithm in \S~\ref{sec-algae} could be implemented as a computer program, although it is not clear whether this is possible.}
    Taking $n = m \pm 1$, we can use {\tt Singular}, or any other computer algebra system, to generate random numbers $\alpha_{ij}^{(\ell)} \in k$ for each $\ell \in \{1, \dots, n \}$ and each pair $(i,j)$ with $i+j \in \{0, \dots, m-1\}$, and consider the elements $\tau_\ell = \sum_{0 \leq i+j \leq m-1} \alpha_{ij}^{(\ell)} \cdot a^ib^j \in \on{SP}^m(f)$. Again using any computer algebra system, we can compute the maximal minors of the matrix
$$\left[\begin{array}{ccc} \wt{\theta}_{1}(\tau_{1}) & \cdots & \wt{\theta}_{1}(\tau_{n}) \\ \vdots & \ddots & \vdots \\ \wt{\theta}_{m}(\tau_{1}) & \cdots & \wt{\theta}_{m}(\tau_{n})  \end{array} \right]$$
Given these minors, call them $\Xi_1, \dots, \Xi_m$, the ideal they generate in the local ring $k[x,y]_{(x,y)}$ can be defined in {\tt Singular} via the code {\tt ideal i := }$\Xi_1$, \dots, $\Xi_m$;. The colength of the ideal {\tt i} can then be computed via the code {\tt ideal j = std(i); vdim(j);}.

Finally, by applying Strategy I or Strategy II to various pairs $(f,m)$ (and running the code repeatedly for each $(f,m)$ to ensure that the same automatic degeneracy value is obtained for several choices of the $m+ \{0, \pm 1\}$ elements of $\wt{\on{SP}}^m(f)$ in Strategy I or of $\on{SP}^m(f)$ in Strategy II, we obtain the expected automatic degeneracy values displayed in Table~\ref{t:1}.

\begin{remark} \label{rem-notequi}
Using either Strategy I or Strategy II, one obtains the following expected values of automatic degeneracies for various $[m : n] \in \BP_k^1 \setminus \{[1:\zeta] : \zeta^4 = 1\}$: $$\on{AD}_{(2)}^4(y^4 - x^4) = 77 > 75 = \on{AD}_{(2)}^4((y^3+xy^2+x^2y+x^3)(m\cdot x-n \cdot y)).$$
Note that the singularities defined by $(y^3+xy^2+x^2y+x^3)(m\cdot x-n \cdot y) = 0$ for $[m : n] \in \BP_k^1 \setminus \{[1:\zeta] : \zeta^4 = 1, \zeta \neq 1\}$ form an \emph{equisingular} family. Thus, it appears that $4^{\mathrm{th}}$-order weight-$2$ type-(a) automatic degeneracy is \emph{not} an equisingularity invariant. It remains open to determine whether any of the countably many different types of automatic degeneracy is an equisingularity invariant.
\end{remark}

\subsection{An Example in the Non-Planar Case} \label{sec-nonplane}

In this section, we compute the $2^{\mathrm{nd}}$-order type-(b) automatic degeneracy of the non-planar ICIS with coordinate ring $$A \coloneqq k[[x,y,z]]/I, \quad \text{where} \quad I = (xy-z^2,yz-x^2).$$ One readily checks that this singularity is an ICIS and is given by the union of four distinct lines (specifically, the lines defined by the ideal $(x,z)$ and the three ideals of the form $(x - \zeta \cdot z, y - \tfrac{1}{\zeta}\cdot z)$, where $\zeta \in k$ runs through the three cube roots of unity) that are concurrent at the origin.

We first need to compute the versal deformation space of the singularity. For a $k$-algebra $S$, let $\Omega_{S/k}^1$ denote the $S$-module of relative differentials of $S$ over $k$. Recall that the first-order deformations of the singularity are in bijective correspondence with the elements of $\on{Ext}^1_A(\Omega_{A/k}^1, A)$ viewed as a $k$-vector space. By dualizing the conormal exact sequence for the inclusion of $\Spec A$ as a closed subscheme of $\Spec k[[x,y,z]]$, we obtain the exact sequence of $A$-modules
\begin{center}
    \begin{tikzcd}
        (\Omega_{A/k}^1)^\vee \arrow{r} & (\Omega_{k[[x,y,z]]/k}^1 \otimes_{k[[x,y,z]]} A)^\vee \arrow{r} & (I/I^2)^\vee \arrow{r} & \on{Ext}_A^1(\Omega_{A/k}^1, A) \arrow{r} & 0
    \end{tikzcd}
\end{center}
As a $k[[x,y,z]]$-module, $\Omega_{k[[x,y,z]]/k}^1$ is free of rank $3$, generated by the differentials $dx$, $dy$, $dz$, so we have that $\Omega_{k[[x,y,z]]/k}^1 \otimes_{k[[x,y,z]]} A \simeq A \cdot dx \oplus A \cdot dy \oplus A \cdot dz$. Moreover, $I/I^2$ is a free $A$-module of rank $2$, generated by the residues of $xy-z^2$ and $yz-x^2$ modulo $I^2$. Thus, the map $I/I^2 \to \Omega_{k[[x,y,z]]/k}^1 \otimes_{k[[x,y,z]]} A$ taking an element of $I$ to its differential is given with respect to the generators specified above by the matrix
$$M = \left[\begin{array}{cc}y & -2x \\x & z \\-2z & y \end{array}\right]$$
It follows that $\on{Ext}_A^1(\Omega_{A/k}^1, A) = \on{coker}(M^T)$, and a calculation reveals that $\on{coker}(M^T)$, viewed as a quotient of $A^{\oplus 2}$, is a $5$-dimensional $k$-vector space with a basis consisting of the elements $(1,0)$, $(0,1)$, $(x,0)$, $(0,x)$, and $(z,z)$. We conclude that the versal deformation space of the singularity is given by the family
\begin{equation} \label{eq-egfamily}
    \begin{tikzcd}
\on{Spf} k[[x,y,z]][[t_1, t_2, t_3, t_4, t_5]]/(xy - z^2 + t_1 + t_3 \cdot x + t_5 \cdot z, yz - x^2 + t_2 + t_4 \cdot x + t_5 \cdot z) \arrow{d} \\ \on{Spf} k[[t_1, t_2, t_3, t_4, t_5]]
\end{tikzcd}
\end{equation}
We now restrict the family of formal schemes in~\eqref{eq-egfamily} to a general $2$-dimensional linear subscheme of the base; i.e., we pull the family back along a map
$$\on{Spf} k[[s,t]] \hookrightarrow \on{Spf} k[[t_1, t_2, t_3, t_4, t_5]], \quad t_i \mapsto a_i \cdot s + b_i \cdot t,$$
where $(a_1, \dots, a_5, b_1, \dots, b_5) \in \BA_k^{10}(k)$ is general. The pulled-back family is given by
\begin{equation} \label{eq-egfamily2}
    \begin{tikzcd}
\on{Spf} k[[x,y,z]][[s,t]]/\big(xy - z^2 + q_1 \cdot s+ r_1 \cdot t, yz - x^2 + q_2 \cdot s + r_2 \cdot t\big) \arrow{d} \\ \on{Spf} k[[s,t]]
\end{tikzcd}
\end{equation}
where for the sake of brevity we put
\begin{align*}
    & q_1  = a_1 + a_3 \cdot x + a_5 \cdot z, \quad   r_1  = b_1 + b_3 \cdot x + b_5 \cdot z, \\
     &   q_2 = a_2 + a_4 \cdot x + a_5 \cdot z, \quad r_2 = b_2 + b_4 \cdot x + b_5 \cdot z.
\end{align*}
The completion at the origin of the local ring of the total space of the family in~\eqref{eq-egfamily2} is given by $D \coloneqq k[[x,y,z]]$. This is necessarily true because the family in~\eqref{eq-egfamily2} is regular, but it can also be easily deduced by observing that the pair of relations
\begin{equation} \label{eq-ressolve}
xy - z^2 + q_1 \cdot s+ r_1 \cdot t = 0 \quad \text{and} \quad  yz - x^2 + q_2 \cdot s + r_2 \cdot t = 0
\end{equation}
in $D[[s,t]]$ can be used to solve for the variables $s,t$ when $(a_1, \dots, a_5, b_1, \dots, b_5) \in \BA_k^{10}(k)$ is general. Indeed, solving the relations in~\eqref{eq-ressolve} for $s,t$ yields
\begin{align}
    s & = f(x,y,z) \coloneqq (q_1r_2 - q_2r_1)^{-1} \cdot (r_2 \cdot (xy - z^2) - r_1\cdot (yz - x^2)), \label{ressolns} \\
    t & = g(x,y,z) \coloneqq -(q_1r_2 - q_2r_1)^{-1} \cdot (q_2 \cdot (xy - z^2) - q_1\cdot (yz - x^2)). \label{ressolnt}
\end{align}
Let $\wt{P}^m(xy-z^2,yz-x^2)$ denote the completion of the stalk of the sheaf of $m^{\mathrm{th}}$-order relative principal parts associated to the family~\eqref{eq-egfamily2} for any integer $m \geq 1$. As we did for the case of a node in the proof of Theorem~\ref{thm-main2}, substituting the family~\eqref{eq-egfamily2} into the definition~\eqref{eq-defmodparts} of the module of principal parts yields that $\wt{P}^2(xy-z^2,yz-x^2)$ is given (as a $D$-module) by
\begin{align}
& \wt{P}^2(xy-z^2,yz-x^2) = \nonumber \\ & D[[s,t]][[u,v,w]]/\big((xy-z^2) + q_1(x,y,z) \cdot s + r_1(x,y,z) \cdot t , (uv-w^2) + q_1(u,v,w) \cdot s + r_1(u,v,w) \cdot t, \nonumber\\
& \hphantom{D[[s,t]][[u,v,w]]/\big(} (yz-x^2) + q_2(x,y,z) \cdot s + r_2(x,y,z) \cdot t , (vw-u^2) + q_2(u,v,w) \cdot s + r_2(u,v,w) \cdot t,\nonumber \\
&  \hphantom{D[[s,t]][[u,v,w]]/\big(} (u-x,v-y,w-z)^2\big),
\nonumber
\end{align}
which by~\eqref{ressolns} and~\eqref{ressolnt} can be rewritten as
\begin{align}
& D[[u,v,w]]/\big(f(u,v,w) - f(x,y,z),  g(u,v,w) - g(x,y,z), (u-x,v-y,w-z)^2\big) = \nonumber \\
& D[[a,b,c]]/\big(f_1 \cdot a + f_2 \cdot b + f_3 \cdot c, g_1 \cdot a + g_2 \cdot b + g_3 \cdot c, (a,b,c)^2\big), \label{eq-easyD}
\end{align}
where $a = u-x$, $b = v-y$, $c = w - z$, and where again for the sake of brevity we put
\scriptsize
\begin{align}
    f_1 & = \frac{2r_1\cdot x+r_2\cdot y+b_3\cdot (x^2-yz)+b_4\cdot (xy-z^2)}{q_1r_2-q_2r_1} + \frac{(a_4r_1-a_3r_2-b_4q_1+b_3q_2)(r_1 \cdot (x^2 - yz) + r_2\cdot (xy-z^2))}{(q_1r_2-q_2r_1)^2}, \label{eq-theeffs}\\
    f_2 & = \frac{r_2 \cdot x - r_1 \cdot z}{q_1r_2-q_2r_1},\nonumber\\
    f_3 & = \frac{-r_1\cdot y-2r_2\cdot z+b_5\cdot (x^2-yz)+b_5\cdot (xy-z^2)}{q_1r_2-q_2r_1} + \frac{(a_5r_1-a_5r_2-b_5q_1+b_5q_2)(r_1 \cdot (x^2 - yz) + r_2\cdot (xy-z^2))}{(q_1r_2-q_2r_1)^2},\nonumber\\
    g_1 & = -\frac{2q_1\cdot x+q_2\cdot y+a_3\cdot (x^2-yz)+a_4\cdot (xy-z^2)}{q_1r_2-q_2r_1} - \frac{(a_4r_1-a_3r_2-b_4q_1+b_3q_2)(q_1 \cdot (x^2 - yz) + q_2\cdot (xy-z^2))}{(q_1r_2-q_2r_1)^2}, \nonumber\\
    g_2 & = -\frac{q_2 \cdot x - q_1 \cdot z}{q_1r_2-q_2r_1},\nonumber\\
    g_3 & = -\frac{-q_1\cdot y-2q_2\cdot z+a_5\cdot (x^2-yz)+a_5\cdot (xy-z^2)}{q_1r_2-q_2r_1} - \frac{(a_5r_1-a_5r_2-b_5q_1+b_5q_2)(q_1 \cdot (x^2 - yz) + q_2\cdot (xy-z^2))}{(q_1r_2-q_2r_1)^2}.\nonumber
\end{align}
\normalsize
The next step is to find a basis for $\wt{P}^2(xy-z^2,yz-x^2)^\vee$, which is a free $D$-module of rank $2$. It follows from the form of $\wt{P}^2(xy-z^2,yz-x^2)$ obtained in~\eqref{eq-easyD} that any functional on $\wt{P}^2(xy-z^2,yz-x^2)$ is defined by where it sends the generators $1$, $a$, $b$, $c$. Recall that the map $\wt{P}^1(f,g)^\vee \hookrightarrow \wt{P}^2(xy-z^2,yz-x^2)^\vee$ is the inclusion of a free rank-$1$ summand, generated by the functional $e_1 \in \wt{P}^2(xy-z^2,yz-x^2)^\vee$ defined by $e_1(1) = 1$ and $e_1(a) = e_1(b) = e_1(c) = 0$. Thus, it remains to find $e_2 \in \wt{P}^2(xy-z^2,yz-x^2)^\vee$ such that $e_2(1) = 0$ and $(e_1, e_2)$ is a basis of $\wt{P}^2(xy-z^2,yz-x^2)^\vee$. From the form of $\wt{P}^2(xy-z^2,yz-x^2)$ obtained in~\eqref{eq-easyD}, we have that $e_2(a)$, $e_2(b)$, $e_2(c)$ must satisfy the following two relations:
\begin{align}
f_1 \cdot e_2(a)+ f_2 \cdot e_2(b) + f_3 \cdot e_2(c) & = 0, \label{eq-easyD2} \\
g_1 \cdot e_2(a)+ g_2 \cdot e_2(b) + g_3 \cdot e_2(c) & = 0. \label{eq-easyD3}
\end{align}
By combining~\eqref{eq-easyD2} and~\eqref{eq-easyD3} in two different ways to eliminate $e_2(c)$ or $e_2(a)$, we respectively obtain the following two relations:
\begin{align}
(f_1g_3-f_3g_1) \cdot e_2(a)+(f_2g_3-f_3g_2) \cdot e_2(b) & = 0, \label{eq-easyD4} \\
(f_1g_2-f_2g_1 )\cdot e_2(b) + ( f_1g_3-f_3g_1  )\cdot e_2(c) & = 0 \label{eq-easyD5}
\end{align}
We claim that for a general choice of $(a_1,\dots, a_5, b_1, \dots, b_5) \in \BA_k^{10}(k)$, the three distinct coefficients
$$f_1g_3-f_3g_1,\,f_2g_3-f_3g_2,\, f_1g_2-f_2g_1$$
that appear in~\eqref{eq-easyD4} and~\eqref{eq-easyD5} are pairwise coprime elements of the unique factorization domain $D$. This claim follows from the observation that the lowest-degree components of these three coefficients, which are given explicitly by
\begin{align*}
    f_1g_3-f_3g_1 & \equiv (q_1r_2-q_2r_1)^{-1}\cdot (y^2-4xz) \pmod{(x,y,z)^3},\\
    f_2g_3-f_3g_2 & \equiv (q_1r_2-q_2r_1)^{-1}\cdot (xy+2z^2) \pmod{(x,y,z)^3},\\
    f_1g_2-f_2g_1 & \equiv (q_1r_2-q_2r_1)^{-1}\cdot (yz+2x^2) \pmod{(x,y,z)^3}
\end{align*}
are irreducible and generate pairwise distinct ideals of $D$. Thus,~\eqref{eq-easyD4} and~\eqref{eq-easyD5} together imply that the set of functionals $\phi \in \wt{P}^2(xy-z^2,yz-x^2)^\vee$ with $\phi(1) = 0$ is precisely the same as the set of functionals $\phi_\gamma \in \wt{P}^2(xy-z^2,yz-x^2)^\vee$ defined by $\phi_\gamma(1) = 0$ and
\begin{align*}
    \phi_\gamma(a) & = (f_2g_3-f_3g_2) \cdot \gamma,\\
    \phi_\gamma(b) & = -(f_1g_3-f_3g_1) \cdot \gamma,\\
    \phi_\gamma(c) & = (f_1g_2-f_2g_1)\cdot \gamma.
\end{align*}
for some $\gamma \in D$. Letting $\gamma = 1 \in D$, we can then take $e_2 = \phi_1$.

Now, given general elements $\tau_i = \alpha_{000}^{(i)} + \alpha_{100}^{(i)} \cdot a + \alpha_{010}^{(i)} \cdot b + \alpha_{001}^{(i)} \cdot c \in \wt{P}^2(xy-z^2,yz-x^2)$ for $i \in \{1, 2, 3\}$, the associated degeneracy scheme is cut out by the ideal $I_{\ol{\tau}}^2 \subset D$ of maximal minors of the matrix
$$M_{\ol{\tau}}^2 = \left[\begin{array}{ccc} e_1(\tau_1) & e_1(\tau_2) & e_1(\tau_3) \\ e_2(\tau_1) & e_2(\tau_2) & e_2(\tau_3) \end{array}\right]$$
It follows from our description of $e_1, e_2$ that
\begin{align*}
I_{\ol{\tau}}^2 & = \big(f_1g_3-f_3g_1,f_2g_3-f_3g_2,f_1g_2-f_2g_1\big).
\end{align*}
Notice that we have been working with a general $2$-parameter deformation of the singularity, so that we could exploit the fact that $D$ is a unique factorization domain. To compute automatic degeneracy, however, we need to further restrict to a general $1$-parameter deformation. To do this, we pull the family in~\eqref{eq-egfamily2} back along a map
$$\on{Spf} k[[t']] \to \on{Spf} k[[s,t]], \quad s \mapsto \rho_1 \cdot t', \, t \mapsto \rho_2 \cdot t'$$
where $(\rho_1, \rho_2) \in \BA_k^{2}(k)$ is general. The pulled-back family is given by
\begin{equation} \label{eq-egfamily3}
    \begin{tikzcd}
\on{Spf} k[x,y,z][[t']]/\big(xy - z^2 + (q_1\rho_1+r_1\rho_2) \cdot (xy-z^2) \cdot t', yz - x^2 + (q_2\rho_1+r_2\rho_2)\cdot t'\big) \arrow{d} \\ \on{Spf} k[[t']]
\end{tikzcd}
\end{equation}
It is clear that the completion at the origin of the local ring of the total space of the family in~\eqref{eq-egfamily3} is given by
$$D' \coloneqq D/\big((q_1\rho_1+r_1\rho_2)^{-1} \cdot (xy-z^2) - (q_2\rho_1+r_2\rho_2)^{-1} \cdot (yz-x^2) \big)$$
It then follows that the desired automatic degeneracy is given by
\begin{align}
& \on{AD}^2_{(1,1)}(xy-z^2,yz-x^2) = \dim_k D'/I_{\ol{\tau}}^2 = \nonumber \\
& \dim_k k[[x,y,z]]/\big(f_1g_3-f_3g_1,f_2g_3-f_3g_2,f_1g_2-f_2g_1,\\
& \hphantom{\dim_k k[[x,y,z]]/\big(,\,\,\,} (q_1\rho_1+r_1\rho_2)^{-1} \cdot (xy-z^2) - (q_2\rho_1+r_2\rho_2)^{-1} \cdot (yz-x^2)\big) \label{eq-extendzeideal} \end{align}
To compute this colength, we modify the formula in~\eqref{eq-anandlem} to suit the present case as follows: letting $\mathfrak{m} = (x,y,z) \subset D'$, we have
\begin{equation} \label{eq-anandlem2}
\dim_k D'/I_{\ol{\tau}}^2 = \sum_{\ell = 0}^\infty \dim_k ((D'/I_{\ol{\tau}}^2) \otimes_{D'} \mathfrak{m}^\ell)/((D'/I_{\ol{\tau}}^2) \otimes_{D'} \mathfrak{m}^{\ell+1}).
\end{equation}
It is evident from~\eqref{eq-extendzeideal} that none of the relations defining $D'/I_{\ol{\tau}}^2$ as a quotient of $k[[x,y,z]]$ are supported in degrees $0$ and $1$. Thus, the terms in the sum on the right-hand side of~\eqref{eq-anandlem2} corresponding to $\ell = 0,1$ are given as follows:
\begin{align}
    & \dim_k ((D'/I_{\ol{\tau}}^2) \otimes_{D'} \mathfrak{m}^0)/((D'/I_{\ol{\tau}}^2) \otimes_{D'} \mathfrak{m}^{1}) = 1, \label{eq-DM0}\\
    & \dim_k ((D'/I_{\ol{\tau}}^2) \otimes_{D'} \mathfrak{m}^1)/((D'/I_{\ol{\tau}}^2) \otimes_{D'} \mathfrak{m}^{2}) = 3. \label{eq-DM1}
\end{align}
For $\ell = 2$, we see that the relations defining $D'/I_{\ol{\tau}}^2$ as a quotient of $k[[x,y,z]]$ give rise to four relations of degree $2$, namely
\begin{equation} \label{eq-linearindeprels}
y^2 - 4xz = xy+2z^2 = 2x^2 + yz = r_1^{-1} \cdot (xy-z^2) - r_2^{-1} \cdot (yz-x^2) = 0.
\end{equation}
It is clear that the relations in~\eqref{eq-linearindeprels} are linearly independent, so we deduce that
\begin{equation} \label{eq-DM2}
    \dim_k ((D'/I_{\ol{\tau}}^2) \otimes_{D'} \mathfrak{m}^2)/((D'/I_{\ol{\tau}}^2) \otimes_{D'} \mathfrak{m}^{3}) = 2.
\end{equation}
Finally, it is not hard to check that the relations defining $D'/I_{\ol{\tau}}^2$ as a quotient of $k[[x,y,z]]$ give rise to twelve maximally independent relations of degree $3$. (An easy way to check this is to use {\tt Macaulay2} to compute the rank of the matrix whose columns express these twelve relations in terms of the basis of monomials of $\mathfrak{m}^3$.) Thus, we have that
\begin{equation} \label{eq-DM3}
\dim_k ((D'/I_{\ol{\tau}}^2) \otimes_{D'} \mathfrak{m}^3)/((D'/I_{\ol{\tau}}^2) \otimes_{D'} \mathfrak{m}^{4}) = 0.
\end{equation}
Note that~\eqref{eq-DM3} implies that the terms in the sum on the right-hand side of~\eqref{eq-anandlem2} corresponding to $\ell \geq 3$ are all equal to $0$. Combining this result with~\eqref{eq-DM0},~\eqref{eq-DM1}, and~\eqref{eq-DM2}, we obtain the following:
\begin{theorem} \label{thm-3dexample}
The $2^{\mathrm{nd}}$-order weight-$2$ type-(b) automatic degeneracy of the ICIS cut out analytically-locally by $xy - z^2 = yz - x^2 = 0$ is given by
$$\on{AD}_{(1,1)}^2(xy - z^2, yz - x^2) = 6.$$
In particular, the number of $2^{\mathrm{nd}}$-order weight-$2$ type-(b) inflection points limiting to this singularity in a general $2$-parameter deformation is equal to $0$.
\end{theorem}
\begin{proof}
We have already computed the automatic degeneracy, so it suffices to prove the second statement in the theorem. Applying the L\^{e}-Greuel formula for the Milnor number of an ICIS (see Theorem~\ref{def-milnnumber}), we find that $\mu_{xy-z^2,yz-x^2} = 5$ and that the total space of the family in~\eqref{eq-egfamily3} has Milnor number $1$. Thus, Proposition~\ref{prop-milnorsum} tells us that $\on{mult}_0 \Delta_{xy-z^2,yz-x^2} = 5 + 1 = 6$. It then follows from~\eqref{eq-fund2} that the desired number of limiting inflection points is given by $6 - 6 = 0$.
\end{proof}
\begin{remark}
  In the course of proving Theorem~\ref{thm-3dexample}, we saw that $\on{AD}_{(1,1)}^2(xy - z^2, yz - x^2) = \on{mult}_0 \Delta_{xy-z^2,yz-x^2}$. Note that for planar ICIS germs $f$, Theorem~\ref{thm-autodeg112ismiln} tells us that $\on{AD}_{(1,1)}^2(f) = \mu_f = \on{mult}_0 \Delta_f$. It remains open to determine whether $\on{AD}_{(1,1)}^2(f) = \on{mult}_0 \Delta_f$ for arbitrary ICIS germs $f$.
\end{remark}

\section{Other Enumerative Applications} \label{sec-eg}

In this section, we apply our results on automatic degeneracy to study three enumerative problems: counting hyperflexes in a pencil of plane curves, counting septactic points in a pencil of plane curves, and computing the classes of certain divisors of Weierstrass points on the moduli space of curves.

\subsection{Counting Hyperflexes in a Pencil of Plane Curves}\label{sec-countatlast}

We begin by summarizing two \emph{ad hoc} strategies that already exist in the literature for circumventing the failure of the sheaves of principal parts to be locally free at singular points. The first strategy applies only to the problem of counting hyperflexes in a pencil of plane curves, whereas the second strategy applies to solve a broader class of enumerative problems, including the hyperflex problem, concerning admissible families of curves whose singular points are no worse than nodal. We then present a third strategy for solving the hyperflex problem that utilizes our results on automatic degeneracy.

\subsubsection{Strategy I: The Universal Point-Line Incidence Variety} \label{sec-ptlininvar}

The standard approach to the problem of counting hyperflexes in a pencil of plane curves of degree $d$ is to count hyperflex \emph{point-line pairs}, rather than hyperflexes themselves. To this end, let
$$Z = \{(\text{point } p, \text{line } L) \in \BP_k^2 \times {\BP_k^2}^\vee \colon p \in L\}$$
be the universal incidence variety parametrizing pairs consisting of a point and a line containing the point. Let $\pi_1 \colon Z \to \BP_k^2$ be the projection map onto the ``point'' factor, and consider the sheaf
$$\scr{V} = \scr{P}_{Z/{\BP_k^2}^\vee}^4\big(\pi_1^* \scr{O}_{\BP_k^2}(d)\big),$$
where we view $Z$ as a family over ${\BP_k^2}^\vee$ via the projection map onto the ``line'' factor. Note that $\scr{V}$ is locally free because the map $Z \to {\BP_k^2}^\vee$ is smooth. Moreover, the fiber of $\scr{V}$ at a pair $(p,L) \in Z$ is
$$\scr{V}|_{(p,L)} = H^0(\scr{O}_L(d) \otimes \scr{O}_L/\scr{I}_p^4),$$
so since $\dim_k H^0(\scr{O}_L(d) \otimes \scr{O}_L/\scr{I}_p^4) = 4$, it follows that $\scr{V}$ is a vector bundle on $Z$ of rank $4$. It is not too hard to see that the number of hyperflexes on a general pencil of plane curves of degree $d$ is given by $\deg c_3(\scr{V})$ (see~\cite[\S~11.3.1]{harris3264} for a proof). To compute this Chern class, one must first determine the Chow ring of $Z$, which can be done by realizing $Z$ as the projectivization of the universal subbundle on ${\BP_k^2}^\vee$ (see~\cite[\S~11.3]{harris3264} for more details). A bit of calculation then yields that the number of hyperflexes is
\begin{equation} \label{eq-hyposform}
\deg c_3(\scr{V}) = 6(d-3)(3d-2).
\end{equation}
\begin{remark}
By shifting the focus from hyperflexes to hyperflex point-line pairs, the above procedure takes advantage of the fact that the projection map $Z \to {\BP_k^2}^\vee$ is smooth --- in other words, lines \emph{do not} degenerate. However, it is not clear how one might generalize this method to study other kinds of inflection points. For instance, if one were interested in counting septactic points in a pencil of plane curves of a given degree (an example that we study in \S~\ref{sec-hyperflexagain}), the analogous procedure would fail because conics \emph{do} degenerate, and so the corresponding universal family of point-conic pairs would fail to be smooth over the parameter space of conics.
\end{remark}

\subsubsection{Strategy II: The ``Hilbert Scheme of Nodal Curves''} \label{sec-ranwithit}

The second method, which was developed by Ran in~\cite{MR3078931}, works only when the singular fibers of the family are nodal, because it relies on specific properties of what Ran terms \emph{the Hilbert scheme of nodal curves}, which is defined to be the punctual flag Hilbert scheme parametrizing schemes of bounded length supported at individual points of the fibers of the family.

Let $X/B$ be an admissible family with the property that each singular fiber is nodal, and let $\scr{V}$ be a vector bundle on $X$. In this general setting, Ran introduces a \emph{tautological bundle} $\Lambda_m(\scr{V})$ defined as follows. Let $X_B^{[m]} = \on{Hilb}_m(X/B)$ denote the relative Hilbert scheme parametrizing length-$m$ subschemes of the fibers of the map $\pi \colon X \to B$, and let $\pi_1, \pi_2$ be the projection maps from $X_B^{[m]} \times_B X$ onto the left and right factors, respectively. Then put
$$\Lambda_m(\scr{V}) = {\pi_1}_*\big(\pi_2^* \scr{V} \otimes (\scr{O}_{X_B^{[m]} \times_B X}/\scr{I}_m)\big),$$
where $\scr{I}_m$ is the universal ideal sheaf of colength $m$ in $\scr{O}_{X_B^{[m]} \times_B X}$. Since we are not interested in all length-$m$ subschemes of the fibers, but only in those subschemes that are supported at a single point, consider the pullback $\Lambda_m(\scr{V})|_{\Gamma_{(m)}}$ of the tautological bundle to the punctual Hilbert scheme $\Gamma_{(m)}$ parametrizing length-$m$ schemes supported at individual points of the fibers.

The Chern classes of these tautological bundles can be computed and applied to solve certain enumerative problems, like the problem of counting $m^{\mathrm{th}}$-order weight-$2$ type-(a) inflection points associated to a linear system $(\scr{L}, \scr{E})$ on an admissible $1$-parameter family of curves. For this particular problem, the class of the desired inflectionary \mbox{locus is}
$$c_2(\Lambda_m(\scr{L})|_{\Gamma_{(m)}}).$$
The computation of the above Chern class is rather involved because the punctual Hilbert scheme $\Gamma_{(m)}$ is typically singular. Thus, Ran works not over $\Gamma_{(m)}$ itself but over the aforementioned punctual flag Hilbert scheme, which turns out to be an iterated blowup of $\Gamma_{(m)}$. After much computation, Ran arrives at the following elegant result.

\begin{theorem}[{\cite[Example 3.21]{MR3078931}}]\label{thm-ran}
Let $X/B$ be an admissible family with each singular fiber nodal, let $\scr{L}$ be a line bundle on $X$, and let the number of singular fibers of $X/B$ be denoted by $\delta$. Then we have that
\begin{align*}
c_2(\Lambda_m(\scr{L})|_{\Gamma_{(m)}}) & = {{m}\choose{2}}\cdot c_1(\scr{L})^2 + \left(3 \cdot {{m+1}\choose{4}} - {{m}\choose{3}}\right) \cdot c_1(\omega_{X/B})^2 + \\
& \hphantom{====} \left(3 \cdot {{m+1}\choose{3}}-2\cdot {{m}\choose{2}}\right) \cdot c_1(\omega_{X/B}) \cdot c_1(\scr{L}) - {{m+1}\choose{4}} \cdot \delta.
\end{align*}
\end{theorem}

\begin{remark}
We make the following observations:
\begin{enumerate}
\item Let the family $X/B$ be a pencil of plane curves of degree $d$, let $m = 4$, and let $\scr{L} = \scr{O}_X(1)$. Then it is not hard to show that Theorem~\ref{thm-ran} gives the formula for the number of hyperflexes in~\eqref{eq-hyposform}; see \S~\ref{sec-hadtorefer} for the proof.
\item It may be possible to generalize Ran's strategy to families of curves acquiring higher-order singularities. Indeed, based on ideas introduced by Ran in~\cite{MR2172162}, H.~Lee has found a description of the punctual Hilbert scheme of length-$m$ schemes supported at a cusp~\cite{MR2922387}. It would certainly be interesting if analogues of the tautological module on families of nodal curves and the consequent enumerative formula can be derived for families of cuspidal curves using Lee's results. In any case, the strategy that we introduce in the following section can be used to handle families acquiring arbitrary plane curve singularities; see \S~\ref{sec-getsworse} for more details.
\item For more examples of how to use the Hilbert scheme of nodal curves to solve interesting enumerative problems on such curves, refer to~\cite{MR2202264} and~\cite{MR2157135}.
\end{enumerate}
\end{remark}

\subsubsection{Strategy III: Using Automatic Degeneracy} \label{sec-hadtorefer} We now apply Theorem~\ref{thm-main2}, our result on automatic degeneracy in the nodal case, to determine the number of hyperflexes in a general pencil of plane curves of a given degree.

Let $X/B$ be an admissible $1$-parameter family, and let $\on{Sing}(X/B) \subset X$ be the locus of singular points of the fibers of the family. Note that $\on{Sing}(X/B)$ is a finite collection, and suppose for each $p \in \on{Sing}(X/B)$ that the analytic-local function cutting out the singularity at $p$ is given by $f_p \in R$. Let $(\scr{L},\scr{E})$ be a linear system on the family. Assuming that the locus of $m^{\mathrm{th}}$-order weight-$2$ type-(a) inflection points of the linear system $(\scr{L},\scr{E})$ is reduced, the number of such points is equal to the degree of the Chern class
\begin{equation} \label{eq-chernref}
c_2(\scr{P}_{X/B}^m(\scr{L})^{\vee\vee})
\end{equation}
minus the support of this class at each of the singular points. As long as the linear system $(\scr{L},\scr{E})$ is general in the sense of Remark~\ref{rem-generalassump}, the support of the Chern class~\eqref{eq-chernref} at the singular point $p \in \on{Sing}(X/B)$ is given by $\on{AD}_{(2)}^m(f_p)$.
If every one of the singular points is nodal and if $\delta = \#(\on{Sing}(X/B))$, then Theorem~\ref{thm-main2} tells us that the total support at the singular points is given by
\begin{equation} \label{eq-totalsupport}
\sum_{p \in \on{Sing}(X/B)} \on{AD}_{(2)}^m(f_p) = {{m+1}\choose {4}} \cdot \delta.
\end{equation}
Subtracting the right-hand side of~\eqref{eq-totalsupport} from the degree of the Chern class~\eqref{eq-chernref}, which we computed in Proposition~\ref{prop-chernconquest}, yields that
\begin{align}
& \deg c_2(\scr{P}_{X/B}^m(\scr{L})^{\vee\vee}) - \sum_{p \in \on{Sing}(X/B)} \on{AD}_{(2)}^m(f_p) = \nonumber \\
& {{m}\choose {2}}\cdot \deg c_1(\scr{L})^2 + \left(3\cdot {{m+1}\choose {4}} - {{m}\choose {3}}\right) \cdot \deg c_1(\omega_{X/B})^2 + \label{eq-theformula} \\
& \hphantom{==} \left(3\cdot {{m+1}\choose {3}}-2 \cdot {{m}\choose {2}}\right) \cdot \deg (c_1(\omega_{X/B}) \cdot c_1(\scr{L})) - {{m+1}\choose {4}} \cdot \delta, \nonumber
\end{align}
which is precisely the formula obtained by Ran in Theorem~\ref{thm-ran}. The fact that we have recovered Ran's formula here suggests that the linear system $(\scr{L},\scr{E})$ is indeed general in the sense of Remark~\ref{rem-generalassump}, although it is not \emph{a priori} clear as to why this is so.

All that remains is to apply the formula to the case where $X/B$ is a pencil of plane curves of degree $d$, $\scr{L} = \scr{O}_X(1)$, and $m = 4$. We first need to provide an explicit construction of the pencil. By definition, the base is $B = \BP_k^1$. Now, suppose the pencil is generated by two homogeneous degree-$d$ polynomials $F,G \in k[X_0, X_1, X_2]$. Consider the rational map
$$\pi' \colon \BP_k^2 \DashedArrow[densely dashed    ] \BP_k^1, \quad [X_0 : X_1 : X_2] \mapsto [F(X_0, X_1, X_2) : G(X_0 : X_1 : X_2)].$$
The map $\pi'$ is not defined at the common vanishing locus $D$ of $F,G$; if the pencil is chosen to be sufficiently general, then $D$ is the union of $d^2$ reduced points. It follows that if we take $X = \on{Bl}_D \BP_k^2$ to be the blowup of $\BP_k^2$ along the locus $D$, then the rational map $\pi'$ defined above extends to a morphism $\pi \colon X \to B$. Furthermore, it is clear that the fiber of $\pi$ above a point $[s : t] \in B = \BP_k^1$ is just the vanishing locus in $\BP_k^2$ of the homogeneous degree-$d$ polynomial $t \cdot F - s \cdot G$, so the fibers of the family $X/B$ are precisely the curves that constitute the pencil generated by $F,G$.

Now that we have explicitly constructed the pencil, we need to compute the degrees of the Chern classes $c_1(\scr{L})^2$, $c_1(\omega_{X/B})^2$, and  $c_1(\scr{L})c_1(\omega_{X/B})$. Let $\phi \colon X \to \BP_k^2$ be the map embedding the fibers of the family in the plane (which is simply given by the blowdown map $\on{Bl}_D \BP_k^2 \to \BP_k^2$). By the compatibility of Chern classes with pullbacks, we know that
$$c_1(\scr{O}_X(1)) = \phi^* c_1(\scr{O}_{\BP_k^2}(1)) = \phi^*\zeta,$$
where $\zeta$ is the hyperplane class in $\BP_k^2$. It follows that
$$\deg c_1(\scr{O}_X(1))^2 = \deg (\phi^* \zeta)^2.$$
By~\cite[part~(b)~of~Proposition~2.19]{harris3264}, we have that $(\phi^* \zeta)^2 = \phi^*(\zeta^2)$. Since $\zeta^2$ is the class of a point, it follows that $\deg (\phi^*\zeta)^2 = 1$, so
$$\deg c_1(\scr{O}_X(1))^2 = 1.$$
To compute $c_1(\omega_{X/B})$,  recall from~\cite[p.~84]{MR1631825} that $\omega_{X/B} = \Omega_X^1 \otimes \pi^* (\Omega_B^1)^\vee$, so
$$c_1(\omega_{X/B}) = K_X - \pi^* K_B.$$
Since $K_{\BP_k^2} = -3 \zeta$, we have that $K_X = -3 \cdot( \phi^* \zeta) + E$, where $E$ is the class of the exceptional locus. Also, $K_B$ is $-2$ times the class of a point in the base, so its pullback $\pi^* K_B$ is $-2$ times the class of a curve in the family, which is $d \cdot (\phi^* \zeta) - E$, so $\pi^* K_B = -2d \cdot (\phi^* \zeta) + 2E$. It follows that
$$c_1(\omega_{X/B}) = (-3 \cdot (\phi^* \zeta) + E) - (-2d \cdot (\phi^* \zeta) + 2E) = (2d-3) \cdot (\phi^* \zeta) - E.$$
Since a general line in $\BP_k^2$ fails to meet the locus $D$ that we have blown up, it follows that $(\phi^* \zeta) \cdot E = 0$. Also, by~\cite[part~(d)~of~Proposition~2.19]{harris3264}, we know that $\deg E^2 = -d^2$. Combining the above results, we deduce that
$$\deg c_1(\omega_{X/B})^2 = (2d-3)^2 - d^2 = 3d^2 - 12d + 9 \quad \text{and} \quad  \deg c_1(\scr{L})c_1(\omega_{X/B}) = 2d-3.$$
Substituting these results in to the formula~\eqref{eq-theformula} and using the fact that $\delta = 3(d-1)^2$ (see~\cite[Proposition 7.4]{harris3264} for the proof), we find that
\begin{align}
& \deg c_2(\scr{P}_{X/B}^m(\scr{L})^{\vee\vee}) - \sum_{p \in \on{Sing}(X/B)} \on{AD}_{(2)}^m(f_p) = \nonumber \\
& \frac{1}{4} m(m-1)(12 + 2 d(d-5) -
   16 m + md(17 - 3 d) + m^2(d-1)(d-4)). \label{eq-laterapply}
\end{align}
Finally, substituting $m = 4$ into~\eqref{eq-laterapply} yields the formula in~\eqref{eq-hyposform}.

\subsection{Counting Septactic Points in a Pencil of Plane Curves} \label{sec-hyperflexagain}

We now apply the formula~\eqref{eq-theformula} to the problem of counting septactic points in a pencil of plane curves.\footnote{To our knowledge, this problem has not been studied in the literature.}

  Let $C \subset \BP_k^2$ be a smooth plane curve of degree $d \geq 3$. A smooth point $p \in C$ is said to be a \emph{septactic point} if the osculating conic of $C$ at $p$ is smooth and has intersection multiplicity at least $7$ with $C$ at $p$. It is easy to see that $p$ is septactic if and only if it is an inflection point with ramification sequence $(0,0,0,0,0,2)$ of the $5$-dimensional linear system $(\scr{L}, W)$ of plane conics, which is defined by taking $\scr{L} = \scr{O}_C(2)$ and $W$ to be the $k$-vector subspace of $H^0(\scr{O}_C(2))$ generated by the pullbacks of global sections of $\scr{O}_{\BP_k^2}(2)$.

  Now let $X/B$ be a general pencil of plane curves of degree $d$.
  Assuming that the linear system of plane conics is general in the sense of Remark~\ref{rem-generalassump}, the total multiplicity of the locus of septactic points is given by~\eqref{eq-theformula}, where we take $\scr{L} = \scr{O}_X(2)$ and $m = 7$. Repeating the calculation of \S~\ref{sec-hadtorefer} with this choice of $\scr{L}$ and $m$ yields
  \begin{equation} \label{eq-septacticmult}
  \deg c_2(\scr{P}_{X/B}^7(\scr{O}_X(2))^{\vee\vee}) - {{m+1}\choose {4}} \cdot \delta = 21(d-3)(15d-11).
  \end{equation}

  Note that the formula on the right-hand side of~\eqref{eq-septacticmult} is merely a multiplicity and should \emph{not} be interpreted as giving the actual number of septactic points in the pencil $X/B$. Indeed, the osculating conic to a curve $C$ at a smooth point $p \in C$ fails to be smooth if and only if $p$ is a flex of $C$, in which case the osculating conic is given by the doubled tangent line to $C$ at $p$. Moreover, a flex is a septactic point if and only if it is a hyperflex: at a flex, the osculating conic meets the curve with intersection multiplicity at least $6$, whereas at a hyperflex, the osculating conic meets the curve with intersection multiplicity at least $8$. Thus, the locus of septactic points contains the locus of hyperflexes as a closed subscheme. It is therefore natural to ask the following question: with what multiplicity does each hyperflex contribute to the formula on the right-hand side of~\eqref{eq-septacticmult}?

  To answer this question, we perform a local calculation in an analytic neighborhood of a hyperflex on a general $2$-parameter family of smooth plane curves. More precisely, let $\wt{B} \coloneqq \on{Spf}k[[s,t]]$, let $\wt{X} \coloneqq \on{Spf}k[[s,t]][[x]]$, and let $\pi \colon \wt{X} \to \wt{B}$ be the obvious map. The family $\wt{X}/\wt{B}$ can be thought of as the pullback of a $2$-parameter family of curves to an analytic-local neighborhood of a smooth point (namely, the point $p_0$ defined by $s = t = x = 0$). Now, observe that if a smooth point $p$ on a plane curve is a hyperflex, then the sequence of orders of vanishing of linear forms at the point $p$ is given by $(0,1,4)$. It follows that the sequence of orders of vanishing of quadratic forms at the point $p$ is given by $(0,1,2,4,5,8)$. Let $\ol{\sigma} = (\sigma_1, \sigma_2, \sigma_3, \sigma_4, \sigma_5, \sigma_6)$ be a list of general elements of $\scr{O}_{\wt{X}}$ that vanish at the point $p_0$ to orders $0, 1, 2, 4, 5, 8$, respectively. The elements of the list $\ol{\sigma}$ can be thought of as germs at $p_0$ of the quadratic forms that vanish to orders $0,1,2,4,5,8$, but to carry out the present computation, we need to assume that these germs are general.

      The module $\scr{P}_{\wt{X}/\wt{B}}^7$ of principal parts is given by~\eqref{eq-defmodparts} to be
  $$\scr{P}_{\wt{X}/\wt{B}}^7 = (\scr{O}_{\wt{X}} \otimes_{\scr{O}_{\wt{B}}} \scr{O}_{\wt{X}})/I_\Delta^7 = k[[s,t,x]][[u]]/(u-x)^7.$$
  A basis of $\scr{P}_{\wt{X}/\wt{B}}^7$ is given by the list $(e_0, \dots, e_6)$ of elements defined by $e_i((u-x)^j) = \delta_{ij}$ for all $i,j \in \{0, \dots, 6\}$. Because the family $\wt{X}/\wt{B}$ is smooth, $\scr{P}_{\wt{X}/\wt{B}}^7$ is a free $k[[s,t,x]]$-module. Now, for each $j$, let $\tau_j \in \scr{P}_{\wt{X}/\wt{B}}^7$ be defined by taking $\sigma_j$ and replacing the variable $x$ with the variable $u$. By the definition of the $\sigma_j$, we can express the elements $\tau_j$ as follows:
  \begin{align}
      \tau_j = \sum_{\ell = 0}^6 \alpha_{j\ell} \cdot (u - x)^\ell, \label{eq-septtaus}
      \end{align}
      where the coefficients $\alpha_{j\ell} \in k[[s,t,x]]$ can be further expressed as follows:
      \begin{itemize}
          \item We have that $\alpha_{20} = \beta_{20} \cdot x + \gamma_{20}$, $\alpha_{3\ell} = \beta_{3\ell} \cdot x^{2-\ell} + \gamma_{3\ell}$ for each $\ell \in \{0,1\}$, $\alpha_{4\ell} = \beta_{4\ell} \cdot x^{4-\ell} + \gamma_{4\ell}$ for each $\ell \in \{0, 1, 2, 3\}$, $\alpha_{5\ell} = \beta_{5\ell} \cdot x^{5-\ell} + \gamma_{5\ell}$ for each $\ell \in \{0, 1, 2, 3, 4\}$, and $\alpha_{6\ell} = \beta_{6\ell} \cdot x^{8-\ell} + \gamma_{6\ell}$ for each $\ell \in \{0,1,2,3,4,5,6\}$, where the coefficients $\beta_{j\ell}$ are units in $k[[s,t,x]]$ and the coefficients $\gamma_{j\ell}$ are contained in the ideal $(s,t) \subset k[[s,t,x]]$; and
      \item all other coefficients $\alpha_{j\ell}$ are units in $k[[s,t,x]]$.
      \end{itemize}
    Given the above setup, the expected multiplicity with which a hyperflex contributes to the formula on the right-hand side of~\eqref{eq-septacticmult} is given by the multiplicity at $p_0$ of the degeneracy locus of the list of elements $\ol{\tau} = (\tau_1, \tau_2, \tau_3, \tau_4, \tau_5, \tau_6)$. Thus, if we denote by $e_i^\vee \in (\scr{P}_{\wt{X}/\wt{B}}^7)^\vee$ the dual functional of $e_i$ for each $i$, the list $(e_0^\vee, \dots, e_6^\vee)$ forms a basis of $(\scr{P}_{\wt{X}/\wt{B}}^7)^\vee$. The degeneracy locus of the list of elements $\ol{\tau}$ is cut out by the ideal $I_{\ol{\tau}}^7 \subset k[[s,t,x]]$ of maximal minors of the $7 \times 6$ matrix $M_{\ol{\tau}}^7$ whose row-$i$, column-$j$ entry is given by $e_i^\vee(\tau_j)$.
  Substituting in the expressions for the elements $\tau_j$ given in~\eqref{eq-septtaus}, we have that
  $$M_{\ol{\tau}}^7 = \left[\begin{array}{cccccc} \alpha_{10} & \beta_{20} \cdot x + \gamma_{20} & \beta_{30} \cdot x^2 + \gamma_{30} & \beta_{40} \cdot x^4 + \gamma_{40} & \beta_{50} \cdot x^5 + \gamma_{50} & \beta_{60} \cdot x^8 + \gamma_{60} \\ \alpha_{11} & \alpha_{21} & \beta_{31} \cdot x + \gamma_{31}  & \beta_{41} \cdot x^3 + \gamma_{41} & \beta_{51} \cdot x^4 + \gamma_{51} & \beta_{61} \cdot x^7 + \gamma_{61} \\ \alpha_{12} & \alpha_{22} & \alpha_{32} & \beta_{42} \cdot x^2 + \gamma_{42} & \beta_{52} \cdot x^3 + \gamma_{52} & \beta_{62} \cdot x^6 + \gamma_{62} \\ \alpha_{13} & \alpha_{23} & \alpha_{33} & \beta_{43} \cdot x + \gamma_{43} & \beta_{53} \cdot x^2 + \gamma_{53} & \beta_{63} \cdot x^5 + \gamma_{63} \\ \alpha_{14} & \alpha_{24} & \alpha_{34} & \alpha_{44} &  \beta_{54} \cdot x + \gamma_{54} & \beta_{64} \cdot x^4 + \gamma_{64} \\ \alpha_{15} & \alpha_{25} & \alpha_{35} & \alpha_{45} & \alpha_{55} & \beta_{65} \cdot x^3 + \gamma_{65} \\ \alpha_{16} & \alpha_{26} & \alpha_{36} & \alpha_{46} & \alpha_{56} & \beta_{66} \cdot x^2 + \gamma_{66} \end{array} \right]$$
  Upon computing the minors of the above matrix and simplifying the ideal they generate using a computer algebra system (e.g., {\tt Macaulay2}), one finds that $$I_{\ol{\tau}}^7 = (s, t, x^2) \subset k[[s,t,x]],$$ if the list of elements $\ol{\sigma}$ is general. (To simplify this computation, it suffices to consider only the three minors of $M_{\ol{\tau}}^7$ obtained by deleting each of the first three rows. Each of these three minors can be expressed as a linear combination of $s, t, x^2$ with unit coefficients, and so it is possible to express each of $s, t, x^2$ as a linear combination of these minors.) Thus, the desired multiplicity is $\dim_k k[[s,t,x]]/I_{\ol{\tau}}^7 = 2$, so hyperflexes contribute to the right-hand side of~\eqref{eq-septacticmult} with multiplicity $2$.\footnote{A similar argument can be used to prove that in a $1$-parameter family of curves, hyperflexes are points of multiplicity $2$ in the divisor of flexes, and flexes are points of multiplicity $1$ in the divisor of \emph{sextactic points}---points at which the osculating conic meets the curve with intersection multiplicity at least $6$. For more on the enumerative geometry of sextactic points, refer to~\cite[\S~341]{MR2866760} and~\cite{maugestenmoesextactic}.} Since the number of hyperflexes in the pencil $X/B$ is given by~\eqref{eq-hyposform} to be $6(d-3)(3d-2)$, we deduce the following result:
  \begin{theorem} \label{thm-septactic}
  The expected number of septactic points --- excluding hyperflexes --- in a general pencil $X/\BP_k^1$ of plane curves of degree $d \geq 3$ is given by
    $$21(d-3)(15d-11)-2 \cdot 6(d-3)(3d-2) = 9(d-3)(31d-23).\footnote{We say ``expected'' because of the generality assumptions that we made in the proof.}$$
  \end{theorem}

\subsection{Calculating Classes of Weierstrass Divisors}

In this section, we apply our results on automatic degeneracy to compute the divisors of curves that possess weight-$2$ Weierstrass points of any given degree in a partial compactification of the moduli space $\mc{M}_g$ \mbox{of curves of genus $g$.}

\subsubsection{Preliminary Calculations} \label{sec-weierbegin}

 Let $C$ be a smooth curve, and let $n \geq 1$ be an integer. Consider the linear system $(\scr{L}, W)$ on $C$ given by taking $\scr{L} = \omega_C^{\otimes n}$ and $W = H^0(\scr{L})$. The inflection points of this linear system are called \emph{degree-$n$ Weierstrass points} of the curve $C$.

 Now, let $\pi \colon X \to B$ be an admissible $1$-parameter family of irreducible curves of genus $g$. The degree-$n$ Weierstrass points of the fibers of the family are given by the inflection points of the linear system $(\omega_{X/B}^{\otimes n}, \pi_* \omega_{X/B}^{\otimes n})$.

 In this preliminary section, we compute the class of the locus of weight-$2$ degree-$n$ Weierstrass points of type (a) (resp., type (b)) in $X$, which is obtained by excising the support at singular points from the class of the degeneracy locus of the composite map
\begin{equation} \label{eq-compmap}
\pi^* (\pi_* \omega_{X/B}^{\otimes n}) \to \scr{P}_{X/B}^{m}(\omega_{X/B}^{\otimes n}) \to \scr{P}_{X/B}^{m}(\omega_{X/B}^{\otimes n})^{\vee\vee},
\end{equation}
 where $m = \on{rk} \pi_* \omega_{X/B}^{\otimes n}+1$ (resp., $m = \on{rk} \pi_* \omega_{X/B}^{\otimes n}-1$). By the Riemann-Roch Theorem, we have that
 \begin{equation} \label{eq-listdems}
     \on{rk} \pi_* \omega_{X/B}^{\otimes n} = \begin{cases} g & \text{if $n = 1$} \\
     (2n-1)(g-1) & \text{if $n > 1$ and $g > 1$} \\
     0 & \text{if $n > 1$ and $g = 0$} \\
     1 & \text{if $n > 1$ and $g = 1$}\end{cases}
 \end{equation}
 It is easy to see from~\eqref{eq-listdems} that smooth curves of genus $0$ and $1$ do not have Weierstrass points of any degree, so in what follows, we restrict to the case $g \geq 2$.

 Now, the Porteous formula (see~\cite[Theorem 12.4]{harris3264}) tells us that the class of the degeneracy locus of the composite map~\eqref{eq-compmap} is given by
\begin{align}
& c_2(\pi^* (\pi_* \omega_{X/B}^{\otimes n}) \to \scr{P}_{X/B}^{m}(\omega_{X/B}^{\otimes n})^{\vee\vee}) =  \nonumber \\
& \quad \begin{cases} c_2([\pi^* (\pi_* \omega_{X/B}^{\otimes n})]^\vee \otimes \scr{P}_{X/B}^{m}(\omega_{X/B}^{\otimes n})^{\vee\vee}) & \text{for type (a)} \\ c_1([\pi^* (\pi_* \omega_{X/B}^{\otimes n})]^\vee \otimes \scr{P}_{X/B}^{m}(\omega_{X/B}^{\otimes n})^{\vee\vee})^2 - c_2([\pi^* (\pi_* \omega_{X/B}^{\otimes n})]^\vee \otimes \scr{P}_{X/B}^{m}(\omega_{X/B}^{\otimes n})^{\vee\vee}) & \text{for type (b)} \end{cases} \label{eq-portchern}
\end{align}
Here, we have used the assumption that the fibers of the family $X/B$ are irreducible, for if one of the singular fibers is reducible, the composite map in~\eqref{eq-compmap} may degenerate along an entire irreducible component of the fiber, causing the degeneracy locus to fail to have the codimension required for the Porteous formula to apply (cf.~Remark~\ref{rem-theend}).

We next compute the Chern classes that appear in~\eqref{eq-portchern}. By our computation of the Chern classes of sheaves of invincible parts in Proposition~\ref{prop-chernconquest}, we have that
\begin{align}
& c_2([\pi^* (\pi_* \omega_{X/B}^{\otimes n})]^\vee \otimes \scr{P}_{X/B}^{m}(\omega_{X/B}^{\otimes n})^{\vee\vee}) = \nonumber \\
& c_2(\scr{P}_{X/B}^{m}(\omega_{X/B}^{\otimes n})^{\vee\vee}) + c_2(\pi^* (\pi_* \omega_{X/B}^{\otimes n})) - c_1(\scr{P}_{X/B}^{m}(\omega_{X/B}^{\otimes n})^{\vee\vee})c_1(\pi^* (\pi_* \omega_{X/B}^{\otimes n})) = \nonumber \\
& \left[{{m}\choose 2} +  \left(3 \cdot {{m+1}\choose {4}} - {{m}\choose{3}}\right) + \left(3 \cdot {{m+1}\choose{3}}-2\cdot {{m}\choose{2}}\right)\right]\cdot c_1(\omega_{X/B}^{\otimes n})^2  - \nonumber \\
    & \qquad \left[m + {{m}\choose{2}} \right]\cdot c_1(\omega_{X/B}^{\otimes n}) \cdot c_1(\pi^* (\pi_* \omega_{X/B}^{\otimes n})) = \nonumber \\
    & \frac{1}{24}m(m-1)(m+1)(3m+2)\cdot c_1(\omega_{X/B}^{\otimes n})^2 - \frac{1}{2}m(m+1)\cdot  c_1(\omega_{X/B}^{\otimes n}) \cdot c_1(\pi^* (\pi_* \omega_{X/B}^{\otimes n})). \label{eq-cherna}
\end{align}
(Note that the term $c_2(\pi^* (\pi_* \omega_{X/B}^{\otimes n}))$ vanishes because Chern classes commute with pullbacks and $\dim B = 1$.) We also have that
\begin{align}
  & c_1([\pi^* (\pi_* \omega_{X/B}^{\otimes n})]^\vee \otimes \scr{P}_{X/B}^{m}(\omega_{X/B}^{\otimes n})^{\vee\vee})^2  = \nonumber \\ &   c_1(\scr{P}_{X/B}^{m}(\omega_{X/B}^{\otimes n})^{\vee\vee})^2 -  2 \cdot c_1(\scr{P}_{X/B}^{m}(\omega_{X/B}^{\otimes n})^{\vee\vee}) \cdot c_1(\pi^* (\pi_* \omega_{X/B}^{\otimes n})) + c_1(\pi^* (\pi_* \omega_{X/B}^{\otimes n}))^2 = \nonumber \\
  & \left(\frac{1}{2}m(m+1)\right)^2 \cdot c_1(\omega_{X/B}^{\otimes n})^2- m(m+1) \cdot  c_1(\omega_{X/B}^{\otimes n}) \cdot c_1(\pi^* (\pi_* \omega_{X/B}^{\otimes n})) + c_1(\pi^* (\pi_* \omega_{X/B}^{\otimes n}))^2. \label{eq-chernb}
\end{align}
Combining the results in~\eqref{eq-portchern}--\eqref{eq-chernb} yields that the class of the degeneracy locus of the composite map~\eqref{eq-compmap} is given by
\begin{align}
  &   - \frac{1}{2}m(m+1)\cdot  c_1(\omega_{X/B}^{\otimes n}) \cdot c_1(\pi^* (\pi_* \omega_{X/B}^{\otimes n}))+ \label{eq-finalcherncalc} \\
& \qquad \begin{cases}\frac{1}{24}m(m-1)(m+1)(3m+2)\cdot c_1(\omega_{X/B}^{\otimes n})^2 & \text{for type (a)} \\
    \frac{1}{24}m(m+1)(m+2)(3m+1)\cdot c_1(\omega_{X/B}^{\otimes n})^2  + c_1(\pi^* (\pi_* \omega_{X/B}^{\otimes n}))^2 & \text{for type (b)}
    \end{cases} \nonumber
\end{align}
Assuming that the linear system $(\omega_{X/B}^{\otimes n}, \pi_* \omega_{X/B}^{\otimes n})$ is general in the sense of Remark~\ref{rem-generalassump}, the class of the locus of weight-$2$ degree-$n$ Weierstrass points of type (a) (resp., type (b)) is then given by subtracting the $m^{\mathrm{th}}$-order automatic degeneracy of type (a) (resp., type (b)) at each of the singular points of the family $X/B$ from the class in~\eqref{eq-finalcherncalc}.

\subsubsection{Notation for Intersection Theory on $\ol{\mc{M}}_g$} \label{sec-bitback} To apply the calculations of \S~\ref{sec-weierbegin} in the context of moduli of curves, we first need to introduce some notation. For $g \geq 2$ an integer, let $\ol{\mc{M}}_g$ and $\ol{\mc{C}}_g$ respectively denote the coarse moduli spaces of stable and pointed stable curves of genus $g$. Let $\pi \colon \ol{\mc{C}}_g \to \ol{\mc{M}}_g$ be the natural map realizing $\ol{\mc{C}}_g$ as the universal family over $\ol{\mc{M}}_g$. 

In his seminal paper~\cite{MR717614}, D.~Mumford establishes a notion of intersection theory on $\ol{\mc{M}}_g$ and $\ol{\mc{C}}_g$.
In particular, Mumford identifies a number of important divisor classes in $\on{Pic}(\ol{\mc{M}}_g) \otimes \BQ$ and $\on{Pic}(\ol{\mc{C}}_g) \otimes \BQ$, some of which are defined as follows:
\begin{itemize}
  \item $K = c_1(\omega_{\ol{\mc{C}}_g/\ol{\mc{M}}_g}) \in \on{Pic}(\ol{\mc{C}}_g) \otimes \BQ$ called the \emph{canonical class};
  \item $\kappa = \pi_* (K^2) \in \on{Pic}(\ol{\mc{M}}_g) \otimes \BQ$ called the \emph{tautological class}; and
  \item $\lambda = c_1(\pi_* \omega_{\ol{\mc{C}}_g/\ol{\mc{M}}_g}) \in \on{Pic}(\ol{\mc{M}}_g) \otimes \BQ$ called the \emph{Hodge class}.
  \item $\delta_i \in \on{Pic}(\ol{\mc{M}}_g) \otimes \BQ$ for $i \in \{0, \dots, \lfloor g/2 \rfloor\}$ are classes corresponding to loci of nodal curves; $\delta_0$ is the class of irreducible curves with a node, and $\delta_i$ for $i > 0$ is the class of curves with two irreducible components, one of genus $i$ and the other of genus $g - i$, that intersect in a node.
\end{itemize}
The classes $(\lambda, \delta_0, \dots, \delta_{\lfloor g/2 \rfloor})$ form a basis of $\on{Pic}(\ol{\mc{M}}_g) \otimes \BQ$. The class $\kappa$ is expressed in terms of this basis via the \emph{Mumford relation}, which states that
\begin{equation} \label{eq-mumford}
\kappa = 12\cdot \lambda - \sum_i \delta_i,
\end{equation}
Also, it follows from the push-pull formula that we have the relation
\begin{equation} \label{eq-pushpullform}
    \pi_*(K \cdot \pi^* \lambda) = (2g-2) \cdot \lambda.
\end{equation}
In addition, as explained in~\cite[(7.5)]{MR1104328}, it follows from the Grothendieck-Riemann-Roch Theorem that
\begin{equation}
    c_1\big(\pi_* \omega_{\ol{\mc{C}}_g/\ol{\mc{M}}_g}^{\otimes n}\big) = (6n^2 - 6n+1) \cdot \lambda - {n \choose 2}\cdot \sum_i \delta_i
\end{equation}
Finally, from~\cite[Proposition 5.1]{MR1016424}, we have that $\pi_*(\pi^* \delta_i \cdot \pi^* \delta_j)$ is linearly independent from $\lambda$ and $\delta_0$ for any $i,j \in \{0, \dots, \lfloor g/2\rfloor\}$, and we also have the following relations:
\begin{equation} \label{eq-cukrels}
    \pi_*((\pi^*\lambda)^2) = \pi_*(\pi^* \delta_0 \cdot \pi^* \lambda) = \pi_*((\pi^* \delta_0)^2) = 0.
\end{equation}
In what follows, we shall use the classes and relations~\eqref{eq-mumford}--\eqref{eq-cukrels} to describe divisors of curves with weight-$2$ degree-$n$ Weierstrass points in $\ol{\mc{M}}_g$.

\subsubsection{The Weierstrass Divisors}

In \S~\ref{sec-weierbegin}, we took an admissible $1$-parameter family $X/B$ and computed the classes in $X$ of the divisors of weight-$2$ Weierstrass points of the fibers of the family. The pushforwards of these classes to the base $B$ are called \emph{Weierstrass divisors} and are of particular interest in the study of moduli of curves. By making the replacements $$X \rightsquigarrow \ol{\mc{C}}_g \quad \text{and} \quad B \rightsquigarrow \ol{\mc{M}}_g,$$ we find that the locus of curves with a weight-$2$ degree-$n$ Weierstrass point of type (a) forms a divisor on $\ol{\mc{M}}_g$, and so does the locus of curves with a weight-$2$ degree-$n$ Weierstrass point of type (b). Let these divisors be denoted $\ol{\mc{W}}_{(2)}^{(n)}$ and $\ol{\mc{W}}_{(1,1)}^{(n)}$, respectively.

We now use our results on automatic degeneracy and the calculations of \S~\ref{sec-weierbegin} to obtain a partial understanding of the classes of $\ol{\mc{W}}_{(2)}^{(n)}$ and $\ol{\mc{W}}_{(1,1)}^{(n)}$. Making the replacements $$c_1(\omega_{X/B}^{\otimes n}) \rightsquigarrow n \cdot K \quad \text{and} \quad  c_1(\pi_* \omega_{X/B}^{\otimes n}) \rightsquigarrow (6n^2 - 6n+1) \cdot \lambda - {n \choose 2}\cdot \sum_i \delta_i$$
in the formula~\eqref{eq-finalcherncalc}, we obtain the following codimension-$2$ classes on $\ol{\mc{C}}_g$:
\begin{align}
  &   - \frac{1}{2}m(m+1)\cdot n \cdot K \cdot \left( (6n^2 - 6n+1) \cdot \pi^* \lambda - {n \choose 2}\cdot \sum_i \pi^* \delta_i\right) + \label{eq-finalcherncalc2} \\
  & \qquad \begin{cases}\frac{1}{24}m(m-1)(m+1)(3m+2)\cdot n^2 \cdot K^2 & \text{for type (a)} \\
    \frac{1}{24}m(m+1)(m+2)(3m+1)\cdot n^2 \cdot K^2 + \left( (6n^2 - 6n+1) \cdot \pi^* \lambda - {n \choose 2}\cdot \sum_i \pi^* \delta_i\right)^2  & \text{for type (b)}
    \end{cases} \nonumber
\end{align}
Pushing the classes in~\eqref{eq-finalcherncalc2} forward along the map $\pi$ yields the following classes in $\on{Pic}(\ol{\mc{M}}_g)$:
\begin{align}
 &   - \frac{1}{2}m(m+1)\cdot n \cdot \left( (6n^2 - 6n+1) \cdot (2g-2) \cdot \lambda - {n \choose 2}\cdot \sum_i  \delta_i\right) + \label{eq-finalcherncalc3} \\
  &  \begin{cases}\frac{1}{24}m(m-1)(m+1)(3m+2)\cdot n^2 \cdot \left(12 \cdot \lambda - \sum_i \delta_i\right) & \text{for type (a)} \\
    \frac{1}{24}m(m+1)(m+2)(3m+1)\cdot n^2 \cdot \left(12 \cdot \lambda - \sum_i \delta_i\right)
    + \wt{\delta}
    & \text{for type (b)}
    \end{cases} \nonumber
\end{align}
where the class $\wt{\delta} \in \on{Pic}(\ol{M}_g)$ is given by
$$\wt{\delta} = \pi_*\left(\left( (6n^2 - 6n+1) \cdot \pi^* \lambda - {n \choose 2}\cdot \sum_i \pi^* \delta_i\right)^2 \right)$$
Note that it follows from the linear independence of $\pi_*(\pi^* \delta_i \cdot \pi^* \delta_j)$ from $\lambda$ and $\delta_0$ and from the relations in~\eqref{eq-cukrels} that the coefficients of $\lambda$ and $\delta_0$ in $\wt{\delta}$ are both equal to $0$.

We now need to subtract out the contribution to the divisor class in~\eqref{eq-finalcherncalc3} coming from the singular points. Recall that the analysis in \S~\ref{sec-weierbegin} requires the assumption that the curves in the family are irreducible. Thus, among the divisor classes $\delta_i \in \on{Pic}(\ol{\mc{M}}_g)$ corresponding to nodal curves, our method can only compute the coefficient of the class $\delta_0$. By Theorem~\ref{thm-main2}, the contribution coming from nodes of irreducible fibers is $$\on{AD}_{(2)}^m(xy) \cdot \delta_0 = {{m+1} \choose 4} \cdot \delta_0 \quad \text{or} \quad \on{AD}_{(1,1)}^m(xy) \cdot \delta_0 = {{m+2} \choose 4} \cdot \delta_0,$$
according as we are interested in type-(a) case or the type-(b) case. Thus, the $\lambda$ and $\delta_0$ terms of the classes of $\ol{\mc{W}}_{(2)}^{(n)}$ and $\ol{\mc{W}}_{(1,1)}^{(n)}$ in $\on{Pic}(\ol{\mc{M}}_g)$ are given as follows:
\begin{align}
  &   - \frac{1}{2}m(m+1)\cdot n \cdot \left( (6n^2 - 6n+1) \cdot (2g-2) \cdot \lambda - {n \choose 2} \cdot \delta_0\right) + \label{eq-finalcherncalc4} \\
  & \qquad \begin{cases}\frac{1}{24}m(m-1)(m+1)(3m+2)\cdot n^2 \cdot \left(12 \cdot \lambda - \delta_0\right) - {{m+1} \choose 4} \cdot \delta_0 & \text{for type (a)} \\
    \frac{1}{24}m(m+1)(m+2)(3m+1)\cdot n^2 \cdot \left(12 \cdot \lambda - \delta_0\right) - {{m+2} \choose 4} \cdot \delta_0
    & \text{for type (b)}
    \end{cases} \nonumber
\end{align}
Substituting in $m = \on{rk} \pi_* \omega_{X/B}^{\otimes n} \pm 1$ using the value of $\on{rk} \pi_* \omega_{X/B}^{\otimes n}$ given in~\eqref{eq-listdems} yields:

\begin{theorem} \label{thm-weight02}
We have the following expected results:\footnote{We say ``expected'' because of the generality assumption made at the end of \S~\ref{sec-weierbegin}.}
\begin{enumerate}
\item The $\lambda$ and $\delta_0$ terms of the class of $\ol{\mc{W}}_{(2)}^{(1)}$ in $\on{Pic}(\ol{\mc{M}}_g)$ are given by
 $$\frac{1}{2} (g + 1) (g + 2) (3g^2 + 3g + 2) \cdot \lambda -\frac{1}{6} g (g + 1)^2 (g + 2) \cdot \delta_0$$
 \item The $\lambda$ and $\delta_0$ terms of the class of $\ol{\mc{W}}_{(1,1)}^{(1)}$ in $\on{Pic}(\ol{\mc{M}}_g)$ are given by
  $$ \frac{1}{2} g^2(g-1) (3g-1) \cdot \lambda -\frac{1}{6}g(g-1)^2 (g+1) \cdot \delta_0$$
\item For $n > 1$, the $\lambda$ and $\delta_0$ terms of the class of $\ol{\mc{W}}_{(2)}^{(n)}$ in $\on{Pic}(\ol{\mc{M}}_g)$ are given by
$$\left(\sum_{i=0}^4 w_{(2)}^{(n)}(i,\lambda) \cdot g^i\right) \cdot \lambda + \left(\sum_{i=0}^4 w_{(2)}^{(n)}(i,\delta_0) \cdot g^i\right)\cdot \delta_0,$$
where the coefficients $w_{(2)}^{(n)}(i,\lambda)$ and $w_{(2)}^{(n)}(i,\delta_0)$ are given by
\begin{align*}
    w_{(2)}^{(n)}(0,\lambda) & = 24n^6-80n^5+78n^4-6n^3-22n^2+6n,\\
    w_{(2)}^{(n)}(1,\lambda) & = -96n^6+288n^5-264n^4+46n^3+37n^2-11n,\\
    w_{(2)}^{(n)}(2,\lambda) & =144n^6-384n^5+330n^4-84n^3-\frac{27}{2}n^2+6n, \\
    w_{(2)}^{(n)}(3,\lambda) & = -96n^6+224n^5-180n^4+56n^3-3n^2-n, \\
    w_{(2)}^{(n)}(4,\lambda) & =24n^6-48n^5+36n^4-12n^3+\frac{3}{2}n^2, \\
    w_{(2)}^{(n)}(0,\delta_0) & = -2n^6+\frac{29}{5}n^5-\frac{53}{3}n^4+\frac{89}{6}n^3-\frac{16}{3}n^2+\frac{1}{2}n,\\
    w_{(2)}^{(n)}(1,\delta_0) & = 8n^6-32n^5+\frac{289}{6}n^4-\frac{413}{12}n^3+\frac{37}{3}n^2-\frac{7}{3}n+\frac{1}{4},\\
    w_{(2)}^{(n)}(2,\delta_0) & = -12n^6+39n^5-\frac{97}{2}n^4+\frac{125}{4}n^3-\frac{299}{24}n^2+\frac{10}{3}n-\frac{11}{24},\\
    w_{(2)}^{(n)}(3,\delta_0) & = 8n^6-\frac{62}{3}n^5+\frac{65}{3}n^4-\frac{27}{2}n^3+\frac{73}{12}n^2-\frac{11}{6}n+\frac{1}{4},\\
    w_{(2)}^{(n)}(4,\delta_0) & = -2n^6+4n^5-\frac{11}{3}n^4+\frac{7}{3}n^3-\frac{9}{8}n^2+\frac{1}{3}n-\frac{1}{24}.
\end{align*}
 \item For $n > 1$, the $\lambda$ and $\delta_0$ terms of the class of $\ol{\mc{W}}_{(1,1)}^{(1)}$ in $\on{Pic}(\ol{\mc{M}}_g)$ are given by
 $$\left(\sum_{i=0}^4 w_{(1,1)}^{(n)}(i,\lambda) \cdot g^i\right) \cdot \lambda + \left(\sum_{i=0}^4 w_{(1,1)}^{(n)}(i,\delta_0) \cdot g^i\right)\cdot \delta_0,$$
 where the coefficients $w_{(1,1)}^{(n)}(i,\lambda)$ and $w_{(1,1)}^{(n)}(i,\delta_0)$ are given by
 \begin{align*}
     w_{(1,1)}^{(n)}(0,\lambda) & = 24n^6-16n^5-18n^4+14n^3-2n^2,\\
     w_{(1,1)}^{(n)}(1,\lambda) & = -96n^6+96n^5+24n^4-46n^3+13n^2-n,\\
     w_{(1,1)}^{(n)}(2,\lambda) & = 144n^6-192n^5+42n^4+36n^3-\frac{35}{2}n^2+2n,\\
     w_{(1,1)}^{(n)}(3,\lambda) & = -96n^6+160n^5-84n^4+8n^3+5n^2-n, \\
     w_{(1,1)}^{(n)}(4,\lambda) & = 24n^6-48n^5+36n^4-12n^3+\frac{3}{2}n^2, \\
     w_{(1,1)}^{(n)}(0,\delta_0) & = -2n^6+\frac{13}{3}n^5-\frac{11}{3}n^4+\frac{4}{3}n^3-\frac{1}{6}n^2-\frac{1}{6}n,\\
     w_{(1,1)}^{(n)}(1,\delta_0) & = 8n^6-16n^5+\frac{85}{6}n^4-\frac{27}{4}n^3+n^2+\frac{1}{3}n-\frac{1}{12},\\
     w_{(1,1)}^{(n)}(2,\delta_0) & = -12n^6+23n^5-\frac{41}{2}n^4+\frac{45}{4}n^3-\frac{83}{24}n^2+\frac{1}{3}n+\frac{1}{24},\\
     w_{(1,1)}^{(n)}(3,\delta_0) & = 8n^6-\frac{46}{3}n^5+\frac{41}{3}n^4-\frac{49}{6}n^3+\frac{41}{12}n^2-\frac{5}{6}n+\frac{1}{12},\\
     w_{(1,1)}^{(n)}(4,\delta_0) & = -2n^6+4n^5-\frac{11}{3}n^4+\frac{7}{3}n^3-\frac{9}{8}n^2+\frac{n}{3}-\frac{1}{24}.
 \end{align*}

 \end{enumerate}
\end{theorem}
\begin{remark} \label{rem-weierstuff}
   There are numerous results in the literature concerning Weierstrass points and singular curves; we briefly summarize a selection of these results as follows.

   \begin{itemize}
   \item In~\cite[Theorem A2.1]{MR791679}, S.~Diaz showed that the number of weight-$1$ degree-$1$ Weierstrass points limiting to the node of a uninodal irreducible curve of genus $g$ in a general $1$-parameter deformation is equal to $g(g-1)$. Note that this result also follows from Theorem~\ref{thm-main2} by substituting $m = \on{rk} \pi_* \omega_{X/B} = g$ into the formula $\on{AD}_{(1)}^m(xy) = m(m-1)$. Also, in~\cite[Example 4.9]{gattoricolfi}), Gatto and Ricolfi use the notion of Widland-Lax multiplicity to show that the number of weight-$1$ degree-$1$ Weierstrass points limiting to the cusp of a unicuspidal curve of genus $g$ in a general $1$-parameter deformation is equal to $g^2-1$.
   \item In~\cite{MR2431731}, C.~Cumino, Esteves, and Gatto use the Eisenbud-Harris theory of limit linear series to prove, roughly speaking, that in a family of smooth curves degenerating to a reducible nodal curve, the number of weight-$2$ inflection points limiting toward the node is equal to zero. Note that this result agrees with that obtained by applying~\eqref{Eq:FundamentalFormula} or~\eqref{eq-fund2} to the case of a node. Further study of the problem of limits of Weierstrass points toward reducible nodal curves is conducted in~\cite{MR2320659}.
   \item The computation of the class of the divisor $\ol{\mc{W}}_{(1,1)}^{(1)}$ was first performed by Diaz and is the primary objective of his article~\cite{MR791679}. Subsequently, in~\cite{MR1016424}, F.~Cukierman used an argument  involving the Riemann-Hurwitz formula to deduce the divisor $[\ol{\mc{W}}_{(2)}^{(1)}]$ from Diaz's computation of $\ol{\mc{W}}_{(1,1)}^{(1)}$. Also, in~\cite{MR3399393}, Esteves computed the class of the locus of hyperelliptic curves in $\ol{\mc{M}}_3$ by means of a technique analogous to that which we used to prove Theorem~\ref{thm-weight02}---namely, using locally free replacements of the sheaves of principal parts. A detailed expository treatment of the results of Cukierman, Diaz, and Esteves on Weierstrass divisors is provided in~\cite[\S~2.2, \S~5]{gattoricolfi}.
   \item The only result on divisors of higher-degree Weierstrass points that we are aware of can be found in~\cite{MR1104328}, where Cukierman and L.-Y.~Fong compute the divisor of weight-$1$ degree-$n$ Weierstrass points in $\ol{\mc{C}}_g$.
   \end{itemize}
\end{remark}

\begin{remark} \label{rem-theend}
   In~\cite[Example~14.4.7]{MR1644323}, W.~Fulton states and proves a version of the Porteous formula for maps that have ``excess degeneracy.'' In~\cite{MR2944484}, T.~Bleier applies Fulton's version of the Porteous formula to compute the $\delta_1$ term of the divisor of hyperelliptic curves in $\ol{\mc{M}}_3$. It remains open as to whether Fulton's version of the Porteous formula can be used to compute the Weierstrass divisors described in Theorem~\ref{thm-weight02}.
\end{remark}

\section{Summary of Open Problems} \label{sec-openended}

We now conclude by briefly listing the numerous open problems relating to inflection points of singular curves that arose throughout the course of the paper.
\begin{itemize}
    \item Given an ICIS, a positive integer $m$, and an integer $n \geq 3$, how many $m^{\mathrm{th}}$-order weight-$n$ inflection points limit toward the ICIS in a general $n$-parameter deformation? Can the notions of automatic degeneracy introduced above be generalized to the higher weight case?
    \item Given an ICIS germ $f$, is it possible to find a closed-form expression for $\on{AD}_{(1)}^m(f)$, $\on{AD}_{(2)}^m(f)$, or $\on{AD}_{(1,1)}^m(f)$ as a function of $m$? Which of these automatic degeneracies are eventually polynomial in $m$? (Note that we can handle the case when $f$ cuts out a node as well as the weight-$2$ type-(a) case when $f$ cuts out a cusp.)
    \item Can the work of Widland and Lax be generalized to the weight-$2$ case? If so, do the results agree with our results on automatic degeneracy, as they did in the weight-$1$ case?
    \item Can the weight-$2$ type-(a) automatic degeneracy be expressed as a Buchsbaum-Rim multiplicity in an ``intrinsic'' way in terms of the local principal parts module as was possible for the weight-$2$ type-(b) case? More broadly, what other connections can be made between automatic degeneracies and well-studied invariants and multiplicities?
    \item Given an ICIS germ $f$, can an explicit basis be obtained for the module $\on{SP}^m(f)^\vee$, as we managed to do in the nodal case? In particular, is it possible to obtain an explicit formula for the basis produced by the algorithm in \S~\ref{sec-algs}, or is it at least possible to implement this algorithm in a computer program?
    \item Can formulas for fixed-order automatic degeneracies be obtained for families of singularities other than those considered in \S~\ref{sec-calc}?
    \item Which automatic degeneracies are equisingularity invariants?
    \item Is $\on{AD}_{(1,1)}^2(f) = \on{mult}_0 \Delta_f$ for arbitrary ICIS germs $f$? (We have shown this for planar singularities and for one non-planar example.)
    \item Can Fulton's excess Porteous formula be used to compute Weierstrass divisors?
\end{itemize}
	
\section*{Acknowledgments}

\noindent This paper is based in part on the second author's senior thesis (see~\cite{swaminathansenior}) at Harvard College. For much of this research, the second author was supported by the Paul and Daisy Soros Fellowship. We thank Joe Harris for suggesting the questions that led to this paper and for providing invaluable advice and encouragement throughout our research. We are grateful to Letterio Gatto and Andrea Ricolfi for taking the time to understand our work and for engaging in several enlightening discussions on the subject of inflection points with us. We thank James Damon, Steve Kleiman, and Vijay Kodiyalam for answering our questions and pointing us in interesting new directions. We are grateful to Hailong Dao for suggesting the use of the fiber cone in the proof of Theorem~\ref{thm-tess}. We thank Michael DiPasquale for suggesting the idea for Strategy I in \S~\ref{sec-probcorrect} and implementing it in {\tt Macaulay2}. We thank the anonymous referee for providing us with numerous useful comments and suggestions. Also, many thanks are due to Sabin Cautis, Daniel Grayson, Aaron Landesman, Jacob Levinson, Saahil Mehta, Ziv Ran, Mike Stillman, and James Tao for helpful conversations. We used {\tt Macaulay2}, {\tt Mathematica}, and {\tt Singular} for explicit calculations.

	\bibliographystyle{alpha}
	\bibliography{bibfile}
	
\end{document}